\documentclass[a4paper, 12pt, dvipsnames, reqno]{amsart}

\usepackage[normalem]{ulem} 
\usepackage{lmodern}
\usepackage[english]{babel}
\usepackage[latin1]{inputenc}
\usepackage{aligned-overset} 
\usepackage{microtype}
\usepackage{a4wide}
\usepackage{keytheorems}
\usepackage{bbold}
\usepackage{xcolor, graphicx}
\usepackage[inline]{enumitem}
\usepackage[all]{xy}
\SelectTips{eu}{12}
\usepackage{quiver}
\usepackage{tikz}
\usepackage{tikz-cd}
\usepackage{scalerel}
\usepackage{versions}
\usepackage{zref-clever}
\zcsetup{capfirst,nameinlink=false}
\usepackage[bookmarks=false,colorlinks,urlcolor=cyan,citecolor=Plum,linkcolor=Blue,hypertexnames=false]{hyperref} 
\usepackage{orcidlink}

\makeatletter
\patchcmd{\@setaddresses}{\indent}{\noindent}{}{}
\patchcmd{\@setaddresses}{\indent}{\noindent}{}{}
\patchcmd{\@setaddresses}{\indent}{\noindent}{}{}
\patchcmd{\@setaddresses}{\indent}{\noindent}{}{}
\makeatother

\makeatletter
\DeclareMathSizes{12}{12}{9}{6}
\makeatother
\makeatletter
\def\namedlabel#1#2{\begingroup
    #2%
    \def\@currentlabel{#2}%
    \phantomsection\label{#1}\endgroup
}
\makeatother

\theoremstyle{plain}

\newtheorem{theorem}{Theorem}[section]
\newkeytheorem{proposition}[
    name=Proposition,
    sibling=theorem]
\newkeytheorem{lemma}[
    name=Lemma,
    sibling=theorem]
\newkeytheorem{corollary}[
    name=Corollary,
    sibling=theorem]
\newtheorem*{theorem*}{Theorem}

\theoremstyle{definition}

\newkeytheorem{definition}[
    name=Definition,
    sibling=theorem,
    refname={definition,definitions},
    Refname={Definition,Definitions},
    qed={\(\lozenge\)}]
\newkeytheorem{example}[
    name=Example,
    sibling=theorem,
    refname={example,examples},
    Refname={Example,Examples},
    qed={\(\circ\)}]
\newtheorem*{notation*}{Notation}

\theoremstyle{remark}

\newkeytheorem{remark}[
    name=Remark,
    sibling=theorem,
    refname={remark,remarks},
    Refname={Remark,Remarks},
    qed={\(\triangle\)}]

\newenvironment{invisible}[1][\unskip]{
	\noindent
	\color{gray}
	[{\textbf{\color{blue}TBH}: \textit{#1}}
}{]}

\newenvironment{rmv}[1][\unskip]{
	\noindent
	\color{gray}
	[{\textbf{\color{blue}Remove}: \textit{#1}}
}{]}

\DeclareMathOperator{\ot}{\mathbin{\mathpalette\myothelper\relax}}

\newcommand{\myothelper}[2]{%
\mathop{\raisebox{1pt}{\scaleobj{0.7}{#1\otimes}}}%
}

\newcommand{\rH}{\mathrm{rH}} 

\newcommand{\Hom}{{\sf Hom}}
\newcommand{\Comon}{\mathsf{Comon}}
\newcommand{\Mon}{\mathsf{Mon}}
\newcommand{\Bimon}{\mathsf{Bimon}}
\newcommand{\Cc}{\mathcal{C}}
\newcommand{\Vect}{\mathsf{Vec}}

\newcommand{\ev}{{\sf ev}}
\newcommand{\db}{{\sf coev}}

\newcommand{\id}{\mathsf{id}}
\newcommand{\Id}{\mathsf{id}}

\def\rHopf{{\sf rHopf}}

\newcommand{\N}{\mathbb{N}}
\newcommand{\Z}{\mathbb{Z}}

\newcommand{\C}{\mathbb{C}}
\newcommand{\K}{\Bbbk}

\newcommand{\cB}{{\mathcal B}}
\newcommand{\cC}{{\mathcal C}}
\newcommand{\cD}{{\mathcal D}}

\newcommand{\cF}{{\mathcal F}}
\newcommand{\cG}{{\mathcal G}}

\newcommand{\cL}{{\mathcal L}}
\newcommand{\cM}{{\mathcal M}}
\newcommand{\cN}{{\mathcal N}}

\newcommand{\cR}{{\mathcal R}}

\newcommand{\mf}[1]{\mathfrak{#1}}
\newcommand{\I}{\mathbb{1}} 
\newcommand{\calpha}{\mf{a}}
\newcommand{\clambda}{\mf{l}}
\newcommand{\crho}{\mf{r}}
\newcommand{\can}{\mathsf{can}}
\newcommand{\rmod}[1]{{#1}_{\raisebox{+2pt}{\(\scriptscriptstyle\bullet\)}}^{}} 

\newcommand{\hopfmod}[1]{{#1}_{\raisebox{+2pt}{\(\scriptscriptstyle\bullet\)}}^{\raisebox{+1pt}{\(\scriptscriptstyle\bullet\)}}} 
\newcommand{\coinv}[2]{{#1}{}^{\mathrm{co}{#2}}} 
\newcommand{\inv}[2]{{\overline{#1}{}^{#2}}}
\newcommand{\brd}{\mf{c}}

\excludeversion{invisible}

\allowdisplaybreaks

\title[Frobenius functors and one-sided Hopf monoids]{Frobenius functors and one-sided Hopf algebras in braided monoidal categories}

\author[L.\@ Bottegoni]{Lucrezia Bottegoni\,\orcidlink{0009-0005-8172-5605}}
\email{lucrezia.bottegoni@edu.unito.it}

\author[D.\@ Ferri]{Davide Ferri\,\orcidlink{0000-0001-8421-496X}}
\address{%
		Dipartimento di Matematica `G.~Peano', Universit\`a degli Studi di Torino, via
			Carlo Alberto 10, 10123 Torino, Italy.\newline
			Vrije Universiteit Brussel, Department of Mathematics and Data Science, Pleinlaan 2, 1050 Brussels, Belgium.}
\email{d.ferri@unito.it, Davide.Ferri@vub.be}
\urladdr{\url{https://sites.google.com/view/davide-ferri}}

\author[P.\@ Saracco]{Paolo Saracco\,\orcidlink{0000-0001-5693-7722}}
\address{Dipartimento di Matematica `G.~Peano', Universit\`a degli Studi di Torino, via Carlo Alberto 10, 10123 Torino, Italy.}
\email{p.saracco@unito.it}
\urladdr{\url{https://sites.google.com/view/paolo-saracco}}

\date{\today}


\begin{document}

\begin{abstract}
  Evidence suggests a tight connection between the existence of antipode-like maps on bialgebra-like structures, and the Frobenius property for the associated free Hopf module functor. In this paper, we prove that in any braided monoidal category satisfying few mild assumptions, a bimonoid is a one-sided Hopf monoid and the antipode is a bimonoid anti-homomorphism, if and only if the free Hopf module functor is Frobenius.

  Our interest in these structures stems from having found genuine examples of one-sided Hopf monoids in contexts where the braiding is non-trivial.  In fact, we provide a construction of the free one-sided Hopf monoid over a comonoid in any symmetric monoidal category with some non-restrictive additional assumptions. We present several examples of this construction, and describe the resulting one-sided (sometimes two-sided) Hopf monoids.
\end{abstract}

\maketitle

\tableofcontents

\section*{Introduction}

One-sided Hopf algebras showed up naturally in the algebraic landscape in Green, Nichols, and Taft \cite{GreenNicholsTaft}, springing from Nichols's effort \cite{Nichols-Quotients} to compute a basis for the free Hopf algebra \(\mathrm{H}\left(\mathrm{M}_n(\Bbbk)^*\right)\) generated by the matrix coalgebra in the sense of Takeuchi \cite{Takeuchi-Free}. By relying on Bergman's Diamond Lemma \cite{Bergman-Diamond}, in \cite[\S2]{Nichols-Quotients} Nichols computed a set of reduction formulas that suffices to resolve all the ambiguities arising in the construction of \(\mathrm{H}\left(\mathrm{M}_n(\Bbbk)^*\right)\) as a quotient of \(\Bbbk\langle X_{i,j}^{r} \mid 1 \leq i,j \leq n, r \geq 0 \rangle\), thus being able to express a basis for \(\mathrm{H}\left(\mathrm{M}_n(\Bbbk)^*\right)\) as the collection of the irreducible words. The relations imposed on \(\Bbbk\langle X_{i,j}^{r} \mid 1 \leq i,j \leq n, r \geq 0 \rangle\) are those that guarantee that \(X_{i,j}^{r} \mapsto X_{j,i}^{r+1}\) is an antipode. Dropping a few of these relations, namely \(\sum_{k=1}^rX_{i,k}^{0}X_{j,k}^{1} - \delta_{i,j}1\) for \(1 \leq i,j \leq n\), still produces a meaningful bialgebra, enhanced with an anti-morphism of bialgebras which is just a left and not a right convolution inverse of the identity (see \cite{GreenNicholsTaft}). 

One-sided Hopf algebras have immediately attracted the attention of the community. In particular, Taft and his collaborators have constructed several new examples: their work led to new insights \cite{BeggsTaft,IyerSmithTaft,NicholsTaft}, novel Hopf algebra structures \cite{IyerTaft,RodriguezTaft-quantumlike,RodriguezTaft-onesided}, and genuine examples of one-sided Hopf algebras \cite{GreenNicholsTaft,LauveTaft,RodriguezTaft-leftquantumgrp}.
Recently, Saracco \cite{Saracco-Frobenius} showed that, for a bialgebra \(B\), the existence of right antipodes which are bialgebra anti-homomorphisms corresponds to the Frobenius property for the free Hopf module functor \(- \ot B \colon \Vect \to \mathrm{HopfMod}^B_B\). 
Even more unexpectedly, Berger, Saracco and Vercruysse \cite{BergerSaraccoVercruysse} observed that one-sided Hopf algebras are natural examples of gabi-algebras: algebras whose category of modules is closed with respect to the closed structure of vector spaces (i.e., the internal hom is given by the linear hom). An extension of this notion, namely one-sided \(n\)-Hopf algebras, plays a central role in simplifying the construction of the Hopf envelope of a bialgebra \cite{ArMeSa}, and one-sided Hopf algebras whose one-sided antipode is a bialgebra anti-homomorphism admit a simpler description of their space of integrals \cite{ArMeSa-integrals}.

Our interest in the topic stems from the following observation, emerged while seeking new interesting and natural examples of one-sided Hopf algebras: 
\(\mathrm{M}_2(\K)\) is a super-coalgebra (a comonoid in the symmetric monoidal category of \(\Z_2\)-graded vector spaces) and the construction of Green, Nichols and Taft of free one-sided Hopf algebras in \(\Vect\) can be adapted to the super-setting to construct a one-sided super-Hopf algebra \(\rH\left(\mathrm{M}_2(\K)\right)\) (see \zcref{ex:MatcomodZ2}). 

Motivated by this example, we extend the construction of the free one-sided Hopf algebra to any symmetric monoidal category satisfying mild assumptions (\zcref{sect:GNTsym-monoidal}). In \zcref{sec:examples} we apply our construction to symmetric monoidal categories \(\cM\) which are not \(\Vect\), and whose braiding does not forget onto the flip braiding on \(\Vect\). As demonstrated in \cite{GreenNicholsTaft}, the free construction in \(\Vect\) may give rise to genuine one-sided Hopf algebras, as well as to two-sided Hopf algebras. In our more general setting, we observe a similar behaviour: the construction may yield a two-sided Hopf monoid in \(\cM\) (\zcref{ex:Z2inZ2}) as well as a genuine one-sided Hopf monoid (\zcref{ex:MatcomodZ2}).

Realising that one-sided Hopf monoids arise just as naturally in the braided monoidal setting, led us to wonder whether they can be characterised as in the linear case. Namely, we wonder whether the Hopf-Frobenius connection, advocated in \cite{SaraccoAntipodes} and \cite{Saracco-Frobenius}, carries over into the braided setting, reinforcing the close relationship between Hopf and Frobenius properties. In the final section of the paper, we answer this question in the positive. 
Even when \((\cM,\ot,\I,\brd)\) is a braided monoidal category, the free right Hopf module functor \(- \ot B \colon \cM \to \cM^B_B\) fits into an adjoint triple 
\[- \ot_B \I \;\dashv\; - \ot B \;\dashv\; \coinv{(-)}{B}.\]
We devote \zcref{sec:HopfFrobenius} to prove the following characterisation of one-sided Hopf monoids:

\begin{theorem*}
   Let \((\cM,\ot,\I,\brd)\) be a braided monoidal category which admits the coequalisers of reflexive pairs and the equalisers of coreflexive pairs. Let \((B,m, u, \Delta,\varepsilon)\) be a bialgebra in \(\cM\) for which the endofunctor \(B \ot -\) preserves reflexive coequalisers.
   Then, the following assertions are equivalent:
   \begin{itemize}[label=\(\diamond\),leftmargin=0.5cm]
       \item \(B\) is a right Hopf algebra in \(\cM\) whose right antipode \(S\) is anti-multiplicative and anti-comultiplicative; 
       \item the free right Hopf module functor \(- \ot B \colon \cM \to \cM^B_B\) is Frobenius;
       \item the canonical morphism \(\sigma_M \colon \coinv{M}{B} \to M \ot_B \I\) is a split-mono for every Hopf module \(M\);
       \item the canonical morphism 
       \[i_B \coloneqq  \left(B \xrightarrow{\eta_B} \coinv{(\rmod{B} \ot \hopfmod{B})}{B} \xrightarrow{\sigma_{B \ot B}} \left(\rmod{B} \ot \hopfmod{B}\right)\ot_B \I\right)\] 
       is an isomorphism.
   \end{itemize}
\end{theorem*}

\begin{notation*} 
    We will denote the identity morphism of an object \(X\) in a category \(\cC\) either by \(\id_X\) or by simply \(X\). For categories \(\mathcal{C}\) and \(\mathcal{D}\), a functor \(\mathcal{F}:\mathcal{C}\to \mathcal{D}\) just means a covariant functor. In notations such as \({}_B\mathfrak{M}\), we use \(\mathfrak{M}\) as a synonym of the category \(\Vect\) of vector spaces. 
\end{notation*}

\section{Preliminaries and first results}\label{sec:prelim}

This section starts presenting preliminary notions and setting the notation. In \zcref{subsect:onesided}, we introduce one-sided Hopf monoids in a braided monoidal category  (\zcref{def:one-sided}). With the exception of \zcref{prop:one_sided_ant_can}, the content of this section is mostly known.

\subsection{Adjoint triples}\label{sec:adjtriples}

Let us recall quickly some facts about adjoint triples and Frobenius functors that we are going to use in the paper. For further details on these objects in connection with our setting, see for example \cite[\S1]{Saracco-Frobenius}. Given categories \(\cC\) and \(\cD\), we say that functors \(\cL,\cR\colon \cC \rightarrow \cD \), \(\cF\colon \cD \rightarrow \cC\) form an \emph{adjoint triple} if \(\cL\) is left adjoint to \(\cF\) which is left adjoint to \(\cR\), i.e.\@ \(\cL\dashv \cF\dashv \cR\). They form an \emph{ambidextrous adjunction} if there is a natural isomorphism \(\cL\cong \cR\). As a matter of notation, we set \(\eta \colon \id \rightarrow \cF\cL,\epsilon \colon \cL\cF \rightarrow \id \) for the unit and counit of the left-most adjunction and \(\gamma \colon \id \rightarrow \cR\cF\), \(\theta \colon \cF\cR \rightarrow \id \) for the right-most one. If in addition \(\cF\) is fully faithful, that is, if \(\epsilon\) and \(\gamma\) are natural isomorphisms, then we have a distinguished natural transformation
\begin{equation}\label{def:sigma}
    \sigma \coloneqq \left( \xymatrix{\cR \ar[r]^-{\left( \epsilon \cR\right) ^{-1}} & \cL\cF\cR \ar[r]^-{\cL\theta} & \cL}\right).
\end{equation}
Recall also that a \emph{Frobenius pair} for the categories \(\cC\) and \(\cD\) is a couple of functors \(\cF\colon \cC \rightarrow \cD\) and \(\cG\colon \cD \rightarrow \cC\) such that \(\cG \) is left and right adjoint to \(\cF\). A functor \(\cF\colon \cC \rightarrow \cD\) is said to be \emph{Frobenius} if there exists a functor \(\cG \colon \cD \rightarrow \cC\) which is at the same time left and right adjoint to \(\cF\).

\begin{lemma}\cite[Lemma 2.3]{Saracco-Frobenius}\label{lemma:frobenius}
    If \(\cF\) is fully faithful, then \(\cF\) is Frobenius if and only if there exist \(\cL,\cR\colon \cD \rightarrow \cC\) such that  \(\cL\dashv \cF\dashv \cR\) is an adjoint triple and \(\sigma\colon \cR\to\cL\) is a natural isomorphism.
\end{lemma}

Since we are interested in adjoint triples whose middle functor is fully faithful, Lemma \ref{lemma:frobenius} allows us to study the Frobenius property by simply looking at the invertibility of the canonical map \(\sigma\). Observe that
\begin{equation}\label{eq:sigmaF}
    \sigma\cF = \epsilon^{-1}   \gamma^{-1} \qquad \text{and} \qquad \cF\sigma = \eta \theta,
\end{equation}
whence, in particular, \(\sigma\cF\) is always a natural isomorphism.

\begin{invisible}
\subsection{Monoidal categories}
    
    Recall that a \emph{monoidal category} \(\left( \cM ,\ot ,\I ,\calpha ,\clambda ,\crho \right) \) is a category \(\cM \) endowed with a functor \(\ot \colon \cM \times \cM \to \cM \) (the \emph{tensor product}), with a distinguished object \(\I \) (the \emph{unit}) and with three natural isomorphisms 
    \begin{align*}
    	\calpha & \colon \ot   (\ot \times \id_{\cM})\to \ot   (\id_{\cM }\times \ot) & &\text{(\emph{associativity constraint})} \\
    	\clambda & \colon \ot    (\I \times \id _{\cM })\to \id _{\cM } & & \text{(\emph{left unit constraint})} \\
    	\crho & \colon \ot    (\id _{\cM }\times \I )\to \id _{\cM } & & \text{(\emph{right unit constraint})} 
    \end{align*}
    that satisfy the \emph{Pentagon} and the \emph{Triangle Axioms}, that is,
    \begin{gather*}
    \calpha_{X,Y,Z\ot W}   \calpha_{X\ot Y,Z,W} = \left(X\ot \calpha_{Y,Z,W}\right)   \calpha_{X,Y\ot Z,W}   \left(\calpha_{X,Y,Z}\ot W\right), \\
    \left(X\ot \clambda_Y\right)   \calpha_{X,\I ,Y} = \crho_X\ot Y,
    \end{gather*}
    for all \(X,Y,Z,W\) objects in \(\cM\).
    
    If the endofunctor \(X\ot-\colon Y\mapsto X\ot Y\) (resp. \(-\ot X\colon Y\mapsto Y\ot X\)) has a right adjoint for every \(X\) in \(\cM\), then \(\cM\) is called a \emph{left-closed} (resp. \emph{right-closed}) monoidal category.
    
    Given two monoidal categories \(\left( \cM ,\ot ,\I ,\calpha,\clambda ,\crho \right) \) and \(\left( \cM ^{\prime },\ot^{\prime} ,\I^{\prime},\calpha ^{\prime },\clambda ^{\prime },\crho ^{\prime }\right) \), a \emph{quasi-monoidal functor} \((\cF,\varphi_0,\varphi)\) between \(\cM\) and \(\cM'\) is a functor \(\cF\colon \cM \to \cM ^{\prime }\) together with an isomorphism \(\varphi _{0}\colon \I^{\prime}\to \cF\left( \I \right) \) and a family of isomorphisms \(\varphi_{X,Y}\colon\cF \left( X\right) \ot^{\prime} \cF\left( Y\right) \to \cF\left( X\ot Y\right) \) for \(X,Y\) objects in \(\cM \), which are natural in both entrances. A quasi-monoidal functor \(\cF\) is said to be \emph{neutral} if 
    \begin{gather}\label{eq:neutral}
    \begin{gathered}
    \cF\left( \clambda_{ {X}} \right)   \varphi_{ {\I,X}}   \left(\varphi _{0}\ot^{\prime}\cF\left( X\right)\right) = \clambda ^{\prime }_{ {\cF(X)}}, \\
    \cF\left(\crho_{ {X}} \right)   \varphi_{ {X,\I}}   \left(\cF\left(X\right) \ot^{\prime}\varphi _{0} \right) = \crho ^{\prime }_{ {\cF(X)}},
    \end{gathered}
    \end{gather}
    and it is said to be \emph{strong monoidal} if, in addition,
    \begin{equation}\label{eq:monoidal}
    \varphi_{X,Y\ot Z}   \left(\cF(X)\ot^{\prime} \varphi_{Y,Z}\right)    \calpha^{\prime }_{\cF(X),\cF(Y),\cF(Z)} = \cF \left( \calpha_{X,Y,Z}\right)   \varphi_{X\ot Y,Z}   \left(\varphi_{X,Y} \ot^{\prime} \cF(Z)\right)
    \end{equation}
    for all \(X,Y,Z\) in \(\cM\). Furthermore, it is said to be \emph{strict} if \(\varphi_0\) and \(\varphi\) are the identities. A strong monoidal functor \((\cF,\varphi_0,\varphi)\) such that \(\cF\) is an equivalence of categories is called a \emph{monoidal equivalence}.
    
    If \(\cF\) comes together with a morphism \(\varphi_0\colon \I'\to\cF(\I)\) and a natural transformation \(\varphi_{X,Y}\colon \cF \left( X\right) \ot^{\prime} \cF\left( Y\right) \to \cF\left( X\ot Y\right) \) that are not necessarily invertible but that satisfy \eqref{eq:neutral} and \eqref{eq:monoidal} then it is called a \emph{monoidal functor} (also termed \emph{lax monoidal functor} in \cite{Aguiar}). If instead \(\cF\) comes together with a morphism \(\psi_0\colon \cF(\I)\to\I'\) and a natural transformation \(\psi_{X,Y}\colon \cF\left( X\ot Y\right) \to \cF \left( X\right) \ot^{\prime} \cF\left( Y\right)\) (not necessarily invertible) satisfying the analogues of \eqref{eq:neutral} and \eqref{eq:monoidal} then it is called an \emph{opmonoidal functor} (also termed \emph{colax monoidal functor} in \cite{Aguiar} and \emph{comonoidal functor} in \cite{BruguieresLackVirelizier, BruguieresVirelizier}).
    
    A natural transformation \(\gamma\) between monoidal functors \((\cF,\varphi_0,\varphi)\) and \((\cG,\phi_0,\phi)\) from a monoidal category \((\cM,\ot,\I,\calpha,\clambda,\crho)\) to \((\cM',\ot',\I',\calpha',\clambda',\crho')\) is said to be \emph{monoidal} if
    \begin{equation}
    \left(\gamma_X\ot'\gamma_Y\right)   \varphi_{X,Y} = \phi_{X,Y}   \gamma_{X\ot Y} \quad \text{and} \quad \gamma_\I  \varphi_0 = \phi_0.
    \end{equation} 
    Similarly, one defines \emph{opmonoidal natural transformations} between opmonoidal functors. An adjoint pair of monoidal functors is called a \emph{monoidal adjunction} (also termed a \emph{lax-lax adjunction} in \cite{Aguiar}) if the unit and the counit are monoidal natural transformations. Analogously, one defines \emph{opmonoidal adjunctions} (also termed \emph{colax-colax adjunctions} in \cite{Aguiar} and \emph{comonoidal adjunctions} in \cite{BruguieresLackVirelizier, BruguieresVirelizier}).
    
    Henceforth, we will often omit the constraints when referring to a monoidal category. 
\end{invisible}


\subsection{Bimonoids in braided monoidal categories}

Recall that a \emph{monoidal category} \(\left( \cM ,\ot ,\I ,\calpha ,\clambda ,\crho \right) \) is a category \(\cM \) endowed with a functor \(\ot \colon \cM \times \cM \to \cM \) (the \emph{tensor product}), with a distinguished object \(\I \) (the \emph{unit}) and with three natural isomorphisms 
\[\calpha \colon \ot   (\ot \times \id_{\cM})\to \ot   (\id_{\cM }\times \ot), \qquad \clambda \colon \ot    (\I \times \id _{\cM })\to \id _{\cM }, \qquad \crho \colon \ot    (\id _{\cM }\times \I )\to \id _{\cM },\]
that satisfy the \emph{pentagon} and the \emph{triangle axioms}, that is,
\begin{gather*}
\calpha_{X,Y,Z\ot W}   \calpha_{X\ot Y,Z,W} = \left(X\ot \calpha_{Y,Z,W}\right)   \calpha_{X,Y\ot Z,W}   \left(\calpha_{X,Y,Z}\ot W\right), \\
\left(X\ot \clambda_Y\right)   \calpha_{X,\I ,Y} = \crho_X\ot Y,
\end{gather*}
for all \(X,Y,Z,W\) objects in \(\cM\). In a monoidal category one can consider the categories \(\Mon(\cM)\) of monoids in \(\cM\) and \(\Comon(\cM)\) of comonoids in \(\cM\). Henceforth, we will often omit the constraints when referring to a monoidal category and we will almost always drop the associativity constraint from the computations. For the sake of consistence, instead, we will usually keep track of the unit constraints.

\medskip

Consider a monoidal category \((\cM, \ot, \I, \calpha, \clambda, \crho)\) and let \(\tau\colon\cM \times \cM \to \cM \times \cM\), \(\tau(X,Y)=(Y,X)\), be the flip functor. Recall from \cite{ JS} that a \emph{braiding} for \(\cM\) is a natural isomorphism \(\brd \colon \ot\to\ot\tau\), i.e. \(\brd =(\brd_{X,Y})_{X,Y\in\cM}\) where \(\brd_{X,Y} \colon X\ot Y\to Y\ot X\) are isomorphisms natural in \(X\) and \(Y\), satisfying the \emph{hexagon axiom}:
\begin{equation}\label{eq:hexagon}
\begin{gathered}
    \xymatrix @!0 @R=50pt @C=29pt {
        &X\ot (Y\ot Z)\ar[rrr]^-{ \brd_{X,Y\ot Z}}&&&(Y\ot Z)\ot X\ar[rd]^-{ \calpha_{Y,Z,X}} & \\ 
        (X\ot Y)\ot Z\ar[ru]^-{ \calpha_{X,Y,Z}} \ar[rd]_-{\brd_{X,Y}\ot \id_Z}&&&&& Y\ot (Z\ot X) \\
		&(Y\ot X)\ot Z\ar[rrr]_-{ \calpha_{Y,X,Z}}&&&Y\ot (X\ot Z)\ar[ru]_-{ \id_Y\ot\brd_{X,Z}} &
	}
\end{gathered} \quad \begin{gathered}
    \xymatrix @!0 @R=50pt @C=29pt {
        &(X\ot Y)\ot Z\ar[rrr]^-{ \brd_{X\ot Y, Z}}&&&Z\ot(X\ot Y)\ar[rd]^-{\calpha^{-1}_{Z,X,Y}}&
		\\ X\ot (Y\ot Z)\ar[rd]_-{\id_X\ot \brd_{Y,Z}} \ar[ru]^-{ \calpha^{-1}_{X,Y,Z}}&&&&& (Z\ot X)\ot Y\\
		&X\ot (Z\ot Y)\ar[rrr]_-{ \calpha^{-1}_{X,Z,Y}}&&& (X\ot Z)\ot Y\ar[ru]_-{ \brd_{X,Z}\ot \id_Y} &}
\end{gathered}
\end{equation}
A \emph{braided monoidal} category \((\cM,\ot ,\I ,\calpha,\clambda, \crho, \brd )\) is a monoidal category with a braiding \(\brd\). Let us recall that in a braided monoidal category we always have that \(\brd_{\I,X} = \brd_{X,\I}^{-1}\) and that
\begin{equation}\label{eq:braidI}
\begin{gathered}
    \xymatrix{
    X \ot \I \ar[dr]_-{\crho_X} \ar[rr]^-{\brd_{X,\I}} & & \I \ot X \ar[dl]^-{\clambda_X} \\
     & X & 
    }
\end{gathered}
\end{equation}
commutes for all \(X\) in \(\cM\); see \cite[\S8.1]{EGNO}. A braiding $\brd$ is called a \emph{symmetry} if $\brd^2 = \brd$. In such a case, $(\cM,\ot,\I,\brd)$ is called a \emph{symmetric monoidal category}.

\medskip

Let \((\cM,\ot ,\I ,\calpha,  \clambda, \crho, \brd )\) be a braided monoidal category. A \emph{bimonoid in \(\cM\)} \cite[pages 192-193]{Ma93} is a datum \((B,m,u,\Delta,\varepsilon)\) where \((B,m,u)\in \Mon(\cM)\) and \((B,\Delta,\varepsilon)\in\Comon(\cM )\), and the two structures are compatible in the sense that the diagrams
\begin{equation}\label{eq:bialgebra-compatibility}
    \begin{gathered}
        \xymatrixcolsep{1.85cm}\xymatrix{B\ot B\ar[d]_-{m}\ar[r]^-{\Delta\ot \Delta}& B\ot B\ot B\ot B\ar[r]^-{\id\ot\brd_{B,B}\ot\id}&B\ot B\ot B\ot B\ar[d]^-{m\ot m}\\ B\ar[rr]_-{\Delta}&& B\ot B } \\
        \xymatrixcolsep{1cm}\xymatrix{\I\ar[r]^-{u}\ar[d]_-{\clambda_\I^{-1}=\crho_{\I}^{-1}}& B\ar[d]^-{\Delta}& B\ot B\ar[d]_-{m}\ar[r]^-{\varepsilon\ot\varepsilon }&\I\ot \I\ar[d]^-{\clambda_\I=\crho_\I}&\I\ar[rd]_-{\id}\ar[r]^-{u}&B\ar[d]^-{\varepsilon} \\ \I\ot \I\ar[r]_-{u\ot u}& B\ot B &B\ar[r]_-{\varepsilon}&\I &&\I}
    \end{gathered}
\end{equation}
commute. A \emph{morphism of bimonoids in \(\cM\)} is a morphism of monoids and comonoids in \(\cM\). Bimonoids in \(\cM\) and their morphisms form a category denoted by \(\Bimon(\cM )\). 

When we need to specify that \(m,u,\Delta,\varepsilon\) are the structure morphisms of \(B\), then we add a subscript \(m_B,u_B,\Delta_B,\varepsilon_B\).

\medskip

If we are given a bimonoid \((B, m, u, \Delta, \varepsilon)\) in a braided monoidal category \((\cM, \ot, \I, \calpha, \clambda, \crho)\), then we can consider modules, comodules and Hopf modules over \(B\). To be consistent with \cite{Saracco-Frobenius}, we will consider them on the \emph{right}, but the reader can easily adapt our arguments to treat the left-hand side case. A \emph{module} is an object \(V\) of \(\cM\) together with an action \(\mu_V \colon V \ot B \to V\) which is associative and unital. A \emph{comodule} is an object \(N\) of \(\cM\) together with a coaction \(\delta_N \colon N \to N \ot B\) which is coassociative and counital. A \emph{Hopf module} is an object \(M\) of \(\cM\) which admits a module structure \((M,\mu_M)\) and a comodule structure \((M,\delta_M)\) that are compatible, in the sense that
\begin{equation}\label{eq:HopfMod}
    (\mu_M \ot B)(M \ot B \ot m)(M \ot \brd_{B,B} \ot B)(\delta_M \ot \Delta) = \delta_M\mu_M.
\end{equation}

As a matter of notation, if the context requires to stress explicitly the module or the comodule structures on a particular object \(M\), we will use a full bullet, such as \(V_\bullet\) or \(V^\bullet\), to denote a given right action or coaction respectively. For example, if \((V,\mu_V)\) is a right \(B\)-module, then \(V_\bullet^{\phantom{\bullet}} \ot B_\bullet^\bullet\) denotes the right Hopf module with regular right \(B\)-comodule structure \(V \ot \Delta\) and diagonal right \(B\)-module structure
\begin{equation}\label{eq:diagonal}
    (\mu_V \ot m)(V \ot \brd_{B,B} \ot B)(V \ot B \ot \Delta).
\end{equation}

\subsection{One-sided Hopf monoids in braided monoidal categories}\label{subsect:onesided}
     
Let \((A,m,u)\in \Mon(\cM)\) and \((C,\Delta, \varepsilon)\in\Comon(\cM)\). The convolution product of two morphisms \(f,g\colon C\to A\) in \(\cM\) is defined by \(f*g\coloneqq m (f\ot g)  \Delta\). The following definition is a straightforward extension of \cite[\S1]{GreenNicholsTaft}.

\begin{definition}\label{def:one-sided}
   A \emph{right Hopf monoid} in \(\cM\) is a bimonoid in \(\cM\) equipped with a \emph{right antipode}, i.e., a morphism \(S \colon B \to B\) such that
    \begin{equation}\label{eq:right-antipode}
        \id_B * S = u\varepsilon.
    \end{equation}
    A morphism \(f \colon  B \to B'\) of right Hopf monoids in \(\cM\) is a morphism of bimonoids in \(\cM\) such that 
    \begin{equation}\label{eq:morph-onesided}
        fS_B=S_{B'} f.
    \end{equation}
    Right Hopf monoids in \(\cM\) and their morphisms form a category that we denote by \(\rHopf(\cM)\). Similarly, one defines \emph{left Hopf monoids} in \(\cM\) and their morphisms.  A \emph{Hopf monoid} in \(\cM\) is a bimonoid which is at the same time a left and a right Hopf monoid. 

    A right antipode is said to be anti-multiplicative and anti-comultiplicative if 
    \begin{equation}\label{eq:antibimon}
        Sm_B = m_B\brd_{B,B}(S \ot S) \qquad \text{and} \qquad (S \ot S)\brd_{B,B}\Delta_B = \Delta_B S
    \end{equation}
    respectively. The full subcategory of right Hopf monoids whose right antipode is anti-multiplicative and anti-comultiplicative will be denoted by \(\underline{\rHopf}(\cM)\).
\end{definition}

\begin{remark}
    Already in \(\cM = \Vect\), a bialgebra morphism between one-sided Hopf algebras does not necessarily preserve the one-sided antipodes, unlike what happens with ordinary Hopf algebras (see, for instance, \cite[Remark 2.9]{ArMeSa}). Therefore, if one is interested in discussing morphisms of one-sided Hopf monoids, the compatibility with the antipodes \eqref{eq:morph-onesided} shall be explicitly required (see, e.g., \cite[p.~400]{GreenNicholsTaft}).
\end{remark}

As the accustomed reader already expects, considering one-sided convolution inverses of the identity is as natural as considering split monos or split epis. The next result relates the existence of one-sided antipodes with one-sided inverses of the canonical Hopf--Galois map.

\begin{proposition}\label{prop:one_sided_ant_can}
     Let \((\cM,\ot,\I,\brd)\) be a braided monoidal category. 
     Let \((B,m, u, \Delta,\varepsilon)\) be a bimonoid in \(\cM\).
     Then, there is an isomorphism of monoids
     \[
    \xymatrix @R=0pt{
    (\Hom(B,B),*,u\varepsilon) \ar@{<->}[r] & \left(\prescript{B}{}{\Hom}_B(\prescript{\bullet}{}{B} \ot B_\bullet, \prescript{\bullet}{}{B} \ot B_\bullet), \circ ,\id \right) \\
    f \ar@{|->}[r] & (B \ot m)(B \ot f \ot B)(\Delta \ot B) \\
    \clambda_B(\varepsilon \ot B)g(B \ot u)\crho_B^{-1} & g \ar@{|->}[l]
    }
    \]
    under which \(\id_B\) corresponds to \(\mathsf{can} = (B \ot m)(\Delta \ot B) \colon B \ot B \to B \ot B\). Therefore, there is a bijective correspondence between left/right/two-sided antipodes on \(B\) and retractions/sections/inverses of \(\can\).
 \begin{invisible}   
    In particular, the following are equivalent:
        \begin{enumerate}[label=\((\roman*)\)]
            \item\label{item:B_right_hopf} \(B\) is a right Hopf monoid in \(\cM\) with anti-multiplicative and anti-comultiplicative right antipode \(S\); 
        
            \item\label{item:can_split}  the map \(\can\colon B\ot B\to B\ot B\) is a split epimorphism in \({}^B\!\mathcal{M}_B\) with respect to the regular action \(B\ot m\) and the regular coaction \(\Delta\ot B\), and with splitting \(\can^\dagger\) such that \(\clambda_B(\varepsilon\ot B)\can^\dagger(B\ot u)\crho_B^{-1}\) is anti-multiplicative and anti-comultiplicative.  
        \end{enumerate}
        \end{invisible}
\end{proposition}

\begin{proof}
Since the forgetful functor \(U \colon {}^B\!\mathcal{M} \to \mathcal{M}\) admits the right adjoint \(B \ot -\), we have a bijection
\[\Hom(B,B) = \Hom(U({}^\bullet\!B),B) \cong {}^B\Hom({}^{\bullet}\!B,{}^{\bullet}\!B \ot B), \qquad f \mapsto (B \ot f)\Delta,\]
with explicit inverse given by the assignment
\(g \mapsto \clambda_B (\varepsilon \ot B)g.\)
\begin{invisible}
    Indeed,
    \[(\Delta \ot B)(B \ot f)\Delta = (B \ot B \ot f)(B \ot \Delta) \Delta\]
    so that \((B \ot f)\Delta\) is colinear. Moreover,
    \[\clambda_B(\varepsilon \ot B)(B \ot f)\Delta = f\]
    and
    \[(B \ot \clambda_B)(B \ot \varepsilon \ot B)(B \ot g) \Delta = (B \ot \clambda_B)(B \ot \varepsilon \ot B)(\Delta \ot B) = (\crho_B \ot B)(B \ot \varepsilon \ot B)(\Delta \ot B)g = g.\]
\end{invisible}%
Furthermore, since \(B\) is a monoid in \({}^B\! \mathcal{M}\), the forgetful functor \(U'\colon {}^B\! \mathcal{M}_B \to {}^B\!\mathcal{M}\) admits the left adjoint \(- \ot B\) and so we have the bijection
\[{}^B\Hom({}^{\bullet}\!B,{}^{\bullet}\!B \ot B) = {}^B\Hom({}^{\bullet}\!B,U'({}^{\bullet}\!B \ot B_\bullet)) \cong {}^B\Hom_B({}^{\bullet}\!B \ot B_\bullet,{}^{\bullet}\!B \ot B_\bullet)\]
mapping
\[f \mapsto (B \ot m)(f \ot B)\]
and with explicit inverse given by the assignment
\[g \mapsto g(B \ot u)\crho^{-1}_B.\]
Now, a direct computation shows that
\begin{align*}
    & (B \ot m)(B \ot g \ot B)(\Delta \ot B)(B \ot m)(B \ot f \ot B)(\Delta \ot B) \\
    & = (B \ot m)(B \ot g \ot B)(B \ot B \ot m)(\Delta \ot B \ot B)(B \ot f \ot B)(\Delta \ot B) \\
    & = (B \ot m)(B \ot B \ot m)(B \ot g \ot B \ot B)(B \ot B \ot f \ot B)(\Delta \ot B \ot B)(\Delta \ot B) \\
    & = (B \ot m)(B \ot m \ot B)(B \ot g \ot f \ot B)(B \ot \Delta \ot B)(\Delta \ot B) \\
    & = (B \ot m)(B \ot g * f \ot B)(\Delta \ot B).
\end{align*}
Clearly, the image of \(\id\) is \(\can\).
\begin{invisible}
The equivalence at the end of the statement follows from the fact that
\[S = \clambda_B(\varepsilon\ot B)\can^\dagger(B\ot u)\crho_B^{-1}\]
under the given bijective correspondence.
\end{invisible}%
\end{proof}

One of the interesting features of one-sided Hopf monoids is that they naturally lift monoidal and closed structures to their categories of modules along the forgetful functors, even if these do not give rise to closed monoidal structures in the sense of \cite{Eilenbrg-Kelly}.

\begin{remark}
    Suppose that \((H,m,u,\Delta,\varepsilon,S)\) is a right Hopf monoid in the symmetric closed monoidal category \((\cM,\ot,\I,\brd,[-])\) with anti-multiplicative and anti-comultiplicative right antipode \(S\). There exists a unique morphism
    \(s_{X,Y} \colon H \ot \left[H \ot X,Y\right] \to \left[X,H \ot Y\right]\) such that
    \[
        \xymatrix @C=55pt @L=7pt {
            H \ot \left[H \ot X,Y\right] \ot X \ar[r]^-{(H \ot S)\Delta \ot [H \! \ot \! X,Y] \ot X} \ar[d]_-{s_{X,Y} \ot X} & H \ot H \ot \left[H \ot X,Y\right] \ot X \ar[r]^-{H \ot \brd_{H,[H \! \ot \! X,Y]} \ot X} & H \ot \left[H \ot X,Y\right] \ot H \ot X  \ar[d]^-{H \ot \ev^{H \ot X}_Y} \\
            \left[X,H \ot Y\right] \ot X \ar[rr]_-{\ev^X_{H \ot Y}} & & H \ot Y 
        }
    \]
    commutes which moreover is natural in $X$ and $Y$ in $\cM$. 
    \begin{invisible}
        The following diagram shows that $s$ is natural in both entries
        {\tiny
        \[\hspace{-2cm}\begin{tikzcd}[ampersand replacement=\&,cramped]
        	{\left[X',H \ot Y'\right] \ot X} \&\&\& {H \ot \left[H \ot X',Y'\right] \ot X} \& {H \ot \left[H \ot X,Y\right] \ot X} \& {\left[X,H \ot Y\right] \ot X} \\
        	\&\& {H \ot \left[H \ot X',Y'\right] \ot X'} \& {H \ot H \ot \left[H \ot X',Y'\right] \ot X} \& {H \ot H \ot \left[H \ot X,Y\right] \ot X} \\
        	\&\&\& {H \ot \left[H \ot X',Y'\right] \ot H \ot X} \& {H \ot \left[H \ot X,Y\right] \ot H \ot X} \\
        	\&\& {H \ot H \ot \left[H \ot X',Y'\right] \ot X'} \& {H \ot \left[H \ot X',Y'\right] \ot H \ot X} \& {H \ot \left[H \ot X,Y\right] \ot H \ot X} \& {H \ot Y} \\
        	\&\& {H \ot \left[H \ot X',Y'\right] \ot H \ot X'} \& {H \ot \left[H \ot X',Y'\right] \ot H \ot X'} \& {H \ot Y'} \\
        	\& {\left[X',H \ot Y'\right] \ot X'} \\
        	{\left[X,H \ot Y'\right] \ot X} \&\&\&\&\& {\left[X,H \ot Y\right] \ot X}
        	\arrow[""{name=0, anchor=center, inner sep=0}, "{\left[X',H \ot Y'\right] \ot f}"{description}, from=1-1, to=6-2]
        	\arrow["{\left[f,H \ot Y'\right] \ot X}"{description}, from=1-1, to=7-1]
        	\arrow[""{name=1, anchor=center, inner sep=0}, "{s_{X',Y'} \ot X}"', from=1-4, to=1-1]
        	\arrow["{H \ot \left[H \ot f,g\right] \ot X}", from=1-4, to=1-5]
        	\arrow["{H \ot \left[H \ot X',Y'\right] \ot f}"{description}, from=1-4, to=2-3]
        	\arrow["{\Delta \ot \left[H \ot X',Y'\right] \ot X}"{description}, from=1-4, to=2-4]
        	\arrow["{{\color{gray}\equiv}}"{description}, draw=none, from=1-4, to=4-5]
        	\arrow["{s_{X, Y} \ot X}", from=1-5, to=1-6]
        	\arrow["{\Delta \ot \left[H \ot X,Y\right] \ot X}"{description}, from=1-5, to=2-5]
        	\arrow["{{\color{gray}\text{def }s}}"{description}, draw=none, from=1-5, to=4-6]
        	\arrow["{\ev^X_{H \ot Y}}"{description}, from=1-6, to=4-6]
        	\arrow["{\Delta \ot \left[H \ot X',Y'\right] \ot X'}"{description}, from=2-3, to=4-3]
        	\arrow["{{\color{gray}\equiv}}"{description}, draw=none, from=2-3, to=5-4]
        	\arrow["{s_{X',Y'} \ot  X'}"{description}, curve={height=24pt}, from=2-3, to=6-2]
        	\arrow["{H \ot \brd_{H,\left[H \ot X',Y'\right]} \ot X}"{description}, from=2-4, to=3-4]
        	\arrow["{H \ot \brd_{H,\left[H \ot X,Y\right]} \ot X}"{description}, from=2-5, to=3-5]
        	\arrow["{H \ot \left[H \ot X',Y'\right] \ot S \ot X}"{description}, from=3-4, to=4-4]
        	\arrow["{H \ot \left[H \ot X,Y\right] \ot S \ot X}"{description}, from=3-5, to=4-5]
        	\arrow["{H \ot \brd_{H,\left[H \ot X',Y'\right]} \ot X'}"{description}, from=4-3, to=5-3]
        	\arrow["{H \ot \left[H \ot f,g\right] \ot H \ot X}", from=4-4, to=4-5]
        	\arrow["{H \ot \left[H \ot X',Y'\right] \ot H \ot f}"{description}, from=4-4, to=5-4]
        	\arrow["{H \ot \ev^{H \ot X}_{Y}}", from=4-5, to=4-6]
        	\arrow["{H \ot \left[H \ot X',Y'\right] \ot S \ot X'}", from=5-3, to=5-4]
        	\arrow["{{\color{gray}\text{naturality and dinaturality}}}"{description}, draw=none, from=5-4, to=4-6]
        	\arrow["{H \ot \ev^{H \ot X'}_{Y'}}"{description}, from=5-4, to=5-5]
        	\arrow["{H \ot g}"', from=5-5, to=4-6]
        	\arrow["{{\color{gray}\text{naturality}}}"{description}, draw=none, from=5-5, to=7-6]
        	\arrow["{{\color{gray}\text{def }s}}"{description}, draw=none, from=6-2, to=5-3]
        	\arrow["{\ev^{X'}_{H \ot Y'}}"{description}, curve={height=12pt}, from=6-2, to=5-5]
        	\arrow["{\ev^{X}_{H \ot Y'}}"{description}, curve={height=12pt}, from=7-1, to=5-5]
        	\arrow["{{\color{gray}\text{dinaturality}}}"{description}, draw=none, from=7-1, to=6-2]
        	\arrow["{\left[X,H \ot g\right] \ot X}"{description}, from=7-1, to=7-6]
        	\arrow["{\ev^{X}_{H \ot Y}}"{description}, from=7-6, to=4-6]
        	\arrow["{{\color{gray}\equiv}}"{description}, draw=none, from=1, to=0]
        \end{tikzcd}\]
        }%
    \end{invisible}%
    The natural transformation $s$ plus the counit $\varepsilon$ equip $H \ot -$ with a structure of gabi-monad \cite[Definition 3.3]{halbig2026gabimonads}, so that $\prescript{}{H}{\cM}$ becomes a skew-closed category in such a way that the forgetful functor to $\cM$ is strictly closed (see \cite[Theorem 3.6]{halbig2026gabimonads} and \cite[Theorem 3.4]{BergerSaraccoVercruysse}). 
    In fact, the mate of $s$ under the tensor-hom adjunction as in \cite[\S3.2]{BergerSaraccoVercruysse} is the natural transformation $t_{X,Y} \colon H \ot X \ot Y \to H \ot X \ot H \ot Y$ given by the composition
    \[H \ot X \ot Y \xrightarrow{\Delta \ot X \ot Y} H \ot H \ot X \ot Y \xrightarrow{H \ot \brd_{H,X} \ot Y} H \ot X \ot H \ot Y \xrightarrow{H \ot X \ot S \ot Y} H \ot X \ot H \ot Y,\]
    \begin{invisible}
        The following diagram shows that $t$ is indeed the mate of $s$:
        \[\begin{tikzcd}[ampersand replacement=\&,column sep=large]
        	{H \ot X \ot Y } \& {H \ot \left[H \ot Y,X \ot H \ot Y\right] \ot Y} \& {\left[Y,H \ot X \ot H \ot Y\right] \ot Y} \\
        	{H \ot H \ot X \ot Y } \& {H \ot H \ot \left[H \ot Y,X \ot H \ot Y\right] \ot Y} \\
        	{H \ot X \ot H \ot Y } \& {H \ot \left[H \ot Y,X \ot H \ot Y\right] \ot H \ot Y} \\
        	{H \ot X \ot H \ot Y } \& {H \ot \left[H \ot Y,X \ot H \ot Y\right] \ot H \ot Y} \& {H \ot X \ot H \ot Y}
        	\arrow["{H \ot \db^{H \ot Y}_X \ot Y }", from=1-1, to=1-2]
        	\arrow["{\Delta \ot X \ot Y }"{description}, from=1-1, to=2-1]
        	\arrow["{s_{Y,X \ot H \ot Y} \ot Y}", from=1-2, to=1-3]
        	\arrow["{\Delta \ot \left[H \ot Y,X \ot H \ot Y\right] \ot Y}"{description}, from=1-2, to=2-2]
        	\arrow["{\ev^Y_{H \ot X \ot H \ot Y}}"{description}, from=1-3, to=4-3]
        	\arrow["{H \ot H \ot \db^{H \ot Y}_X \ot Y }", from=2-1, to=2-2]
        	\arrow["{H \ot \brd_{H,X} \ot Y}"{description}, from=2-1, to=3-1]
        	\arrow["{H \ot \brd_{H,\left[H \ot Y,X \ot H \ot Y\right]} \ot Y}"{description}, from=2-2, to=3-2]
        	\arrow["{H \ot \db^{H \ot Y}_X \ot H \ot Y }", from=3-1, to=3-2]
        	\arrow["{H \ot X \ot S \ot Y }"{description}, from=3-1, to=4-1]
        	\arrow["{H \ot \left[H \ot Y,X \ot H \ot Y\right] \ot S \ot Y}"{description}, from=3-2, to=4-2]
        	\arrow["{H \ot \db^{H \ot Y}_X \ot H \ot Y }", from=4-1, to=4-2]
        	\arrow["{=}", curve={height=30pt}, from=4-1, to=4-3]
        	\arrow["{H \ot \ev^{H \ot Y}_{X \ot H \ot Y}}", from=4-2, to=4-3]
        \end{tikzcd}\]
    \end{invisible}%
    \begin{invisible}
        The next diagrams show that $t$ satisfies the conditions from \cite[Remark 3.8]{BergerSaraccoVercruysse} to have a gabi-monad:
        \[\begin{tikzcd}[ampersand replacement=\&,cramped]
        	\&\&\& {H \ot X \ot Y} \\
        	{X \ot Y} \&\&\& {H \ot H \ot X \ot Y} \\
        	\&\&\& {H \ot X \ot H \ot Y} \\
        	\&\&\& {H \ot X \ot H \ot Y}
        	\arrow["{\Delta \ot X \ot Y}", from=1-4, to=2-4]
        	\arrow["{u \ot X \ot Y}"{description}, curve={height=-12pt}, from=2-1, to=1-4]
        	\arrow["{u \ot u \ot X \ot Y}"{description}, from=2-1, to=2-4]
        	\arrow["{u \ot X \ot u \ot Y}"{description}, curve={height=6pt}, from=2-1, to=3-4]
        	\arrow["{u \ot X \ot u \ot Y}"{description}, curve={height=18pt}, from=2-1, to=4-4]
        	\arrow["{H \ot \brd_{H,X} \ot Y}", from=2-4, to=3-4]
        	\arrow["{H \ot X \ot S \ot Y}", from=3-4, to=4-4]
        \end{tikzcd}\]
        is \cite[(3.19)]{BergerSaraccoVercruysse};
        {\tiny
        \[\begin{tikzcd}[ampersand replacement=\&,cramped]
        	{H \ot H \ot H \ot X \ot Y} \&\&\& {H \ot H \ot X \ot Y} \& {H \ot X \ot Y} \\
        	\& {H \ot H \ot H \ot H \ot X \ot Y} \&\& {H \ot H \ot H \ot H \ot X \ot Y} \\
        	\&\& {H \ot H \ot H \ot H \ot X \ot Y} \& {H \ot H \ot H \ot H \ot X \ot Y} \& {H \ot H \ot X \ot Y} \\
        	\&\&\& {H \ot H \ot X \ot H \ot H \ot Y} \& {H \ot X \ot H \ot Y} \\
        	\& {H \ot H \ot H \ot X \ot Y} \& {H \ot H \ot H \ot H \ot X \ot Y} \& {H \ot H \ot X \ot H \ot H \ot Y} \\
        	{H \ot H \ot X \ot H \ot Y} \&\& {H \ot H \ot H \ot X \ot H \ot Y} \& {H \ot H \ot X \ot H \ot H \ot Y} \& {H \ot X \ot H \ot Y}
        	\arrow["{H \ot H \ot (H \ot S)\Delta \ot X \ot Y}"{description}, from=1-1, to=2-2]
        	\arrow["{H \ot \brd_{H,H} \ot X \ot Y}"{description}, from=1-1, to=5-2]
        	\arrow[""{name=0, anchor=center, inner sep=0}, "{H \ot \brd_{H,H \ot X} \ot Y}"', from=1-1, to=6-1]
        	\arrow[""{name=1, anchor=center, inner sep=0}, "{(H \ot S)\Delta \ot H \ot X \ot Y}"', from=1-4, to=1-1]
        	\arrow["{m \ot X \ot Y}", from=1-4, to=1-5]
        	\arrow["{\Delta \ot \Delta \ot X \ot Y}"{description}, from=1-4, to=2-4]
        	\arrow["{{\color{gray}\text{bimonoid}}}"{description}, draw=none, from=1-4, to=3-5]
        	\arrow["{\Delta \ot X \ot Y}", from=1-5, to=3-5]
        	\arrow["{H \ot \brd_{H,H} \ot H \ot X \ot Y}"{description}, from=2-2, to=3-3]
        	\arrow["{{\color{gray}\text{naturality}}}"{description}, draw=none, from=2-2, to=5-2]
        	\arrow[""{name=2, anchor=center, inner sep=0}, "{H \ot \brd_{H,H \ot H} \ot X \ot Y}"{description}, from=2-2, to=5-3]
        	\arrow[""{name=3, anchor=center, inner sep=0}, "{H \ot S \ot H \ot S \ot X \ot Y}"', from=2-4, to=2-2]
        	\arrow["{H \ot \brd_{H,H} \ot H \ot X \ot Y}"{description}, from=2-4, to=3-4]
        	\arrow["{H \ot H \ot \brd_{H,H} \ot X \ot Y}"{description}, from=3-3, to=5-3]
        	\arrow["{{\color{gray}\text{naturality + hexagon}}}"{description}, draw=none, from=3-3, to=6-4]
        	\arrow[""{name=4, anchor=center, inner sep=0}, "{H \ot H \ot S \ot S \ot X \ot Y}"', from=3-4, to=3-3]
        	\arrow["{m \ot m \ot X \ot Y}"', from=3-4, to=3-5]
        	\arrow["{H \ot H \ot \brd_{H \ot H,X} \ot Y}"{description}, from=3-4, to=4-4]
        	\arrow["{{\color{gray}\text{naturality}}}"{description}, draw=none, from=3-4, to=4-5]
        	\arrow["{H \ot \brd_{H,X} \ot Y}", from=3-5, to=4-5]
        	\arrow["{m \ot X \ot m \ot Y}"', from=4-4, to=4-5]
        	\arrow["{H \ot H \ot X \ot S \ot S \ot Y}"{description}, from=4-4, to=5-4]
        	\arrow["{{\color{gray}\text{anti-multiplicativity}}}"{description}, draw=none, from=4-4, to=6-5]
        	\arrow["{H \ot X \ot S \ot Y}", from=4-5, to=6-5]
        	\arrow["{H \ot (H \ot S)\Delta \ot H \ot X \ot Y}"', from=5-2, to=5-3]
        	\arrow["{H \ot H \ot \brd_{H,X} \ot Y}"{description}, from=5-2, to=6-1]
        	\arrow["{H \ot H \ot H \ot \brd_{H,X} \ot Y}"{description}, from=5-3, to=6-3]
        	\arrow["{H \ot H \ot X \ot \brd_{H,H} \ot Y}"{description}, from=5-4, to=6-4]
        	\arrow[""{name=5, anchor=center, inner sep=0}, "{H \ot (H \ot S)\Delta \ot X \ot H \ot Y}"', from=6-1, to=6-3]
        	\arrow["{H \ot H \ot \brd_{H,X} \ot H \ot Y}"', from=6-3, to=6-4]
        	\arrow["{m \ot X \ot m \ot Y}"', from=6-4, to=6-5]
        	\arrow["{{\color{gray}\text{hexagon}}}"{description}, draw=none, from=0, to=5-2]
        	\arrow["{{\color{gray}\equiv}}"{description}, draw=none, from=1, to=3]
        	\arrow["{{\color{gray}\text{naturality}}}"{description}, draw=none, from=3, to=4]
        	\arrow["{{\color{gray}\text{hexagom}}}"{description}, draw=none, from=3-3, to=2]
        	\arrow["{{\color{gray}\text{naturality}}}"{description}, draw=none, from=5-2, to=5]
        \end{tikzcd}\]
        }%
        is \cite[(3.20)]{BergerSaraccoVercruysse};
        \[\begin{tikzcd}[ampersand replacement=\&,cramped,sep=3em]
        	{H \ot X} \&\&\&\& {H \ot X \ot \I} \\
        	\\
        	\&\&\& {H \ot X \ot \I \ot \I} \\
        	{H \ot X \ot \I} \&\& {H \ot \I \ot X \ot \I} \\
        	{H \ot H \ot X \ot \I} \&\& {H \ot H \ot X \ot \I} \&\& {H \ot X \ot H \ot \I}
        	\arrow[""{name=0, anchor=center, inner sep=0}, draw=none, from=5-1, to=4-3]
        	\arrow["{{\color{gray}\text{counitality}}}"{description}, draw=none, from=4-1, to=0]
        	\arrow["{{\color{gray}S\text{ counital}}}"{description}, draw=none, from=5-3, to=0]
        	\arrow["{H \ot \varepsilon \ot X \ot \I}"{description}, from=5-1, to=4-3]
        	\arrow["{H \ot \crho_X^{-1}}", from=1-1, to=1-5]
        	\arrow["{\crho_{H \ot X}^{-1} = H \ot \crho_X^{-1}}"{description}, from=1-1, to=4-1]
        	\arrow["{H \ot X \ot m_\I = H \ot \crho_X \ot \I}"{description}, shift right, from=3-4, to=1-5]
        	\arrow["{=}"{description}, curve={height=-24pt}, from=4-1, to=1-5]
        	\arrow["{{\color{gray}\text{triangle identity}}}"{description, pos=0.3}, draw=none, from=4-1, to=1-5]
        	\arrow["{\crho_H^{-1} \ot X \ot \I}", from=4-1, to=4-3]
        	\arrow["{\Delta \ot X \ot \I}"', from=4-1, to=5-1]
        	\arrow["{H \ot \clambda_X \ot \I}"{description, pos=0.3}, curve={height=-24pt}, from=4-3, to=1-5]
        	\arrow["{{\color{gray}\eqref{eq:braidI}}}"{description}, draw=none, from=4-3, to=1-5]
        	\arrow["{H \ot \brd_{\I,X} \ot \I}"{description}, from=4-3, to=3-4]
        	\arrow["{H \ot S \ot X \ot \I}"', from=5-1, to=5-3]
        	\arrow["{H \ot \varepsilon \ot X \ot \I}"{description}, from=5-3, to=4-3]
        	\arrow["{H \ot \brd_{H,X} \ot \I}"', from=5-3, to=5-5]
        	\arrow[""{name=1, anchor=center, inner sep=0}, "{H \ot X \ot m_\I(\varepsilon \ot \I)}"', from=5-5, to=1-5]
        	\arrow["{{\color{gray}\text{naturality}}}"{description}, draw=none, from=5-3, to=1]
        \end{tikzcd}\]
        is \cite[(3.21)]{BergerSaraccoVercruysse};
        \[\begin{tikzcd}[ampersand replacement=\&,cramped,sep=3.5em]
        	{H \ot \I \ot M} \&\&\& \\
        	\& {H \ot M} \& {\I \ot M} \\
        	{H \ot H\ot \I \ot M} \& {H \ot H\ot M} \\
        	{H \ot H\ot \I \ot M} \\
        	{H\ot \I \ot H \ot M} \& {H\ot H \ot M} \& {H \ot M} \& M
        	\arrow["{H \ot \crho_{H} \ot M}"{description}, from=1-1, to=2-2]
        	\arrow[""{name=0, anchor=center, inner sep=0}, "{\mu_\I \ot M}"{description}, curve={height=-12pt}, from=1-1, to=2-3]
        	\arrow["{\Delta \ot \I \ot M}"{description}, from=1-1, to=3-1]
        	\arrow["{{\color{gray}\text{naturality}}}"{description}, draw=none, from=1-1, to=5-2]
        	\arrow["{\varepsilon \ot M}", from=2-2, to=2-3]
        	\arrow["{\Delta \ot M}"{description}, from=2-2, to=3-2]
        	\arrow["{{\color{gray}\text{right antipode}}}"{description}, draw=none, from=2-2, to=5-3]
        	\arrow["{u \ot M}"{description}, from=2-3, to=5-3]
        	\arrow[""{name=1, anchor=center, inner sep=0}, "{\clambda_M}", from=2-3, to=5-4]
        	\arrow["{H \ot S \ot \I \ot M}"{description}, from=3-1, to=4-1]
        	\arrow["{H \ot S \ot M}"{description}, from=3-2, to=5-2]
        	\arrow["{H \ot \brd_{H,\I} \ot M}"', from=4-1, to=5-1]
        	\arrow[""{name=2, anchor=center, inner sep=0}, "{H \ot \crho_{H} \ot M}"{description}, from=4-1, to=5-2]
        	\arrow["{H \ot \clambda_{H \ot M}}"', from=5-1, to=5-2]
        	\arrow["{m \ot M}"', from=5-2, to=5-3]
        	\arrow["{\mu_M}"', from=5-3, to=5-4]
        	\arrow["{{\color{gray}\text{def } \mu_\I}}"{description}, draw=none, from=2-2, to=0]
        	\arrow["{{\color{gray}\eqref{eq:braidI}}}"{description}, draw=none, from=5-1, to=2]
        	\arrow["{{\color{gray}\text{unitality}}}"{description}, draw=none, from=5-3, to=1]
        \end{tikzcd}\]
        is \cite[(3.22)]{BergerSaraccoVercruysse};
        \[\rotatebox{90}{\tiny
        \begin{tikzcd}[ampersand replacement=\&,cramped]
        	{H \, X \, Y \, M} \&\& {H \, H \, X \, Y \, M} \&\&\&\& {H \, X \, H \, Y \, M} \&\& {H \, X \, H \, Y \, M} \\
        	\&\&\& {H \, \I \, H \, X \, Y \, M} \&\& {H \, X \, \I \, H \, Y \, M} \&\& {H \, X \, H \, Y \, \I \, M} \\
        	\&\& {H \, H \, H \, X \, Y \, M} \& {H \, X \, H \, H \, Y \, M} \&\&\& {H \, X \, H \, Y \, H \, M} \& {H \, X \, H \, Y \, H \, M} \& {H \, X \, H \, Y \, H \, M} \\
        	{H \, H \, X \, Y \, M} \&\& {H \, H \, H \, H \, X \, Y \, M} \&\& {H \, X \, H \, H \, H \, Y \, M} \&\& {H \, X \, H \, Y \, H \, H \, M} \& {H \, X \, H \, Y \, H \, H \, M} \\
        	\& {H \, H \, H \, X \, Y \, M} \&\& {H \, H \, X \, H \, H \, Y \, M} \\
        	\& {H \, H \, H \, X \, Y \, M} \& {H \, X \, H \, H \, Y \, M} \&\& {H \, X \, H \, H \, Y \, H \, M} \&\& {H \, X \, H \, H \, Y \, H \, M} \& {H \, X \, H \, Y \, H \, H \, M} \& {H \, X \, H \, Y \, H \, H \, M} \\
        	\&\&\& {H \, H \, X \, H \, H \, Y \, M} \&\& {H \, X \, H \, H \, Y \, H \, M} \&\&\& {H \, X \, H \, H \, Y \, H \, M} \\
        	{H \, H \, X \, Y \, M} \&\& {H \, H \, H \, X \, Y \, M} \& {H \, X \, H \, H \, Y \, M} \& {H \, H \, X \, H \, H \, Y \, M} \&\& {H \, H \, X \, H \, H \, Y \, M} \&\& {H \, X \, H \, H \, Y \, H \, M} \\
        	{H \, X \, H \, Y \, M} \&\& {H \, H \, X \, H \, Y \, M} \&\& {H \, H \, X \, H \, Y \, M} \&\&\&\& {H \, X \, H \, Y \, H \, M}
        	\arrow["{\Delta \, X \, Y \, M}", from=1-1, to=1-3]
        	\arrow["{\Delta \, X \, Y \, M}"{description}, from=1-1, to=4-1]
        	\arrow["{H \, \brd_{H,X} \, Y \, M}"{description}, from=1-3, to=1-7]
        	\arrow["{H \, \clambda_H^{-1} \, X \, Y \, M}"{description}, from=1-3, to=2-4]
        	\arrow["{H \, \Delta \, X \, Y \, M}"{description}, from=1-3, to=3-3]
        	\arrow["{H \, X \, S \, Y \, M}"{description}, from=1-7, to=1-9]
        	\arrow["{H \, \brd_{\I \, H,X} \, Y \, M}"{description}, from=2-4, to=2-6]
        	\arrow["{H \, X \, \clambda_H \, Y \, M}"{description}, from=2-6, to=1-7]
        	\arrow["{H \, X \, \brd_{\I,H \, Y} \, M}"{description}, from=2-6, to=2-8]
        	\arrow["{H \, X \, S \, Y \, \crho_M}"{description}, from=2-8, to=1-9]
        	\arrow["{H \, X \, S \, Y \, u \, M}"{description, pos=0.4}, from=2-8, to=3-8]
        	\arrow["{H \, X \, S \, Y \, u \, M}", curve={height=-12pt}, from=2-8, to=3-9]
        	\arrow["{H \, \varepsilon \, H \, X \, Y \, M}"{description}, from=3-3, to=2-4]
        	\arrow["{H \, \brd_{H \, H,X} \, Y \, M}"', from=3-3, to=3-4]
        	\arrow["{H \, \Delta \, H \, X \, Y \, M}"{description}, from=3-3, to=4-3]
        	\arrow["{H \, X \, \brd_{H,H \, Y} \, M}"{description}, from=3-4, to=3-7]
        	\arrow["{H \, X \, \Delta \, H \, Y \, M}"{description}, from=3-4, to=4-5]
        	\arrow["{H \, X \, S \, Y \, \varepsilon \, M}"{description, pos=0.4}, from=3-7, to=2-8]
        	\arrow["{H \, X \, H \, Y \, \Delta \, M}"{description}, from=3-7, to=4-7]
        	\arrow["{H \, X \, H \, Y \, S \, M}"', from=3-8, to=3-9]
        	\arrow["{H \, X \, H \, Y \, \mu_M}"{description}, from=3-9, to=1-9]
        	\arrow["{H \, \Delta \, X \, Y \, M}"{description}, from=4-1, to=5-2]
        	\arrow["{H \, S \, X \, Y \, M}"{description}, from=4-1, to=8-1]
        	\arrow["{H \, \brd_{H \, H \, H,X} \, Y \, M}"{description}, from=4-3, to=4-5]
        	\arrow["{H \, H \, \brd_{H \, H,X} \, Y \, M}"{description}, from=4-3, to=5-4]
        	\arrow["{H \, X \, \brd_{H \, H,H \, Y} \, M}"{description}, from=4-5, to=4-7]
        	\arrow["{H \, X \, \brd_{H,H \, H \, Y} \, M}"{description}, from=4-5, to=6-5]
        	\arrow["{H \, X \, S \, Y \, H \, S \, M}", from=4-7, to=4-8]
        	\arrow["{H \, X \, H \, Y \, m \, M}"{description}, from=4-8, to=3-8]
        	\arrow["{\Delta \, H \, H \, X \, Y \, M}"{description}, from=5-2, to=4-3]
        	\arrow["{H \, \brd_{H,H} \, X \, Y \, M}"{description}, from=5-2, to=6-2]
        	\arrow["{H \, \brd_{H,X} \, H \, H \, Y \, M}"{description}, from=5-4, to=4-5]
        	\arrow["{H \, \brd_{H,X \, H \, H \, Y} \, M}"{description}, from=5-4, to=6-5]
        	\arrow["{H \, H \, X \, \brd_{H,H} \, Y \, M}"{description}, from=5-4, to=7-4]
        	\arrow["{H \, \brd_{H \, H,X} \, Y \, M}"', from=6-2, to=6-3]
        	\arrow["{H \, S \, S \, X \, Y \, M}"{description}, from=6-2, to=8-3]
        	\arrow["{\Delta \, X \, H \, H \, Y \, M}"{description}, from=6-3, to=7-4]
        	\arrow[""{name=0, anchor=center, inner sep=0}, "{H \, X \, \brd_{H,H \, Y \, H} \, M}"{description, pos=0.6}, from=6-5, to=4-7]
        	\arrow["{H \, X \, \brd_{H,H} \, Y \, H \, M}"{description}, from=6-5, to=7-6]
        	\arrow["{H \, X \, H \, \brd_{H,Y \, H} \, M}"{description, pos=0.3}, from=6-7, to=4-8]
        	\arrow["{H \, X \, H \, \brd_{H,Y} \, H \, M}"', from=6-7, to=6-8]
        	\arrow["{H \, X \, H \, Y \, \brd_{H,H} \, M}"{description}, from=6-8, to=4-8]
        	\arrow["{H \, X \, H \, Y \, S \, S \, M}", from=6-8, to=6-9]
        	\arrow["{H \, X \, H \, Y \, m \, M}"{description}, from=6-9, to=3-9]
        	\arrow["{H \, \brd_{H,X \, H \, H \, Y} \, M}"{description}, from=7-4, to=7-6]
        	\arrow["{H \, X \, H \, \brd_{H,Y \, H} \, M}"{description}, from=7-6, to=4-7]
        	\arrow["{H \, X \, S \, S \, Y \, H \, M}"{description}, from=7-6, to=6-7]
        	\arrow["{H \, X \, S \, S \, Y \, S \, M}"{description}, from=7-6, to=8-9]
        	\arrow["{H \, X \, H \, \brd_{H,Y} \, H \, M}"{description}, from=7-9, to=6-9]
        	\arrow["{H \, \Delta \, X \, Y \, M}"{description}, from=8-1, to=8-3]
        	\arrow["{H \, \brd_{H,X} \, Y \, M}"{description}, from=8-1, to=9-1]
        	\arrow["{H \, \brd_{H \, H,X} \, Y \, M}"', from=8-3, to=8-4]
        	\arrow["{\Delta \, X \, H \, H \, Y \, M}"', from=8-4, to=8-5]
        	\arrow["{H \, S \, X \, H \, H \, Y \, M}"{description}, from=8-5, to=8-7]
        	\arrow["{H \, \brd_{H,X \, H \, H \, Y} \, M}"{description}, from=8-7, to=8-9]
        	\arrow["{H \, X \, H \, S \, Y \, H \, M}"{description}, from=8-9, to=7-9]
        	\arrow["{\Delta \, X \, H \, Y \, M}"{description}, from=9-1, to=9-3]
        	\arrow["{H \, S \, X \, H \, Y \, M}"{description}, from=9-3, to=9-5]
        	\arrow["{H \, \brd_{H,X \, H \, Y} \, M}", from=9-5, to=9-9]
        	\arrow["{H \, X \, \Delta \, Y \, H \, M}"{description}, from=9-9, to=8-9]
        	\arrow["{{\color{gray}\text{braided diagram check + symmetry}}}"{description}, draw=none, from=4-5, to=0]
        \end{tikzcd}
        } \]
        where we omitted the $\ot$ for the sake of compactness, is \cite[(3.23)]{BergerSaraccoVercruysse}.
    \end{invisible}%
    and one can verify that $t$ satisfies the conditions \cite[(3.19)-(3.23)]{BergerSaraccoVercruysse}. Hence, $H \ot -$ is a gabi-monad and $\prescript{}{H}{\cM}$ is skew-closed with internal homs lifted from $\cM$ along the forgetful functor. For the convenience of the interested reader, (3.19) follows from unitality of $S$, (3.20) from its anti-multiplicativity, (3.21) from its counitality, (3.22) from the fact that it is a right antipode, and (3.23) from anti-multiplicativity, anti-comultiplicativity, unitality, and the fact that it is a right antipode. 
    \begin{invisible}
        The action of $H$ on the internal hom $[M,N]$ is given by the unique \(\mu_{\left[M,N\right]} \colon H \ot \left[M,N\right] \to \left[M,N\right]\) such that
        \[
            \xymatrix @C=40pt {
                H \ot \left[M,N\right] \ot M \ar[d]_-{\Delta \ot \left[M,N\right] \ot M} \ar[r]^-{\mu_{\left[M,N\right]} \ot M} & \left[M,N\right] \ot M \ar[dddd]^-{\ev^M_N} \\
                H \ot H \ot \left[M,N\right] \ot M \ar[d]_-{H \ot \brd_{H,\left[M,N\right]} \ot N} & \\ 
                H \ot \left[M,N\right] \ot H \ot M \ar[d]_-{H \ot \left[M,N\right] \ot \mu_{M}(S \ot M)} & \\
                H \ot \left[M,N\right] \ot M \ar[d]_-{H \ot \ev^M_N} & \\
                H \ot N \ar[r]_-{\mu_N} & N
            }
        \]
        commutes.
    \end{invisible}%
\end{remark}

\section{Symmetric monoidal categories: the free one-sided Hopf monoid}\label{sect:GNTsym-monoidal} 

Given a comonoid \(C\) in a symmetric monoidal category \(\cM\) satisfying some mild assumptions, we now construct a one-sided Hopf monoid \(\rH(C)\). This will be a free construction from the category of comonoids \(\Comon(\cM)\) to the category \(\underline{\rHopf}(\cM)\) of right Hopf monoids whose one-sided antipode is anti-multiplicative and anti-comultiplicative. The reader not interested in this level of generality may refer to \zcref{sec:examples}, where the same construction is unwrapped in the case where the Heyneman--Sweedler notation can be employed.

\begin{invisible}
    {\paolo
    \begin{remark}\label{rem:brd_cop_is_bialgebra}
        Let \((A,m,u)\) be a monoid, \((C,\Delta,\varepsilon)\) be a comonoid, and \((B,m,u,\Delta,\varepsilon)\) be a bimonoid in a braided monoidal category \((\cM,\ot,\I,\brd)\). The coopposites \(C^\mathrm{cop} = (C,\brd_{C,C}^{-1}\Delta,\varepsilon)\) and \(B^\mathrm{cop} = (B,m,u,\brd_{B,B}^{-1}\Delta,\varepsilon)\) are a comonoid and a bimonoid, respectively, in the braided monoidal category \(\cM^{\otimes\mathrm{op}} = (\cM,\ot,\I,\brd^{-1})\). Analogously, the opposites \(A^\mathrm{op} = (A,m\brd_{A,A}^{-1},u)\) and \(B^\mathrm{op} = (B,m\brd_{B,B}^{-1},u,\Delta,\varepsilon)\) are a monoid and a bimonoid, respectively, in \(\cM^{\otimes\mathrm{op}}\). In addition, \(B^{\mathrm{op},\mathrm{cop}}=(B,m\brd_{B,B}^{-1},u,\brd_{B,B}\Delta,\varepsilon)\) and \(B^{\mathrm{cop},\mathrm{op}}=(B,m\brd_{B,B},u,\brd_{B,B}^{-1}\Delta,\varepsilon)\) are again bimonoids in \((\cM,\ot,\I,\brd)\), and \(B^{\mathrm{op},\mathrm{op}} = B = B^{\mathrm{cop},\mathrm{cop}}\). In view of these observations, one may adapt the construction presented in the present section to the more general context of braided, but not necessarily symmetric, monoidal categories at the cost of caring about the braiding and its inverse. We decided not to pursue this line.
        
        However, if one pursues this line, the construction of the candidate one-sided antipode $S$ shall start from $\id \colon V_n \to V_{n+1}$, which is \emph{not} of comonoids, and hence the final $S$ will be a right antipode which is an anti-morphism of monoids but not necessarily of comonoids.

        The next commutative diagram shows why anti-comultiplicativity may be an issue when starting from the even degree:
        \[\begin{tikzcd}[ampersand replacement=\&,cramped]
        	\& {V_{2n+1} \otimes V_{2n+1}} \&\& {V \otimes V} \& {T(V) \otimes T(V)} \\
        	\& {V_{2n+1} \otimes V_{2n+1}} \&\& {V \otimes V} \& {T(V) \otimes T(V)} \\
        	\& {V_{2n+1}} \&\& V \& {T(V)} \\
        	\& {V_{2n}} \&\& V \& {T(V)} \\
        	{V_{2n+1}} \&\& {V_{2n} \otimes V_{2n}} \& {V\otimes V} \& {T(V) \otimes T(V)} \\
        	\& {V_{2n+1} \otimes V_{2n+1}} \&\& {V \otimes V} \& {T(V) \otimes T(V)}
        	\arrow["{\iota_{2n+1} \otimes \iota_{2n+1}}"{description}, from=1-2, to=1-4]
        	\arrow["{\iota_V \otimes \iota_V}", from=1-4, to=1-5]
        	\arrow["{\brd^{-1}_{V_{2n+1},V_{2n+1}} = \brd^{-1}_{C,C}}"{description}, from=2-2, to=1-2]
        	\arrow["{\iota_{2n+1} \otimes \iota_{2n+1}}"{description}, from=2-2, to=2-4]
        	\arrow["{\brd^{-1}_{V,V}}"', from=2-4, to=1-4]
        	\arrow["{\iota_V \otimes \iota_V}"', from=2-4, to=2-5]
        	\arrow["{\brd^{-1}}"', from=2-5, to=1-5]
        	\arrow["{\Delta_{2n+1}=\brd_{C,C}^{-1}\Delta_C}"{description}, from=3-2, to=2-2]
        	\arrow["{\iota_{2n+1}}"', from=3-2, to=3-4]
        	\arrow["{\Delta_V}"', from=3-4, to=2-4]
        	\arrow["{\iota_V}"', from=3-4, to=3-5]
        	\arrow["\Delta"', from=3-5, to=2-5]
        	\arrow["id"', from=4-2, to=3-2]
        	\arrow["{\iota_{2n}}", from=4-2, to=4-4]
        	\arrow["id"', from=4-2, to=5-1]
        	\arrow["{\Delta_{2n} = \Delta_C}"{description}, from=4-2, to=5-3]
        	\arrow["{S_V}"', from=4-4, to=3-4]
        	\arrow["{\iota_V}", from=4-4, to=4-5]
        	\arrow["{\Delta_V}", from=4-4, to=5-4]
        	\arrow["S"', from=4-5, to=3-5]
        	\arrow["\Delta", from=4-5, to=5-5]
        	\arrow["{\brd_{C,C}\Delta_{2n+1} = \Delta_C}"{description}, from=5-1, to=6-2]
        	\arrow["{\iota_{2n} \otimes \iota_{2n}}"', from=5-3, to=5-4]
        	\arrow["{id \otimes id}"{description}, from=5-3, to=6-2]
        	\arrow["{\iota_V \otimes \iota_V}", from=5-4, to=5-5]
        	\arrow["{S_V \otimes S_V}"{description}, from=5-4, to=6-4]
        	\arrow["{S \otimes S}", from=5-5, to=6-5]
        	\arrow["{\iota_{2n+1} \otimes \iota_{2n+1}}"', from=6-2, to=6-4]
        	\arrow["{\iota_V \otimes \iota_V}"', from=6-4, to=6-5]
        \end{tikzcd}\]
    \end{remark}
    }
\end{invisible}

%
\medskip 

Let \((\cM,\ot,\I,\brd)\) be a symmetric monoidal category such that \(\cM\) admits countable colimits and these are preserved by \(X \ot -\) (and hence by \(- \ot X\)) for every object \(X\) in \(\cM\).
\begin{invisible}
    We claim that the category of comonoids admits countable coproducts. Let \(\{C_n\}_{n \in \N}\) be a family of comonoids in \(\cM\). Consider their coproduct \(C \coloneqq \bigsqcup_{n \in \N}C_n\) in \(\cM\), with structure maps \(\iota_n \colon C_n \to C\). By the universal property of the coproduct, there exists a unique morphism \(\varepsilon \colon C \to \I\) in \(\cM\) such that \(\varepsilon \iota_n = \varepsilon_n\) for every \(n \in \N\). For every \(n \in \N\), consider also the composition
    \[C_n \xrightarrow{\Delta_n} C_n \ot C_n \xrightarrow{\iota_n \ot \iota_n} C \ot C\]
    in \(\cM\). By the universal property of the coproduct again, there exists a unique \(\Delta \colon C \to C \ot C\) in \(\cM\) such that \(\Delta \iota_n = (\iota_n \ot \iota_n)\Delta_n\), for every \(n \in \N\). Let us show first that \((\Delta,\varepsilon)\) is a structure of comonoid on \(C\). Since 
    \begin{align*}
        (\Delta \ot C)\Delta \iota_n & = (\Delta \ot C)(\iota_n \ot \iota_n)\Delta_n = (\iota_n \ot \iota_n \ot \iota_n)(\Delta_n \ot C_n) \Delta_n \\
        & = (\iota_n \ot \iota_n \ot \iota_n)(C_n \ot \Delta_n) \Delta_n = (C \ot \Delta)\Delta \iota_n
    \end{align*}
    for every \(n \in \N\), \(\Delta\) is coassociative. Since
    \begin{align*}
        \clambda_C(\varepsilon \ot C)\Delta \iota_n & = \clambda_C(\varepsilon \ot C)(\iota_n \ot \iota_n)\Delta_n = \clambda_C(\I \ot \iota_n)(\varepsilon_n \ot C_n)\Delta_n  \\
        & = \iota_n\clambda_{C_n}(\varepsilon_n \ot C_n)\Delta_n = \iota_n
    \end{align*}
    for every \(n \in \N\), \(\Delta\) is left counital. Similarly, \(\Delta\) is right counital. Now, let us show that the comonoid \((C,\Delta,\varepsilon)\) satisfies the universal property of the coproduct. Suppose that we have a family \(f_n \colon C_n \to D\) of comonoid morphisms. Since they are, in particular, morphisms in \(\cM\), there exists a unique morphism \(F \colon C \to D\) in \(\cM\) such that \(F\iota_n = f_n\). One can check directly that
    \begin{align*}
        \Delta_D F \iota_n & = \Delta_D f_n = (f_n \ot f_n) \Delta_{n} = (F \ot F)(\iota_n \ot \iota_n)\Delta_{n} = (F \ot F)\Delta \iota_n, \\
        \varepsilon_D F \iota_n & = \varepsilon_D f_n = \varepsilon_n = \varepsilon \iota_n,
    \end{align*}
    for every \(n \in \N\) and so \(F\) is a morphism of comonoids.
\end{invisible}

%
%
Let \(C\) be a comonoid in \(\mathcal{M}\). For \(n\geq 0\), let 
\[
    V_{2n}\coloneqq C\quad \text{and} \quad V_{2n+1}\coloneqq C^{\mathrm{cop}},\quad\text{where}\quad \Delta^\mathrm{cop}=\brd_{C,C} \Delta.
\] 

First, we consider the coproduct \(V\coloneqq \bigsqcup _{n}V_n\) in \(\cM\) with structure maps \(\iota_n \colon V_n \to V\), which is a comonoid again when equipped with the unique morphisms \(\varepsilon_V \colon V \to \I\) and \(\Delta_V \colon V \to V \ot V\) in \(\cM\) such that 
\begin{equation}\label{eq:coprod_coalg}
    \varepsilon_V \iota_n = \varepsilon_{V_n} \qquad \text{and} \qquad \Delta_V \iota_n = (\iota_n \ot \iota_n)\Delta_{V_n}
\end{equation}
for every \(n \in \N\).
%
%
\begin{invisible}
    {\color{purple}
    In fact, \(V\) is even the coproduct in \(\Comon(\cM)\), as shown in the invisible above. There are a number of way of deducing this fact from the literature without having to prove it formally. One possibility is to cite Abdulwahid and Iovanov's \cite[Proposition 2.3]{AbdulwahidIovanov}, where they simply claim that \(\Comon(\cM)\) is cocomplete and the forgetful functor preserves colimits, if \(\cM\) is cocomplete. However, they refer to the coring case in \cite{BrWi}. Another possibility is to cite Porst's \cite[2.7]{PorstMonoids}, where he claims that if \(\cM\) is admissible monoidal, then \(\Comon(\cM)\) is comonadic over \(\cM\), and hence it admits all the colimits of \(\cM\) lifted along the forgetful functor (in \cite[\S2.2, pp 2895 onwards]{Vasilakopoulou}, Porst's argument is resumed by Christina in a more comprehensible way). However, this requires a bit more than we need, because we do not care about accessibility here. Apparently, the argument is based on the idea that a cocontinuous functor from a locally presentable category automatically has a right adjoint, and by the dual of \cite[3.1 Fact]{Porst}, the forgetful functor \(\Comon(\cM) \to \cM\) is comonadic provided it has a right adjoint (a similar argument appears also in \cite[\S8.2.3]{NiuSpivak}). As far as we are concerned, our argument is sufficient.
    }%
\end{invisible}%

Then, we consider the free monoid \(T(V)=\bigsqcup_n V^{\ot n}\) in \(\mathcal{M}\); see \cite[ \S VII.3, Theorem 2]{Maclane}. 
\begin{invisible}
    It satisfies the universal property: for every morphism \(f\colon V\to A\) into a monoid \(A\) in \(\mathcal{M}\), there is a unique monoid morphism \(\hat{f}\colon T(V)\to A\) such that \(\hat{f}  \iota_V=f\), where \(\iota_V\colon V\to T(V)\) is the canonical injection.
\end{invisible}
By its universal property, there exist unique monoid morphisms \(\Delta_{T(V)} \colon T(V) \to T(V) \ot T(V)\) and \(\varepsilon_{T(V)} \colon T(V) \to \I\) in \(\mathcal{M}\) such that the following diagrams commute:
\begin{equation}\label{eq:bialgTV}
    \begin{gathered}
        \xymatrix @C=70pt {
            V \ar[r]^{\iota_V} \ar[d]_-{\Delta_V} & T(V) \ar@{.>}[d]^-{\Delta_{T(V)}} \\
            V \ot V \ar[r]_-{\iota_V \ot \iota_V} & T(V) \ot T(V)
            }
            \qquad 
            \xymatrix{
            V \ar[rr]^-{\iota_V} \ar[dr]_-{\varepsilon_V} & & T(V) \ar@{.>}[dl]^-{\varepsilon_{T(V)}} \\
            & \I & 
        }
    \end{gathered}
\end{equation}
Then, \(T(V)\) becomes a bimonoid in \(\mathcal{M}\).

Now, observe that \(V^\mathrm{cop}\cong \bigsqcup_n V_n^{\mathrm{cop}}\): indeed, in the diagram
\[
    \xymatrix @C=45pt {
        V_n \ar@{}[dr]|-{\color{gray}\eqref{eq:coprod_coalg}} \ar[d]_-{\iota_n} \ar[r]^-{\Delta_n} & V_n \ot V_n \ar@{}[dr]|-{{\color{gray}\brd \text{ natural}}} \ar[d]|-{\iota_n \ot \iota_n} \ar[r]^-{\brd_{V_n,V_n}} & V_n \ot V_n \ar[d]^-{\iota_n \ot \iota_n} \\
        V \ar[r]_-{\Delta_V} & V \ot V \ar[r]_-{\brd_{V,V}} & V\ot V
    }
\]
the left-hand side square commutes by definition of \(\Delta_V\), and the right-hand side square commutes by naturality of \(\brd\). By definition, \(\Delta_{V^\mathrm{cop}}\) is the composition of the lower row, thus it makes the large rectangle commute. By the universal property of the coproduct, there exists a unique morphism \(V\to V\ot V\) making the large rectangle commute: this morphism is by definition \(\Delta_{\bigsqcup_n V_n^{\mathrm{cop}}}\), and hence, by uniqueness, one has \(\Delta_{\bigsqcup_n V_n^{\mathrm{cop}}}=\Delta_{V^\mathrm{cop}}\).

Since \(\brd\) is a symmetry, we can consider the unique morphism of comonoids \(S_V \colon V^{\mathrm{cop}}\rightarrow V\) such that 
\begin{equation}\label{eq:defs}
    \begin{gathered}
        \begin{tikzcd}
        	{V_n^{\mathrm{cop}}} & {V^{\mathrm{cop}}=\bigsqcup_n V_n^{\mathrm{cop}}} \\
        	{V_{n+1}} & V
        	\arrow[hook, from=1-1, to=1-2]
        	\arrow[equals, from=1-1, to=2-1]
        	\arrow["S_V", dotted, from=1-2, to=2-2]
        	\arrow[hook, from=2-1, to=2-2]
        \end{tikzcd}
    \end{gathered}
\end{equation}
commutes for every \(n \geq 0\). This is morally an index-shift operator: if we were working in \(\Vect\), then \(S_V\) would be
\begin{equation*}
S_V ( 0,\ldots ,0,\underset{n}{\underbrace{v}},0,0,\ldots ) = (
0,\ldots ,0,0,\underset{n+1}{\underbrace{v}},0,\ldots )
\end{equation*}%
where on the left-hand side \(v\) is in the entry \(n\), while in the right-hand side it is in the entry \(n+1\); see \zcref{ex:Z2inZ2,ex:MatcomodZ2} in the forthcoming \zcref{sec:examples}. 

We can now lift \(S_V\) to a morphism of monoids \(S\colon T\left( V^{\mathrm{cop}}\right) \rightarrow T\left( V\right) ^{\mathrm{op}}\) by considering the composition
\begin{equation*}
    V^{\mathrm{cop}} \overset{S_V}{\longrightarrow }V\overset{\iota }{\longrightarrow }T\left( V\right)^{\mathrm{op}}.
\end{equation*}%

The following classical observation will be useful here.

\begin{remark}\label{rem:cop_is_bialgebra}
    Let \((B,m,u,\Delta,\varepsilon)\) be a bimonoid in a symmetric monoidal category \((\cM,\ot,\I,\brd)\). Then, its coopposite 
    \(B^\mathrm{cop} = (B,m,u,\brd_{B,B}\Delta,\varepsilon)\) is again a bimonoid in \((\cM,\ot,\I,\brd)\). 
\end{remark}

In symmetric monoidal categories, the free monoid \(T(D)\) over the object underlying a comonoid \(D\) is the free bimonoid over the comonoid \(D\). Therefore, by the universal property of the free bimonoid functor \(T(-)\), there exists a unique morphism of bimonoids \(S \colon T(V^{\mathrm{cop}}) \to T(V)^{\mathrm{op}}\).
Since both \(\brd_{T(V),T(V)} \Delta_{T(V)}\) and \(\Delta_{T(V^{\mathrm{cop}})}\) are monoid morphisms (by \zcref{rem:cop_is_bialgebra}) making the large rectangle 
\[
    \xymatrix @C=50pt{
        V \ar@{}[dr]|-{\color{gray}\eqref{eq:bialgTV}} \ar[d]_-{\iota_V} \ar[r]^-{\Delta_V} & V \ot V \ar@{}[dr]|-{{\color{gray}\brd \text{ natural}}} \ar[d]_-{\iota_V \ot \iota_V} \ar[r]^-{\brd_{V,V}} & V \ot V \ar[d]^-{\iota_V \ot \iota_V} \\
        T(V) \ar[r]_-{\Delta_{T(V)}} & T(V) \ot T(V) \ar[r]_-{\brd_{T(V),T(V)}} & T(V)\ot T(V)
    }
\]
commute, they are equal, and hence \(T(V)^{\mathrm{cop}} = T(V^{\mathrm{cop}})\), so that \(S\) becomes actually a morphism of bimonoids \(S \colon T(V)^{\mathrm{cop}} \to T(V)^{\mathrm{op}}\).

By following \cite{GreenNicholsTaft}, we should now perform the analogue of quotienting \(T(V)\) by the relations that turn \(S\) into a one-sided antipode. If we were to work in \(\Vect\), then our aim would be performing \(T(V)/K\), where
\[K = \stretchleftright[600]{\langle}
{\left.\begin{gathered}x_{(1)}S(x_{(2)}) - \varepsilon(x)1_{T(V)}, \\[3pt] S(y_{(1)})y_{(2)} - \varepsilon(y)1_{T(V)} \end{gathered} ~\right|~ \begin{gathered} x \in V_n, n\geq 0, \\[3pt] y \in V_m, m \geq 1 \end{gathered}}
{\rangle}.\]
This will be done in several steps.

Call \(B \coloneqq T(V)\). Consider \(W \coloneqq \bigsqcup_{n\geq 1} V_n\) and denote by \(\jmath_n \colon V_n \to W\) the associated structure morphisms:
there exists a unique map \(\gamma \colon W \to V\) such that 
\begin{equation}\label{eq:defgamma}
    \gamma \jmath_{n} = \iota_{n} \qquad \text{for any } n\geq 1
\end{equation}
induced by \(\id_{V_n}\).

\subsection*{Step 1}
Define \(Q\) in \(\cM\) to be the coequaliser
    \[
        \xymatrix @C=35pt {
            B \ot V \ot B \ar@<+0.6ex>[rrr]^-{m^2(B \ot (\id  * S ) \iota_V \ot B)} \ar@<-0.6ex>[rrr]_-{m^2(B \ot u \varepsilon  \iota_V \ot B)} &&& B \ar[r]^-{q} & Q
        }
    \]
in \(\cM\). Hereafter, we often omit the canonical map \(\iota_V\), when not explicitly needed.

\begin{proposition}\label{prop:Qbimonoid}
    The object \(Q\) inherits a bimonoid structure from \(B\).
\end{proposition}

\begin{invisible}
    \begin{proof} 
    Let us begin by showing that \(Q\) inherits a monoid structure. We use the following diagram as a reference:
    \[\begin{tikzcd}[ampersand replacement=\&,sep=3em]
    	{Q \ot B \ot V \ot B} \&\& {Q \ot B} \&\& {Q \ot Q} \\
    	\&\&\&\& Q \\
    	{B \ot B \ot V \ot B} \&\& {B\ot B} \&\& B \\
    	\\
    	\&\& {B \ot V \ot B \ot B} \&\& {B \ot V \ot B}
    	\arrow["{Q \ot m^2(B \ot \id*S \ot B)}"', shift right, from=1-1, to=1-3]
    	\arrow["{Q \ot m^2(B \ot u\varepsilon \ot B)}", shift left, from=1-1, to=1-3]
    	\arrow["{Q \ot q}", from=1-3, to=1-5]
    	\arrow[""{name=0, anchor=center, inner sep=0}, "{\exists! m}", dashed, from=1-5, to=2-5]
    	\arrow["{q \ot B \ot V \ot B}"{description}, from=3-1, to=1-1]
    	\arrow["\equiv"{description}, draw=none, from=3-1, to=1-3]
    	\arrow["{B \ot m^2(B \ot u\varepsilon \ot B)}", shift left, from=3-1, to=3-3]
    	\arrow["{B \ot m^2(B \ot \id*S \ot B)}"', shift right, from=3-1, to=3-3]
    	\arrow[""{name=2, anchor=center, inner sep=0}, "{q \ot B}", from=3-3, to=1-3]
    	\arrow[""{name=3, anchor=center, inner sep=0}, "qm", from=3-3, to=2-5]
    	\arrow["m", from=3-3, to=3-5]
    	\arrow["q"', from=3-5, to=2-5]
    	\arrow["{m^2(B \ot u\varepsilon \ot B) \ot B}"'{pos=0.7}, shift right, from=5-3, to=3-3]
    	\arrow["{m^2(B \ot \id*S \ot B) \ot B}"{pos=0.6}, shift left, from=5-3, to=3-3]
    	\arrow["{ \color{gray}\mathsf{associativity}}"{description}, draw=none, from=5-3, to=3-5]
    	\arrow["{B \ot V \ot m}"', from=5-3, to=5-5]
    	\arrow["{m^2(B \ot \id*S \ot B)}"{pos=0.3}, shift left, from=5-5, to=3-5]
    	\arrow["{m^2(B \ot u\varepsilon \ot B)}"'{pos=0.4}, shift right, from=5-5, to=3-5]
    	\arrow["{\color{gray}\begin{array}{c} \mathsf{universal} \\[-2pt]  \mathsf{property} \end{array}}"{description, pos=0.7}, draw=none, from=1-3, to=0]
    	\arrow["{\color{gray}\begin{array}{c}  \mathsf{universal} \\[-2pt]  \mathsf{property} \end{array}}"{description, pos=0.6}, draw=none, from=2-5, to=2]
    	\arrow["{\exists!\rho}"{description}, dashed, from=1-3, to=2-5]
    	\arrow["\color{gray}\equiv"{description}, draw=none, from=3-5, to=3]
        \arrow[""{name=1, anchor=center, inner sep=0}, "{m \ot V \ot B}"', from=3-1, to=5-5, rounded corners,
            to path={let \p1=([yshift=-0.6cm]\tikztostart.center), \p2=([yshift=-0.75cm]\tikztotarget.center) in
                     -- (\p1) -- (\x1, \y2) -- (\p2) \tikztonodes -- (\tikztotarget.south)}]
    	\arrow["{ \color{gray}\mathsf{(void)}}"{description}, draw=none, from=3-1, to=1]
    \end{tikzcd}\]
    where the notation \(m^2\) denotes indifferently the compositions \(m  (m  \ot B) = m  (B \ot m )\).
    
    By the associativity of the multiplication of \(B\), the bottom right square of the diagram commutes sequentially and hence, by the universal property of the coequaliser \((Q \ot B, q \ot B)\), there exists a unique morphism \(\rho \colon Q\ot B\to Q\) such that \(\rho (q\ot B)= qm \).
    
    The morphism \(qm \) coequalises \(B\ot m ^2(B\ot u \varepsilon \ot B)\) and \(B\ot m ^2 (B\ot (\id * S)\ot B)\) by associativity. Moreover, the top left square commutes sequentially. These two facts together imply that \(\rho\) coequalises the top left pair of parallel arrows, pre-composed with \(q\ot B\ot V\ot B\). Since \(q\ot B\ot V\ot B\) is an epimorphism, \(\rho\) actually coequalises the top left pair directly. By the universal property of the coequaliser \((Q\ot Q, Q\ot q)\), this induces a unique map \(m_Q \colon Q\ot Q\to Q\) such that \(m_Q(Q\ot q) = \rho\). By the commutativity of the top right square, \(q\colon B\to Q\) is multiplicative, and hence the associativity of \(m_Q\) follows. Clearly, the unit of \(Q\) is \(qu \).
    
    We proceed to show that \(Q\) inherits a comonoid structure. To this aim, remark first of all that since \(q\) coequalises \(m ^2(B\ot (\id * S)\iota_V \ot B)\) and \(m ^2(B\ot u \varepsilon \iota_V \ot B)\) by definition, \(q\) also coequalises \((\id * S)\iota_V\) and \(u  \varepsilon \iota_V\), as shown by the commutativity of
    \begin{equation}\label{eq:qonV}
    \begin{gathered}
        \begin{tikzcd}[ampersand replacement=\&,row sep=3.5em,column sep=5em]
        	{B \ot V \ot B} \& {B \ot B \ot B} \& B \\
        	V \& B \& Q.
        	\arrow["{B \ot u\varepsilon \ot B}"', shift right, from=1-1, to=1-2]
        	\arrow["{B \ot \id*S \ot B}", shift left, from=1-1, to=1-2]
        	\arrow["{m^2}", from=1-2, to=1-3]
        	\arrow["q", from=1-3, to=2-3]
        	\arrow["{u \ot V \ot u}", from=2-1, to=1-1]
        	\arrow["\color{gray}\equiv"{description}, draw=none, from=2-1, to=1-2]
        	\arrow["{u\varepsilon}"', shift right, from=2-1, to=2-2]
        	\arrow["{\id*S}", shift left, from=2-1, to=2-2]
        	\arrow["{u \ot B \ot u}", from=2-2, to=1-2]
        	\arrow["{\color{gray}\mathsf{unitality}}"{description}, draw=none, from=2-2, to=1-3]
        	\arrow["q"', from=2-2, to=2-3]
        \end{tikzcd}
    \end{gathered}
    \end{equation}
    Now, if we proved that \((q\ot q)\Delta \) coequalises \(m ^2(B\ot \id * S\ot B)\) and \(m ^2(B\ot u \varepsilon \ot B)\), then a unique morphism \(\Delta_Q\colon Q\to Q\ot Q\) would be induced, such that \(\Delta_Q q = (q\ot q)\Delta \), whence the coassociativity would easily follow. Consider the following diagram:
    
    \[\begin{tikzcd}[ampersand replacement=\&]
    	{B \ot V \ot B} \& {B \ot B \ot B} \& B \& Q \\
    	\& {B \ot B \ot B \ot B \ot B \ot B} \& {B \ot B} \\
    	\& {Q \ot Q \ot Q \ot Q \ot Q \ot Q} \&\& {Q \ot Q}
    	\arrow["{B \ot \id*S \ot B}", shift left, from=1-1, to=1-2]
    	\arrow["{B \ot u\varepsilon \ot B}"', shift right, from=1-1, to=1-2]
    	\arrow["{m^2}", from=1-2, to=1-3]
    	\arrow["{\Delta^{\ot 3}}", from=1-2, to=2-2]
    	\arrow["q", from=1-3, to=1-4]
    	\arrow["\Delta", dashed, from=1-4, to=3-4]
    	\arrow["{\color{gray} \mathsf{bimonoid}}"{description}, draw=none, from=2-2, to=1-3]
    	\arrow["{m_{\ot 2}^2}", from=2-2, to=2-3]
    	\arrow["{q^{\ot 6}}", from=2-2, to=3-2]
    	\arrow["{\color{gray} q \mathsf{\,of\,monoids}}"{description}, draw=none, from=2-2, to=3-4]
    	\arrow["{\begin{array}{c} \begin{array}{c} \color{gray}\mathsf{universal} \\[-6pt] \color{gray}\mathsf{property} \end{array} \end{array}}"{description}, draw=none, from=2-3, to=1-4]
    	\arrow["\Delta", from=1-3, to=2-3]
    	\arrow["{q \ot q}"{description}, from=2-3, to=3-4]
    	\arrow["{m_{\ot 2}^2}"', from=3-2, to=3-4]
    \end{tikzcd}\]
    where by \(m_{\ot 2}^2\) we mean
    \[m_{\ot 2}^2 = \left({m }^2 \ot {m }^2\right)\left(B \ot B \ot \brd_{B,B} \ot B \ot B\right)\left(B \ot \brd_{B,B} \ot \brd_{B,B} \ot B\right).\]
    This shows that \((q\ot q)\Delta \) coequalises \(m ^2(B\ot \id * S\ot B)\) and \(m ^2(B\ot u \varepsilon \ot B)\) if \((q\ot q)\Delta \) coequalises \(\id * S\) and \(u \varepsilon \). The latter follows from the direct computation:
    \begin{align*}
        & (q \ot q)\Delta m (B \ot S)\Delta   \\
        \overset{\eqref{eq:bialgebra-compatibility}}&{=}  (q \ot q)(m  \ot m )(B \ot \brd_{B,B} \ot B)(\Delta  \ot \Delta )(B \ot S)\Delta  \\
        \overset{(\dagger)}&{=} (q \ot q)(m  \ot m )(B \ot \brd_{B,B} \ot B)(B \ot B \ot S \ot S)(B \ot B \ot \brd_{B,B})(\Delta  \ot \Delta )\Delta  \\
        \overset{(*)}&{=} (q \ot q)(m  \ot m )(B \ot S \ot B \ot S)(B \ot \brd_{B,B} \ot B)(B \ot B \ot \brd_{B,B})(\Delta  \ot \Delta )\Delta  \\
        \overset{\eqref{eq:hexagon}}&{=} (q \ot q)(m  \ot m )(B \ot S \ot B \ot S)(B \ot \brd_{B \ot B,B})(B \ot \Delta  \ot B)(B \ot \Delta )\Delta  \\
        \overset{(*)}&{=} (q \ot q)(m  \ot m )(B \ot S \ot B \ot S)(B \ot B \ot \Delta )(B \ot \brd_{B,B})(B \ot \Delta )\Delta  \\
        \overset{\eqref{eq:qonV}}&{=} (q \ot q)(m  \ot u )(B \ot S \ot \varepsilon )(B \ot \brd_{B,B})(B \ot \Delta )\Delta  \\
        \overset{(\ddagger)}&{=} (q \ot q)(m  \ot u )(B \ot S \ot \I)(B \ot \brd_{\I,B})(B \ot \clambda ^{-1})\Delta  \\
        \overset{\eqref{eq:braidI}}&{=} (q \ot q)(m  \ot u )(B \ot S \ot \I)(B \ot \crho ^{-1})\Delta  \\
        & = (Q \ot q)(Q \ot u )\crho_Q^{-1}qm (B \ot S)\Delta  \\
        \overset{\eqref{eq:qonV}}&{=} (q \ot q)(u  \ot u )\crho_\I^{-1}\varepsilon  \\
    \overset{\eqref{eq:bialgebra-compatibility}}&{=} (q \ot q)\Delta  u \varepsilon  
    \end{align*}
    where \((\dagger)\) follows by anti-comultiplicativity of \(S\), \((*)\) by naturality of \(\brd\), and \((\ddagger)\) by counitality of \(\Delta \).
    Since \(\varepsilon \) coequalises \(\id  * S\) and \(u  \varepsilon \), the diagram
    \[\begin{tikzcd}[ampersand replacement=\&,column sep=4em]
    	{B \ot V \ot B} \& {B \ot B \ot B} \& B \& Q \\
    	\& {\I \ot \I \ot \I} \&\& \I
    	\arrow["{B \ot \id*S \ot B}", shift left, from=1-1, to=1-2]
    	\arrow["{B \ot u\varepsilon \ot B}"', shift right, from=1-1, to=1-2]
    	\arrow["{m^2}", from=1-2, to=1-3]
    	\arrow["{\varepsilon \ot \varepsilon \ot \varepsilon}", from=1-2, to=2-2]
    	\arrow["q", from=1-3, to=1-4]
    	\arrow["\varepsilon"{description}, from=1-3, to=2-4]
    	\arrow["\varepsilon", dashed, from=1-4, to=2-4]
    	\arrow[""{name=0, anchor=center, inner sep=0}, "\cong"{description}, from=2-2, to=2-4]
    	\arrow["{\color{gray}\mathsf{bimonoid}}"{description}, draw=none, from=1-3, to=0]
    \end{tikzcd}\]
    entails that there exists a unique \(\varepsilon_Q\colon Q \to \I\) such that \(\varepsilon_Qq = \varepsilon \), which is a counit for \(\Delta_Q\).
    
    Since the monoid and comonoid structure on \(Q\) are induced by those on \(B\), the bimonoid compatibility follows.
    \end{proof}
\end{invisible}

We break the proof down into smaller results that we can reuse later on.

\begin{lemma}[{\cite[Lemma 2.1.1]{ArdiLaiachiMen}, later published as \cite[2.1 Lemma]{Porst}}]\label{lem:coeq_monoid}
    Let \((A,m,u)\) be a monoid in a monoidal category \((\cN,\ot,\I)\) with coequalisers. Let \(f,g \colon X \to A\) be two morphisms in \(\cN\) and consider the coequaliser
    \[
        \xymatrix @C=35pt {
            A \ot X \ot A \ar@<+0.6ex>[rr]^-{m ^2(A \ot f \ot A)} \ar@<-0.6ex>[rr]_-{m ^2(A \ot g \ot A)} && A \ar[r]^-{q} & Q
        }
    \]
    in \(\cN\). Assume that \(\ot\) preserves coequalisers. Then \(Q\) carries a unique monoid structure such that \(q\) is a morphism of monoids.
\end{lemma}

\begin{invisible}
    \begin{lemma}
        Let \(X\) be an object, \((A,m,u)\) be a monoid, and \(f,g \colon X \to A\) be two morphisms in \(\cM\). Define \(Q\) in \(\cM\) via the coequaliser
        \[
            \xymatrix @C=35pt {
                A \ot X \ot A \ar@<+0.6ex>[rr]^-{m ^2(A \ot f \ot A)} \ar@<-0.6ex>[rr]_-{m ^2(A \ot g \ot A)} && A \ar[r]^-{q} & Q
            }
        \]
        in \(\cM\). Then, the object \(Q\) inherits a monoid structure from \(A\), i.e., such that \(q\) is a morphism of monoids.
    \end{lemma}

    \begin{proof}
        We use the following diagram as a reference:
        \[\begin{tikzcd}[ampersand replacement=\&,sep=3em]
        	{Q \ot A \ot X \ot A} \&\& {Q \ot A} \&\& {Q \ot Q} \\
        	\&\&\&\& Q \\
        	{A \ot A \ot X \ot A} \&\& {A\ot A} \&\& A \\
        	\\
        	\&\& {A \ot X \ot A \ot A} \&\& {A \ot X \ot A}
        	\arrow["{Q \ot m^2(A \ot f \ot A)}"', shift right, from=1-1, to=1-3]
        	\arrow["{Q \ot m^2(A \ot g \ot A)}", shift left, from=1-1, to=1-3]
        	\arrow["{Q \ot q}", from=1-3, to=1-5]
        	\arrow[""{name=0, anchor=center, inner sep=0}, "{\exists! m}", dashed, from=1-5, to=2-5]
        	\arrow["{q \ot A \ot X \ot A}"{description}, from=3-1, to=1-1]
        	\arrow["\color{gray}\equiv"{description}, draw=none, from=3-1, to=1-3]
        	\arrow["{A \ot m^2(A \ot g \ot A)}", shift left, from=3-1, to=3-3]
        	\arrow["{A \ot m^2(A \ot f \ot A)}"', shift right, from=3-1, to=3-3]
        	\arrow[""{name=2, anchor=center, inner sep=0}, "{q \ot A}", from=3-3, to=1-3]
        	\arrow[""{name=3, anchor=center, inner sep=0}, "qm", from=3-3, to=2-5]
        	\arrow["m", from=3-3, to=3-5]
        	\arrow["q"', from=3-5, to=2-5]
        	\arrow["{m^2(A \ot g \ot A) \ot A}"'{pos=0.7}, shift right, from=5-3, to=3-3]
        	\arrow["{m^2(A \ot f \ot A) \ot A}"{pos=0.6}, shift left, from=5-3, to=3-3]
        	\arrow["{ \color{gray}\mathsf{associativity}}"{description}, draw=none, from=5-3, to=3-5]
        	\arrow["{A \ot X \ot m}"', from=5-3, to=5-5]
        	\arrow["{m^2(A \ot f \ot A)}"{pos=0.3}, shift left, from=5-5, to=3-5]
        	\arrow["{m^2(A \ot g \ot A)}"'{pos=0.4}, shift right, from=5-5, to=3-5]
        	\arrow["{\color{gray}\begin{array}{c} \mathsf{universal} \\[-2pt]  \mathsf{property} \end{array}}"{description, pos=0.7}, draw=none, from=1-3, to=0]
        	\arrow["{\color{gray}\begin{array}{c}  \mathsf{universal} \\[-2pt]  \mathsf{property} \end{array}}"{description, pos=0.6}, draw=none, from=2-5, to=2]
        	\arrow["{\exists!\rho}"{description}, dashed, from=1-3, to=2-5]
        	\arrow["\color{gray}\equiv"{description}, draw=none, from=3-5, to=3]
            \arrow[""{name=1, anchor=center, inner sep=0}, "{m \ot X \ot A}"', from=3-1, to=5-5, rounded corners,
                to path={let \p1=([yshift=-0.6cm]\tikztostart.center), \p2=([yshift=-0.75cm]\tikztotarget.center) in
                         -- (\p1) -- (\x1, \y2) -- (\p2) \tikztonodes -- (\tikztotarget.south)}]
        	\arrow["{ \color{gray}\mathsf{(void)}}"{description}, draw=none, from=3-1, to=1]
        \end{tikzcd}\]
        where \(m^2\) denotes indifferently the compositions \(m  (m  \ot A) = m  (A \ot m )\).
        
        By the associativity of the multiplication of \(A\), the bottom right square of the diagram commutes sequentially and hence, by the universal property of the coequaliser \((Q \ot A, q \ot A)\), there exists a unique morphism \(\rho \colon Q\ot A\to Q\) such that \(\rho (q\ot A)= qm \).
        
        The morphism \(qm \) coequalises \(A\ot m ^2(A\ot f \ot A) \otimes A\) and \(A\ot m ^2 (A\ot g \ot A) \otimes A\) by associativity. Moreover, the top left square commutes sequentially. These two facts together imply that \(\rho\) coequalises the top left pair of parallel arrows, pre-composed with \(q\ot A\ot X\ot A\). Since \(q\ot A\ot X\ot A\) is an epimorphism, \(\rho\) actually coequalises the top left pair directly. By the universal property of the coequaliser \((Q\ot Q, Q\ot q)\), this induces a unique map \(m_Q \colon Q\ot Q\to Q\) such that \(m_Q(Q\ot q) = \rho\). By the commutativity of the top right square, \(q\colon A\to Q\) is multiplicative, and hence the associativity of \(m_Q\) follows. Clearly, the unit of \(Q\) is \(qu \).
    \end{proof}
\end{invisible}

\begin{remark}[{\cite[Lemma 2.1.2]{ArdiLaiachiMen}}]\label{rem:coeq_algs}
    In the setting of \zcref{lem:coeq_monoid}, \(em^2(A \ot f \ot A) \!=\! em^2(A \ot g \ot A)\) implies \(ef=eg\) for any morphism \(e \colon A \to Y\) in \(\cN\). The converse is true whenever \(e\) is a morphism of monoids.
    In particular, \(q \colon A \to Q\) satisfies \(qf=qg\).
\end{remark}

\begin{invisible}
    \begin{remark}
        In the setting of \zcref{lem:coeq_monoid}, an monoid map \(e \colon A \to A'\) coequalises 
        \[
            \xymatrix @C=75pt {
                A \ot X \ot A \ar@<+0.7ex>[r]^-{m^2(A \ot f \ot A)} \ar@<-0.7ex>[r]_-{m^2(A \ot g \ot A)} & A
            }
        \]
        if and only if it coequalises \(f\) and \(g\).
        In fact, one direction is trivial, while the other is shown by the commutativity of
        \begin{equation}\label{eq:qonV}
            \begin{gathered}
                \begin{tikzcd}[ampersand replacement=\&,row sep=3.5em,column sep=5em]
                    {A \ot X \ot A} \& {A \ot A \ot A} \& A \\
                    X \& A \& A'.
                    \arrow["{A \ot g \ot A}"', shift right, from=1-1, to=1-2]
                    \arrow["{A \ot f \ot A}", shift left, from=1-1, to=1-2]
                    \arrow["{m^2}", from=1-2, to=1-3]
                    \arrow["e", from=1-3, to=2-3]
                    \arrow["{u \ot X \ot u}", from=2-1, to=1-1]
                    \arrow["\color{gray}\equiv"{description}, draw=none, from=2-1, to=1-2]
                    \arrow["{g}"', shift right, from=2-1, to=2-2]
                    \arrow["{f}", shift left, from=2-1, to=2-2]
                    \arrow["{u \ot A \ot u}", from=2-2, to=1-2]
                    \arrow["{\color{gray}\mathsf{unitality}}"{description}, draw=none, from=2-2, to=1-3]
                    \arrow["e"', from=2-2, to=2-3]
                \end{tikzcd}
            \end{gathered}
        \end{equation}
        In particular, \(q \colon A \to Q\) coequalises \(g\) and \(f\).
    \end{remark}
\end{invisible}

\begin{proof}[Proof of \zcref{prop:Qbimonoid}]
    \zcref{lem:coeq_monoid} ensures that \(Q\) inherits a monoid structure such that \(q \colon B \to Q\) is a morphism of monoids. 
    Hence, we proceed to show that \(Q\) inherits also a comonoid structure such that \(q\) is a morphism of comonoids, too. 
    
    If we proved that \((q\ot q)\Delta \) coequalises \(m ^2(B\ot (\id  * S )\ot B)\) and \(m ^2(B\ot u \varepsilon \ot B)\), then a unique morphism \(\Delta_Q\colon Q\to Q\ot Q\) would be induced, such that \(\Delta_Q q = (q\ot q)\Delta \), whence the coassociativity would easily follow.
    
    By \zcref{rem:coeq_algs}, \((q\ot q)\Delta \) coequalises \(m ^2(B\ot (\id  * S )\ot B)\) and \(m ^2(B\ot u \varepsilon \ot B)\) if and only if it coequalises \(\id  * S \) and \(u \varepsilon \). The fact that \((q\ot q)\Delta \) does coequalise \(\id  * S \) and \(u \varepsilon \) follows from the fact that
    \begin{equation}\label{eq:qonV}
        q(\id * S) \overset{\text{Rem } \ref{rem:coeq_algs}}{=} qu\varepsilon
    \end{equation}
    and the direct computation:
    \begin{align*}
        & (q \ot q)\Delta m (B \ot S )\Delta \\ 
        \overset{\eqref{eq:bialgebra-compatibility}}&{=}  (q \ot q)(m  \ot m )(B \ot \brd_{B,B} \ot B)(\Delta  \ot \Delta )(B \ot S )\Delta  \\
        \overset{(\dagger)}&{=} (q \ot q)(m  \ot m )(B \ot \brd_{B,B} \ot B)(B \ot B \ot S  \ot S )(B \ot B \ot \brd_{B,B})(\Delta  \ot \Delta )\Delta  \\
        \overset{(*)}&{=} (q \ot q)(m  \ot m )(B \ot S  \ot B \ot S )(B \ot \brd_{B,B} \ot B)(B \ot B \ot \brd_{B,B})(\Delta  \ot \Delta )\Delta  \\
        \overset{\eqref{eq:hexagon}}&{=} (q \ot q)(m  \ot m )(B \ot S  \ot B \ot S )(B \ot \brd_{B \ot B,B})(B \ot \Delta  \ot B)(B \ot \Delta )\Delta  \\
        \overset{(*)}&{=} (q \ot q)(m  \ot m )(B \ot S  \ot B \ot S )(B \ot B \ot \Delta )(B \ot \brd_{B,B})(B \ot \Delta )\Delta  \\
        \overset{\eqref{eq:qonV}}&{=} (q \ot q)(m  \ot u )(B \ot S  \ot \varepsilon )(B \ot \brd_{B,B})(B \ot \Delta )\Delta  \\
        \overset{(\ddagger)}&{=} (q \ot q)(m  \ot u )(B \ot S  \ot \I)(B \ot \brd_{\I,B})(B \ot \clambda_B ^{-1})\Delta  \\
        \overset{\eqref{eq:braidI}}&{=} (q \ot q)(m  \ot u )(B \ot S  \ot \I)(B \ot \crho_B^{-1})\Delta  \\
        & = (Q \ot q)(Q \ot u)\crho_Q^{-1}qm (B \ot S )\Delta  \\
        \overset{\eqref{eq:qonV}}&{=} (q \ot q)(u  \ot u )\crho_\I^{-1}\varepsilon \overset{\eqref{eq:bialgebra-compatibility}}{=} (q \ot q)\Delta  u \varepsilon  
    \end{align*}
    where \((\dagger)\) follows by anti-comultiplicativity of \(S \), \((*)\) by naturality of \(\brd\), and \((\ddagger)\) by counitality of \(\Delta \).
    Similarly, since \(\varepsilon \) coequalises \(\id  * S \) and \(u  \varepsilon \), the diagram
    \[\begin{tikzcd}[ampersand replacement=\&,column sep=4em]
    	{B \ot V \ot B} \& {B \ot B \ot B} \& B \& Q \\
    	\& {\I \ot \I \ot \I} \&\& \I
    	\arrow["{B \ot \id*S  \ot B}", shift left, from=1-1, to=1-2]
    	\arrow["{B \ot u\varepsilon \ot B}"', shift right, from=1-1, to=1-2]
    	\arrow["{m^2}", from=1-2, to=1-3]
    	\arrow["{\varepsilon \ot \varepsilon \ot \varepsilon}", from=1-2, to=2-2]
    	\arrow["q", from=1-3, to=1-4]
    	\arrow["\varepsilon"{description}, from=1-3, to=2-4]
    	\arrow["\varepsilon", dashed, from=1-4, to=2-4]
    	\arrow[""{name=0, anchor=center, inner sep=0}, "\cong"{description}, from=2-2, to=2-4]
    	\arrow["{\color{gray}\mathsf{bimonoid}}"{description}, draw=none, from=1-3, to=0]
    \end{tikzcd}\]
    entails that there exists a unique \(\varepsilon_Q\colon Q \to \I\) such that \(\varepsilon_Qq = \varepsilon \), which is a counit for \(\Delta_Q\).
    
    Since the monoid and comonoid structure on \(Q\) are induced by those on \(B\), the bimonoid compatibility follows.
\end{proof}

\subsection*{Step 2}
Define \(P\) in \(\cM\) to be the coequaliser
    \[
        \xymatrix @C=35pt {
            Q \ot W \ot Q \ar@<+0.6ex>[rrr]^-{m^2(Q \ot q(S  * \id ) \iota_V \gamma \ot Q)} \ar@<-0.6ex>[rrr]_-{m^2(Q \ot qu \varepsilon  \iota_V \gamma \ot Q)} &&& Q \ar[r]^-{p} & P
        }
    \]
in \(\cM\). By an argument similar to the one above, \(P\) is a bimonoid in \(\cM\). In fact, \(P\) is a monoid by \zcref{lem:coeq_monoid} and it is a bimonoid if and only if \((p \ot p)\Delta_Q\) and \(\varepsilon_Q\) coequalise \(q(S  * \id )\) and \(qu \varepsilon\). As before, the last two facts can be checked directly. For instance, an easy string diagram verification, or mimicking the argument above, should convince the reader that
\begin{align*}
    & (p \ot p)\Delta_Qq(S  * \id ) = (pq \ot pq)\Delta(S  * \id ) \\
    & = (pq \ot pq)(S*\id \ot B) (B \ot m) (B \ot S \ot B) (\brd_{B,B} \ot B)(\Delta \ot B)\Delta \\
    & = (pq \ot pq)(u\varepsilon \ot B) (B \ot m) (B \ot S \ot B) (\brd_{B,B} \ot B)(\Delta \ot B)\Delta \\
    & = (pq \ot pq)(u \ot B)(\I \ot m) (\I \ot S \ot B) (\brd_{B,\I} \ot B)(\crho_B^{-1} \ot B)\Delta \\
    & = (pq \ot pq)(u \ot B)\clambda_B^{-1}(S*\id) = (pq \ot pq)(u \ot B)\clambda_B^{-1}u\varepsilon \\
    & = (p \ot p)\Delta_Qqu\varepsilon.
\end{align*}

Note that \(P\) will be our candidate right Hopf monoid in \(\cM\). However, in order to show that it inherits a right antipode, we need a few more steps to show that \(S \colon B \to B\) induces \(S_P \colon P \to P\).

\subsection*{Step 3}
Define \(Q'\) in \(\cM\) to be the coequaliser
    \[
        \xymatrix @C=35pt {
            B \ot W \ot B \ar@<+0.6ex>[rrr]^-{m^2(B \ot (S * \id) \iota_V\gamma \ot B)} \ar@<-0.6ex>[rrr]_-{m^2(B \ot u\varepsilon \iota_V\gamma \ot B)} &&& B \ar[r]^-{q'} & Q'
        }
    \]
in \(\cM\) and \(P'\) to be the coequaliser
    \[
        \xymatrix @C=35pt {
            Q' \ot V \ot Q' \ar@<+0.6ex>[rrr]^-{m^2(Q' \ot q'(\id*S) \iota_V \ot Q')} \ar@<-0.6ex>[rrr]_-{m^2(Q' \ot q'u\varepsilon \iota_V \ot Q')} &&& Q' \ar[r]^-{p'} & P'
        }
    \]
in \(\cM\).
We show that \(pq = p'q'\), by taking advantage of the fact that  colimits commute with colimits; see \cite[IX.2]{Maclane}. Remark that, by construction, \(V\) is the coproduct of the \(V_k\)'s in \(\cM\).
If we consider the morphisms \(V_0 \xrightarrow{=} V_1\) and \(V_i \xrightarrow{=} V_i\) for every \(i \geq 1\) in \(\cM\), then there exists a unique morphism \(\pi \colon V\to W\) in \(\cM\) such that
\[
    \xymatrix{
        V_0 \ar[r]^-{\iota_0} \ar[d]_-{=} & V \ar@{.>}[d]^-{\pi} & W \ar[d]^-{=} \ar[l]_-{\gamma} \\
        V_1 \ar[r]_-{\jmath_1} & W & W \ar[l]^-{=}
    }
\]
commutes. In particular, \(\pi \gamma = \id_W\) and \(\pi\) is an epimorphism. Since in the following diagram all the rows and all the columns are coequaliser diagrams 
{\small
\[\begin{tikzcd}[ampersand replacement=\&,cramped,column sep=small]
	{(B \!\ot V \!\ot B) \!\ot V \!\ot (B \!\ot V \!\ot B)} \&\&\& {(B \!\ot V \!\ot B) \!\ot W \!\ot (B \!\ot V \!\ot B)} \& {B \!\ot W \!\ot B} \& {Q \!\ot W \!\ot Q} \\
	{(B \!\ot W \!\ot B) \!\ot V \!\ot (B \!\ot W \!\ot B)} \\
	{B \!\ot V \!\ot B} \&\&\&\& B \& Q \\
	{Q'\!\ot V \!\ot Q'} \&\&\&\& {Q'} \& {P'\cong P}
	\arrow["{(B \!\ot V \!\ot B) \!\ot \pi \!\ot (B \!\ot V \!\ot B)}"{yshift=5pt}, from=1-1, to=1-4]
	\arrow["{(B \!\ot \pi \!\ot B) \!\ot V \!\ot (B \!\ot \pi \!\ot B)}"{description}, from=1-1, to=2-1]
	\arrow[shift right, from=1-4, to=1-5]
	\arrow[shift left, from=1-4, to=1-5]
	\arrow["{q \!\ot W \!\ot q}"{yshift=5pt}, dashed, from=1-5, to=1-6]
	\arrow[shift right, from=1-5, to=3-5]
	\arrow[shift left, from=1-5, to=3-5]
	\arrow[shift left, from=1-6, to=3-6]
	\arrow[shift right, from=1-6, to=3-6]
	\arrow[shift right, from=2-1, to=3-1]
	\arrow[shift left, from=2-1, to=3-1]
	\arrow[shift right, from=3-1, to=3-5]
	\arrow[shift left, from=3-1, to=3-5]
	\arrow["{q'\!\ot V \!\ot q'}"{description}, dashed, from=3-1, to=4-1]
	\arrow["q", dashed, from=3-5, to=3-6]
	\arrow["{q'}"', dashed, from=3-5, to=4-5]
	\arrow["p", dashed, from=3-6, to=4-6]
	\arrow[shift left, from=4-1, to=4-5]
	\arrow[shift right, from=4-1, to=4-5]
	\arrow["{p'}"', dashed, from=4-5, to=4-6]
\end{tikzcd}\]
}
we conclude that 
\begin{equation}\label{eq:pq}
    pq= p'q'.
\end{equation}

\subsection*{Step 4}

We now show that \(S\colon B\to B\) induces \(S'\colon Q \to Q'\), which in turn induces \(S_P \colon P\to P\). To this aim, one needs the following lemma. 
\begin{lemma}
Let \(B\) be a bimonoid in \(\cM\) and let \(S\colon B\to B\) be an anti-homomorphism of bimonoids. We have 
\begin{equation}\label{eq:SidS}
    S(\id * S) = (S * \id)S \qquad \text{and} \qquad S(S * \id) = (\id * S)S.
\end{equation}
\end{lemma}

\begin{proof}
The first equality follows from the commutativity of the diagram
%
\[\begin{tikzcd}
	{B \ot B} &&& {B \ot B} &&& {B \ot B} \\
	B & {B \ot B} && {B \ot B} && {B \ot B} & {B \ot B} \\
	V &&{}&{}& B && B
    \arrow["{\color{gray}\quad\qquad\text{\sf anti-comultiplicativity}}"{description}, draw=none, from=1-1, to=2-2]
	\arrow["{\brd_{B,B}}", from=1-1, to=1-4]
	\arrow["{\brd_{B,B}}", from=1-4, to=1-7]
	\arrow[sloped,"{B \ot S}", from=1-4, to=2-6]
	\arrow["{S \ot B}", from=1-7, to=2-7]
	\arrow["\Delta", from=2-1, to=1-1]
	\arrow[sloped,"{S \ot S}", from=2-2, to=1-4]
	\arrow[sloped,"{B \ot S}", from=2-2, to=2-4]
	\arrow["{S \ot S}", from=2-4, to=2-6]
	\arrow["m", from=2-4, to=3-5]
	\arrow["{\brd_{B,B}}", from=2-6, to=2-7]
	\arrow["m", from=2-7, to=3-7]
	\arrow["S", from=3-1, to=2-1]
	\arrow[sloped, "\Delta", from=3-1, to=2-2]
	\arrow["{\id*S}"', from=3-1, to=3-5]
	\arrow["S"', from=3-5, to=3-7]
    \arrow["\color{gray}{\equiv}"{description}, draw=none, from=1-4, to=2-4]
    \arrow["{\color{gray}\brd\text{ \sf natural}}"{xshift=-10pt,description}, draw=none, from=2-6, to=1-7]
    \arrow["{\color{gray}\text{\sf anti-multiplicativity}}"{description}, draw=none, from=2-4, to=3-7]
    \arrow["{{\color{gray}\mathsf{definition}}}"{description, pos=0.4}, draw=none,  from=2-2, to=3-4]
    \arrow["{=}", from=1-1, to=1-7, rounded corners,
        to path={-- ([yshift=0.7cm]\tikztostart.center) -- ([yshift=0.7cm]\tikztotarget.center) \tikztonodes -- (\tikztotarget.north)}]
\end{tikzcd}\]
Similarly, one can check the other equality. 
\end{proof}

Let us consider the diagram
%
\[
\begin{tikzcd}[ampersand replacement=\&,row sep=3.15em]
	{B \ot V \ot B} \&\& {B \ot B \ot B} \&\& B \& Q \\
	{B \ot B \ot B} \&\& {B \ot B \ot B} \&\& B \\
	\&\& {Q'\ot Q'\ot Q'} \&\&\& {Q'}
	\arrow["{\color{gray}\begin{array}{c}\mathsf{anti\mbox{-}}\\ \mathsf{multiplicativity}\end{array}}"{description,xshift=5pt}, draw=none, from=1-3, to=2-5]
	\arrow["{B \ot (\id*S)\iota_V \ot B}", shift left, from=1-1, to=1-3]
	\arrow["{B \ot u\varepsilon\iota_V \ot B}"', shift right, from=1-1, to=1-3]
	\arrow["{S \ot S_V \ot S}"{description}, from=1-1, to=2-1]
	\arrow["\color{gray}{\eqref{eq:SidS}}"{description}, draw=none, from=1-1, to=2-3]
	\arrow["{m^2}", from=1-3, to=1-5]
	\arrow["{S \ot S \ot S}"{description}, from=1-3, to=2-3]
	\arrow["q", from=1-5, to=1-6]
	\arrow["S", from=1-5, to=2-5]
	\arrow["{\exists!S'}", dashed, from=1-6, to=3-6]
	\arrow["{B \ot (S*\id)\iota_V \ot B}", shift left, from=2-1, to=2-3]
	\arrow["{B \ot u\varepsilon\iota_V \ot B}"', shift right, from=2-1, to=2-3]
	\arrow["{m'}"', from=2-3, to=2-5]
	\arrow["{q'\ot q'\ot q'}"{description}, from=2-3, to=3-3]
	\arrow["\color{gray}{q' \, \mathsf{multiplicative}}"{description}, draw=none, from=2-3, to=3-6]
	\arrow["{q'}", from=2-5, to=3-6]
	\arrow["{m'}"'{pos=0.4}, from=3-3, to=3-6]
\end{tikzcd}
\]
where \(m'= m^{\mathrm{op}}(m^{\mathrm{op}} \ot \id)\). By \zcref{rem:coeq_algs}, \(q'(S * \id) \iota_V = q'u \varepsilon \iota_V\) and therefore \(q'S\) coequalises \(m^2(B \ot (\id*S)\iota_V \ot B)\) and \(m^2(B \ot u\varepsilon\iota_V \ot B)\) and so there exists a unique \(S'\colon Q \to Q'\) such that
\begin{equation}\label{eq:S'q'}
    S'q = q'S
\end{equation}
as in the diagram. Note that \(S'\) becomes anti-multiplicative and anti-comultiplicative. 

Then, consider the diagram
\[\begin{tikzcd}[ampersand replacement=\&,row sep=3.15em,column sep=3.4em]
	{Q \ot V \ot Q} \&\& {Q \ot B \ot Q} \& {Q \ot Q \ot Q} \& Q \& P \\
	{Q' \ot V \ot Q'} \&\& {Q' \ot B \ot Q'} \& {Q'\ot Q'\ot Q'} \& {Q'} \\
	\&\& {P \ot Q \ot P} \& {P \ot P \ot P} \&\& P
	\arrow["{\color{gray}\begin{array}{c}\mathsf{anti\mbox{-}}\\ \mathsf{multiplicativity}\end{array}}"{description}, draw=none, from=1-4, to=2-5]
	\arrow["{Q \ot (S*\id)\iota_V\gamma \ot Q}", shift left, from=1-1, to=1-3]
	\arrow["{Q \ot u\varepsilon \iota_V\gamma \ot Q}"', shift right, from=1-1, to=1-3]
	\arrow["{S' \ot s\gamma \ot S'}"{description}, from=1-1, to=2-1]
	\arrow["{\eqref{eq:SidS}}"{description}, draw=none, from=1-1, to=2-3]
	\arrow["{Q \ot q \ot Q}", from=1-3, to=1-4]
	\arrow["{S' \ot S \ot S'}"{description}, from=1-3, to=2-3]
	\arrow["{\eqref{eq:S'q'}}"{description}, draw=none, from=1-3, to=2-4]
	\arrow["{m^2}", from=1-4, to=1-5]
	\arrow["{S'\ot S'\ot S'}"{description}, from=1-4, to=2-4]
	\arrow["p", from=1-5, to=1-6]
	\arrow["{S'}"', from=1-5, to=2-5]
	\arrow["S_P", dashed, from=1-6, to=3-6]
	\arrow["{Q' \ot u\varepsilon \iota_V \ot Q'}"', shift right, from=2-1, to=2-3]
	\arrow["{Q' \ot (\id*S)\iota_V \ot Q'}", shift left, from=2-1, to=2-3]
	\arrow["{Q'\ot q'\ot Q'}"', from=2-3, to=2-4]
	\arrow["{p'\ot q \ot p'}"{description}, from=2-3, to=3-3]
	\arrow["{m''}", from=2-4, to=2-5]
	\arrow["{p'\ot p'\ot p'}"{description}, from=2-4, to=3-4]
	\arrow["{\color{gray}p'\,\mathsf{multiplicative}}"{description, pos=0.4}, draw=none, from=2-4, to=3-6]
	\arrow["{p'}"{description}, from=2-5, to=3-6]
	\arrow["{P \ot p \ot P}"', from=3-3, to=3-4]
	\arrow["{m''}"', from=3-4, to=3-6]
\end{tikzcd}\]
Since \(q\) coequalises \((\id * S)\iota_V\) and \(u_B \varepsilon_B\iota_V\), as we remarked in \zcref{rem:coeq_algs}, \(p'S'\) coequalises \(m^2(Q \ot q(S*\id)\iota_V\gamma \ot Q)\) and \(m^2(Q \ot qu\varepsilon \iota_V\gamma \ot Q)\) and therefore there exists a unique \(S_P \colon P \to P\) such that
\begin{equation}\label{eq:Spp'S'}
    S_Pp = p'S'
\end{equation}
as in the diagram.

\subsection*{Step 5}

To conclude, we need to verify that \(S_P\) is a right antipode. Since \(p\) and \(q\) are epimorphisms and since the external paths in the following diagram coincide because both \(q\) and \(p\) are morphisms of bimonoids,
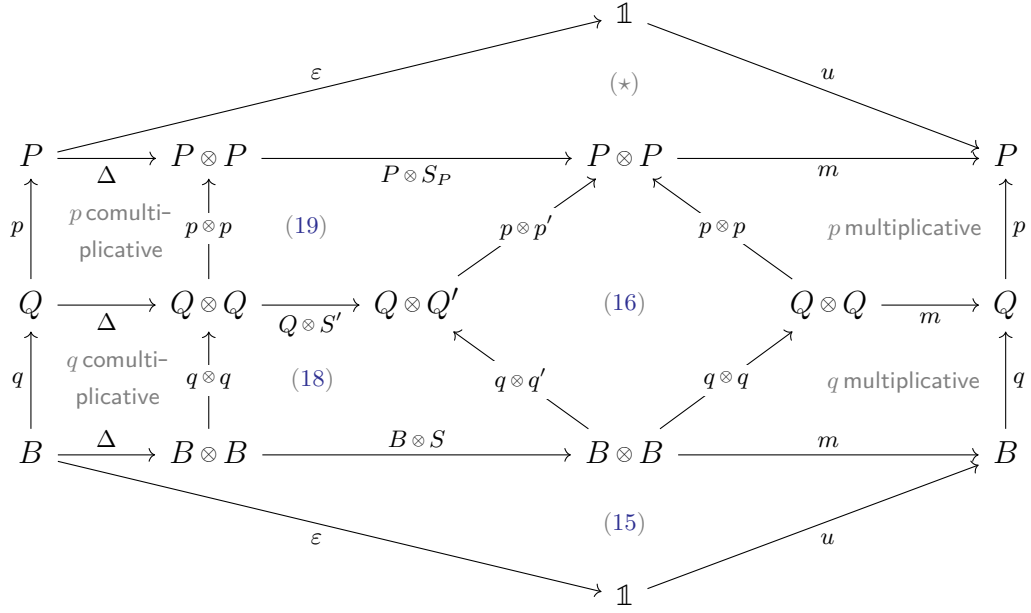
\[\begin{tikzcd}[ampersand replacement=\&,sep=3.15em]
	\&\&\& \I \&\& \\
	P \& {P \ot P} \&\& {P \ot P} \&\& P \\
	Q \& {Q \ot Q} \& {Q \ot Q'} \&\& {Q \ot Q} \& Q \\
	B \& {B \ot B} \&\& {B \ot B} \&\& B \\
	\&\&\& \I
	\arrow["u", from=1-4, to=2-6]
	\arrow["\varepsilon", from=2-1, to=1-4]
	\arrow["\Delta"', from=2-1, to=2-2]
	\arrow["{P \ot S_P}"', from=2-2, to=2-4]
	\arrow["{\color{gray}\eqref{eq:Spp'S'}}"{description}, draw=none, from=2-2, to=3-3]
	\arrow["{\color{gray}(\star)}"{description}, draw=none, from=2-4, to=1-4]
	\arrow["m"', from=2-4, to=2-6]
	\arrow["\color{gray}{\eqref{eq:pq}}"{description}, draw=none, from=2-4, to=4-4]
	\arrow["p", from=3-1, to=2-1]
	\arrow["{\color{gray}\begin{array}{c}p\,\mathsf{comulti\mbox{-}}\\ \mathsf{plicative}\end{array}}"{description}, draw=none, from=3-1, to=2-2]
	\arrow["\Delta"', from=3-1, to=3-2]
	\arrow["{\color{gray}\begin{array}{c}q\,\mathsf{comulti\mbox{-}}\\ \mathsf{plicative}\end{array}}"{description}, draw=none, from=3-1, to=4-2]
	\arrow["{p \ot p}"{description}, from=3-2, to=2-2]
	\arrow["{Q \ot S'}"', from=3-2, to=3-3]
	\arrow["{p \ot p'}"{description}, from=3-3, to=2-4]
	\arrow["{p \ot p}"{description}, from=3-5, to=2-4]
	\arrow["{\color{gray}p\,\mathsf{multiplicative}}\quad"{description}, draw=none, from=3-5, to=2-6]
	\arrow["m"', from=3-5, to=3-6]
	\arrow["{\color{gray}q\,\mathsf{multiplicative}\quad}"{description}, draw=none, from=3-5, to=4-6]
	\arrow["p"', from=3-6, to=2-6]
	\arrow["q", from=4-1, to=3-1]
	\arrow["\Delta", from=4-1, to=4-2]
	\arrow["\varepsilon"', from=4-1, to=5-4]
	\arrow["{q \ot q}"{description}, from=4-2, to=3-2]
	\arrow["{\color{gray}\eqref{eq:S'q'}}"{description}, draw=none, from=4-2, to=3-3]
	\arrow["{B \ot S}", from=4-2, to=4-4]
	\arrow["{q \ot q'}"{description}, from=4-4, to=3-3]
	\arrow["{q \ot q}"{description}, from=4-4, to=3-5]
	\arrow["m", from=4-4, to=4-6]
	\arrow["{\color{gray}\eqref{eq:qonV}}"{description}, draw=none, from=4-4, to=5-4]
	\arrow["q"', from=4-6, to=3-6]
	\arrow["u"', from=5-4, to=4-6]
\end{tikzcd}\]
also the rooftop pentagon \((\star)\) commutes, so that \(S_P \colon P \to P\) is a right antipode.

Given a comonoid \(C\) in \(\cM\), we denote by \(\rH(C)\) the right Hopf monoid \(P\) constructed above.

\begin{theorem}
There is an adjunction
\[\begin{tikzcd}
	{\Comon(\cM)} & {\underline{\rHopf}(\cM)}
	\arrow[""{name=0, anchor=center, inner sep=0}, "\rH", shift left=2, from=1-1, to=1-2]
	\arrow[""{name=1, anchor=center, inner sep=0}, "{\mathrm{U}}", shift left=2, from=1-2, to=1-1]
	\arrow["\dashv"{anchor=center, rotate=-90}, draw=none, from=0, to=1]
\end{tikzcd}\]
between the category \(\Comon(\cM)\) of comonoids in \(\cM\), and the category \(\underline{\rHopf}(\cM)\) of right Hopf monoids in \(\cM\) whose right antipode is anti-multiplicative and anti-comultiplicative, where \(\mathrm{U}\) denotes the forgetful functor.
\end{theorem}
\begin{proof}
For any comonoid \(C\), we consider the composition
\begin{equation}\label{eq:etaC}
    \eta_C \coloneqq \left(C \xrightarrow{\iota_0} V \xrightarrow{\iota_V} T(V) = B \xrightarrow{q} Q \xrightarrow{p} P = \mathrm{U} \left(\rH(C)\right)\right),
\end{equation}
which is a comonoid map as composition of comonoid maps.

Let \(C\) be a comonoid and \(H\) be an object in \(\underline{\rHopf}(\cM)\). With every morphism of comonoids \(f\colon C\to \mathrm{U}(H)\), we naturally associate a unique morphism of right Hopf monoids \(F_P\colon \rH(C)\to H\) such that \(\mathrm{U}(F_P)  \eta_C = f\).

Mimicking \cite[Proposition 14]{GreenNicholsTaft}, we define \(f_0 \coloneqq f\) and 
\begin{equation}\label{eq:fn}
    f_n \coloneqq S_H^n f.
\end{equation}
Each \(f_n\colon V_n\to H\) is a comonoid morphism, which can be seen by induction starting from \(f_0\), since \(S_H\) is a comonoid anti-homomorphism and the \(V_n\)'s are alternatively equal to \(C\) or \(C^{\mathrm{cop}}\). Let \(\iota_n\colon V_n\to V\) be the canonical morphism of the coproduct. We set 
\begin{equation}\label{eq:tildef}
    \tilde f \coloneqq \bigsqcup_{n\ge 0}f_n\colon V\to H \quad \text{uniquely determined by} \quad \tilde f \iota_n = f_n \quad \text{ for all } n\geq 0.
\end{equation}
This is a comonoid morphism, and so it induces a bimonoid morphism \(F_B \colon B= T(V)\to H\) such that 
\begin{equation}\label{eq:FB}
    F_B \iota_V = \tilde f.
\end{equation}
For all \(n \geq 0\), one has 
\[\tilde f S_V \iota_n \overset{\eqref{eq:defs}}{=} \tilde f \iota_{n+1} \overset{\eqref{eq:tildef}}{=} f_{n+1} \overset{\eqref{eq:fn}}{=} S_H^{n+1} f = S_H S_H^n f \overset{\eqref{eq:fn}}{=} S_H f_n \overset{\eqref{eq:tildef}}{=} S_H \tilde f \iota_n,\] 
and so, by uniqueness in the universal property of the coproduct, one gets \(\tilde f S_V = S_H \tilde f\), whence also 
\begin{equation}\label{eq:FS=SF}
    F_BS_B = S_H F_B
\end{equation} 
because
\[F_B S_B \iota_V \overset{\text{def }S_B}{=} F_B \iota_V S_V \overset{\eqref{eq:FB}}{=} \tilde f S_V = S_H \tilde f \overset{\eqref{eq:FB}}{=} S_H F_B \iota_V.\]

We now prove that \(F_B\colon B\to H\) coequalises \(m_B^2(B\ot \id_B * S_B \ot B)\) and \(m_B^2(B\ot u_B\varepsilon_B \ot B)\), and hence induces a unique morphism \(F_Q \colon Q\to H\) such that
\begin{equation}\label{eq:FQq}
    F_Q q = F_B. 
\end{equation}
This is an easy computation, using that \(F_B\) is a bimonoid morphism:
\begin{align*}
    F_B m_B^2(B\ot \id * S_B \ot B)&= m_H^2 (F_B\ot F_B\ot F_B) (B\ot m_B(B \ot S_B)\Delta_B \ot B)\\
    &= m^2_H(F_B\ot F_Bm_B(B\ot S_B)\Delta_B \ot F_B)\\
    &= m^2_H(F_B\ot m_H(F_B\ot F_B)(B\ot S_B)\Delta_B \ot F_B)\\
    \overset{\eqref{eq:FS=SF}}&{=} m^2_H(F_B\ot m_H(H \otimes S_H)(F_B\ot F_B)\Delta_B \ot F_B)\\
    &= m^2_H(H\ot m(H\ot S_H)\Delta_H \ot H)(F_B\ot F_B\ot F_B)\\
    \overset{\eqref{eq:right-antipode}}&{=} m_H^2(H\ot u_H\varepsilon_H\ot H)(F_B\ot F_B\ot F_B)\\
    &= F_B m_B^2(B\ot u_B\varepsilon_B\ot B).
\end{align*}

Then, we prove that the resulting \(F_Q\) coequalises \(m_Q^2(Q \ot q(S_B * \id_B) \iota_V \gamma \ot Q)\) and \(m_Q^2(Q \ot qu_B\varepsilon_B\iota_V \gamma \ot Q)\), inducing a unique morphism \(F_P \colon P \to H\) such that
\begin{equation}\label{eq:FP}
    F_P p = F_Q.
\end{equation}
In view of \eqref{eq:SidS}, \(S_H\) satisfies 
\begin{equation}\label{eq:SHleft}
    m_H(S_H \ot \id)\Delta_HS_H = u_H \varepsilon_HS_H.
\end{equation}
Hence, by using again that \(F_B\) and \(F_Q\) are bimonoid morphisms, one has 
\begin{align*}
    & F_Qm_Q^2(Q \ot q(S * \id) \iota_V \gamma\jmath_{n+1} \ot Q) \\
    & = m_H^2(F_Q \ot F_Q \ot F_Q)(Q \ot q(S * \id) \iota_V \gamma\jmath_{n+1} \ot Q) \\
    \overset{\eqref{eq:FQq}}&{=} m_H^2(F_Q \ot F_B \ot F_Q)(Q \ot m_B(S_B \ot B)\Delta_B \ot Q)(Q \ot \iota_V \gamma\jmath_{n+1} \ot Q) \\
    & = m_H^2(H \ot m_H(S_H \ot H)\Delta_H \ot H)(F_Q \ot F_B \ot F_Q)(Q \ot \iota_V \gamma\jmath_{n+1} \ot Q) \\
    \overset{\eqref{eq:FB}}&{=} m_H^2(H \ot m_H(S_H \ot H)\Delta_H \ot H)(F_Q \ot \tilde f \gamma\jmath_{n+1} \ot F_Q) \\
    \overset{\eqref{eq:defgamma}}&{=} m_H^2(H \ot m_H(S_H \ot H)\Delta_H \ot H)(F_Q \ot \tilde f \iota_{n+1} \ot F_Q) \\
    \overset{\eqref{eq:tildef}}&{=} m_H^2(H \ot m_H(S_H \ot H)\Delta_H \ot H)(F_Q \ot f_{n+1} \ot F_Q) \\
    \overset{\eqref{eq:fn}}&{=} m_H^2(H \ot m_H(S_H \ot H)\Delta_H \ot H)(F_Q \ot  S_H f_{n} \ot F_Q) \\
    \overset{\eqref{eq:tildef}}&{=} m_H^2(H \ot m_H(S_H \ot H)\Delta_H \ot H)(F_Q \ot  S_H \tilde f \iota_{n} \ot F_Q) \\
    \overset{\eqref{eq:SHleft}}&{=} m_H^2(H \ot u_H\varepsilon_H \ot H)(F_Q \ot  S_H \tilde f \iota_{n} \ot F_Q) \\
    & = m_H^2(H \ot u_H\varepsilon_H \ot H)(F_Q \ot  \tilde f \gamma \jmath_{n+1} \ot F_Q) \\
    \overset{\eqref{eq:FB}}&{=} m_H^2(H \ot u_H\varepsilon_H \ot H)(F_Q \ot  F_B \iota_V \gamma \jmath_{n+1} \ot F_Q) \\
    & = F_Qm_Q^2(Q \ot qu_B\varepsilon_B\iota_V \gamma \jmath_{n+1} \ot Q),
\end{align*}
for all \(n \geq 0\), whence 
\[ F_Qm_Q^2(Q \ot q(S * \id) \iota_V \gamma \ot Q) =  F_Qm_Q^2(Q \ot qu_B\varepsilon_B\iota_V \gamma \ot Q) \]
because the tensor product preserves colimits. Therefore \(F_Q\colon Q\to H\) induces a unique morphism \(F_P\colon P\to H\) such that \(F_P p = F_Q\). 
In addition, since
\begin{align*}
    S_HF_Ppq \overset{\eqref{eq:FP}}&{=} S_HF_Qq \overset{\eqref{eq:FQq}}{=} S_HF_B \overset{\eqref{eq:FS=SF}}{=} F_BS_B \overset{\eqref{eq:FQq}}{=} F_QqS_B \overset{\eqref{eq:FP}}{=} F_PpqS_B \\
    \overset{\eqref{eq:pq}}&{=} F_Pp'q'S_B     \overset{\eqref{eq:S'q'}}{=} F_Pp'S'q \overset{\eqref{eq:Spp'S'}}{=} F_PS_Ppq
\end{align*}
and since both \(p\) and \(q\) are epimorphisms, we also have that \(F_P\) is compatible with the right antipodes: \(S_HF_P = F_PS_P\). To conclude, observe that \(F_P \eta_C = f\) as desired: in fact,
\[F_P\eta_C \overset{\eqref{eq:etaC}}{=} F_Ppq\iota_V\iota_0 \overset{\eqref{eq:FP}}{=} F_Qq\iota_V\iota_0 \overset{\eqref{eq:FQq}}{=} F_B\iota_V\iota_0 \overset{\eqref{eq:FB}}{=} \tilde{f}\iota_0 \overset{\eqref{eq:tildef}}{=} f.\]

This proves the universal property of \(\rH(C)\). By \cite[\S IV, Theorem 2 (\textit{ii})]{Maclane}, the universal property suffices to have the functoriality of \(\rH\) and the adjunction.
\end{proof}

\section{Examples}\label{sec:examples}

\begin{example}[The cartesian monoidal case]
    In a braided cartesian monoidal category \((\cM,\times,\I,\brd)\), any one-sided Hopf monoid object is automatically two-sided, i.e., a Hopf monoid object. Indeed, recall that in this case any object in \(\cM\) admits a unique comonoid structure where \(\Delta\) is the diagonal map and \(\varepsilon\) is the unique morphism to the terminal object \(\I\). Suppose that an object \(M\) in this setting admits also a monoid structure \((M,m,u)\) and an endomorphism \(S\) which is right (or left) convolution inverse of the identity, that is,
    \begin{equation}\label{eq:cartesian}
        m   (\id_M \times S)   \Delta = u   \varepsilon.
    \end{equation}
    Since \(S\) is a morphism in \(\cM\), we have that \(\Delta   S = (S \times S)   \Delta\) and that \(\varepsilon   S = \varepsilon\). In particular, from \eqref{eq:cartesian} we deduce that
    \begin{equation}\label{eq:cartesian2}
        u   \varepsilon = u   \varepsilon   S = m   (\id_M \times S)   \Delta   S = m   (S \times S^2)   \Delta,
    \end{equation}
    which allows us to conclude that
    \begin{align*}
        \id_M & = m  \, (\id_M \times u)  \, (\id_M \times \varepsilon)  \, \Delta \\
        \overset{\eqref{eq:cartesian2}}&{=} m  \, (\id_M \times m)  \, (\id_M \times S \times S^2)  \, (\id_M \times \Delta)  \, \Delta \\
        & = m  \, (m \times \id_M)  \, (\id_M \times S \times S^2)  \, (\Delta \times \id_M)  \, \Delta \\
        \overset{\eqref{eq:cartesian}}&{=} m  \, (u \times \id_M)  \, (\id_M \times S^2)  \, (\varepsilon \times \id_M)  \, \Delta = S^2.
    \end{align*}
    Since we have proven that \(S^2 = \id_M\), now \eqref{eq:cartesian2} becomes
    \(m   (S \times \id_M)   \Delta = u   \varepsilon,\)
    i.e.\@ \(S\) is also a left convolution inverse of \(\id_M\). 
\end{example}

\begin{example}\label{ex:trivial}
    Any classical right Hopf algebra in the sense of \cite{GreenNicholsTaft} 
    naturally gives rise to many right Hopf monoids in braided categories of modules over cocommutative bialgebras.
    
    Indeed, let \(H\) be a right Hopf \(\Bbbk\)-algebra. 
    For instance, \(H\) can be the free right Hopf algebra over the matrix coalgebra \(M_n(\Bbbk)\), obtained via a right-hand side counterpart of \cite[\S3]{GreenNicholsTaft}. This \(H\) becomes a right Hopf monoid in the symmetric monoidal category \(\mathcal{M} = {}_B\mathfrak{M}\) of left modules over any cocommutative \(\K\)-bialgebra \(B\), when endowed with the trivial action induced by the algebra morphism \(u_H   \varepsilon_B\). For instance, \(B\) can be \(\Bbbk[X]\) with \(X\) grouplike (i.e., \(\Bbbk \mathbb{N}\)), \(\Bbbk[X]\) with \(X\) primitive (i.e., \(U(\Bbbk X) \cong \mathcal{O}(\Bbbk,+,0)\)), \(\Bbbk[X^{\pm 1}]\) with \(X\) grouplike (i.e., \(\mathcal{O}(\Bbbk^\times,\cdot,1)\)), or the trigonometric bialgebra \(\Bbbk[c,s]\) with \(\Delta(c) = c \ot c - s \ot s\) and \(\Delta(s) = c \ot s + s \ot c\). In all these cases, the monoidal structure is given by the tensor product of \(\Bbbk\)-vector spaces and the braiding in \(\mathcal{M}\) is given by the usual flip. 
\end{example}

\begin{example}\label{ex:Z2inZ2} 
    Denote by \(e,g\) the elements of \(\Z_2\), where \(e\) is the neutral element. Consider the \(2\)-dimensional group Hopf algebra \(\Bbbk \Z_2\) with unit \(e\). When \(\mathrm{char}(\Bbbk)\neq 2\), this is a quasitriangular Hopf algebra with universal \(\mathcal{R}\)-matrix
    \[
        \mathcal{R} = \mathcal{R}^i \ot \mathcal{R}_i =\frac{1}{2}\left(e \ot e + e \ot g + g \ot e - g \ot g \right).
    \]
    The category \({}_{\K \Z_2}\mathfrak{M}\) of left \(\K\Z_2\)-modules is a symmetric monoidal category with non-trivial braiding \(\brd\) induced by the universal \(\mathcal{R}\)-matrix,
    \begin{invisible}
        i.e., \(\brd_{M,N} \colon M \ot N \to N \ot M\) given by
        \[
        m\otimes n \mapsto \mathcal{R}^{2}n\otimes \mathcal{R}^{1}m = \frac{1}{2}\left( n\otimes m+n\otimes g\cdot m+g\cdot n\otimes m-g\cdot n\otimes g\cdot m\right),
        \]
    \end{invisible}%
    see \cite[Example 2.1.6]{MajidBook}. 
    \begin{invisible}
        This is a symmetry because if we do it twice
        \begin{eqnarray*}
        x\otimes y &\mapsto &\frac{1}{2}\left( y\otimes x+y\otimes g\cdot x+g\cdot y\otimes x-g\cdot y\otimes g\cdot x\right) \\
        &\mapsto &\frac{1}{4}\left(
        \begin{array}{c}
            x\otimes y+x\otimes g\cdot y+g\cdot x\otimes y-g\cdot x\otimes g\cdot y+ \\ 
            +g\cdot x\otimes y+g\cdot x\otimes g\cdot y+x\otimes y-x\otimes g\cdot y+ \\ 
            +x\otimes g\cdot y+x\otimes y+g\cdot x\otimes g\cdot y-g\cdot x\otimes y+ \\ 
            -g\cdot x\otimes g\cdot y-g\cdot x\otimes y-x\otimes g\cdot y+x\otimes y
        \end{array}
        \right) \\
        &=&\frac{1}{4}\left( x\otimes y+x\otimes y+x\otimes y+x\otimes y\right) = x\otimes y
        \end{eqnarray*}
    \end{invisible}%
    Take \(C=\K\Z_{2}\) with the regular left action. It is a left \(\K\Z_2\)-module coalgebra with its own grouplike comultiplication. Then \(C^{\mathrm{cop}}\) in \(\Bbbk \mathbb{Z}_{2}\)-modules has comultiplication uniquely determined by
    \begin{align*}
        \brd_{C,C} \Delta_C \left( e\right) & = \cR \left( e\otimes e\right) = 
        \frac{1}{2}\left( e\otimes e+e\otimes g+g\otimes e-g\otimes g\right) , \\
        \brd_{C,C} \Delta_C \left( g\right) & 
        = \frac{1}{2}\left( g\otimes g+g\otimes e+e\otimes g-e\otimes e\right) .
    \end{align*}
    
    We are in a condition to perform the construction of \zcref{sect:GNTsym-monoidal}. Set
    \[
        V_{2i} \coloneqq C, \qquad V_{2i+1} \coloneqq C^{\mathrm{cop}} \qquad \text{and} \qquad V \coloneqq \bigoplus_{i\ge 0}V_{i}.
    \]
    Then \(V\) is still a \(\K\Z_{2}\)-module coalgebra and 
    \begin{invisible}
        hence we can perform 
        \[
            T(V) \coloneqq \bigoplus_{n} V^{\ot n}
        \]
        By the universal property of the tensor algebra, there exist unique algebra morphisms \(\Delta_{T(V)} \colon T(V) \to T(V) \ot T(V)\) and \(\varepsilon_{T(V)} \colon T(V) \to \I\) in \({}_{\K\Z_2}\mathfrak{M}\) such that the following diagrams commute
        \[
        \xymatrix @C=70pt {
        V \ar[r]^{\iota_V} \ar[d]_-{\Delta_V} & T(V) \ar@{.>}[d]^-{\Delta_{T(V)}} \\
        V \ot V \ar[r]_-{\iota_V \ot \iota_V} & T(V) \ot T(V)
        }
        \qquad 
        \xymatrix{
        V \ar[rr]^-{\iota_V} \ar[dr]_-{\varepsilon_V} & & T(V) \ar@{.>}[dl]^-{\varepsilon_{T(V)}} \\
        & \I & 
        },
        \]
        where \(\iota_V \colon V \to T(V)\) is the canonical morphism associated with the tensor algebra. Then, 
    \end{invisible}%
    \(T(V)\) becomes a bimonoid in \({}_{\K\Z_2}\mathfrak{M}\).
    \begin{invisible}
        Observe that \(V^\mathrm{cop}\cong \bigoplus_n V_n^{\mathrm{cop}}\): indeed, in the diagram    
        \[
            \xymatrix{
                V_n \ar[d]_-{\iota_n} \ar[r]^-{\Delta_n} & V_n \ot V_n \ar[d]^-{\iota_n \ot \iota_n} \ar[r]^-{\brd_{V_n,V_n}} & V_n \ot V_n \ar[d]^-{\iota_n \ot \iota_n} \\
                V \ar[r]_-{\Delta_V} & V \ot V \ar[r]_-{\brd_{V,V}} & V\ot V
            }
        \]
        the left-hand side square commutes by definition of \(\Delta_V\), and the right-hand side square commutes by naturality of \(\brd\). By definition, \(\Delta_{V^\mathrm{cop}}\) is the composition of the lower row, thus it makes the large rectangle commute. By the universal property of the coproduct, there exists a unique morphism \(V\to V\ot V\) making the large rectangle commute: this morphism is by definition \(\Delta_{\bigoplus_n V_n^{\mathrm{cop}}}\), and hence, by uniqueness, one has \(\Delta_{\bigoplus_n V_n^{\mathrm{cop}}}=\Delta_{V^\mathrm{cop}}\).
    \end{invisible}%
    The unique morphism of \(\Bbbk \mathbb{Z}_{2}\)-module coalgebras \(S_V \colon V^{\mathrm{cop}}\rightarrow V\) such that 
    \[\begin{tikzcd}
    	{V_n^{\mathrm{cop}}} & {V^{\mathrm{cop}}=\bigoplus_n V_n^{\mathrm{cop}}} \\
    	{V_{n+1}} & V
    	\arrow[hook, from=1-1, to=1-2]
    	\arrow[equals, from=1-1, to=2-1]
    	\arrow["S_V", dotted, from=1-2, to=2-2]
    	\arrow[hook, from=2-1, to=2-2]
    \end{tikzcd}\]
    commutes is explicitly given by
    \begin{equation*}
        S_V ( 0,\ldots ,0,\underset{n}{\underbrace{v}},0,0,\ldots ) = (0,\ldots ,0,0,\underset{n+1}{\underbrace{v}},0,\ldots ).
    \end{equation*}
    If we denote by \(g_{n}\) the element \(g\in \Bbbk \mathbb{Z}_{2}\) when considered in \(V_{n}\), then \(S_V\left(g_{n}\right) =g_{n+1}\). 
    \begin{invisible}
        We can now lift \(S_V\) to a morphism of monoids \(S\colon T\left( V^{\mathrm{cop}}\right) \rightarrow T\left( V\right) ^{\mathrm{op}}\) by considering the left \(\Bbbk \mathbb{Z}_{2}\)-linear composition
        \begin{equation*}
            V^{\mathrm{cop}} \overset{s}{\longrightarrow }V\overset{\iota }{\longrightarrow }T\left(V\right)^{\mathrm{op}}.
        \end{equation*}
        By the universal property of the free bialgebra functor \(T(-)\), there exists a unique morphism of bimonoids \(S \colon T(V^{\mathrm{cop}}) \to T(V)^{\mathrm{op}}\).
        Since both \(\brd_{TV,TV} \Delta_{TV}\) and \(\Delta_{T(V^{\mathrm{cop}})}\) are algebra morphisms satisfying
        \[
            \brd_{TV,TV} \Delta_{TV} \iota_V = \brd_{TV,TV} (\iota_V \ot \iota_V) \Delta_V = (\iota_V \ot \iota_V) \brd_{V,V} \Delta_V = (\iota_V \ot \iota_V)\Delta_{V^{\mathrm{cop}}} = \Delta_{T(V^{\mathrm{cop}})}\iota_V,
        \]
        they are equal and hence \(T(V)^{\mathrm{cop}} = T(V^{\mathrm{cop}})\), so that we have a morphism of bimonoids \(S \colon T(V)^{\mathrm{cop}} \to T(V)^{\mathrm{op}}\). 
    \end{invisible}%
    
    Now, we lift \(S_V\) to a morphism of bimonoids \(S \colon T(V)^{\mathrm{cop}} \to T(V)^{\mathrm{op}}\) and we perform the construction of \(\rH(C)\).
    \begin{invisible}
        We claim that, in fact, \(S\colon T\left( V\right) ^{\mathrm{cop}}\rightarrow T\left(V\right) ^{\mathrm{op}}\) is a morphism of bimonoids. To this aim, let us observe first of all that \(T\left( V\right) ^{\mathrm{cop}}\) is, indeed, a bialgebra: since the braiding is symmetric, \(\mathfrak{c}\colon T\left( V\right) \otimes T\left(V\right) \rightarrow T\left( V\right) \otimes T\left( V\right)\) is an algebra morphism and hence \(\mathfrak{c}\Delta\) is an algebra morphism. Similarly, \(T\left( V\right)^{\mathrm{op}}\) is still a bialgebra because \(m  \mathfrak{c}\) is still a coalgebra morphism (again, because \(\mathfrak{c}\) is symmetric).
        Secondly, showing that \(S\colon T\left( V\right) ^{\mathrm{cop}}\rightarrow T\left(V\right) ^{\mathrm{op}}\) is a morphism of bimonoids is the same showing that \(S\colon T\left( V\right) \rightarrow T\left( V\right) ^{\mathrm{op,cop}}\) is a morphism of bimonoids, and since \(S\), \(\Delta\), and \(\Delta ^{\mathrm{cop}}\) are all algebra morphisms, it is enough to show that \(\Delta ^{\mathrm{cop}}\circ S=\left( S\otimes S\right) \circ \Delta\) on elements of \(V\), and by the universal property of the coproduct it is enough to check it on elements of \(V_{n}\). If \(V_{n}=\Bbbk \mathbb{Z}_{2}\), then \(V_{n+1}=\Bbbk \mathbb{Z}_{2}^{\mathrm{cop}}\) and we have
        \begin{equation*}
            \left( S\otimes S\right) \Delta (u_{n})=\left( S\otimes S\right) \left(u_{n}\otimes u_{n}\right) =u_{n+1}\otimes u_{n+1}
        \end{equation*}
        while
        \begin{equation*}
            \Delta ^{\mathrm{cop}}S\left( u_{n}\right) =\mathfrak{c}\Delta \left( u_{n+1}\right) = \mathfrak{c}\Delta _{\Bbbk \mathbb{Z}_{2}}^{\mathrm{cop}}\left( u_{n+1}\right) =\Delta _{\Bbbk \mathbb{Z}_{2}}\left( u_{n+1}\right) =u_{n+1}\otimes u_{n+1}.
        \end{equation*}
        
        If \(V_{n}=\Bbbk \mathbb{Z}_{2}^{\mathrm{cop}}\), then \(V_{n+1}=\Bbbk \mathbb{Z}_{2}\) and we have
        \begin{equation*}
            \left( S\otimes S\right) \Delta (u_{n})=\left( S\otimes S\right) \mathfrak{c}\left( u_{n}\otimes u_{n}\right) =\mathfrak{c}\left( u_{n+1}\otimes u_{n+1}\right)
        \end{equation*}
        while
        \begin{equation*}
            \Delta ^{\mathrm{cop}}S\left( u_{n}\right) =\mathfrak{c}\Delta \left( u_{n+1}\right) = \mathfrak{c}\Delta _{\Bbbk \mathbb{Z}_{2}}\left( u_{n+1}\right) =\mathfrak{c} \left( u_{n+1}\otimes u_{n+1}\right) .
        \end{equation*}
        Similarly for \(g_{n}\). Therefore, \(S:T\left( V\right) ^{\mathrm{cop}}\rightarrow T\left( V\right) ^{\mathrm{op}}\) is indeed a morphism of bimonoids.
    \end{invisible}%
    In the present setting, this amounts to considering \(X=\left\{ x_{\left( 1\right) }S\left( x_{\left(2\right) }\right) -\varepsilon \left( x\right) 1\mid x\in V_{n},n\geq 0\right\}\) and 
    \[K \coloneqq \left\langle X\right\rangle +\left\langle S\left( X\right) \right\rangle = 
    \stretchleftright[600]{\langle}
    {\left. \begin{gathered} x_{(1)}S\left(x_{(2)}\right) - \varepsilon(x)1, \\[3pt] S\left(y_{(1)}\right)y_{(2)} - \varepsilon(y)1 \end{gathered} ~\right|~ \begin{gathered} x \in V_n, n \geq 0, \\[5pt] y \in V_m, m \geq 1\end{gathered}}
    {\rangle}.\]
    \begin{invisible}
        In the present setting, we can consider the ideal \(I\) in \(T\left( V\right)\) generated by \(X=\left\{ x_{\left( 1\right) }S\left( x_{\left(2\right) }\right) -\varepsilon \left( x\right) 1\mid x\in V_{n},n\geq 0\right\}\). 
        Note that \(I\) is not necessarily \(S\)-stable, but we can observe the following: if \(r=x_{\left( 1\right) }S\left( x_{\left( 2\right) }\right) -\varepsilon \left( x\right) 1\) and \(z\) is a generic element of \(\mathbb{Z}_{2}\), then
        \begin{eqnarray*}
            z\cdot \left( x_{\left( 1\right) }S\left( x_{\left( 2\right) }\right) -\varepsilon \left( x\right) 1\right) &=&z\cdot \left( x_{\left( 1\right) }S\left( x_{\left( 2\right) }\right) \right) -\varepsilon \left( x\right) \left( z\cdot 1\right) \\
            &=&\left( z\cdot x_{\left( 1\right) }\right) \left( z\cdot S\left( x_{\left( 2\right) }\right) \right) -\varepsilon \left( x\right) \varepsilon \left( z\right) 1 \\
            &=&\left( z\cdot x_{\left( 1\right) }\right) S\left( z\cdot x_{\left( 2\right) }\right) -\varepsilon \left( z\cdot x\right) 1 \\
            &=&\left( z\cdot x\right) _{\left( 1\right) }S\left( \left( z\cdot x\right) _{\left( 2\right) }\right) -\varepsilon \left( z\cdot x\right) 1\in X
        \end{eqnarray*}
        because \(T\left( V\right)\) is a \(\Bbbk \mathbb{Z}_{2}\)-module algebra, \(S\) is \(\Bbbk \mathbb{Z}_{2}\)-linear, and \(V_{n}\) is a \(\Bbbk \mathbb{Z}_{2}\)-module coalgebra, and if moreover \(a,b\) are generic elements of \(T\left(
        V\right)\), then
        \begin{align*}
        S\left( arb\right) &=S\mu \left( a\otimes rb\right) =\mu ^{op}\left( S\otimes S\right) \left( a\otimes rb\right) =\mu \mathfrak{c}\left( S\left( a\right) \otimes S\left( rb\right) \right) \\
         & =\left( \mathcal{R}^{2}\cdot S\left(rb\right) \right) \left( \mathcal{R}^{1}\cdot S\left( a\right) \right) =\left( \mathcal{R}^{2}\cdot S\mu \left( r\otimes b\right) \right) \left( \mathcal{R}^{1}\cdot S\left( a\right) \right) \\
        & =\left( \mathcal{R}^{2}\cdot \mu ^{op}\left( S\otimes S\right) \left( r\otimes b\right) \right) \left( \mathcal{R}^{1}\cdot S\left( a\right) \right) \\
        & = \left( \mathcal{R}^{2}\cdot \mu \mathfrak{c}\left(  \left( r\right) \otimes S\left( b\right) \right) \right) \left( \mathcal{R}^{1}\cdot S\left( a\right) \right)  \\
        & = \left( \mathcal{R}^{2}\cdot \left( \mathcal{P}^{2}\cdot S\left(b\right) \right) \left( \mathcal{P}^{1}\cdot S\left( r\right) \right) \right) \left( \mathcal{R}^{1}\cdot S\left( a\right) \right)   \\
        & = \left( \left( \mathcal{R}^{2}_{(1)}\cdot \mathcal{P}^{2}\cdot S\left(b\right) \right) \left( \mathcal{R}^{2}_{(2)}\cdot \mathcal{P}^{1}\cdot S\left( r\right) \right) \right) \left( \mathcal{R}^{1}\cdot S\left( a\right) \right)    \\
        & = \left( \mathcal{R}^{2}_{(1)} \mathcal{P}^{2}\cdot S\left(b\right) \right) S\left( \mathcal{R}^{2}_{(2)} \mathcal{P}^{1}\cdot r\right)\left( \mathcal{R}^{1}\cdot S\left( a\right) \right) \in T\left( V\right) S\left( X\right) T\left( V\right) .
        \end{align*}
        
        That is to say, \(S\left( I\right) =S\left( T\left( V\right) XT\left(
        V\right) \right) \subseteq T\left( V\right) S\left( X\right) T\left(
        V\right)\). Now, an element in \(S\left( X\right)\) has the form%
        \begin{align*}
        S\left( x_{\left( 1\right) }S\left( x_{\left( 2\right) }\right) -\varepsilon
        \left( x\right) 1\right) &=S\left( x_{\left( 1\right) }S\left( x_{\left(
        2\right) }\right) \right) -\varepsilon \left( x\right) 1 \\
        &=\left( \mathcal{R}^{2}\cdot S\left( S\left( x_{\left( 2\right) }\right)
        \right) \right) \left( \mathcal{R}^{1}\cdot S\left( x_{\left( 1\right)
        }\right) \right) -\varepsilon \left( x\right) 1 \\
        &=\left( \mathcal{R}^{2}\cdot S\left( \mathcal{P}^{1}\cdot S\left( x\right)
        _{\left( 1\right) }\right) \right) \left( \mathcal{R}^{1}\cdot \mathcal{P}%
        ^{2}\cdot S\left( x\right) _{\left( 2\right) }\right) -\varepsilon \left(
        x\right) 1 \\
        &=\left( \mathcal{R}^{2}\cdot S\left( \mathcal{P}^{1}\cdot S\left( x\right)
        _{\left( 1\right) }\right) \right) \left( \mathcal{R}^{1}\mathcal{P}%
        ^{2}\cdot S\left( x\right) _{\left( 2\right) }\right) -\varepsilon \left(
        x\right) 1 \\
        &=S\left( \mathcal{R}^{2}\cdot \mathcal{P}^{1}\cdot S\left( x\right)
        _{\left( 1\right) }\right) \left( \mathcal{R}^{1}\mathcal{P}^{2}\cdot
        S\left( x\right) _{\left( 2\right) }\right) -\varepsilon \left( x\right) 1 \\
        &=S\left( \mathcal{R}^{2}\mathcal{P}^{1}\cdot S\left( x\right) _{\left(
        1\right) }\right) \left( \mathcal{R}^{1}\mathcal{P}^{2}\cdot S\left(
        x\right) _{\left( 2\right) }\right) -\varepsilon \left( x\right) 1 \\
        &=S\left( S\left( x\right) _{\left( 1\right) }\right) S\left( x\right)
        _{\left( 2\right) }-\varepsilon \left( x\right) 1,
        \end{align*}%
        because \(S\) is left \(\Bbbk \mathbb{Z}_{2}\)-linear and \(\mathfrak{c}\) is a
        symmetry (\(\mathfrak{c}^{2}=\id\)).  A string diagram verification
        reveals that%
        \begin{equation*}
        SS\left( x_{\left( 1\right) }S\left( x_{\left( 2\right) }\right) \right)
        =SS\left( x\right) _{\left( 1\right) }S\left( SS\left( x\right) _{\left(
        2\right) }\right)
        \end{equation*}%
        therefore \(S\left( T\left( V\right) S\left( X\right) T\left( V\right)
        \right) \subseteq T\left( V\right) SS\left( X\right) T\left( V\right)\) and%
        \begin{equation*}
        SS\left( x_{\left( 1\right) }S\left( x_{\left( 2\right) }\right)
        -\varepsilon \left( x\right) 1\right) =SS\left( x\right) _{\left( 1\right)
        }S\left( SS\left( x\right) _{\left( 2\right) }\right) -\varepsilon \left(
        SS\left( x\right) \right) 1\in X
        \end{equation*}%
        entail that%
        \begin{equation*}
        S\left( T\left( V\right) S\left( X\right) T\left( V\right) \right) \subseteq
        I.
        \end{equation*}%
        Thus,%
        \begin{equation*}
        K:=\left\langle X\right\rangle +\left\langle S\left( X\right) \right\rangle
        \end{equation*}%
        is an \(S\)-stable ideal. It is, in fact, an \(S\)-stable biideal. Indeed, on
        the one hand%
        \begin{equation*}
        \varepsilon \left( x_{\left( 1\right) }S\left( x_{\left( 2\right) }\right)
        -\varepsilon \left( x\right) 1\right) =\varepsilon \left( x\right)
        -\varepsilon \left( x\right) =0
        \end{equation*}%
        entails that%
        \begin{equation*}
        \varepsilon \left( K\right) =\varepsilon \left( \left\langle X\right\rangle
        \right) +\varepsilon \left( \left\langle S\left( X\right) \right\rangle
        \right) =0
        \end{equation*}%
        while on the other hand, working modulo \(I\ot  T\left( V\right) +T\left(
        V\right) \ot  I\),%
        \begin{align*}
        &\Delta \left( x_{\left( 1\right) }S\left( x_{\left( 2\right) }\right)
        -\varepsilon \left( x\right) 1\right) \\&=\left( x_{\left( 1\right) \left(
        1\right) }\ot  x_{\left( 1\right) \left( 2\right) }\right) \left( S\left(
        x_{\left( 2\right) }\right) _{\left( 1\right) }\ot  S\left( x_{\left(
        2\right) }\right) _{\left( 2\right) }\right) -\varepsilon \left( x\right)
        \left( 1\ot  1\right) \\
        &=\left( x_{\left( 1\right) }\left( \mathcal{R}^{2}\cdot S\left( x_{\left(
        3\right) }\right) _{\left( 1\right) }\right) \ot  \left( \mathcal{R}%
        ^{1}\cdot x_{\left( 2\right) }\right) S\left( x_{\left( 3\right) }\right)
        _{\left( 2\right) }\right) -\varepsilon \left( x\right) \left( 1\ot 
        1\right) \\
        &=\left( x_{\left( 1\right) }\left( \mathcal{R}^{2}\mathcal{P}^{2}\cdot
        S\left( x_{\left( 4\right) }\right) \right) \ot  \left( \mathcal{R}%
        ^{1}\cdot x_{\left( 2\right) }\right) \left( \mathcal{P}^{1}\cdot S\left(
        x_{\left( 3\right) }\right) \right) \right) -\varepsilon \left( x\right)
        \left( 1\ot  1\right) \\
        &=\left( x_{\left( 1\right) }\left( \mathcal{R}^{2}\cdot S\left( x_{\left(
        4\right) }\right) \right) \ot  \left( \mathcal{R}_{\left( 1\right)
        }^{1}\cdot x_{\left( 2\right) }\right) \left( \mathcal{R}_{\left( 2\right)
        }^{1}\cdot S\left( x_{\left( 3\right) }\right) \right) \right) -\varepsilon
        \left( x\right) \left( 1\ot  1\right) \\
        &=\left( x_{\left( 1\right) }\left( \mathcal{R}^{2}\cdot S\left( x_{\left(
        3\right) }\right) \right) \ot  \left( \mathcal{R}^{1}\cdot x_{\left(
        2\right) }\right) _{\left( 1\right) }S\left( \left( \mathcal{R}^{1}\cdot
        x_{\left( 2\right) }\right) _{\left( 2\right) }\right) \right) -\varepsilon
        \left( x\right) \left( 1\ot  1\right) \\
        &=\left( x_{\left( 1\right) }\left( \mathcal{R}^{2}\cdot S\left( x_{\left(
        3\right) }\right) \right) \ot  \varepsilon \left( \mathcal{R}^{1}\cdot
        x_{\left( 2\right) }\right) \right) -\varepsilon \left( x\right) \left(
        1\ot  1\right) \\
        &=\left( x_{\left( 1\right) }\left( \mathcal{R}^{2}\cdot S\left( x_{\left(
        3\right) }\right) \right) \ot  \varepsilon \left( \mathcal{R}^{1}\right)
        \varepsilon \left( x_{\left( 2\right) }\right) \right) -\varepsilon \left(
        x\right) \left( 1\ot  1\right) \\
        &=\left( x_{\left( 1\right) }S\left( x_{\left( 2\right) }\right) \ot 
        1\right) -\varepsilon \left( x\right) \left( 1\ot  1\right) =0,
        \end{align*}%
        so that we can conclude that \(\Delta \left( X\right) \in I\otimes T\left(
        V\right) +T\left( V\right) \otimes I\) and 
        \begin{equation*}
        \Delta \left( S\left( X\right) \right) \subseteq \mathfrak{c}\left( S\left(
        I\right) \otimes T\left( V\right) +T\left( V\right) \otimes S\left( I\right)
        \right) \subseteq S\left( I\right) \otimes T\left( V\right) +T\left(
        V\right) \otimes S\left( I\right) .
        \end{equation*}%
        Thus, \(\Delta \left( K\right) \subseteq K\otimes T\left( V\right) +T\left(
        V\right) \otimes K\).
        %
        %
        In conclusion, the quotient \(H\coloneqq T\left( V\right) /K\) in \(_{\Bbbk \mathbb{Z}_{2}}\mathfrak{M}\) is a bimonoid with a bimonoid morphism \(\mathcal{S}\colon H\rightarrow H^\mathrm{op,cop}\) such that
        \begin{eqnarray*}
            x_{\left( 1\right) }\mathcal{S}\left( x_{\left( 2\right) }\right) &=& \varepsilon \left( x\right) 1,\qquad \forall x\in V_{n},\forall n\geq 0, \\
            \mathcal{S}\left( x_{\left( 1\right) }\right) x_{\left( 2\right) } &=& \varepsilon \left( x\right) 1,\qquad \forall x\in V_{n},\forall n\geq 1.
        \end{eqnarray*}%
    \end{invisible}%
    The quotient \(H\coloneqq T\left( V\right) /K\) in \(_{\Bbbk \mathbb{Z}_{2}}\mathfrak{M}\) is, in principle, a right Hopf monoid, in view of \zcref{sect:GNTsym-monoidal}. However, in this case we are going to prove that \(H\) is also a Hopf monoid.
    
    In order to make the expression of \(H\) more explicit, let us isolate a more convenient set of generators for the ideal \(K\).
    In view of the \(\Bbbk\)-linearity of the maps involved, for the relations
    \begin{equation*}
        x_{\left( 1\right) }S\left( x_{\left( 2\right) }\right) -\varepsilon \left(x\right) 1\qquad \text{for}\qquad x\in V_{n},n\geq 0,
    \end{equation*}
    it is enough to consider
    \begin{align*}
        \left( g_{2n}\right) _{\left( 1\right) }S\left( \left( g_{2n}\right) _{\left( 2\right) }\right) -\varepsilon \left( g_{2n}\right) 1 & = g_{2n}S\left( g_{2n}\right) -1=g_{2n}g_{2n+1}-1, \\
        \left( e_{2n}\right) _{\left( 1\right) }S\left( \left( e_{2n}\right) _{\left( 2\right) }\right) -\varepsilon \left( e_{2n}\right) 1 & =e_{2n}S\left( e_{2n}\right) -1=e_{2n}e_{2n+1}-1,
    \end{align*}
    for all \(n\geq 0\) for the even degrees, and
    \begin{align*}
        & \left( g_{2n+1}\right) _{\left( 1\right) }S\left( \left( g_{2n+1}\right)_{\left( 2\right) }\right) -\varepsilon \left( g_{2n+1}\right) 1 \\
        & \hspace{2em} = \frac{1}{2}\left( g_{2n+1}S\left( g_{2n+1}\right) +g_{2n+1}S\left( e_{2n+1}\right) +e_{2n+1}S\left( g_{2n+1}\right) -e_{2n+1}S\left( e_{2n+1}\right) \right) -1 \\
        & \hspace{2em}=\frac{1}{2}\left( g_{2n+1}g_{2n+2}+g_{2n+1}e_{2n+2}+e_{2n+1}g_{2n+2}-e_{2n+1}e_{2n+2}\right) -1, \\
        & \left( e_{2n+1}\right) _{\left( 1\right) }S\left( \left( e_{2n+1}\right) _{\left( 2\right) }\right) -\varepsilon \left( e_{2n+1}\right) 1 \\
        & \hspace{2em}=\frac{1}{ 2}\left( e_{2n+1}S\left( e_{2n+1}\right) +e_{2n+1}S\left( g_{2n+1}\right) +g_{2n+1}S\left( e_{2n+1}\right) -g_{2n+1}S\left( g_{2n+1}\right) \right) -1 \\
        & \hspace{2em}=\frac{1}{2}\left( e_{2n+1}e_{2n+2}+e_{2n+1}g_{2n+2}+g_{2n+1}e_{2n+2}-g_{2n+1}g_{2n+2}\right) -1,
    \end{align*}
    for all \(n\geq 0\) for the odd degrees. For the relations
    \begin{equation*}
        S\left( y_{\left( 1\right) }\right) y_{\left( 2\right) }=\varepsilon \left( y\right) 1\qquad \text{for}\qquad y\in V_{n},\forall n\geq 1,
    \end{equation*}
    it is enough to consider
    \begin{align*}
        S\left( \left( g_{2n}\right) _{\left( 1\right) }\right) \left( g_{2n}\right)_{\left( 2\right) }-\varepsilon \left( g_{2n}\right) 1 & = S\left(g_{2n}\right) g_{2n}-1=g_{2n+1}g_{2n}-1, \\
        S\left( \left( e_{2n}\right) _{\left( 1\right) }\right) \left( e_{2n}\right)_{\left( 2\right) }-\varepsilon \left( e_{2n}\right) 1 & = S\left(e_{2n}\right) e_{2n}-1=e_{2n+1}e_{2n}-1,
    \end{align*}
    for all \(n\geq 1\) for the even degrees, and
    \begin{align*}
    & S\left( \left( g_{2n+1}\right) _{\left( 1\right) }\right) \left(g_{2n+1}\right) _{\left( 2\right) }-\varepsilon \left( g_{2n+1}\right) 1 \\
    & \hspace{2em} = \frac{1}{2}\left( S\left( g_{2n+1}\right) g_{2n+1}+S\left( g_{2n+1}\right) e_{2n+1}+S\left( e_{2n+1}\right) g_{2n+1}-S\left( e_{2n+1}\right) e_{2n+1}\right) -1 \\
    & \hspace{2em} = \frac{1}{2}\left(g_{2n+2}g_{2n+1}+g_{2n+2}e_{2n+1}+e_{2n+2}g_{2n+1}-e_{2n+2}e_{2n+1}\right) -1, \\
    & S\left( \left( e_{2n+1}\right) _{\left( 1\right) }\right) \left(e_{2n+1}\right) _{\left( 2\right) }-\varepsilon \left( e_{2n+1}\right) 1 \\
    & \hspace{2em} = \frac{1}{2}\left( S\left( e_{2n+1}\right) e_{2n+1}+S\left( e_{2n+1}\right) g_{2n+1}+S\left( g_{2n+1}\right) e_{2n+1}-S\left( g_{2n+1}\right) g_{2n+1}\right) -1 \\
    & \hspace{2em} = \frac{1}{2}\left( e_{2n+2}e_{2n+1}+e_{2n+2}g_{2n+1}+g_{2n+2}e_{2n+1}-g_{2n+2}g_{2n+1}\right) -1,
    \end{align*}
    for all \(n\geq 0\) for the odd degrees. Summing up, \(K\) is the ideal generated by the elements
    \begin{align*}
        g_{2n}g_{2n+1}-1,\qquad \forall n &\geq 0, \\
        e_{2n}e_{2n+1}-1,\qquad \forall n &\geq 0, \\
        g_{2n+1}g_{2n+2}+g_{2n+1}e_{2n+2}+e_{2n+1}g_{2n+2}-e_{2n+1}e_{2n+2}-2,\qquad
        \forall n &\geq 0, \\
        e_{2n+1}e_{2n+2}+e_{2n+1}g_{2n+2}+g_{2n+1}e_{2n+2}-g_{2n+1}g_{2n+2}-2,\qquad
        \forall n &\geq 0, \\
        g_{2n+2}g_{2n+1}+g_{2n+2}e_{2n+1}+e_{2n+2}g_{2n+1}-e_{2n+2}e_{2n+1}-2,\qquad
        \forall n &\geq 0, \\
        e_{2n+2}e_{2n+1}+e_{2n+2}g_{2n+1}+g_{2n+2}e_{2n+1}-g_{2n+2}g_{2n+1}-2,\qquad
        \forall n &\geq 0, \\
        g_{2n+1}g_{2n}-1,\qquad \forall n &\geq 1, \\
        e_{2n+1}e_{2n}-1,\qquad \forall n &\geq 1,
    \end{align*}
    or, equivalently, by the elements
    \begin{subequations}
        \begin{align}
            g_{2n}g_{2n+1}-1,\qquad \forall n &\geq 0, \label{relation1}\\
            e_{2n}e_{2n+1}-1,\qquad \forall n &\geq 0, \label{relation2}\\
            g_{2n+1}g_{2n+2}-e_{2n+1}e_{2n+2},\qquad \forall n &\geq 0, \label{relation3}\\
            e_{2n+1}g_{2n+2}+g_{2n+1}e_{2n+2}-2,\qquad \forall n &\geq 0, \label{relation4}\\
            g_{2n+2}g_{2n+1}-e_{2n+2}e_{2n+1},\qquad \forall n &\geq 0, \label{relation5}\\
            e_{2n+2}g_{2n+1}+g_{2n+2}e_{2n+1}-2,\qquad \forall n &\geq 0, \label{relation6}\\
            g_{2n+1}g_{2n}-1,\qquad \forall n &\geq 1, \label{relation7}\\
            e_{2n+1}e_{2n}-1,\qquad \forall n &\geq 1. \label{relation8}
        \end{align}
    \end{subequations}
    \begin{invisible}
        Because, for instance, the relations
        \begin{eqnarray*}
            g_{2n+1}g_{2n+2}+g_{2n+1}u_{2n+2}+u_{2n+1}g_{2n+2}-u_{2n+1}u_{2n+2}-2\qquad
            \forall n &\geq &0, \\
            u_{2n+1}u_{2n+2}+u_{2n+1}g_{2n+2}+g_{2n+1}u_{2n+2}-g_{2n+1}g_{2n+2}-2\qquad
            \forall n &\geq &0,
        \end{eqnarray*}
        are equivalent to the relations
        \begin{eqnarray*}
            R1+R2 &:&\qquad g_{2n+1}u_{2n+2}+u_{2n+1}g_{2n+2}-2\qquad \forall n\geq 0, \\
            R1-R2 &:&\qquad g_{2n+1}g_{2n+2}-u_{2n+1}u_{2n+2}\qquad \forall n\geq 0.
        \end{eqnarray*}
    \end{invisible}%
    
    Remark, however, that from \eqref{relation1} for \(n = 1\), \eqref{relation2} for \(n = 0\), and \eqref{relation3} for \(n = 0\), we find that
    \begin{align*}
        K & \ni \left( e_{1}e_{0}-1\right) \left( g_{1}\left( g_{2}g_{3}-1\right)
        -\left( g_{1}g_{2}-e_{1}e_{2}\right) g_{3}\right) -e_{1}\left(
        e_{0}e_{1}-1\right) e_{2}g_{3} \\
        & = \left( e_{1}e_{0}-1\right) \left( g_{1}\left( g_{2}g_{3}-1\right) -\left(
        g_{1}g_{2}-e_{1}e_{2}\right) g_{3}-e_{1}e_{2}g_{3}\right)  \\
        & = \left( e_{1}e_{0}-1\right) \left(
        g_{1}g_{2}g_{3}-g_{1}-g_{1}g_{2}g_{3}+e_{1}e_{2}g_{3}-e_{1}e_{2}g_{3}\right) = \left( 1 - e_{1}e_{0}\right) g_{1}.
    \end{align*}
    Therefore, from this and from \eqref{relation2} for \(n = 0\) and \eqref{relation4} for \(n = 0\), we conclude that
    \begin{align*}
        K & \ni \left( e_{1}e_{0}-1\right) g_{1}e_{2}+\left(1 - e_1e_0\right)\left(e_{1}g_{2}+g_{1}e_{2}-2\right) + e_{1}\left( e_{0}e_{1}-1\right)g_{2} \\
        & = \left(1 - e_{1}e_{0}\right)\left(-g_{1}e_{2} + e_{1}g_{2} + g_{1}e_{2} - 2 - e_{1}g_{2}\right) = 2(e_{1}e_{0}-1),
    \end{align*}
    i.e., \(e_1e_0 - 1 \in K\).
    Similarly, we have
    \begin{align*}
        K & \ni \left( g_{1}g_{0}-1\right) \left( e_{1}\left( e_{2}e_{3}-1\right)
        -\left( g_{1}g_{2}-e_{1}e_{2}\right) e_{3}\right) -g_{1}\left(
        g_{0}g_{1}-1\right) g_{2}e_{3} \\
        & = \left( g_{1}g_{0}-1\right) \left( e_{1}\left( e_{2}e_{3}-1\right) -\left(
        g_{1}g_{2}-e_{1}e_{2}\right) e_{3}-g_{1}g_{2}e_{3}\right)  \\
        & = \left( g_{1}g_{0}-1\right) \left(
        e_{1}e_{2}e_{3}-e_{1}-g_{1}g_{2}e_{3}+e_{1}e_{2}e_{3}-g_{1}g_{2}e_{3}\right) = \left( 1-g_{1}g_{0}\right) e_{1},
    \end{align*}
    and hence
    \begin{align*}
        K & \ni \left( g_{1}g_{0}-1\right) e_{1}g_{2}+\left(1-g_{1}g_{0}\right)\left(e_{1}g_{2}+g_{1}e_{2}-2\right) +g_{1}\left( g_{0}g_{1}-1\right) e_{2} \\
         & =\left( 1-g_{1}g_{0}\right) \left(-e_{1}g_{2}+g_{1}e_{2}+e_{1}g_{2}-2-g_{1}e_{2}\right) = 2\left( g_{1}g_{0}-1\right),
    \end{align*}
    i.e., \(g_{1}g_{0}-1 \in K\). The relations \(e_1e_0 = 1\) and \(g_{1}g_{0} = 1\) in \(T(V)/K\) tell us that \(S(y_{(1)})y_{(2)} = \varepsilon(y)1\) also for \(y \in V_0\), and hence the right antipode induced on \(T(V)/K\) is, in fact, a two-sided inverse of the identity.
\end{example}

\begin{invisible}
    \paolo{
    \textbf{Questions:} Is \(T(V)/K\), for \(K = I + T(V)S(I)T(V)\), pointed as a coalgebra? Which is its coradical? Is it cocommutative as a coalgebra? It can be that this case falls under one of the conditions in GNT, which can still be applied to the braided setting.
    }
\end{invisible}

\begin{invisible}[Original matrix example in dim 2x2]
    \begin{example}\label{ex:comodZ2}
        Consider again the Hopf algebra \(H = \K \Z_2 = \K[g \mid g^2 = 1]\) as in \zcref{ex:Z2inZ2}. The category of left \(H\)-comodules  \({}^H\!\cM\) is the category of \(\Z_2\)-graded vector spaces and it is a symmetric monoidal category with respect to the braiding uniquely determined by 
        \[
            \brd\left( u\ot v\right) =(-1)^{\overline{u} \cdot \overline{v}}v\ot u,
        \]
        for \(u\) and \(v\) homogeneous elements of degree \(\overline{u}\) and \(\overline{v}\), respectively. The matrix coalgebra \(C \coloneqq M_{2}(\K)\) is an \(H\)-comodule coalgebra with respect to the coaction
        \[
            \delta (E_{ij}) = g^{i+j}\ot E_{ij}, \qquad \left(i,j \in \{1,2\}\right),
        \]
        where \(E_{ij}\) is the matrix with \(1\) in position \((i,j)\) and \(0\) elsewhere. In \({}^H\!\cM\) with the braiding above, the coopposite comultiplication \(\Delta^\mathrm{cop} = \brd \Delta\) takes the explicit form
        \begin{align*}
            \Delta ^{\mathrm{cop}}\left( E_{ij}\right)
            & = (-1)^{g^{i+1}g^{1+j}}E_{1j}\ot
            E_{i1}+(-1)^{g^{i+2}g^{2+j}}E_{2j}\ot E_{i2} \\
            & = (-1)^{g^{i+j+2}}E_{1j}\ot E_{i1}+(-1)^{g^{i+j+4}}E_{2j}\ot E_{i2}
            \\
            & = (-1)^{g^{i+j}}E_{1j}\ot E_{i1}+(-1)^{g^{i+j}}E_{2j}\ot E_{i2} \\
            & = (-1)^{g^{i+j}}\left( E_{1j}\ot E_{i1}+E_{2j}\ot E_{i2}\right).
        \end{align*}
        We apply the construction described in  \zcref{sect:GNTsym-monoidal}. For \(i\geq 0\), let 
        \[
            V_{2k}\coloneqq C, \qquad  V_{2k+1}\coloneqq C^{\mathrm{cop}},\qquad V\coloneqq \bigoplus_{k} V_k,
        \] 
        and consider the tensor algebra \(T(V) = \bigoplus_n V^{\ot n}\). 
        Denote by \(E_{ij}^{k}\) the matrix \(E_{ij}\) in \(V_k\).
        Then, \(\rH(C) = T(V)/K\) where \(K\) is the ideal generated by the relations
        {\paolo 
        \begin{align*}
            &E_{i1}^{2n} E_{1j}^{2n+1} + E_{i2}^{2n} E_{2j}^{2n+1} - \delta_{ij}1, & n \geq 0, \\
            & E_{1j}^{2n+1} E_{i1}^{2n+2} + E_{2j}^{2n+1} E_{i2}^{2n+2}  -(-1)^{g^{i+j}}\delta_{ij}1, & n \geq 0, \\
            & E_{i1}^{2n+1}E_{1j}^{2n}+E_{i2}^{2n+1}E_{2j}^{2n}-\delta _{ij}1, & n \geq 1, \\
            & E_{1j}^{2n+2}E_{i1}^{2n+1}+E_{2j}^{2n+2}E_{i2}^{2n+1}-(-1)^{g^{i+j}}\delta_{ij}1, & n \geq 0,
        \end{align*}
        for all \(i,j \in \{1,2\}\).}
        \begin{rmv}
            \color{gray}
            By the universal property of the tensor algebra 
            there exist unique algebra morphisms \(\Delta_{T(V)} \colon T(V) \to T(V) \ot T(V)\) and \(\varepsilon_{T(V)} \colon T(V) \to \I\) in \({}^H\!\mathcal{M}\) such that the following diagrams commute
            \[
            \xymatrix @C=70pt {
            V \ar[r]^{\iota_V} \ar[d]_-{\Delta_V} & T(V) \ar@{.>}[d]^-{\Delta_{T(V)}} \\
            V \ot V \ar[r]_-{\iota_V \ot \iota_V} & T(V) \ot T(V)
            }
            \qquad 
            \xymatrix{
            V \ar[rr]^-{\iota_V} \ar[dr]_-{\varepsilon_V} & & T(V) \ar@{.>}[dl]^-{\varepsilon_{T(V)}} \\
            & \I & 
            },
            \]
            where \(\iota_V \colon V \to T(V)\) is the canonical morphism associated with the tensor algebra. Then, \(T(V)\) becomes a bialgebra in \({}^H\!\cM\).
            
            One can further define an anti-morphism of coalgebras \(s \colon V^{\mathrm{cop}} \to V\) extending the identity via
            \[
                \xymatrix{
                    V_n^{\mathrm{cop}} \ar@{=}[d] \ar[r]^-{\imath_n} & V^{\mathrm{cop}} \ar@{.>}[d]^{\exists ! \, s \text{ of coalgebras}} \\
                    V_{n+1} \ar[r]_-{\imath_{n+1}} & V
                }
            \]
            which in turn induces a unique morphism of bialgebras \(S\colon  T(V^{\mathrm{cop}})\to T(V)^{\mathrm{op}}\) in \({}^H\!\cM\) such that
            \[
                \xymatrix{
                    V^{\mathrm{cop}} \ar[r]^-{\iota_V} \ar[d]_-{s} & T(V^{\mathrm{cop}}) \ar@{.>}[d]^-{S} \\
                    V \ar[r]_-{\iota_V} & T(V)^\mathrm{op}.
                }
            \]
            Let \(I\) be the ideal of \(T(V)\) generated by 
            \[
                \left\{m_{T(V)}(T(V) \ot S)\Delta_{T(V)}(v) - u_{T(V)}\varepsilon_{T(V)}(v)\mid v \in V\right\}
            \]
            and let \(K = I + T(V)S(I)T(V)\). Then, \(K\) is an \(S\)-stable bi-ideal of \(T(V)\). In fact, by 
            \begin{align*}
                S^2m_{T(V)}(T(V) \ot S)\Delta_{T(V)} = m_{T(V)}(S^2 \ot S^2)(T(V) \ot S)\Delta_{T(V)} = m_{T(V)}(T(V) \ot S)\Delta_{T(V)} S^2,
            \end{align*}
            we have \(S(T(V)S(I)T(V)) \subseteq I\) and so \(K\) is \(S\)-stable.
        \end{rmv}
        {\paolo This is a right Hopf monoid with anti-multiplicative and anti-comultiplicative right antipode \(S\) uniquely determined by \(S(E_{ij}^k) = E_{ij}^{k+1}\) for \(i,j \in \{1,2\}, k \geq 0\).
        We claim that this is not a left antipode because the relations leading to the left antipode condition, i.e. 
        \begin{align*}
        	& E_{i1}^{1}E_{1j}^0+E_{i2}^1E_{2j}^0-\delta_{ij}1=0, 
        \end{align*}
        for all \(i,j\in\{1,2\}\), does not hold in \(\rH(C)\). In order to do so, we determine a basis of irreducible words for \(T(V)/K\) via Bergman's Diamond Lemma. To this aim, we declare}
        \(E_{ij}^p<E_{hk}^q\) if \(p<q\); if \(p=q\), then we set \(E_{11}^p<E_{12}^p<E_{21}^p<E_{22}^p\). Then, we say that a word \(w\) is bigger than or
        equal to a word \(v\) if and only if:
        \begin{enumerate}[label = \roman*),leftmargin=*]
            \item the word \(w\) is longer than the word \(v\), or
            
            \item the first letter in \(w\) which is different from the corresponding
            letter in \(v\) is bigger than the corresponding letter in \(v\), i.e., in the
            first spot in which they differ we have \(w_{j}\geq v_{j}\).
        \end{enumerate}
        Adopting this order, the reduction formulas are given by 
        \begin{align}
            E_{i2}^{2n}E_{2j}^{2n+1} & ~\longrightarrow~ -E_{i1}^{2n}E_{1j}^{2n+1}+\delta_{ij}1, & n\geq 0,  \label{rel1_old} \\
            E_{2j}^{2n+1}E_{i2}^{2n+2} & ~\longrightarrow~ -E_{1j}^{2n+1}E_{i1}^{2n+2}+(-1)^{g^{i+j}}\delta _{ij}1, & n\geq 0, \label{rel2_old} \\
            E_{i2}^{2n+1}E_{2j}^{2n} & ~\longrightarrow~ -E_{i1}^{2n+1}E_{1j}^{2n}+\delta_{ij}1, & n\geq 1,  \label{rel3_old} \\
            E_{2j}^{2n+2}E_{i2}^{2n+1} & ~\longrightarrow~ -E_{1j}^{2n+2}E_{i1}^{2n+1}+(-1)^{g^{i+j}}\delta _{ij}1, & {\paolo n\geq 0}. \label{rel4_old}
        \end{align}
        To apply Bergman's Diamond Lemma, we analyse the ambiguities arising from \eqref{rel1_old}-\eqref{rel2_old}, \eqref{rel2_old}-\eqref{rel1_old}, {\paolo\sout{for \(n\geq 0\), and from}}
        \eqref{rel1_old}-\eqref{rel3_old}, \eqref{rel2_old}-\eqref{rel4_old}, \eqref{rel4_old}-\eqref{rel2_old}, \eqref{rel3_old}-\eqref{rel1_old}, \eqref{rel3_old}-\eqref{rel4_old}, \eqref{rel4_old}-\eqref{rel3_old}. {\paolo\sout{, for \(n\geq 1\).}}
        One can check that we need to add the following new reduction formulas:
        for \(n\geq 0\)
        \begin{align}
            E_{i2}^{2n} E_{1j}^{2n+1} E_{k1}^{2n+2} & ~\longrightarrow~ E_{i1}^{2n} E_{1j}^{2n+1} E_{k2}^{2n+2} - \delta _{ij}E_{k2}^{2n+2} + (-1)^{g^{j+k}}\delta _{jk}E_{i2}^{2n} \label{rel5_old} \\
            E_{2j}^{2n+1} E_{i1}^{2n+2} E_{1k}^{2n+3} & ~\longrightarrow~ E_{1j}^{2n+1} E_{i1}^{2n+2} E_{2k}^{2n+3} - (-1)^{g^{i+j}}\delta _{ij}E_{2k}^{2n+3} + \delta _{ik}E_{2j}^{2n+1} \label{rel6_old} \\
            {\paolo E_{i2}^{2n+3} E_{1j}^{2n+2} E_{k1}^{2n+1} } & {\paolo  ~\longrightarrow~  E_{i1}^{2n+3} E_{1j}^{2n+2} E_{k2}^{2n+1} + (-1)^{g^{j+k}}\delta_{jk}E_{i2}^{2n+3} - \delta _{ij}E_{k2}^{2n+1} } \label{rel7_old}
        \end{align}
        and for \(n\geq 1\)
        \begin{equation}
        E_{2j}^{2n+2}E_{i1}^{2n+1}E_{1k}^{2n}~\longrightarrow~ E_{1j}^{2n+2}E_{i1}^{2n+1}E_{2k}^{2n}+\delta_{ik}E_{2j}^{2n+2}-(-1)^{g^{i+j}}\delta_{ij}E_{2k}^{2n}\label{rel8_old}
        \end{equation}
        
        For instance, \eqref{rel5_old} is obtained, for \(n\geq 0\), by reducing \(E_{i2}^{2n}E_{2j}^{2n+1}E_{k2}^{2n+2}\) through \eqref{rel1_old} and \eqref{rel2_old}. Using \eqref{rel1_old} we get
        \(-E_{i1}^{2n}E_{1j}^{2n+1}E_{k2}^{2n+2}+\delta _{ij}E_{k2}^{2n+2}\), 
        while by \eqref{rel2_old} we get 
        \(-E_{i2}^{2n}E_{1j}^{2n+1}E_{k1}^{2n+2}+(-1)^{g^{j+k}}\delta _{jk}E_{i2}^{2n}\). Hence \eqref{rel5_old} follows.

        We now analyse the ambiguities arising from \eqref{rel2_old}-\eqref{rel5_old}, 
        \eqref{rel1_old}-\eqref{rel6_old}, \eqref{rel6_old}-\eqref{rel3_old}, for \(n\geq 0\), and from 
        \eqref{rel5_old}-\eqref{rel8_old}, \eqref{rel8_old}-\eqref{rel5_old}, \eqref{rel3_old}-\eqref{rel5_old}, \eqref{rel5_old}-\eqref{rel4_old}, \eqref{rel6_old}-\eqref{rel7_old}, \eqref{rel7_old}-\eqref{rel6_old}, \eqref{rel4_old}-\eqref{rel6_old}, \eqref{rel1_old}-\eqref{rel7_old}, \eqref{rel7_old}-\eqref{rel2_old}, \eqref{rel4_old}-\eqref{rel7_old}, 
        \eqref{rel8_old}-\eqref{rel1_old}, 
        \eqref{rel2_old}-\eqref{rel8_old}, for \(n\geq 1\).
        
        One can check that no new reduction formula needs to be added. For example, we consider the ambiguity arising from \eqref{rel5_old}-\eqref{rel8_old} for \(n\geq 1\). 
        Indeed, looking at \(E_{i2}^{2n}E_{1j}^{2n+1}E_{21}^{2n+2}E_{k1}^{2n+1}E_{1h}^{2n}\) for \(n\geq 1\), it reduces to 
        	\begin{eqnarray*}
        		E_{i1}^{2n}E_{1j}^{2n+1}E_{22}^{2n+2}E_{k1}^{2n+1}E_{1h}^{2n}-\delta
        		_{ij}E_{22}^{2n+2}E_{k1}^{2n+1}E_{1h}^{2n}+(-1)^{g^{2+j}}\delta _{2j}E_{i2}^{2n}E_{k1}^{2n+1}E_{1h}^{2n}
        \end{eqnarray*}
        using \eqref{rel5_old}, and then to 
        	\begin{equation}\label{eq:5-8_old}
        		\begin{split}
        &E_{i1}^{2n}E_{1j}^{2n+1}E_{12}^{2n+2}E_{k1}^{2n+1}E_{2h}^{2n}+\delta_{kh}E_{i1}^{2n}E_{1j}^{2n+1}E_{22}^{2n+2}\\&-(-1)^{g^{k+2}}\delta_{k2}E_{i1}^{2n}E_{1j}^{2n+1}E_{2h}^{2n}-\delta_{ij}E_{12}^{2n+2}E_{k1}^{2n+1}E_{2h}^{2n}-\delta_{ij}\delta_{kh}E_{22}^{2n+2}\\&+\delta_{ij}(-1)^{g^{k+2}}\delta_{k2}E_{2h}^{2n} +(-1)^{g^{2+j}}\delta _{2j}E_{i2}^{2n}E_{k1}^{2n+1}E_{1h}^{2n}
        \end{split}\end{equation}
        using \eqref{rel8_old}. On the other hand, using \eqref{rel8_old}, \(E_{i2}^{2n}E_{1j}^{2n+1}E_{21}^{2n+2}E_{k1}^{2n+1}E_{1h}^{2n}\) reduces to
        \[
        E_{i2}^{2n}E_{1j}^{2n+1}E_{11}^{2n+2}E_{k1}^{2n+1}E_{2h}^{2n}+\delta_{kh}E_{i2}^{2n}E_{1j}^{2n+1}E_{21}^{2n+2}-(-1)^{g^{k+1}}\delta_{k1}E_{i2}^{2n}E_{1j}^{2n+1}E_{2h}^{2n}
        \]
        and then to
        \begin{equation}\label{eq:8-5_old}
        	\begin{split}
        		&E_{i1}^{2n}E_{1j}^{2n+1}E_{12}^{2n+2}E_{k1}^{2n+1}E_{2h}^{2n}-\delta
        _{ij}E_{12}^{2n+2}E_{k1}^{2n+1}E_{2h}^{2n}\\&+(-1)^{g^{j+1}}\delta _{j1}E_{i2}^{2n}E_{k1}^{2n+1}E_{2h}^{2n}+\delta_{kh}E_{i1}^{2n}E_{1j}^{2n+1}E_{22}^{2n+2}-\delta_{kh}\delta
        _{ij}E_{22}^{2n+2}\\&+(-1)^{g^{j+2}}\delta_{kh}\delta _{j2}E_{i2}^{2n}-(-1)^{g^{k+1}}\delta_{k1}E_{i2}^{2n}E_{1j}^{2n+1}E_{2h}^{2n}
        \end{split}
        \end{equation}
        using \eqref{rel5_old}. From \eqref{eq:5-8_old} and \eqref{eq:8-5_old}, we need to compare 
        	\begin{align*}
        		&-(-1)^{g^{k+2}}\delta_{k2}E_{i1}^{2n}E_{1j}^{2n+1}E_{2h}^{2n}+\delta_{ij}(-1)^{g^{k+2}}\delta_{k2}E_{2h}^{2n} +(-1)^{g^{2+j}}\delta _{2j}E_{i2}^{2n}E_{k1}^{2n+1}E_{1h}^{2n}\quad \mathrm{with} \\& (-1)^{g^{j+1}}\delta _{j1}E_{i2}^{2n}E_{k1}^{2n+1}E_{2h}^{2n}+(-1)^{g^{j+2}}\delta_{kh}\delta _{j2}E_{i2}^{2n}-(-1)^{g^{k+1}}\delta_{k1}E_{i2}^{2n}E_{1j}^{2n+1}E_{2h}^{2n},
        	\end{align*}
        for \(i,j,h,k\) in \(\{1,2\}\).
        If \(k=1\), we have 
        	\begin{align*}
        		& (-1)^{g^{2+j}}\delta _{2j}E_{i2}^{2n}E_{11}^{2n+1}E_{1h}^{2n}\qquad 
        		\mathrm{and} \\
        		& (-1)^{g^{j+1}}\delta
        		_{j1}E_{i2}^{2n}E_{11}^{2n+1}E_{2h}^{2n}+(-1)^{g^{j+2}}\delta _{1h}\delta
        		_{j2}E_{i2}^{2n}-E_{i2}^{2n}E_{1j}^{2n+1}E_{2h}^{2n}:
        	\end{align*}
        if \(j=1\) we get \(0\) and \(E_{i2}^{2n}E_{11}^{2n+1}E_{2h}^{2n}-E_{i2}^{2n}E_{11}^{2n+1}E_{2h}^{2n}=0\); if \(j=2\) we get \(E_{i2}^{2n}E_{11}^{2n+1}E_{1h}^{2n}\) and 
        \[
            \delta _{1h}E_{i2}^{2n}-E_{i2}^{2n}E_{12}^{2n+1}E_{2h}^{2n}\overset{\eqref{rel3_old}}{=}\delta _{1h}E_{i2}^{2n}+E_{i2}^{2n}E_{11}^{2n+1}E_{1h}^{2n}-\delta _{1h}E_{i2}^{2n}=E_{i2}^{2n}E_{11}^{2n+1}E_{1h}^{2n}.
        \]
        So, independently from \(i\) and \(h\), we do not obtain new relations.
        
        If \(k=2\), we have 
        	\begin{align*}
        		& -E_{i1}^{2n}E_{1j}^{2n+1}E_{2h}^{2n}+\delta
        		_{ij}E_{2h}^{2n}+(-1)^{g^{2+j}}\delta
        		_{2j}E_{i2}^{2n}E_{21}^{2n+1}E_{1h}^{2n}\qquad \mathrm{and} \\
        		& (-1)^{g^{j+1}}\delta
        		_{j1}E_{i2}^{2n}E_{21}^{2n+1}E_{2h}^{2n}+(-1)^{g^{j+2}}\delta _{2h}\delta
        		_{j2}E_{i2}^{2n}:
        	\end{align*}
        if \(j=1\), we have \(-E_{i1}^{2n}E_{11}^{2n+1}E_{2h}^{2n}+\delta _{i1}E_{2h}^{2n}\) and \(E_{i2}^{2n}E_{21}^{2n+1}E_{2h}^{2n}\overset{\eqref{rel1_old}}{=}-E_{i1}^{2n}E_{11}^{2n+1}E_{2h}^{2n}+\delta _{i1}E_{2h}^{2n}\). Independently from \(i\) and \(h\), we do not obtain new relations. If \(j=2\), we have
        		\begin{align*}
        			& -E_{i1}^{2n}E_{12}^{2n+1}E_{2h}^{2n}+\delta
        			_{i2}E_{2h}^{2n}+E_{i2}^{2n}E_{21}^{2n+1}E_{1h}^{2n}\qquad \mathrm{and} \qquad \delta _{2h}E_{i2}^{2n}.
        		\end{align*}
        By \eqref{rel1_old} \(E_{i2}^{2n}E_{21}^{2n+1}E_{1h}^{2n}\) reduces to \(-E_{i1}^{2n}E_{11}^{2n+1}E_{1h}^{2n}+\delta _{i1}E_{1h}^{2n}\); by \eqref{rel3_old} \(E_{i1}^{2n}E_{12}^{2n+1}E_{2h}^{2n}\) reduces to \(-E_{i1}^{2n}E_{11}^{2n+1}E_{1h}^{2n}+\delta _{1h}E_{i1}^{2n}\), so that we need to compare
        		\begin{align*}
        			& -\delta _{1h}E_{i1}^{2n}+\delta _{i2}E_{2h}^{2n}+\delta
        			_{i1}E_{1h}^{2n}\qquad \mathrm{and} \qquad \delta _{2h}E_{i2}^{2n}.
        		\end{align*}
        If \(h=1\), we have \(-E_{i1}^{2n}+\delta _{i2}E_{21}^{2n}+\delta _{i1}E_{11}^{2n}\) and \(0\); if \(h=2\), we have	\(\delta _{i2}E_{22}^{2n}+\delta _{i1}E_{12}^{2n}\) and \(E_{i2}^{2n}\). Hence, for all \(i\in\{1,2\}\), no new reduction is needed.
        
        {\davide Therefore the set of reductions \eqref{rel1_old}-\eqref{rel8_old} does not give rise to new ambiguities, and is compatible with our chosen monomial order.}  
        {\lucrezia \sout{The left-antipode relations} The relations leading to a left antipode, i.e. 
        \begin{align*}
    E_{i1}^{1}E_{1j}^0+E_{i2}^1E_{2j}^0-\delta_{ij}1=0,
        \end{align*}
        for all \(i,j\in\{1,2\}\), 
        }{\davide
        do not appear in this set, hence by Bergman's Diamond Lemma they are not deducible from the set of reductions: this proves that \(T(V)\) is not a left Hopf algebra.}

        {\paolo Adding a sentence in which we observe that, by Bergman, the irreducible words form a basis and hence the relation we do not want consists of linearly independent elements. Thus, it is not a relation in the quotient and hence there exists at least one element on which the right antipode does not behave as a left antipode.}
    \end{example}
\end{invisible}

\begin{example}\label{ex:MatcomodZ2}
    Consider again the Hopf algebra \(H = \K \Z_2 = \K[g \mid g^2 = e]\) as in \zcref{ex:Z2inZ2}. The category of left \(H\)-comodules  \({}^H\mathfrak{M}\) is the category of \(\Z_2\)-graded vector spaces and it is a symmetric monoidal category with respect to the braiding uniquely determined by 
    \[
        \brd\left( u\ot v\right) =(-1)^{\overline{u} \cdot \overline{v}}v\ot u,
    \]
    for \(u\) and \(v\) homogeneous elements of degree \(\overline{u}\) and \(\overline{v}\), respectively, where \(\left( -1\right)^{e}=1,\left( -1\right) ^{g}=-1\). The matrix coalgebra \(C \coloneqq M_{r}(\K)\) is an \(H\)-comodule coalgebra with respect to the coaction uniquely determined by
    \[
        \delta (E_{ij}) = g^{i+j}\ot E_{ij}, \qquad \left(i,j \in \{1,2,\ldots,r\}\right),
    \]
    where \(E_{ij}\) is the matrix with \(1\) in position \((i,j)\) and \(0\) elsewhere. 
    \begin{invisible}
        Indeed
        \begin{eqnarray*}
            \delta _{\otimes }\left( \Delta \left( E_{ij}\right) \right) &=&\delta
            _{\otimes }\left( \sum_{k=1}^{r}E_{ik}\otimes E_{kj}\right)
            =\sum_{k=1}^{r}g^{i+2k+j}\otimes E_{ik}\otimes E_{kj} \\
            &=&g^{i+j}\otimes \sum_{k=1}^{r}E_{ik}\otimes E_{kj}=\left( H\otimes \Delta
            \right) \left( \delta \left( E_{ij}\right) \right) ,
        \end{eqnarray*}
        so that \(\Delta\) is colinear, while 
        \begin{equation*}
            \delta _{k}\left( \varepsilon \left( E_{ij}\right) \right) =\delta
            _{k}\left( \delta _{ij}\right) =\delta _{ij}1=\left\{ 
            \begin{array}{c}
                0\qquad i\neq j \\ 
                1\qquad i=j
            \end{array}
            \right.
        \end{equation*}
        and
        \begin{equation*}
            \left( H\otimes \varepsilon \right) \left( \delta \left( E_{ij}\right)
            \right) =g^{i+j}\varepsilon \left( E_{ij}\right) =g^{i+j}\delta
            _{ij}=\left\{ 
            \begin{array}{c}
                0\qquad i\neq j \\ 
                1\qquad i=j
            \end{array}
            \right.
        \end{equation*}
        and so also \(\varepsilon\) is colinear.
    \end{invisible}
    In \({}^H\mathfrak{M}\) with the braiding above, the coopposite comultiplication \(\Delta^\mathrm{cop} = \brd \Delta\) takes the explicit form
    \begin{align*}
        \Delta ^{\mathrm{cop}}\left( E_{ij}\right) = \sum_{k=1}^{r} (-1)^{g^{i+2k+j}} E_{kj} \otimes E_{ik} = (-1)^{g^{i+j}} \left( \sum_{k=1}^{r} E_{kj} \otimes E_{ik}\right).
    \end{align*}
    We apply the construction described in  \zcref{sect:GNTsym-monoidal}. For \(n\geq 0\), let 
    \[
        V_{2n}\coloneqq C, \qquad  V_{2n+1}\coloneqq C^{\mathrm{cop}},\qquad V\coloneqq \bigoplus_{n} V_n,
    \] 
    and consider the tensor algebra \(T(V) = \bigoplus_t V^{\ot t}\).
    Denote by \(E_{ij}^{k}\) the matrix \(E_{ij}\) in \(V_k\).
    Then, \(\rH(C) = T(V)/K\) where \(K\) is the ideal generated by the relations
    \begin{align*}
        &\sum_{k=1}^{r}E_{ik}^{2n}E_{kj}^{2n+1}-\delta _{ij}1, && n \geq 0, \\
        &\sum_{k=1}^{r}E_{kj}^{2n+1}E_{ik}^{2n+2}-(-1)^{g^{i+j}}\delta _{ij}1, &&
        n \geq 0, \\
        &\sum_{k=1}^{r}E_{ik}^{2n+1}E_{kj}^{2n}-\delta _{ij}1, && n \geq 1, \\
        &\sum_{k=1}^{r}E_{kj}^{2n+2}E_{ik}^{2n+1}-(-1)^{g^{i+j}}\delta _{ij}1, &&
        n \geq 0,
    \end{align*}
    for all \(i,j \in \{1,2,\ldots,r\}\).
    This is a right Hopf monoid with anti-multiplicative and anti-comultiplicative right antipode \(S\) uniquely determined by \(S(E_{ij}^k) = E_{ij}^{k+1}\) for \(i,j \in \{1,2,\ldots,r\}, k \geq 0\).
    We claim that this is not a left antipode because the relations leading to the left antipode condition, i.e. 
    \begin{equation}\label{eq:badbadbad}
    	\sum_{k=1}^r E_{ik}^{1} E_{kj}^0 - \delta_{ij}1 = 0,
    \end{equation}
    for all \(i,j\in\{1,2,\ldots,r\}\), do not hold in \(\rH(C)\). In order to do so, we determine a basis of irreducible words for \(T(V)/K\) via Bergman's Diamond Lemma. To this aim, we declare \(E_{ij}^{p}<E_{hk}^{q}\) if \(p<q\); if \(p=q\) we set 
    \begin{equation*}
        E_{11}^{p} < E_{12}^{p} < \cdots < E_{1r}^{p} < E_{21}^{p} < E_{22}^{p} < \cdots <E_{2r}^{p} < \cdots < E_{r\left( r-1\right) }^{p} < E_{rr}^{p}.
    \end{equation*}
    Then, we say that a word \(w\) is bigger than or equal to a word \(v\) if and only if:
    \begin{enumerate}[label=\roman*)]
        \item the word \(w\) is longer than the word \(v\), or
        
        \item \(w\) and \(v\) have the same length and the first letter in \(w\) which is different from the corresponding letter in \(v\) is bigger than the corresponding letter in \(v\), i.e., in the first spot in which they differ we have \(w_{j}\geq v_{j}\).
    \end{enumerate}
    Ordered with respect to this order, the reduction formulas are given by
    \begin{subequations}\label{eq:1st_rels}
        \begin{align}
            E_{ir}^{2n}E_{rj}^{2n+1} & \longrightarrow -\sum_{h=1}^{r-1} E_{ih}^{2n} E_{hj}^{2n+1} + \delta _{ij}1, && n\geq 0, \label{eq:rel1} \\
            E_{rj}^{2n+1} E_{ir}^{2n+2} & \longrightarrow - \sum_{h=1}^{r-1} E_{hj}^{2n+1} E_{ih}^{2n+2} +  \delta_{ij}1,&& n\geq 0,  \label{eq:rel2} \\
            E_{ir}^{2n+1} E_{rj}^{2n} & \longrightarrow -\sum_{h=1}^{r-1} E_{ih}^{2n+1} E_{hj}^{2n} + \delta _{ij}1, && n\geq 1, \label{eq:rel3} \\
            E_{rj}^{2n+2} E_{ir}^{2n+1} & \longrightarrow - \sum_{h=1}^{r-1} E_{hj}^{2n+2} E_{ih}^{2n+1} + \delta_{ij}1, && n\geq 0,  \label{eq:rel4}
        \end{align}        
    \end{subequations}
    for \(i,j \in \{1,2,\ldots,r\}\), because obviously \((-1)^{g^{i+j}}\delta_{ij}1 = (-1)^{g^{2i}}\delta_{ij}1 = \delta_{ij}1\). Therefore, as a \(\K\)-algebra, \(\rH(C)\) coincides with the one obtained by Green, Nichols, and Taft in \cite{GreenNicholsTaft} and so \eqref{eq:badbadbad} is not a valid relation in \(\rH(C)\). Nevertheless, let us verify this fact again, for the sake of completeness.
    
    To apply Bergman's Diamond Lemma, we analyse the ambiguities arising from \eqref{eq:1st_rels}. Some tedious, but otherwise straightforward, verifications show that the ambiguities arising from \eqref{eq:rel1}-\eqref{eq:rel3}, \eqref{eq:rel3}-\eqref{eq:rel1}, \eqref{eq:rel2}-\eqref{eq:rel4}, and \eqref{eq:rel4}-\eqref{eq:rel2} are resolvable. For instance, \eqref{eq:rel2} and \eqref{eq:rel4}, in this order, give rise to the overlap ambiguity     \(E_{rj}^{2n+1}E_{rr}^{2n+2}E_{ir}^{2n+1}\) for every \(n\geq 0\). However, the following chains of reductions
    \begin{align*}    
        & E_{rj}^{2n+1 }E_{rr}^{2n+2} E_{ir}^{2n+1} \xrightarrow{\eqref{eq:rel2}} 
        - \sum_{h=1}^{r-1} E_{hj}^{2n+1} E_{rh}^{2n+2} E_{ir}^{2n+1} + \delta _{rj} E_{ir}^{2n+1} \\
        & \xrightarrow{\eqref{eq:rel4}} \sum_{h=1}^{r-1} \sum_{k=1}^{r-1} E_{hj}^{2n+1} E_{kh}^{2n+2} E_{ik}^{2n+1} - \sum_{h=1}^{r-1} \delta _{ih} E_{hj}^{2n+1} + \delta_{rj}E_{ir}^{2n+1} \\
        & = \sum_{h=1}^{r-1} \sum_{k=1}^{r-1} E_{hj}^{2n+1} E_{kh}^{2n+2} E_{ik}^{2n+1} - \left( 1 - \delta_{ir} - \delta_{rj}\right) E_{ij}^{2n+1}
    \end{align*}
    and
    \begin{align*}
        & E_{rj}^{2n+1} E_{rr}^{2n+2} E_{ir}^{2n+1} \xrightarrow{\eqref{eq:rel4}} - \sum_{k=1}^{r-1} E_{rj}^{2n+1} E_{kr}^{2n+2} E_{ik}^{2n+1} + \delta _{ir} E_{rj}^{2n+1} \\
        & \xrightarrow{\eqref{eq:rel2}} \sum_{k=1}^{r-1} \sum_{h=1}^{r-1} E_{hj}^{2n+1} E_{kh}^{2n+2} E_{ik}^{2n+1} - \sum_{k=1}^{r-1} \delta _{kj} E_{ik}^{2n+1} + \delta_{ir} E_{rj}^{2n+1} \\
        & = \sum_{k=1}^{r-1} \sum_{h=1}^{r-1} E_{hj}^{2n+1} E_{kh}^{2n+2} E_{ik}^{2n+1} - \left( 1 - \delta_{rj} - \delta_{ir}\right) E_{ij}^{2n+1},
    \end{align*}
    show that the ambiguity is resolvable. Some more computations show that, instead, the ambiguities arising from \eqref{eq:rel1}-\eqref{eq:rel2}, \eqref{eq:rel2}-\eqref{eq:rel1}, \eqref{eq:rel3}-\eqref{eq:rel4}, and \eqref{eq:rel4}-\eqref{eq:rel3}, are not resolvable and so they lead us to consider the additional reduction formulas
    \begin{subequations}\label{eq:2nd_rels}
        \begin{align}
            & E_{ir}^{2n} E_{\left( r-1\right) j}^{2n+1} E_{k\left( r-1\right) }^{2n+2}  && \label{eq:rel5} \\
            & \rightarrow -\sum_{h=1}^{r-2} E_{ir}^{2n} E_{hj}^{2n+1} E_{kh}^{2n+2} + \sum_{h=1}^{r-1} E_{ih}^{2n} E_{hj}^{2n+1} E_{kr}^{2n+2} - \delta_{ij} E_{kr}^{2n+2} + \delta _{kj} E_{ir}^{2n}, && n\geq 0, \notag \\
            & E_{rj}^{2n+1} E_{i\left( r-1\right) }^{2n+2} E_{\left( r-1\right) k}^{2n+3} && \label{eq:rel6} \\
            & \rightarrow -\sum_{h=1}^{r-2} E_{rj}^{2n+1} E_{ih}^{2n+2} E_{hk}^{2n+3} + \sum_{h=1}^{r-1} E_{hj}^{2n+1} E_{ih}^{2n+2} E_{rk}^{2n+3} - \delta_{ij} E_{rk}^{2n+3} + \delta_{ik} E_{rj}^{2n+1}, && n\geq 0, \notag  \\
            & E_{rj}^{2n+2} E_{i\left( r-1\right) }^{2n+1} E_{\left( r-1\right) k}^{2n}  && \label{eq:rel7} \\
            & \rightarrow - \sum_{h=1}^{r-2} E_{rj}^{2n+2} E_{ih}^{2n+1} E_{hk}^{2n} + \sum_{h=1}^{r-1} E_{hj}^{2n+2} E_{ih}^{2n+1} E_{rk}^{2n} + \delta_{ik} E_{rj}^{2n+2} - \delta _{ij} E_{rk}^{2n}, && n\geq 1, \notag \\
            & E_{ir}^{2n+3} E_{\left( r-1\right) j}^{2n+2} E_{k\left( r-1\right) }^{2n+1} && \label{eq:rel8} \\
            & \rightarrow - \sum_{h=1}^{r-2} E_{ir}^{2n+3} E_{hj}^{2n+2} E_{kh}^{2n+1} + \sum_{h=1}^{r-1} E_{ih}^{2n+3} E_{hj}^{2n+2} E_{kr}^{2n+1} + \delta_{kj} E_{ir}^{2n+3} - \delta _{ij} E_{kr}^{2n+1}, && n\geq 0. \notag  
        \end{align}
    \end{subequations}    
    For instance, \eqref{eq:rel5} is obtained from reducing \(E_{ir}^{2n} E_{\left( r-1\right) j}^{2n+1} E_{k\left( r-1\right) }^{2n+2}\) via \eqref{eq:rel1} and \eqref{eq:rel2}, and then equating the results.

    Now, all the ambiguities arising from \eqref{eq:1st_rels} are resolvable, whence we analyse the additional ambiguities arising from \eqref{eq:2nd_rels} and from the interaction between \eqref{eq:1st_rels} and \eqref{eq:2nd_rels}. In total, they are \eqref{eq:rel1}-\eqref{eq:rel6}, \eqref{eq:rel1}-\eqref{eq:rel8}, \eqref{eq:rel2}-\eqref{eq:rel5}, \eqref{eq:rel2}-\eqref{eq:rel7}, \eqref{eq:rel3}-\eqref{eq:rel5}, \eqref{eq:rel3}-\eqref{eq:rel7}, \eqref{eq:rel4}-\eqref{eq:rel6}, \eqref{eq:rel4}-\eqref{eq:rel8}, \eqref{eq:rel5}-\eqref{eq:rel4}, \eqref{eq:rel5}-\eqref{eq:rel7}, \eqref{eq:rel6}-\eqref{eq:rel3}, \eqref{eq:rel6}-\eqref{eq:rel8}, \eqref{eq:rel7}-\eqref{eq:rel1}, \eqref{eq:rel7}-\eqref{eq:rel5}, \eqref{eq:rel8}-\eqref{eq:rel2}, \eqref{eq:rel8}-\eqref{eq:rel6}. Long and tedious computations show that all these ambiguities are resolvable. 
    \begin{invisible}
        For example, if we consider the ambiguity \(E_{rj}^{2n+2} E_{i\left( r-1\right)}^{2n+1} E_{\left( r-1\right) r}^{2n} E_{\left( r-1\right) l}^{2n+1} E_{k\left(r-1\right) }^{2n+2}\) for \(i,j,k,l \in \{1,2,\ldots,r\}\) and  \(n\geq 1\) arising from \eqref{eq:rel7} and \eqref{eq:rel5}, then the following chains of reductions
        {\small
        \begin{align*}
            E_{rj}^{2n+2} & E_{i\left( r-1\right) }^{2n+1} E_{\left( r-1\right) r}^{2n} E_{\left( r-1\right) l}^{2n+1} E_{k\left( r-1\right) }^{2n+2} \\
            \xrightarrow{\eqref{eq:rel7}} & - \sum_{h=1}^{r-2} E_{rj}^{2n+2} E_{ih}^{2n+1} E_{hr}^{2n} E_{\left( r-1\right)l}^{2n+1} E_{k\left( r-1\right)}^{2n+2} + \sum_{h=1}^{r-1} E_{hj}^{2n+2} E_{ih}^{2n+1} E_{rr}^{2n} E_{\left(r-1\right) l}^{2n+1} E_{k\left( r-1\right) }^{2n+2} + \\ 
            & \quad + \delta _{ir} E_{rj}^{2n+2} E_{\left( r-1\right) l}^{2n+1} E_{k\left( r-1\right)}^{2n+2} - \delta _{ij} E_{rr}^{2n} E_{\left( r-1\right)l}^{2n+1} E_{k\left( r-1\right) }^{2n+2} 
            \\
            = & -\sum_{h=1}^{r-2}E_{rj}^{2n+2} E_{ih}^{2n+1} E_{hr}^{2n} E_{\left(r-1\right) l}^{2n+1} E_{k\left( r-1\right) }^{2n+2} + \left(\sum_{h=1}^{r-1} E_{hj}^{2n+2} E_{ih}^{2n+1} -  \delta_{ij}1\right) E_{rr}^{2n}E_{\left( r-1\right) l}^{2n+1} E_{k\left( r-1\right)
            }^{2n+2} + \\
            & \quad + \delta _{ir} E_{rj}^{2n+2} E_{\left( r-1\right) l}^{2n+1} E_{k\left(
            r-1\right) }^{2n+2} \\
            \xrightarrow{\eqref{eq:rel5}} & \sum_{h=1}^{r-2} E_{rj}^{2n+2} E_{ih}^{2n+1} \left( \sum_{m=1}^{r-2} E_{hr}^{2n} E_{ml}^{2n+1} E_{km}^{2n+2} -\sum_{m=1}^{r-1}E_{hm}^{2n}E_{ml}^{2n+1}E_{kr}^{2n+2}+\delta_{hl}E_{kr}^{2n+2} - \delta _{kl}E_{hr}^{2n}\right) + \\
            & \quad - \left( \sum_{h=1}^{r-1}E_{hj}^{2n+2}E_{ih}^{2n+1}-\delta_{ij}1\right) \Bigg( \sum_{h=1}^{r-2}E_{rr}^{2n}E_{hl}^{2n+1}E_{kh}^{2n+2} - \sum_{h=1}^{r-1}E_{rh}^{2n}E_{hl}^{2n+1}E_{kr}^{2n+2} + \\ 
            & \quad\quad  +\delta_{rl}E_{kr}^{2n+2} - \delta _{kl}E_{rr}^{2n}\Bigg) +\delta_{ir}E_{rj}^{2n+2}E_{\left( r-1\right) l}^{2n+1}E_{k\left( r-1\right) }^{2n+2} \\
            = & E_{rj}^{2n+2}\left( \sum_{h=1}^{r-2}E_{ih}^{2n+1}\left( E_{hr}^{2n}\left( \sum_{m=1}^{r-2}E_{ml}^{2n+1}E_{km}^{2n+2}-\delta _{kl}1\right) -\left( \sum_{m=1}^{r-1}E_{hm}^{2n}E_{ml}^{2n+1}-\delta _{hl}1\right) E_{kr}^{2n+2}\right) \right) + \\
            & \quad -\left( \sum_{h=1}^{r-1}E_{hj}^{2n+2}E_{ih}^{2n+1}-\delta _{ij}1\right) \Bigg( E_{rr}^{2n}\left( \sum_{h=1}^{r-2}E_{hl}^{2n+1}E_{kh}^{2n+2}-\delta _{kl}1\right) + \\
            & \quad \quad -\left( \sum_{h=1}^{r-1}E_{rh}^{2n}E_{hl}^{2n+1}-\delta _{rl}1\right) E_{kr}^{2n+2}\Bigg) + \delta _{ir}E_{rj}^{2n+2}E_{\left( r-1\right) l}^{2n+1}E_{k\left( r-1\right) }^{2n+2} \\
            = & E_{rj}^{2n+2}\left( \sum_{h=1}^{r-2}E_{ih}^{2n+1}E_{hr}^{2n}\right) \left( \sum_{m=1}^{r-2}E_{ml}^{2n+1}E_{km}^{2n+2}-\delta _{kl}1\right) + \\ 
            & \quad - E_{rj}^{2n+2}\left( \sum_{h=1}^{r-2}E_{ih}^{2n+1}\left( \sum_{m=1}^{r-1}E_{hm}^{2n}E_{ml}^{2n+1}-\delta _{hl}1\right) \right) E_{kr}^{2n+2}+ \\
            & \quad \quad -\left( \sum_{h=1}^{r-1}E_{hj}^{2n+2}E_{ih}^{2n+1}-\delta _{ij}1\right) E_{rr}^{2n}\left( \sum_{h=1}^{r-2}E_{hl}^{2n+1}E_{kh}^{2n+2}-\delta _{kl}1\right) + \\ 
            & \quad \quad \quad +\left( \sum_{h=1}^{r-1}E_{hj}^{2n+2}E_{ih}^{2n+1}-\delta _{ij}1\right) \left( \sum_{h=1}^{r-1}E_{rh}^{2n}E_{hl}^{2n+1}-\delta _{rl}1\right) E_{kr}^{2n+2}+\delta _{ir}E_{rj}^{2n+2}E_{\left( r-1\right) l}^{2n+1}E_{k\left( r-1\right) }^{2n+2} \\
            = & E_{rj}^{2n+2}\left( \sum_{h=1}^{r-2}E_{ih}^{2n+1}E_{hr}^{2n}\right)
            \left( \sum_{m=1}^{r-2}E_{ml}^{2n+1}E_{km}^{2n+2}-\delta
            _{kl}1\right) + \\ 
            & \quad -E_{rj}^{2n+2}\left(\sum_{h=1}^{r-2}\sum_{m=1}^{r-1}E_{ih}^{2n+1}E_{hm}^{2n}E_{ml}^{2n+1}\right)
            E_{kr}^{2n+2}+\left( 1-\delta _{rl}-\delta _{\left( r-1\right) l}1\right)
            E_{rj}^{2n+2}E_{il}^{2n+1}E_{kr}^{2n+2}+ \\
            & \quad \quad  -\left(
            \sum_{h=1}^{r-1}E_{hj}^{2n+2}E_{ih}^{2n+1}-\delta
            _{ij}1\right) E_{rr}^{2n}\left(
            \sum_{h=1}^{r-2}E_{hl}^{2n+1}E_{kh}^{2n+2}-\delta
            _{kl}1\right) + \\ \
            & \quad \quad \quad +\left(
            \sum_{h=1}^{r-1}E_{hj}^{2n+2}E_{ih}^{2n+1}-\delta
            _{ij}1\right) \left( \sum_{h=1}^{r-1}E_{rh}^{2n}E_{hl}^{2n+1}\right)
            E_{kr}^{2n+2}+ \\
            & \quad \quad \quad \quad  -\delta _{rl}\left(
            \sum_{h=1}^{r-1}E_{hj}^{2n+2}E_{ih}^{2n+1}-\delta
            _{ij}1\right) E_{kr}^{2n+2}+\delta _{ir}E_{rj}^{2n+2}E_{\left( r-1\right)
            l}^{2n+1}E_{k\left( r-1\right) }^{2n+2} \\
            = & E_{rj}^{2n+2}\left( \sum_{h=1}^{r-2}E_{ih}^{2n+1}E_{hr}^{2n}\right)
            \left( \sum_{m=1}^{r-2}E_{ml}^{2n+1}E_{km}^{2n+2}-\delta
            _{kl}1\right) + \\ 
            & \quad -E_{rj}^{2n+2}\left(
            \sum_{h=1}^{r-2}\sum_{m=1}^{r-1}E_{ih}^{2n+1}E_{hm}^{2n}E_{ml}^{2n+1}\right)
            E_{kr}^{2n+2}+\left( 1-\delta _{\left( r-1\right) l}1\right)
            E_{rj}^{2n+2}E_{il}^{2n+1}E_{kr}^{2n+2}+ \\
            & \quad \quad  -\left(
            \sum_{h=1}^{r-1}E_{hj}^{2n+2}E_{ih}^{2n+1}-\delta
            _{ij}1\right) E_{rr}^{2n}\left(
            \sum_{h=1}^{r-2}E_{hl}^{2n+1}E_{kh}^{2n+2}-\delta
            _{kl}1\right) + \\
            & \quad \quad \quad +\left(\sum_{h=1}^{r-1}E_{hj}^{2n+2}E_{ih}^{2n+1}-\delta
            _{ij}1\right) \left( \sum_{h=1}^{r-1}E_{rh}^{2n}E_{hl}^{2n+1}\right)
            E_{kr}^{2n+2}+ \\
            & \quad \quad \quad \quad  -\delta _{rl}\left(
            E_{rj}^{2n+2}E_{ir}^{2n+1} +
            \sum_{h=1}^{r-1}E_{hj}^{2n+2}E_{ih}^{2n+1}-\delta
            _{ij}1\right) E_{kr}^{2n+2} + \\
            & \quad \quad \quad \quad \quad +\delta _{ir}E_{rj}^{2n+2}E_{\left( r-1\right)
            l}^{2n+1}E_{k\left( r-1\right) }^{2n+2} \\
            \xrightarrow{\eqref{eq:rel4}} & E_{rj}^{2n+2}\left(
            \sum_{h=1}^{r-2}E_{ih}^{2n+1}E_{hr}^{2n}\right) \left(
            \sum_{m=1}^{r-2}E_{ml}^{2n+1}E_{km}^{2n+2}-\delta _{kl}1\right) + \\
            & \quad -E_{rj}^{2n+2}\left(
            \sum_{h=1}^{r-2}\sum_{m=1}^{r-1}E_{ih}^{2n+1}E_{hm}^{2n}E_{ml}^{2n+1}\right)
            E_{kr}^{2n+2}+\left( 1-\delta _{\left( r-1\right) l}1\right)
            E_{rj}^{2n+2}E_{il}^{2n+1}E_{kr}^{2n+2}+ \\
            & \quad \quad  -\left(
            \sum_{h=1}^{r-1}E_{hj}^{2n+2}E_{ih}^{2n+1}-\delta
            _{ij}1\right) E_{rr}^{2n}\left(
            \sum_{h=1}^{r-2}E_{hl}^{2n+1}E_{kh}^{2n+2}-\delta
            _{kl}1\right) + \\
            & \quad \quad \quad +\left(
            \sum_{h=1}^{r-1}E_{hj}^{2n+2}E_{ih}^{2n+1}-\delta
            _{ij}1\right) \left( \sum_{h=1}^{r-1}E_{rh}^{2n}E_{hl}^{2n+1}\right)
            E_{kr}^{2n+2}+\delta _{ir}E_{rj}^{2n+2}E_{\left( r-1\right)
            l}^{2n+1}E_{k\left( r-1\right) }^{2n+2}
        \end{align*}    }
        and
        {\small
        \begin{align*}
            E_{rj}^{2n+2}& E_{i\left( r-1\right) }^{2n+1}E_{\left( r-1\right) r}^{2n}E_{\left( r-1\right) l}^{2n+1}E_{k\left( r-1\right) }^{2n+2} \\
            \xrightarrow{\eqref{eq:rel5}} & -\sum_{m=1}^{r-2}E_{rj}^{2n+2}E_{i\left( r-1\right) }^{2n+1}E_{\left( r-1\right) r}^{2n}E_{ml}^{2n+1}E_{km}^{2n+2}+\sum_{m=1}^{r-1}E_{rj}^{2n+2}E_{i\left( r-1\right) }^{2n+1}E_{\left( r-1\right) m}^{2n}E_{ml}^{2n+1}E_{kr}^{2n+2} \\
            & \quad -\delta _{\left( r-1\right) l}E_{rj}^{2n+2}E_{i\left( r-1\right) }^{2n+1}E_{kr}^{2n+2}+\delta _{kl}E_{rj}^{2n+2}E_{i\left( r-1\right) }^{2n+1}E_{\left( r-1\right) r}^{2n} \\
            = & E_{rj}^{2n+2}E_{i\left( r-1\right) }^{2n+1}E_{\left( r-1\right) r}^{2n}\left( -\sum_{m=1}^{r-2}E_{ml}^{2n+1}E_{km}^{2n+2}+\delta _{kl}1\right)  \\
            & \quad +\sum_{m=1}^{r-1}E_{rj}^{2n+2}E_{i\left( r-1\right) }^{2n+1}E_{\left( r-1\right) m}^{2n}E_{ml}^{2n+1}E_{kr}^{2n+2}-\delta _{\left( r-1\right) l}E_{rj}^{2n+2}E_{i\left( r-1\right) }^{2n+1}E_{kr}^{2n+2} \\
            \xrightarrow{\eqref{eq:rel7}} & \Bigg( -\sum_{h=1}^{r-2}E_{rj}^{2n+2}E_{ih}^{2n+1}E_{hr}^{2n}+ \sum_{h=1}^{r-1}E_{hj}^{2n+2}E_{ih}^{2n+1}E_{rr}^{2n} + \\
            & \quad +\delta _{ir}E_{rj}^{2n+2}-\delta _{ij}E_{rr}^{2n}\Bigg) \left( -\sum_{m=1}^{r-2}E_{ml}^{2n+1}E_{km}^{2n+2}+\delta _{kl}1\right) + \\
            & \quad \quad  +\sum_{m=1}^{r-1}\Bigg( -\sum_{h=1}^{r-2}E_{rj}^{2n+2}E_{ih}^{2n+1}E_{hm}^{2n}+ \sum_{h=1}^{r-1}E_{hj}^{2n+2}E_{ih}^{2n+1}E_{rm}^{2n} + \\
            & \quad \quad \quad +\delta _{im}E_{rj}^{2n+2}-\delta _{ij}E_{rm}^{2n}\Bigg) E_{ml}^{2n+1}E_{kr}^{2n+2}-\delta _{\left( r-1\right) l}E_{rj}^{2n+2}E_{i\left( r-1\right) }^{2n+1}E_{kr}^{2n+2} \\
            = & \left( E_{rj}^{2n+2}\left( \sum_{h=1}^{r-2}E_{ih}^{2n+1}E_{hr}^{2n}-\delta _{ir}1\right) -\left( \sum_{h=1}^{r-1}E_{hj}^{2n+2}E_{ih}^{2n+1}-\delta _{ij}1\right) E_{rr}^{2n}\right) \cdot \\
            & \quad \cdot \left( \sum_{m=1}^{r-2}E_{ml}^{2n+1}E_{km}^{2n+2}-\delta _{kl}1\right) +\Bigg( \sum_{m=1}^{r-1}\Bigg( -E_{rj}^{2n+2}\left( \sum_{h=1}^{r-2}E_{ih}^{2n+1}E_{hm}^{2n}-\delta _{im}1\right) + \\
            & \quad \quad +\left( \sum_{h=1}^{r-1}E_{hj}^{2n+2}E_{ih}^{2n+1}-\delta _{ij}1\right) E_{rm}^{2n}\Bigg) E_{ml}^{2n+1}\Bigg) E_{kr}^{2n+2}-\delta _{\left( r-1\right) l}E_{rj}^{2n+2}E_{i\left( r-1\right) }^{2n+1}E_{kr}^{2n+2} \\
            = & E_{rj}^{2n+2}\left( \sum_{h=1}^{r-2}E_{ih}^{2n+1}E_{hr}^{2n} -\delta _{ir}1\right) \left( \sum_{m=1}^{r-2}E_{ml}^{2n+1}E_{km}^{2n+2}-\delta _{kl}1\right) + \\
            & \quad -\left( \sum_{h=1}^{r-1}E_{hj}^{2n+2}E_{ih}^{2n+1}-\delta _{ij}1\right) E_{rr}^{2n}\left( \sum_{m=1}^{r-2}E_{ml}^{2n+1}E_{km}^{2n+2}-\delta _{kl}1\right) + \\
            & \quad \quad  -E_{rj}^{2n+2}\left( \sum_{m=1}^{r-1}\left( \sum_{h=1}^{r-2}E_{ih}^{2n+1}E_{hm}^{2n}-\delta _{im}1\right) E_{ml}^{2n+1}\right) E_{kr}^{2n+2} + \\
            & \quad \quad \quad +\left( \sum_{h=1}^{r-1}E_{hj}^{2n+2}E_{ih}^{2n+1}-\delta _{ij}1\right) \left( \sum_{m=1}^{r-1}E_{rm}^{2n}E_{ml}^{2n+1}\right) E_{kr}^{2n+2} + \\
            & \quad \quad \quad \quad -\delta _{\left( r-1\right) l}E_{rj}^{2n+2}E_{i\left( r-1\right) }^{2n+1}E_{kr}^{2n+2} \\
            = & E_{rj}^{2n+2}\left( \sum_{h=1}^{r-2}E_{ih}^{2n+1}E_{hr}^{2n}\right) \left( \sum_{m=1}^{r-2}E_{ml}^{2n+1}E_{km}^{2n+2} -\delta _{kl}1\right) + \\
            & \quad -\delta _{ir}E_{rj}^{2n+2}\left( \sum_{m=1}^{r-2}E_{ml}^{2n+1}E_{km}^{2n+2}-\delta _{kl}1\right) + \\
            & \quad \quad - \left( \sum_{h=1}^{r-1}E_{hj}^{2n+2}E_{ih}^{2n+1}-\delta _{ij}1\right) E_{rr}^{2n}\left( \sum_{m=1}^{r-2}E_{ml}^{2n+1}E_{km}^{2n+2}-\delta _{kl}1\right) + \\
            & \quad \quad \quad  -E_{rj}^{2n+2}\left( \sum_{m=1}^{r-1}\left( \sum_{h=1}^{r-2}E_{ih}^{2n+1}E_{hm}^{2n}-\delta _{im}1\right) E_{ml}^{2n+1}\right) E_{kr}^{2n+2} + \\
            & \quad \quad \quad \quad +\left( \sum_{h=1}^{r-1}E_{hj}^{2n+2}E_{ih}^{2n+1}-\delta _{ij}1\right) \left( \sum_{m=1}^{r-1}E_{rm}^{2n}E_{ml}^{2n+1}\right) E_{kr}^{2n+2} + \\
            & \quad \quad \quad \quad \quad -\delta _{\left( r-1\right) l}E_{rj}^{2n+2}E_{i\left( r-1\right) }^{2n+1}E_{kr}^{2n+2} \\
            =&E_{rj}^{2n+2}\left( \sum_{h=1}^{r- 2}E_{ih}^{2n+1}E_{hr}^{2n}\right) \left( \sum_{m=1}^{r-2}E_{ml}^{2n+1}E_{km}^{2n+2} -\delta _{kl}1\right) + \\
            & \quad -\delta _{ir}E_{rj}^{2n+2}\left( \sum_{m=1}^{r-2}E_{ml}^{2n+1}E_{km}^{2n+2}-\delta _{kl}1\right) + \\
            & \quad \quad -\left( \sum_{h=1}^{r-1}E_{hj}^{2n+2}E_{ih}^{2n+1}-\delta _{ij}1\right) E_{rr}^{2n}\left( \sum_{m=1}^{r-2}E_{ml}^{2n+1}E_{km}^{2n+2}-\delta _{kl}1\right) + \\
            & \quad \quad \quad  -E_{rj}^{2n+2}\left( \sum_{m=1}^{r-1}\sum_{h=1}^{r-2}E_{ih}^{2n+1}E_{hm}^{2n}E_{ml}^{2n+1}\right) E_{kr}^{2n+2}+\left( 1-\delta _{ir}1\right) E_{rj}^{2n+2}E_{il}^{2n+1}E_{kr}^{2n+2} + \\
            & \quad \quad \quad \quad +\left( \sum_{h=1}^{r-1}E_{hj}^{2n+2}E_{ih}^{2n+1}-\delta _{ij}1\right) \left( \sum_{m=1}^{r-1}E_{rm}^{2n}E_{ml}^{2n+1}\right) E_{kr}^{2n+2} + \\
            & \quad \quad \quad \quad \quad -\delta _{\left( r-1\right) l}E_{rj}^{2n+2}E_{i\left( r-1\right) }^{2n+1}E_{kr}^{2n+2} \\
            = & E_{rj}^{2n+2}\left( \sum_{h=1}^{r-2}E_{ih}^{2n+1}E_{hr}^{2n}\right) \left( \sum_{m=1}^{r-2}E_{ml}^{2n+1}E_{km}^{2n+2}-\delta _{kl}1\right) + \\
            & \quad -\delta _{ir}E_{rj}^{2n+2}\left( E_{rl}^{2n+1}E_{kr}^{2n+2}+\sum_{m=1}^{r-2}E_{ml}^{2n+1}E_{km}^{2n+2} -\delta _{kl}1\right) + \\
            & \quad \quad -\left( \sum_{h=1}^{r-1}E_{hj}^{2n+2}E_{ih}^{2n+1}-\delta _{ij}1\right) E_{rr}^{2n}\left( \sum_{m=1}^{r-2}E_{ml}^{2n+1}E_{km}^{2n+2}-\delta _{kl}1\right) + \\
            & \quad \quad \quad -E_{rj}^{2n+2}\left(
            \sum_{m=1}^{r-1}\sum_{h=1}^{r-2}E_{ih}^{2n+1}E_{hm}^{2n}E_{ml}^{2n+1}\right)
            E_{kr}^{2n+2} + \\
            & \quad \quad \quad \quad +\left( \sum_{h=1}^{r-1}E_{hj}^{2n+2}E_{ih}^{2n+1}-\delta _{ij}1\right) \left( \sum_{m=1}^{r-1}E_{rm}^{2n}E_{ml}^{2n+1}\right) E_{kr}^{2n+2} + \\
            & \quad \quad \quad \quad \quad +\left( 1-\delta _{\left( r-1\right) l}1\right) E_{rj}^{2n+2}E_{i\left( r-1\right) }^{2n+1}E_{kr}^{2n+2} \\
            \xrightarrow{\eqref{eq:rel2}} & E_{rj}^{2n+2}\left( \sum_{h=1}^{r-2}E_{ih}^{2n+1}E_{hr}^{2n}\right) \left( \sum_{m=1}^{r-2}E_{ml}^{2n+1}E_{km}^{2n+2}-\delta _{kl}1\right) +\delta _{ir}E_{rj}^{2n+2}E_{\left( r-1\right) l}^{2n+1}E_{k\left( r-1\right) }^{2n+2} + \\
            & \quad -\left( \sum_{h=1}^{r-1}E_{hj}^{2n+2}E_{ih}^{2n+1}-\delta _{ij}1\right) E_{rr}^{2n}\left(
            \sum_{m=1}^{r-2}E_{ml}^{2n+1}E_{km}^{2n+2}-\delta _{kl}1\right) + \\
            & \quad \quad  -E_{rj}^{2n+2}\left( \sum_{m=1}^{r-1}\sum_{h=1}^{r-2}E_{ih}^{2n+1}E_{hm}^{2n}E_{ml}^{2n+1}\right) E_{kr}^{2n+2} + \\
            & \quad \quad \quad + \left( \sum_{h=1}^{r-1}E_{hj}^{2n+2}E_{ih}^{2n+1}-\delta _{ij}1\right) \left( \sum_{m=1}^{r-1}E_{rm}^{2n}E_{ml}^{2n+1}\right) E_{kr}^{2n+2} + \\
            & \quad \quad \quad \quad +\left( 1-\delta _{\left( r-1\right) l}1\right) E_{rj}^{2n+2}E_{i\left( r-1\right) }^{2n+1}E_{kr}^{2n+2}
        \end{align*}    }
        show that the ambiguity is resolvable.
    \end{invisible}
    Therefore, by Bergman's Diamond Lemma, the words that are irreducible with respect to \eqref{eq:1st_rels} and \eqref{eq:2nd_rels} form a basis of the quotient \(T(V)/K\). Since the summands of 
    \begin{equation}\label{eq:uffa}
        \sum_{k=1}^r E_{ik}^{1}E_{kj}^0 - \delta_{ij}1
    \end{equation}
    for all \(i,j\in\{1,2,\ldots,r\}\), are irreducible words, they are linearly independent and so \eqref{eq:uffa} cannot be \(0\). Hence, \(\rH(C)\) is not a Hopf monoid.
\end{example}

\begin{invisible}[Rough computations of the \(M_r(\K)\) example]
    \(C=M_{r}(k)\) is a comodule over \(H=k\mathbb{Z}_{2}=k\left[ g\mid g^{2}=e%
    \right]\) with respect to%
    \begin{equation*}
    \delta (E_{ij})=g^{i+j}\otimes E_{ij}.
    \end{equation*}%
    It is a comodule coalgebra, too, with respect to%
    \begin{equation*}
    \Delta (E_{ij})=\sum_{k=1}^{r}E_{ik}\otimes E_{kj}\qquad \mathrm{and}\qquad
    \varepsilon (E_{ij})=\delta _{ij}.
    \end{equation*}%
    Indeed%
    \begin{eqnarray*}
    \delta _{\otimes }\left( \Delta \left( E_{ij}\right) \right) &=&\delta
    _{\otimes }\left( \sum_{k=1}^{r}E_{ik}\otimes E_{kj}\right)
    =\sum_{k=1}^{r}g^{i+2k+j}\otimes E_{ik}\otimes E_{kj} \\
    &=&g^{i+j}\otimes \sum_{k=1}^{r}E_{ik}\otimes E_{kj}=\left( H\otimes \Delta
    \right) \left( \delta \left( E_{ij}\right) \right) ,
    \end{eqnarray*}%
    so that \(\Delta\) is colinear, while 
    \begin{equation*}
    \delta _{k}\left( \varepsilon \left( E_{ij}\right) \right) =\delta
    _{k}\left( \delta _{ij}\right) =\delta _{ij}1=\left\{ 
    \begin{array}{c}
    0\qquad i\neq j \\ 
    1\qquad i=j%
    \end{array}%
    \right.
    \end{equation*}%
    and%
    \begin{equation*}
    \left( H\otimes \varepsilon \right) \left( \delta \left( E_{ij}\right)
    \right) =g^{i+j}\varepsilon \left( E_{ij}\right) =g^{i+j}\delta
    _{ij}=\left\{ 
    \begin{array}{c}
    0\qquad i\neq j \\ 
    1\qquad i=j%
    \end{array}%
    \right.
    \end{equation*}%
    and so also \(\varepsilon\) is colinear.
    
    We can look at \(M_{n}(k)\) as a comodule coalgebra in the braided category of 
    \(H\)-comodules with respect to the braiding%
    \begin{equation*}
    \mathfrak{c}:U\otimes V\rightarrow V\otimes U,\qquad u\otimes v\mapsto (-1)^{%
    \overline{u}\cdot \overline{v}}v\otimes u,
    \end{equation*}%
    where \(u\in U\) is homogeneous of degree \(\overline{u}\) and \(v\in V\) is
    homogeneous of degree \(\overline{v}\), and where \(\left( -1\right)
    ^{e}=1,\left( -1\right) ^{g}=-1\). Then,%
    \begin{equation*}
    \Delta ^{\mathrm{cop}}\left( E_{i,j}\right)
    =\sum_{k=1}^{r}(-1)^{g^{i+2k+j}}E_{kj}\otimes E_{ik}=(-1)^{g^{i+j}}\left(
    \sum_{k=1}^{r}E_{kj}\otimes E_{ik}\right)
    \end{equation*}%
    If we apply GNT's procedure, we shall quotient by the relations%
    \begin{eqnarray*}
    \sum_{k=1}^{r}E_{ik}^{2n}E_{kj}^{2n+1}-\delta _{ij}1,\qquad n &\geq &0, \\
    \sum_{k=1}^{r}E_{kj}^{2n+1}E_{ik}^{2n+2}-(-1)^{g^{i+j}}\delta _{ij}1,\qquad
    n &\geq &0, \\
    \sum_{k=1}^{r}E_{ik}^{2n+1}E_{kj}^{2n}-\delta _{ij}1,\qquad n &\geq &1, \\
    \sum_{k=1}^{r}E_{kj}^{2n+2}E_{ik}^{2n+1}-(-1)^{g^{i+j}}\delta _{ij}1,\qquad
    n &\geq &0.
    \end{eqnarray*}
    
    To apply Bergman's Diamond Lemma, let us analyse ambiguity by ambiguity.
    
    We declare \(E_{ij}^{p}<E_{hk}^{q}\) if \(p<q\); if \(p=q\) we have 
    \begin{equation*}
    E_{11}^{p}<E_{12}^{p}<\cdots <E_{1r}^{p}<E_{21}^{p}<E_{22}^{p}<\cdots
    <E_{2r}^{p}<\cdots <E_{r\left( r-1\right) }^{p}<E_{rr}^{p}
    \end{equation*}%
    . Then, a word \(w\) is bigger than or equal to a word \(v\) if and only if:
    
    \begin{enumerate}
    \item[i)] the word \(w\) is longer than the word \(v\), or
    
    \item[ii)] the first letter in \(w\) which is different from the corresponding
    letter in \(v\) is bigger than the corresponding letter in \(v\), i.e., in the
    first spot in which they differ we have \(w_{j}\geq v_{j}\).
    \end{enumerate}
    
    Ordered with respect to this order, the reduction formulas are given by%
    \begin{eqnarray}
    E_{ir}^{2n}E_{rj}^{2n+1} &\rightarrow
    &-\sum_{h=1}^{r-1}E_{ih}^{2n}E_{hj}^{2n+1}+\delta _{ij}1,\qquad n\geq 0,
    \label{rel1} \\
    E_{rj}^{2n+1}E_{ir}^{2n+2} &\rightarrow
    &-\sum_{h=1}^{r-1}E_{hj}^{2n+1}E_{ih}^{2n+2}+(-1)^{g^{i+j}}\delta
    _{ij}1,\qquad n\geq 0,  \label{rel2} \\
    E_{ir}^{2n+1}E_{rj}^{2n} &\rightarrow
    &-\sum_{h=1}^{r-1}E_{ih}^{2n+1}E_{hj}^{2n}+\delta _{ij}1,\qquad n\geq 1,
    \label{rel3} \\
    E_{rj}^{2n+2}E_{ir}^{2n+1} &\rightarrow
    &-\sum_{h=1}^{r-1}E_{hj}^{2n+2}E_{ih}^{2n+1}+(-1)^{g^{i+j}}\delta
    _{ij}1,\qquad n\geq 0.  \label{rel4}
    \end{eqnarray}
    
    \subsection{I round of ambiguites}
    
    The ambiguities are then
    
    \begin{itemize}
    \item \ref{rel1}-\ref{rel2} \(E_{ir}^{2n}E_{rj}^{2n+1}E_{kr}^{2n+2}\) for \(%
    n\geq 0\)%
    \begin{equation*}
    E_{ir}^{2n}E_{rj}^{2n+1}E_{kr}^{2n+2}\overset{\ref{rel1}}{\rightarrow }%
    -\sum_{h=1}^{r-1}E_{ih}^{2n}E_{hj}^{2n+1}E_{kr}^{2n+2}+\delta
    _{ij}E_{kr}^{2n+2}
    \end{equation*}%
    while%
    \begin{equation*}
    E_{ir}^{2n}E_{rj}^{2n+1}E_{kr}^{2n+2}\overset{\ref{rel2}}{\rightarrow }%
    -\sum_{h=1}^{r-1}E_{ir}^{2n}E_{hj}^{2n+1}E_{kh}^{2n+2}+(-1)^{g^{k+j}}\delta
    _{kj}E_{ir}^{2n}
    \end{equation*}%
    gives the additional reduction%
    \begin{equation*}
    E_{ir}^{2n}E_{\left( r-1\right) j}^{2n+1}E_{k\left( r-1\right)
    }^{2n+2}\rightarrow
    -\sum_{h=1}^{r-2}E_{ir}^{2n}E_{hj}^{2n+1}E_{kh}^{2n+2}+%
    \sum_{h=1}^{r-1}E_{ih}^{2n}E_{hj}^{2n+1}E_{kr}^{2n+2}-\delta
    _{ij}E_{kr}^{2n+2}+(-1)^{g^{k+j}}\delta _{kj}E_{ir}^{2n}.
    \end{equation*}
    
    \item \ref{rel1}-\ref{rel3} \(E_{ir}^{2n}E_{rr}^{2n+1}E_{rj}^{2n}\) for \(n\geq
    1\)%
    \begin{eqnarray*}
    &&E_{ir}^{2n}E_{rr}^{2n+1}E_{rj}^{2n}\overset{\ref{rel1}}{\rightarrow }%
    -\sum_{h=1}^{r-1}E_{ih}^{2n}E_{hr}^{2n+1}E_{rj}^{2n}+\delta _{ir}E_{rj}^{2n}%
    \overset{\ref{rel3}}{\rightarrow }\sum_{h=1}^{r-1}%
    \sum_{k=1}^{r-1}E_{ih}^{2n}E_{hk}^{2n+1}E_{kj}^{2n}-\sum_{h=1}^{r-1}\delta
    _{hj}E_{ih}^{2n}+\delta _{ir}E_{rj}^{2n} \\
    &=&\sum_{h=1}^{r-1}\sum_{k=1}^{r-1}E_{ih}^{2n}E_{hk}^{2n+1}E_{kj}^{2n}-%
    \left( 1-\delta _{rj}\right) E_{ij}^{2n}+\delta _{ir}E_{rj}^{2n},
    \end{eqnarray*}%
    while%
    \begin{eqnarray*}
    &&E_{ir}^{2n}E_{rr}^{2n+1}E_{rj}^{2n}\overset{\ref{rel3}}{\rightarrow }%
    -\sum_{h=1}^{r-1}E_{ir}^{2n}E_{rh}^{2n+1}E_{hj}^{2n}+\delta _{rj}E_{ir}^{2n}%
    \overset{\ref{rel1}}{\rightarrow }\sum_{h=1}^{r-1}%
    \sum_{k=1}^{r-1}E_{ik}^{2n}E_{kh}^{2n+1}E_{hj}^{2n}-\sum_{h=1}^{r-1}\delta
    _{ih}E_{hj}^{2n}+\delta _{rj}E_{ir}^{2n} \\
    &=&\sum_{h=1}^{r-1}\sum_{k=1}^{r-1}E_{ik}^{2n}E_{kh}^{2n+1}E_{hj}^{2n}-%
    \left( 1-\delta _{ir}\right) E_{ij}^{2n}+\delta _{rj}E_{ir}^{2n}.
    \end{eqnarray*}%
    Since%
    \begin{equation*}
    \left( 1-\delta _{rj}\right) E_{ij}^{2n}-\delta _{ir}E_{rj}^{2n}=\left(
    1-\delta _{ir}\right) E_{ij}^{2n}-\delta _{rj}E_{ir}^{2n}
    \end{equation*}%
    because case by case%
    \begin{equation*}
    \left\{ 
    \begin{array}{c}
    i=r,j=r,\qquad -E_{rr}^{2n}\qquad \mathrm{VS}\qquad -E_{rr}^{2n}\qquad
    \left( OK\right)  \\ 
    i=r,j<r,\qquad E_{rj}^{2n}-E_{rj}^{2n}\qquad \mathrm{VS}\qquad 0\qquad
    \left( OK\right)  \\ 
    i<r,j=r,\qquad 0\qquad \mathrm{VS}\qquad E_{ir}^{2n}-E_{ir}^{2n}\qquad
    \left( OK\right)  \\ 
    i<r,j<r,\qquad E_{ij}^{2n}\qquad \mathrm{VS}\qquad E_{ij}^{2n}\qquad \left(
    OK\right) 
    \end{array}%
    \right. 
    \end{equation*}%
    so that no new reduction needs to be added.
    
    \item \ref{rel2}-\ref{rel1} \(E_{rj}^{2n+1}E_{ir}^{2n+2}E_{rk}^{2n+3}\) for \(%
    n\geq 0\)%
    \begin{equation*}
    E_{rj}^{2n+1}E_{ir}^{2n+2}E_{rk}^{2n+3}\overset{\ref{rel2}}{\rightarrow }%
    -\sum_{h=1}^{r-1}E_{hj}^{2n+1}E_{ih}^{2n+2}E_{rk}^{2n+3}+(-1)^{g^{i+j}}%
    \delta _{ij}E_{rk}^{2n+3}
    \end{equation*}%
    while%
    \begin{equation*}
    E_{rj}^{2n+1}E_{ir}^{2n+2}E_{rk}^{2n+3}\overset{\ref{rel1}}{\rightarrow }%
    -\sum_{h=1}^{r-1}E_{rj}^{2n+1}E_{ih}^{2n+2}E_{hk}^{2n+3}+\delta
    _{ik}E_{rj}^{2n+1}
    \end{equation*}%
    gives the additional reduction%
    \begin{equation*}
    E_{rj}^{2n+1}E_{i\left( r-1\right) }^{2n+2}E_{\left( r-1\right)
    k}^{2n+3}\rightarrow
    -\sum_{h=1}^{r-2}E_{rj}^{2n+1}E_{ih}^{2n+2}E_{hk}^{2n+3}+%
    \sum_{h=1}^{r-1}E_{hj}^{2n+1}E_{ih}^{2n+2}E_{rk}^{2n+3}-(-1)^{g^{i+j}}\delta
    _{ij}E_{rk}^{2n+3}+\delta _{ik}E_{rj}^{2n+1}
    \end{equation*}
    
    \item \ref{rel2}-\ref{rel4} \(E_{rj}^{2n+1}E_{rr}^{2n+2}E_{ir}^{2n+1}\) for \(%
    n\geq 0\)%
    \begin{eqnarray*}
    &&E_{rj}^{2n+1}E_{rr}^{2n+2}E_{ir}^{2n+1}\overset{\ref{rel2}}{\rightarrow }%
    -\sum_{h=1}^{r-1}E_{hj}^{2n+1}E_{rh}^{2n+2}E_{ir}^{2n+1}+(-1)^{g^{r+j}}%
    \delta _{rj}E_{ir}^{2n+1} \\
    &&\overset{\ref{rel4}}{\rightarrow }\sum_{h=1}^{r-1}%
    \sum_{k=1}^{r-1}E_{hj}^{2n+1}E_{kh}^{2n+2}E_{ik}^{2n+1}-%
    \sum_{h=1}^{r-1}(-1)^{g^{i+h}}\delta _{ih}E_{hj}^{2n+1}+(-1)^{g^{r+j}}\delta
    _{rj}E_{ir}^{2n+1} \\
    &=&\sum_{h=1}^{r-1}\sum_{k=1}^{r-1}E_{hj}^{2n+1}E_{kh}^{2n+2}E_{ik}^{2n+1}-%
    \left( 1-\delta _{ir}\right) E_{ij}^{2n+1}+(-1)^{g^{r+j}}\delta
    _{rj}E_{ir}^{2n+1},
    \end{eqnarray*}%
    while%
    \begin{eqnarray*}
    &&E_{rj}^{2n+1}E_{rr}^{2n+2}E_{ir}^{2n+1}\overset{\ref{rel4}}{\rightarrow }%
    -\sum_{h=1}^{r-1}E_{rj}^{2n+1}E_{hr}^{2n+2}E_{ih}^{2n+1}+(-1)^{g^{i+r}}%
    \delta _{ir}E_{rj}^{2n+1} \\
    &&\overset{\ref{rel2}}{\rightarrow }\sum_{h=1}^{r-1}%
    \sum_{k=1}^{r-1}E_{kj}^{2n+1}E_{hk}^{2n+2}E_{ih}^{2n+1}-%
    \sum_{h=1}^{r-1}(-1)^{g^{h+j}}\delta _{hj}E_{ih}^{2n+1}+(-1)^{g^{i+r}}\delta
    _{ir}E_{rj}^{2n+1} \\
    &=&\sum_{h=1}^{r-1}\sum_{k=1}^{r-1}E_{kj}^{2n+1}E_{hk}^{2n+2}E_{ih}^{2n+1}-%
    \left( 1-\delta _{rj}\right) E_{ij}^{2n+1}+(-1)^{g^{i+r}}\delta
    _{ir}E_{rj}^{2n+1}.
    \end{eqnarray*}%
    Since%
    \begin{equation*}
    -\left( 1-\delta _{ir}\right) E_{ij}^{2n+1}+(-1)^{g^{r+j}}\delta
    _{rj}E_{ir}^{2n+1}=-\left( 1-\delta _{rj}\right)
    E_{ij}^{2n+1}+(-1)^{g^{i+r}}\delta _{ir}E_{rj}^{2n+1}
    \end{equation*}%
    because case by case%
    \begin{equation*}
    \left\{ 
    \begin{array}{c}
    i=r,j=r\qquad E_{rr}^{2n+1}=E_{rr}^{2n+1} \\ 
    i=r,j<r\qquad 0=-E_{rj}^{2n+1}+E_{rj}^{2n+1} \\ 
    i<r,j=r\qquad -E_{ir}^{2n+1}+E_{ir}^{2n+1}=0 \\ 
    i<r,j<r\qquad -E_{ij}^{2n+1}=-E_{ij}^{2n+1}%
    \end{array}%
    \right. 
    \end{equation*}%
    so that no new reduction needs to be added.
    
    \item \ref{rel3}-\ref{rel1} \(E_{ir}^{2n+1}E_{rr}^{2n}E_{rj}^{2n+1}\) for \(%
    n\geq 1\),%
    \begin{eqnarray*}
    &&E_{ir}^{2n+1}E_{rr}^{2n}E_{rj}^{2n+1}\overset{\ref{rel3}}{\rightarrow }%
    -\sum_{h=1}^{r-1}E_{ih}^{2n+1}E_{hr}^{2n}E_{rj}^{2n+1}+\delta
    _{ir}E_{rj}^{2n+1} \\
    &&\overset{\ref{rel1}}{\rightarrow }\sum_{h=1}^{r-1}%
    \sum_{k=1}^{r-1}E_{ih}^{2n+1}E_{hk}^{2n}E_{kj}^{2n+1}-\sum_{h=1}^{r-1}\delta
    _{hj}E_{ih}^{2n+1}+\delta _{ir}E_{rj}^{2n+1} \\
    &=&\sum_{h=1}^{r-1}\sum_{k=1}^{r-1}E_{ih}^{2n+1}E_{hk}^{2n}E_{kj}^{2n+1}-%
    \left( 1-\delta _{rj}\right) E_{ij}^{2n+1}+\delta _{ir}E_{rj}^{2n+1},
    \end{eqnarray*}%
    while
    
    \begin{eqnarray*}
    &&E_{ir}^{2n+1}E_{rr}^{2n}E_{rj}^{2n+1}\overset{\ref{rel1}}{\rightarrow }%
    -\sum_{h=1}^{r-1}E_{ir}^{2n+1}E_{rh}^{2n}E_{hj}^{2n+1}+\delta
    _{rj}E_{ir}^{2n+1} \\
    &&\overset{\ref{rel3}}{\rightarrow }\sum_{h=1}^{r-1}%
    \sum_{k=1}^{r-1}E_{ik}^{2n+1}E_{kh}^{2n}E_{hj}^{2n+1}-\sum_{h=1}^{r-1}\delta
    _{ih}E_{hj}^{2n+1}+\delta _{rj}E_{ir}^{2n+1} \\
    &=&\sum_{h=1}^{r-1}\sum_{k=1}^{r-1}E_{ik}^{2n+1}E_{kh}^{2n}E_{hj}^{2n+1}-%
    \left( 1-\delta _{ir}\right) E_{ij}^{2n+1}+\delta _{rj}E_{ir}^{2n+1}.
    \end{eqnarray*}%
    Since%
    \begin{equation*}
    -\left( 1-\delta _{rj}\right) E_{ij}^{2n+1}+\delta
    _{ir}E_{rj}^{2n+1}=-\left( 1-\delta _{ir}\right) E_{ij}^{2n+1}+\delta
    _{rj}E_{ir}^{2n+1}
    \end{equation*}%
    because case by case%
    \begin{eqnarray*}
    i &=&r,j=r\qquad E_{rr}^{2n+1}=E_{rr}^{2n+1}, \\
    i &=&r,j<r\qquad -E_{rj}^{2n+1}+E_{rj}^{2n+1}=0, \\
    i &<&r,j=r\qquad 0=-E_{ir}^{2n+1}+E_{ir}^{2n+1}, \\
    i &<&r,j<r\qquad -E_{ij}^{2n+1}=-E_{ij}^{2n+1},
    \end{eqnarray*}%
    so that no new reduction needs to be added.
    
    \item \ref{rel3}-\ref{rel4} \(E_{ir}^{2n+3}E_{rj}^{2n+2}E_{kr}^{2n+1}\) for \(%
    n\geq 0\).
    
    \begin{equation*}
    E_{ir}^{2n+3}E_{rj}^{2n+2}E_{kr}^{2n+1}\overset{\ref{rel3}}{\rightarrow }%
    -\sum_{h=1}^{r-1}E_{ih}^{2n+3}E_{hj}^{2n+2}E_{kr}^{2n+1}+\delta
    _{ij}E_{kr}^{2n+1}
    \end{equation*}%
    while 
    \begin{equation*}
    E_{ir}^{2n+3}E_{rj}^{2n+2}E_{kr}^{2n+1}\overset{\ref{rel4}}{\rightarrow }%
    -\sum_{h=1}^{r-1}E_{ir}^{2n+3}E_{hj}^{2n+2}E_{kh}^{2n+1}+(-1)^{g^{k+j}}%
    \delta _{kj}E_{ir}^{2n+3}
    \end{equation*}%
    so we have the additional reduction%
    \begin{equation*}
    E_{ir}^{2n+3}E_{\left( r-1\right) j}^{2n+2}E_{k\left( r-1\right)
    }^{2n+1}\rightarrow
    -\sum_{h=1}^{r-2}E_{ir}^{2n+3}E_{hj}^{2n+2}E_{kh}^{2n+1}+%
    \sum_{h=1}^{r-1}E_{ih}^{2n+3}E_{hj}^{2n+2}E_{kr}^{2n+1}+(-1)^{g^{k+j}}\delta
    _{kj}E_{ir}^{2n+3}-\delta _{ij}E_{kr}^{2n+1}
    \end{equation*}
    
    \item \ref{rel4}-\ref{rel2} \(E_{rj}^{2n+2}E_{rr}^{2n+1}E_{ir}^{2n+2}\) for \(%
    n\geq 0\),%
    \begin{eqnarray*}
    &&E_{rj}^{2n+2}E_{rr}^{2n+1}E_{ir}^{2n+2}\overset{\ref{rel4}}{\rightarrow }%
    -\sum_{h=1}^{r-1}E_{hj}^{2n+2}E_{rh}^{2n+1}E_{ir}^{2n+2}+(-1)^{g^{r+j}}%
    \delta _{rj}E_{ir}^{2n+2} \\
    &&\overset{\ref{rel2}}{\rightarrow }\sum_{h=1}^{r-1}%
    \sum_{k=1}^{r-1}E_{hj}^{2n+2}E_{kh}^{2n+1}E_{ik}^{2n+2}-%
    \sum_{h=1}^{r-1}(-1)^{g^{i+h}}\delta _{ih}E_{hj}^{2n+2}+(-1)^{g^{r+j}}\delta
    _{rj}E_{ir}^{2n+2} \\
    &=&\sum_{h=1}^{r-1}\sum_{k=1}^{r-1}E_{hj}^{2n+2}E_{kh}^{2n+1}E_{ik}^{2n+2}-%
    \left( 1-\delta _{ir}\right) E_{ij}^{2n+2}+(-1)^{g^{r+j}}\delta
    _{rj}E_{ir}^{2n+2}
    \end{eqnarray*}%
    while 
    \begin{eqnarray*}
    &&E_{rj}^{2n+2}E_{rr}^{2n+1}E_{ir}^{2n+2}\overset{\ref{rel2}}{\rightarrow }%
    -\sum_{h=1}^{r-1}E_{rj}^{2n+2}E_{hr}^{2n+1}E_{ih}^{2n+2}+(-1)^{g^{i+r}}%
    \delta _{ir}E_{rj}^{2n+2} \\
    &&\overset{\ref{rel4}}{\rightarrow }\sum_{h=1}^{r-1}%
    \sum_{k=1}^{r-1}E_{kj}^{2n+2}E_{hk}^{2n+1}E_{ih}^{2n+2}-%
    \sum_{h=1}^{r-1}(-1)^{g^{h+j}}\delta _{hj}E_{ih}^{2n+2}+(-1)^{g^{i+r}}\delta
    _{ir}E_{rj}^{2n+2} \\
    &=&\sum_{h=1}^{r-1}\sum_{k=1}^{r-1}E_{kj}^{2n+2}E_{hk}^{2n+1}E_{ih}^{2n+2}-%
    \left( 1-\delta _{rj}\right) E_{ij}^{2n+2}+(-1)^{g^{i+r}}\delta
    _{ir}E_{rj}^{2n+2}.
    \end{eqnarray*}%
    Since%
    \begin{equation*}
    -\left( 1-\delta _{ir}\right) E_{ij}^{2n+2}+(-1)^{g^{r+j}}\delta
    _{rj}E_{ir}^{2n+2}=-\left( 1-\delta _{rj}\right)
    E_{ij}^{2n+2}+(-1)^{g^{i+r}}\delta _{ir}E_{rj}^{2n+2}
    \end{equation*}%
    because, case by case,%
    \begin{eqnarray*}
    i &=&r,j=r,\qquad E_{rr}^{2n+2}=E_{rr}^{2n+2} \\
    i &=&r,j<r,\qquad 0=-E_{rj}^{2n+2}+E_{rj}^{2n+2} \\
    i &<&r,j=r,\qquad -E_{ir}^{2n+2}+E_{ir}^{2n+2}=0 \\
    i &<&r,j<r,\qquad -E_{ij}^{2n+2}=-E_{ij}^{2n+2}
    \end{eqnarray*}%
    so no new reduction needs to be added.
    
    \item \ref{rel4}-\ref{rel3} \(E_{rj}^{2n+2}E_{ir}^{2n+1}E_{rk}^{2n}\) for \(%
    n\geq 1\)%
    \begin{equation*}
    E_{rj}^{2n+2}E_{ir}^{2n+1}E_{rk}^{2n}\overset{\ref{rel4}}{\rightarrow }%
    -\sum_{h=1}^{r-1}E_{hj}^{2n+2}E_{ih}^{2n+1}E_{rk}^{2n}+(-1)^{g^{i+j}}\delta
    _{ij}E_{rk}^{2n}
    \end{equation*}%
    while 
    \begin{equation*}
    E_{rj}^{2n+2}E_{ir}^{2n+1}E_{rk}^{2n}\overset{\ref{rel3}}{\rightarrow }%
    -\sum_{h=1}^{r-1}E_{rj}^{2n+2}E_{ih}^{2n+1}E_{hk}^{2n}+\delta
    _{ik}E_{rj}^{2n+2}
    \end{equation*}%
    so we have the additional reduction%
    \begin{equation*}
    E_{rj}^{2n+2}E_{i\left( r-1\right) }^{2n+1}E_{\left( r-1\right)
    k}^{2n}\rightarrow
    -\sum_{h=1}^{r-2}E_{rj}^{2n+2}E_{ih}^{2n+1}E_{hk}^{2n}+%
    \sum_{h=1}^{r-1}E_{hj}^{2n+2}E_{ih}^{2n+1}E_{rk}^{2n}+\delta
    _{ik}E_{rj}^{2n+2}-(-1)^{g^{i+j}}\delta _{ij}E_{rk}^{2n}
    \end{equation*}
    \end{itemize}
    
    \begin{eqnarray*}
    E_{ir}^{2n}E_{rj}^{2n+1} &\rightarrow
    &-\sum_{h=1}^{r-1}E_{ih}^{2n}E_{hj}^{2n+1}+\delta _{ij}1,\qquad n\geq 0, \\
    E_{rj}^{2n+1}E_{ir}^{2n+2} &\rightarrow
    &-\sum_{h=1}^{r-1}E_{hj}^{2n+1}E_{ih}^{2n+2}+(-1)^{g^{i+j}}\delta
    _{ij}1,\qquad n\geq 0, \\
    E_{ir}^{2n+1}E_{rj}^{2n} &\rightarrow
    &-\sum_{h=1}^{r-1}E_{ih}^{2n+1}E_{hj}^{2n}+\delta _{ij}1,\qquad n\geq 1, \\
    E_{rj}^{2n+2}E_{ir}^{2n+1} &\rightarrow
    &-\sum_{h=1}^{r-1}E_{hj}^{2n+2}E_{ih}^{2n+1}+(-1)^{g^{i+j}}\delta
    _{ij}1,\qquad n\geq 0.
    \end{eqnarray*}
    
    \begin{itemize}
    \item 
    \medskip 
    
    Recall: 
    \begin{eqnarray*}
    E_{ir}^{2n}E_{rj}^{2n+1} &\rightarrow
    &-\sum_{h=1}^{r-1}E_{ih}^{2n}E_{hj}^{2n+1}+\delta _{ij}1,\qquad n\geq 0, \\
    E_{rj}^{2n+1}E_{ir}^{2n+2} &\rightarrow
    &-\sum_{h=1}^{r-1}E_{hj}^{2n+1}E_{ih}^{2n+2}+(-1)^{g^{i+j}}\delta
    _{ij}1,\qquad n\geq 0, \\
    E_{ir}^{2n+1}E_{rj}^{2n} &\rightarrow
    &-\sum_{h=1}^{r-1}E_{ih}^{2n+1}E_{hj}^{2n}+\delta _{ij}1,\qquad n\geq 1, \\
    E_{rj}^{2n+2}E_{ir}^{2n+1} &\rightarrow
    &-\sum_{h=1}^{r-1}E_{hj}^{2n+2}E_{ih}^{2n+1}+(-1)^{g^{i+j}}\delta
    _{ij}1,\qquad n\geq 0.
    \end{eqnarray*}
    
    To sum up, we need to add:%
    \begin{eqnarray}
    E_{ir}^{2n}E_{\left( r-1\right) j}^{2n+1}E_{k\left( r-1\right) }^{2n+2}
    &\rightarrow
    &-\sum_{h=1}^{r-2}E_{ir}^{2n}E_{hj}^{2n+1}E_{kh}^{2n+2}+%
    \sum_{h=1}^{r-1}E_{ih}^{2n}E_{hj}^{2n+1}E_{kr}^{2n+2}-\delta
    _{ij}E_{kr}^{2n+2}+(-1)^{g^{k+j}}\delta _{kj}E_{ir}^{2n},\qquad n\geq 0,
    \label{rel5} \\
    E_{rj}^{2n+1}E_{i\left( r-1\right) }^{2n+2}E_{\left( r-1\right) k}^{2n+3}
    &\rightarrow
    &-\sum_{h=1}^{r-2}E_{rj}^{2n+1}E_{ih}^{2n+2}E_{hk}^{2n+3}+%
    \sum_{h=1}^{r-1}E_{hj}^{2n+1}E_{ih}^{2n+2}E_{rk}^{2n+3}-(-1)^{g^{i+j}}\delta
    _{ij}E_{rk}^{2n+3}+\delta _{ik}E_{rj}^{2n+1},\qquad n\geq 0,  \label{rel6} \\
    E_{rj}^{2n+2}E_{i\left( r-1\right) }^{2n+1}E_{\left( r-1\right) k}^{2n}
    &\rightarrow
    &-\sum_{h=1}^{r-2}E_{rj}^{2n+2}E_{ih}^{2n+1}E_{hk}^{2n}+%
    \sum_{h=1}^{r-1}E_{hj}^{2n+2}E_{ih}^{2n+1}E_{rk}^{2n}+\delta
    _{ik}E_{rj}^{2n+2}-(-1)^{g^{i+j}}\delta _{ij}E_{rk}^{2n},\qquad n\geq 1,
    \label{rel7} \\
    E_{ir}^{2n+3}E_{\left( r-1\right) j}^{2n+2}E_{k\left( r-1\right) }^{2n+1}
    &\rightarrow
    &-\sum_{h=1}^{r-2}E_{ir}^{2n+3}E_{hj}^{2n+2}E_{kh}^{2n+1}+%
    \sum_{h=1}^{r-1}E_{ih}^{2n+3}E_{hj}^{2n+2}E_{kr}^{2n+1}+(-1)^{g^{k+j}}\delta
    _{kj}E_{ir}^{2n+3}-\delta _{ij}E_{kr}^{2n+1},\qquad \ n\geq 0,  \label{rel8}
    \end{eqnarray}
    \end{itemize}
    
    \begin{remark}
    Note that, in order not to have a two-sided antipode, we need to not obtain
    the following relations, for all \(i,j\in \{1,2\}\), 
    \begin{equation*}
    E_{i1}^{1}E_{1j}^{0}+E_{i2}^{1}E_{2j}^{0}-\delta _{ij}1=0.
    \end{equation*}
    \end{remark}
    
    \subsection{\textbf{II round of ambiguities}:}
    
    \begin{itemize}
    \item \ref{rel1}-\ref{rel6} \(E_{ir}^{2n}E_{rj}^{2n+1}E_{l\left( r-1\right)
    }^{2n+2}E_{\left( r-1\right) k}^{2n+3}\) for \(n\geq 0\)%
    \begin{equation*}
    E_{ir}^{2n}E_{rj}^{2n+1}E_{l\left( r-1\right) }^{2n+2}E_{\left( r-1\right)
    k}^{2n+3}\overset{\ref{rel1}}{\rightarrow }-\left(
    \sum_{h=1}^{r-1}E_{ih}^{2n}E_{hj}^{2n+1}-\delta _{ij}1\right) E_{l\left(
    r-1\right) }^{2n+2}E_{\left( r-1\right) k}^{2n+3},
    \end{equation*}%
    and%
    \begin{eqnarray*}
    &&E_{ir}^{2n}E_{rj}^{2n+1}E_{l\left( r-1\right) }^{2n+2}E_{\left( r-1\right)
    k}^{2n+3}\overset{\ref{rel6}}{\rightarrow }%
    -\sum_{h=1}^{r-2}E_{ir}^{2n}E_{rj}^{2n+1}E_{lh}^{2n+2}E_{hk}^{2n+3}+%
    \sum_{h=1}^{r-1}E_{ir}^{2n}E_{hj}^{2n+1}E_{lh}^{2n+2}E_{rk}^{2n+3}-(-1)^{g^{l+j}}\delta _{lj}E_{ir}^{2n}E_{rk}^{2n+3}+\delta _{lk}E_{ir}^{2n}E_{rj}^{2n+1}
    \\
    &=&-\left( E_{ir}^{2n}E_{rj}^{2n+1}\right) \left(
    \sum_{h=1}^{r-2}E_{lh}^{2n+2}E_{hk}^{2n+3}-\delta _{lk}1\right)
    +E_{ir}^{2n}E_{\left( r-1\right) j}^{2n+1}E_{l\left( r-1\right)
    }^{2n+2}E_{rk}^{2n+3}+%
    \sum_{h=1}^{r-2}E_{ir}^{2n}E_{hj}^{2n+1}E_{lh}^{2n+2}E_{rk}^{2n+3}-(-1)^{g^{l+j}}\delta _{lj}E_{ir}^{2n}E_{rk}^{2n+3}
    \\
    &&\overset{\ref{rel5}}{\rightarrow }-\left( E_{ir}^{2n}E_{rj}^{2n+1}\right)
    \left( \sum_{h=1}^{r-2}E_{lh}^{2n+2}E_{hk}^{2n+3}-\delta _{lk}1\right)
    -\sum_{h=1}^{r-2}E_{ir}^{2n}E_{hj}^{2n+1}E_{lh}^{2n+2}E_{rk}^{2n+3}+%
    \sum_{h=1}^{r-1}E_{ih}^{2n}E_{hj}^{2n+1}E_{lr}^{2n+2}E_{rk}^{2n+3}-\delta
    _{ij}E_{lr}^{2n+2}E_{rk}^{2n+3}+(-1)^{g^{l+j}}\delta
    _{lj}E_{ir}^{2n}E_{rk}^{2n+3}+ \\
    &&\qquad \qquad
    +\sum_{h=1}^{r-2}E_{ir}^{2n}E_{hj}^{2n+1}E_{lh}^{2n+2}E_{rk}^{2n+3}-(-1)^{g^{l+j}}\delta _{lj}E_{ir}^{2n}E_{rk}^{2n+3}
    \\
    &=&-\left( E_{ir}^{2n}E_{rj}^{2n+1}\right) \left(
    \sum_{h=1}^{r-2}E_{lh}^{2n+2}E_{hk}^{2n+3}-\delta _{lk}1\right) +\left(
    \sum_{h=1}^{r-1}E_{ih}^{2n}E_{hj}^{2n+1}-\delta _{ij}1\right) \left(
    E_{lr}^{2n+2}E_{rk}^{2n+3}\right)  \\
    &&\overset{\ref{rel1}}{\rightarrow }-\left(
    -\sum_{h=1}^{r-1}E_{ih}^{2n}E_{hj}^{2n+1}+\delta _{ij}1\right) \left(
    \sum_{h=1}^{r-2}E_{lh}^{2n+2}E_{hk}^{2n+3}-\delta _{lk}1\right) +\left(
    \sum_{h=1}^{r-1}E_{ih}^{2n}E_{hj}^{2n+1}-\delta _{ij}1\right) \left(
    -\sum_{h=1}^{r-1}E_{lh}^{2n}E_{hk}^{2n+1}+\delta _{lk}1\right)  \\
    \qquad \qquad  &=&\left( \sum_{h=1}^{r-1}E_{ih}^{2n}E_{hj}^{2n+1}-\delta
    _{ij}1\right) \left( \sum_{h=1}^{r-2}E_{lh}^{2n+2}E_{hk}^{2n+3}-\delta
    _{lk}1\right) -\left( \sum_{h=1}^{r-1}E_{ih}^{2n}E_{hj}^{2n+1}-\delta
    _{ij}1\right) \left( \sum_{h=1}^{r-1}E_{lh}^{2n}E_{hk}^{2n+1}-\delta
    _{lk}1\right)  \\
    &=&-\left( \sum_{h=1}^{r-1}E_{ih}^{2n}E_{hj}^{2n+1}-\delta _{ij}1\right)
    E_{l\left( r-1\right) }^{2n}E_{\left( r-1\right) k}^{2n+1}.
    \end{eqnarray*}%
    Hence, the ambiguity is resolvable.
    
    \item \ref{rel1}-\ref{rel8} \(E_{ir}^{2n+2}E_{rr}^{2n+3}E_{\left( r-1\right)
    j}^{2n+2}E_{k\left( r-1\right) }^{2n+1}\) for \(n\geq 0\)%
    \begin{eqnarray*}
    &&E_{ir}^{2n+2}E_{rr}^{2n+3}E_{\left( r-1\right) j}^{2n+2}E_{k\left(
    r-1\right) }^{2n+1}\overset{\ref{rel1}}{\rightarrow }%
    -\sum_{h=1}^{r-1}E_{ih}^{2n+2}E_{hr}^{2n+3}E_{\left( r-1\right)
    j}^{2n+2}E_{k\left( r-1\right) }^{2n+1}+\delta _{ir}E_{\left( r-1\right)
    j}^{2n+2}E_{k\left( r-1\right) }^{2n+1} \\
    &&\overset{\ref{rel8}}{\rightarrow }\sum_{h=1}^{r-1}%
    \sum_{l=1}^{r-2}E_{ih}^{2n+2}E_{hr}^{2n+3}E_{lj}^{2n+2}E_{kl}^{2n+1}-%
    \sum_{h=1}^{r-1}%
    \sum_{l=1}^{r-1}E_{ih}^{2n+2}E_{hl}^{2n+3}E_{lj}^{2n+2}E_{kr}^{2n+1}-(-1)^{g^{k+j}}\delta _{kj}\left( \sum_{h=1}^{r-1}E_{ih}^{2n+2}E_{hr}^{2n+3}\right) +\left( \sum_{h=1}^{r-1}\delta _{hj}E_{ih}^{2n+2}\right) E_{kr}^{2n+1}+\delta _{ir}E_{\left( r-1\right) j}^{2n+2}E_{k\left( r-1\right) }^{2n+1}
    \\
    &=&\left( \sum_{h=1}^{r-1}E_{ih}^{2n+2}E_{hr}^{2n+3}\right) \left(
    \sum_{l=1}^{r-2}E_{lj}^{2n+2}E_{kl}^{2n+1}-(-1)^{g^{k+j}}\delta
    _{kj}1\right)
    -\sum_{h=1}^{r-1}%
    \sum_{l=1}^{r-1}E_{ih}^{2n+2}E_{hl}^{2n+3}E_{lj}^{2n+2}E_{kr}^{2n+1}+\left(
    1-\delta _{jr}\right) E_{ij}^{2n+2}E_{kr}^{2n+1}+\delta _{ir}E_{\left(
    r-1\right) j}^{2n+2}E_{k\left( r-1\right) }^{2n+1} \\
    &=&\left( \sum_{h=1}^{r-1}E_{ih}^{2n+2}E_{hr}^{2n+3}\right) \left(
    \sum_{l=1}^{r-2}E_{lj}^{2n+2}E_{kl}^{2n+1}-(-1)^{g^{k+j}}\delta
    _{kj}1\right)
    -\sum_{h=1}^{r-1}%
    \sum_{l=1}^{r-1}E_{ih}^{2n+2}E_{hl}^{2n+3}E_{lj}^{2n+2}E_{kr}^{2n+1}+\left(
    1-\delta _{jr}\right) E_{ij}^{2n+2}E_{kr}^{2n+1}+\delta _{ir}E_{\left(
    r-1\right) j}^{2n+2}E_{k\left( r-1\right) }^{2n+1}
    \end{eqnarray*}%
    and%
    \begin{eqnarray*}
    &&E_{ir}^{2n+2}E_{rr}^{2n+3}E_{\left( r-1\right) j}^{2n+2}E_{k\left(
    r-1\right) }^{2n+1}\overset{\ref{rel8}}{\rightarrow }%
    -\sum_{h=1}^{r-2}E_{ir}^{2n+2}E_{rr}^{2n+3}E_{hj}^{2n+2}E_{kh}^{2n+1}+%
    \sum_{h=1}^{r-1}E_{ir}^{2n+2}E_{rh}^{2n+3}E_{hj}^{2n+2}E_{kr}^{2n+1}+(-1)^{g^{k+j}}\delta _{kj}E_{ir}^{2n+2}E_{rr}^{2n+3}-\delta _{rj}E_{ir}^{2n+2}E_{kr}^{2n+1}
    \\
    &=&-E_{ir}^{2n+2}E_{rr}^{2n+3}\left(
    \sum_{h=1}^{r-2}E_{hj}^{2n+2}E_{kh}^{2n+1}-(-1)^{g^{k+j}}\delta
    _{kj}1\right)
    +\sum_{h=1}^{r-1}E_{ir}^{2n+2}E_{rh}^{2n+3}E_{hj}^{2n+2}E_{kr}^{2n+1}-\delta
    _{rj}E_{ir}^{2n+2}E_{kr}^{2n+1} \\
    &&\overset{\ref{rel1}}{\rightarrow }\left(
    \sum_{l=1}^{r-1}E_{il}^{2n+2}E_{lr}^{2n+3}-\delta _{ir}1\right) \left(
    \sum_{h=1}^{r-2}E_{hj}^{2n+2}E_{kh}^{2n+1}-(-1)^{g^{k+j}}\delta
    _{kj}1\right)
    -\sum_{h=1}^{r-1}%
    \sum_{l=1}^{r-1}E_{il}^{2n+2}E_{lh}^{2n+3}E_{hj}^{2n+2}E_{kr}^{2n+1}+\left(
    \sum_{h=1}^{r-1}\delta _{ih}E_{hj}^{2n+2}\right) E_{kr}^{2n+1}-\delta
    _{rj}E_{ir}^{2n+2}E_{kr}^{2n+1} \\
    &=&\left( \sum_{l=1}^{r-1}E_{il}^{2n+2}E_{lr}^{2n+3}\right) \left(
    \sum_{h=1}^{r-2}E_{hj}^{2n+2}E_{kh}^{2n+1}-(-1)^{g^{k+j}}\delta
    _{kj}1\right)
    -\sum_{h=1}^{r-1}%
    \sum_{l=1}^{r-1}E_{il}^{2n+2}E_{lh}^{2n+3}E_{hj}^{2n+2}E_{kr}^{2n+1}-\delta
    _{ir}\left( \sum_{h=1}^{r-2}E_{hj}^{2n+2}E_{kh}^{2n+1}-(-1)^{g^{k+j}}\delta
    _{kj}1\right) +\left( 1-\delta _{ir}\right)
    E_{ij}^{2n+2}E_{kr}^{2n+1}-\delta _{rj}E_{ir}^{2n+2}E_{kr}^{2n+1}.
    \end{eqnarray*}%
    Now, we proceed to compare 
    \begin{equation*}
    \left( 1-\delta _{jr}\right) E_{ij}^{2n+2}E_{kr}^{2n+1}+\delta
    _{ir}E_{\left( r-1\right) j}^{2n+2}E_{k\left( r-1\right) }^{2n+1}\qquad 
    \mathrm{and}\qquad -\delta _{ir}\left(
    \sum_{h=1}^{r-2}E_{hj}^{2n+2}E_{kh}^{2n+1}-(-1)^{g^{k+j}}\delta
    _{kj}1\right) +\left( 1-\delta _{ir}\right)
    E_{ij}^{2n+2}E_{kr}^{2n+1}-\delta _{rj}E_{ir}^{2n+2}E_{kr}^{2n+1}
    \end{equation*}%
    case by case. If \(j=r\), we shall compare%
    \begin{equation*}
    \delta _{ir}E_{\left( r-1\right) r}^{2n+2}E_{k\left( r-1\right) }^{2n+1}
    \end{equation*}%
    and%
    \begin{eqnarray*}
    &&-\delta _{ir}\left(
    \sum_{h=1}^{r-2}E_{hr}^{2n+2}E_{kh}^{2n+1}-(-1)^{g^{k+r}}\delta
    _{kr}1\right) +\left( 1-\delta _{ir}\right)
    E_{ir}^{2n+2}E_{kr}^{2n+1}-E_{ir}^{2n+2}E_{kr}^{2n+1} \\
    &=&-\delta _{ir}\left(
    \sum_{h=1}^{r-2}E_{hr}^{2n+2}E_{kh}^{2n+1}-(-1)^{g^{k+r}}\delta
    _{kr}1+E_{ir}^{2n+2}E_{kr}^{2n+1}\right) .
    \end{eqnarray*}%
    In this case, if \(i<r\) both are \(0\) and so they are equal. If \(i=r\), the
    second one can be reduced further as%
    \begin{eqnarray*}
    &&-\sum_{h=1}^{r-2}E_{hr}^{2n+2}E_{kh}^{2n+1}+(-1)^{g^{k+r}}\delta
    _{kr}1-E_{rr}^{2n+2}E_{kr}^{2n+1} \\
    &&\overset{\ref{rel4}}{\rightarrow }%
    -\sum_{h=1}^{r-2}E_{hr}^{2n+2}E_{kh}^{2n+1}+(-1)^{g^{k+r}}\delta
    _{kr}1+\sum_{h=1}^{r-1}E_{hr}^{2n+2}E_{kh}^{2n+1}-(-1)^{g^{k+r}}\delta _{kr}1
    \\
    &=&E_{\left( r-1\right) r}^{2n+2}E_{k\left( r-1\right) }^{2n+1}
    \end{eqnarray*}%
    and so the two terms can be reduced to the same. If \(j<r\), then we shall
    compare%
    \begin{equation*}
    E_{ij}^{2n+2}E_{kr}^{2n+1}+\delta _{ir}E_{\left( r-1\right)
    j}^{2n+2}E_{k\left( r-1\right) }^{2n+1}
    \end{equation*}%
    and%
    \begin{equation*}
    -\delta _{ir}\left(
    \sum_{h=1}^{r-2}E_{hj}^{2n+2}E_{kh}^{2n+1}-(-1)^{g^{k+j}}\delta
    _{kj}1\right) +\left( 1-\delta _{ir}\right) E_{ij}^{2n+2}E_{kr}^{2n+1}.
    \end{equation*}%
    In this case, if \(i<r\) then both are \(E_{ij}^{2n+2}E_{kr}^{2n+1}\) and so
    they are equal. If \(i=r\), then the first one can be reduced further as%
    \begin{eqnarray*}
    &&E_{rj}^{2n+2}E_{kr}^{2n+1}+E_{\left( r-1\right) j}^{2n+2}E_{k\left(
    r-1\right) }^{2n+1}\overset{\ref{rel4}}{\rightarrow }%
    -\sum_{h=1}^{r-1}E_{hj}^{2n+2}E_{kh}^{2n+1}+(-1)^{g^{k+j}}\delta
    _{kj}1+E_{\left( r-1\right) j}^{2n+2}E_{k\left( r-1\right) }^{2n+1} \\
    &=&-\sum_{h=1}^{r-2}E_{hj}^{2n+2}E_{kh}^{2n+1}+(-1)^{g^{k+j}}\delta _{kj}1
    \end{eqnarray*}%
    which coincides with the second one. Since we exhausted all the cases, the
    original ambiguity is resolvable.
    
    \item \ref{rel2}-\ref{rel5} \(E_{rj}^{2n+1}E_{ir}^{2n+2}E_{\left( r-1\right)
    l}^{2n+3}E_{k\left( r-1\right) }^{2n+4}\) for \(n\geq 0\)%
    \begin{equation*}
    E_{rj}^{2n+1}E_{ir}^{2n+2}E_{\left( r-1\right) l}^{2n+3}E_{k\left(
    r-1\right) }^{2n+4}\overset{\ref{rel2}}{\rightarrow }\left(
    -\sum_{h=1}^{r-1}E_{hj}^{2n+1}E_{ih}^{2n+2}+(-1)^{g^{i+j}}\delta
    _{ij}1\right) E_{\left( r-1\right) l}^{2n+3}E_{k\left( r-1\right) }^{2n+4}
    \end{equation*}%
    while%
    \begin{eqnarray*}
    &&E_{rj}^{2n+1}E_{ir}^{2n+2}E_{\left( r-1\right) l}^{2n+3}E_{k\left(
    r-1\right) }^{2n+4}\overset{\ref{rel5}}{\rightarrow }%
    -\sum_{h=1}^{r-2}E_{rj}^{2n+1}E_{ir}^{2n+2}E_{hl}^{2n+3}E_{kh}^{2n+3}+%
    \sum_{h=1}^{r-1}E_{rj}^{2n+1}E_{ih}^{2n+2}E_{hl}^{2n+3}E_{kr}^{2n+4}-\delta
    _{il}E_{rj}^{2n+1}E_{kr}^{2n+4}+(-1)^{g^{k+l}}\delta
    _{kl}E_{rj}^{2n+1}E_{ir}^{2n+2} \\
    &=&E_{rj}^{2n+1}E_{ir}^{2n+2}\left(
    -\sum_{h=1}^{r-2}E_{hl}^{2n+3}E_{kh}^{2n+3}+(-1)^{g^{k+l}}\delta
    _{kl}1\right) +E_{rj}^{2n+1}E_{i\left( r-1\right) }^{2n+2}E_{\left(
    r-1\right)
    l}^{2n+3}E_{kr}^{2n+4}+%
    \sum_{h=1}^{r-2}E_{rj}^{2n+1}E_{ih}^{2n+2}E_{hl}^{2n+3}E_{kr}^{2n+4}-\delta
    _{il}E_{rj}^{2n+1}E_{kr}^{2n+4} \\
    &&\overset{\ref{rel6}}{\rightarrow }E_{rj}^{2n+1}E_{ir}^{2n+2}\left(
    -\sum_{h=1}^{r-2}E_{hl}^{2n+3}E_{kh}^{2n+3}+(-1)^{g^{k+l}}\delta
    _{kl}1\right)
    -\sum_{h=1}^{r-2}E_{rj}^{2n+1}E_{ih}^{2n+2}E_{hl}^{2n+3}E_{kr}^{2n+4}+%
    \sum_{h=1}^{r-1}E_{hj}^{2n+1}E_{ih}^{2n+2}E_{rl}^{2n+3}E_{kr}^{2n+4}-(-1)^{g^{i+j}}\delta _{ij}E_{rl}^{2n+3}E_{kr}^{2n+4}+
    \\
    &&\qquad \qquad +\delta
    _{il}E_{rj}^{2n+1}E_{kr}^{2n+4}+%
    \sum_{h=1}^{r-2}E_{rj}^{2n+1}E_{ih}^{2n+2}E_{hl}^{2n+3}E_{kr}^{2n+4}-\delta
    _{il}E_{rj}^{2n+1}E_{kr}^{2n+4} \\
    &=&E_{rj}^{2n+1}E_{ir}^{2n+2}\left(
    -\sum_{h=1}^{r-2}E_{hl}^{2n+3}E_{kh}^{2n+3}+(-1)^{g^{k+l}}\delta
    _{kl}1\right) -\left(
    -\sum_{h=1}^{r-1}E_{hj}^{2n+1}E_{ih}^{2n+2}+(-1)^{g^{i+j}}\delta
    _{ij}1\right) E_{rl}^{2n+3}E_{kr}^{2n+4} \\
    &&\overset{\ref{rel2}}{\rightarrow }\left(
    -\sum_{h=1}^{r-1}E_{hj}^{2n+1}E_{ih}^{2n+2}+(-1)^{g^{i+j}}\delta
    _{ij}1\right) \left(
    -\sum_{h=1}^{r-2}E_{hl}^{2n+3}E_{kh}^{2n+3}+(-1)^{g^{k+l}}\delta
    _{kl}1\right) -\left(
    -\sum_{h=1}^{r-1}E_{hj}^{2n+1}E_{ih}^{2n+2}+(-1)^{g^{i+j}}\delta
    _{ij}1\right) \left(
    -\sum_{h=1}^{r-1}E_{hl}^{2n+3}E_{kh}^{2n+4}+(-1)^{g^{k+l}}\delta
    _{kl}1\right)  \\
    &=&\left( -\sum_{h=1}^{r-1}E_{hj}^{2n+1}E_{ih}^{2n+2}+(-1)^{g^{i+j}}\delta
    _{ij}1\right) E_{\left( r-1\right) l}^{2n+3}E_{k\left( r-1\right) }^{2n+4},
    \end{eqnarray*}%
    and so the ambiguity is resolvable.
    
    \item \ref{rel2}-\ref{rel7} \(E_{rj}^{2n+1}E_{rr}^{2n+2}E_{i\left( r-1\right)
    }^{2n+1}E_{\left( r-1\right) k}^{2n}\) for \(n\geq 1\)%
    \begin{eqnarray*}
    &&E_{rj}^{2n+1}E_{rr}^{2n+2}E_{i\left( r-1\right) }^{2n+1}E_{\left(
    r-1\right) k}^{2n}\overset{\ref{rel2}}{\rightarrow }%
    -\sum_{h=1}^{r-1}E_{hj}^{2n+1}E_{rh}^{2n+2}E_{i\left( r-1\right)
    }^{2n+1}E_{\left( r-1\right) k}^{2n}+(-1)^{g^{r+j}}\delta _{rj}E_{i\left(
    r-1\right) }^{2n+1}E_{\left( r-1\right) k}^{2n} \\
    &&\overset{\ref{rel7}}{\rightarrow }\sum_{h=1}^{r-1}%
    \sum_{l=1}^{r-2}E_{hj}^{2n+1}E_{rh}^{2n+2}E_{il}^{2n+1}E_{lk}^{2n}-%
    \sum_{h=1}^{r-1}%
    \sum_{l=1}^{r-1}E_{hj}^{2n+1}E_{lh}^{2n+2}E_{il}^{2n+1}E_{rk}^{2n}-\delta
    _{ik}\left( \sum_{h=1}^{r-1}E_{hj}^{2n+1}E_{rh}^{2n+2}\right)
    +\sum_{h=1}^{r-1}(-1)^{g^{i+h}}\delta
    _{ih}E_{hj}^{2n+1}E_{rk}^{2n}+(-1)^{g^{r+j}}\delta _{rj}E_{i\left(
    r-1\right) }^{2n+1}E_{\left( r-1\right) k}^{2n} \\
    &=&\left( \sum_{h=1}^{r-1}E_{hj}^{2n+1}E_{rh}^{2n+2}\right) \left(
    \sum_{l=1}^{r-2}E_{il}^{2n+1}E_{lk}^{2n}-\delta _{ik}1\right)
    -\sum_{h=1}^{r-1}%
    \sum_{l=1}^{r-1}E_{hj}^{2n+1}E_{lh}^{2n+2}E_{il}^{2n+1}E_{rk}^{2n}+%
    \sum_{h=1}^{r-1}(-1)^{g^{i+h}}\delta
    _{ih}E_{hj}^{2n+1}E_{rk}^{2n}+(-1)^{g^{r+j}}\delta _{rj}E_{i\left(
    r-1\right) }^{2n+1}E_{\left( r-1\right) k}^{2n}
    \end{eqnarray*}%
    while%
    \begin{eqnarray*}
    &&E_{rj}^{2n+1}E_{rr}^{2n+2}E_{i\left( r-1\right) }^{2n+1}E_{\left(
    r-1\right) k}^{2n}\overset{\ref{rel7}}{\rightarrow }%
    -\sum_{h=1}^{r-2}E_{rj}^{2n+1}E_{rr}^{2n+2}E_{ih}^{2n+1}E_{hk}^{2n}+%
    \sum_{h=1}^{r-1}E_{rj}^{2n+1}E_{hr}^{2n+2}E_{ih}^{2n+1}E_{rk}^{2n}+\delta
    _{ik}E_{rj}^{2n+1}E_{rr}^{2n+2}-(-1)^{g^{i+r}}\delta
    _{ir}E_{rj}^{2n+1}E_{rk}^{2n} \\
    &=&-E_{rj}^{2n+1}E_{rr}^{2n+2}\left(
    \sum_{h=1}^{r-2}E_{ih}^{2n+1}E_{hk}^{2n}-\delta _{ik}1\right)
    +\sum_{h=1}^{r-1}E_{rj}^{2n+1}E_{hr}^{2n+2}E_{ih}^{2n+1}E_{rk}^{2n}-(-1)^{g^{i+r}}\delta _{ir}E_{rj}^{2n+1}E_{rk}^{2n}
    \\
    &&\overset{\ref{rel2}}{\rightarrow }\left(
    \sum_{h=1}^{r-1}E_{hj}^{2n+1}E_{rh}^{2n+2}-(-1)^{g^{r+j}}\delta
    _{rj}1\right) \left( \sum_{h=1}^{r-2}E_{ih}^{2n+1}E_{hk}^{2n}-\delta
    _{ik}1\right)
    -\sum_{h=1}^{r-1}%
    \sum_{l=1}^{r-1}E_{lj}^{2n+1}E_{hl}^{2n+2}E_{ih}^{2n+1}E_{rk}^{2n}+%
    \sum_{h=1}^{r-1}(-1)^{g^{h+j}}\delta
    _{hj}E_{ih}^{2n+1}E_{rk}^{2n}-(-1)^{g^{i+r}}\delta
    _{ir}E_{rj}^{2n+1}E_{rk}^{2n} \\
    &=&\left( \sum_{h=1}^{r-1}E_{hj}^{2n+1}E_{rh}^{2n+2}\right) \left(
    \sum_{h=1}^{r-2}E_{ih}^{2n+1}E_{hk}^{2n}-\delta _{ik}1\right)
    -\sum_{h=1}^{r-1}%
    \sum_{l=1}^{r-1}E_{lj}^{2n+1}E_{hl}^{2n+2}E_{ih}^{2n+1}E_{rk}^{2n}-(-1)^{g^{r+j}}\delta _{rj}\left( \sum_{h=1}^{r-2}E_{ih}^{2n+1}E_{hk}^{2n}-\delta _{ik}1\right) +\sum_{h=1}^{r-1}(-1)^{g^{h+j}}\delta _{hj}E_{ih}^{2n+1}E_{rk}^{2n}-(-1)^{g^{i+r}}\delta _{ir}E_{rj}^{2n+1}E_{rk}^{2n}
    \end{eqnarray*}%
    and so we need to compare%
    \begin{equation*}
    \sum_{h=1}^{r-1}(-1)^{g^{i+h}}\delta
    _{ih}E_{hj}^{2n+1}E_{rk}^{2n}+(-1)^{g^{r+j}}\delta _{rj}E_{i\left(
    r-1\right) }^{2n+1}E_{\left( r-1\right) k}^{2n}\qquad \mathrm{and}\qquad
    -(-1)^{g^{r+j}}\delta _{rj}\left(
    \sum_{h=1}^{r-2}E_{ih}^{2n+1}E_{hk}^{2n}-\delta _{ik}1\right)
    +\sum_{h=1}^{r-1}(-1)^{g^{h+j}}\delta
    _{hj}E_{ih}^{2n+1}E_{rk}^{2n}-(-1)^{g^{i+r}}\delta
    _{ir}E_{rj}^{2n+1}E_{rk}^{2n}
    \end{equation*}%
    case by case. As before, if \(j=r\) then we need to compare%
    \begin{equation*}
    \sum_{h=1}^{r-1}(-1)^{g^{i+h}}\delta
    _{ih}E_{hr}^{2n+1}E_{rk}^{2n}+E_{i\left( r-1\right) }^{2n+1}E_{\left(
    r-1\right) k}^{2n}\overset{\ref{rel3}}{\rightarrow }-\sum_{h=1}^{r-1}%
    \sum_{l=1}^{r-1}(-1)^{g^{i+h}}\delta
    _{ih}E_{hl}^{2n+1}E_{lk}^{2n}+\sum_{h=1}^{r-1}(-1)^{g^{i+h}}\delta
    _{ih}\delta _{hk}1+E_{i\left( r-1\right) }^{2n+1}E_{\left( r-1\right) k}^{2n}
    \end{equation*}%
    and%
    \begin{equation*}
    -\sum_{h=1}^{r-2}E_{ih}^{2n+1}E_{hk}^{2n}+\delta _{ik}1-(-1)^{g^{i+r}}\delta
    _{ir}E_{rr}^{2n+1}E_{rk}^{2n}\overset{\ref{rel3}}{\rightarrow }%
    -\sum_{h=1}^{r-2}E_{ih}^{2n+1}E_{hk}^{2n}+\delta
    _{ik}1+\sum_{h=1}^{r-1}(-1)^{g^{i+r}}\delta
    _{ir}E_{rh}^{2n+1}E_{hk}^{2n}-(-1)^{g^{i+r}}\delta _{ir}\delta _{rk}1.
    \end{equation*}%
    In this case, if \(i=r\) then the first term is \(E_{r\left( r-1\right)
    }^{2n+1}E_{\left( r-1\right) k}^{2n}\) which coincides with the second as%
    \begin{equation*}
    -\sum_{h=1}^{r-2}E_{rh}^{2n+1}E_{hk}^{2n}+\delta
    _{rk}1+\sum_{h=1}^{r-1}E_{rh}^{2n+1}E_{hk}^{2n}-\delta _{rk}1=E_{r\left(
    r-1\right) }^{2n+1}E_{\left( r-1\right) k}^{2n}.
    \end{equation*}%
    If, instead, \(i<r\) then%
    \begin{equation*}
    -\sum_{h=1}^{r-1}\sum_{l=1}^{r-1}(-1)^{g^{i+h}}\delta
    _{ih}E_{hl}^{2n+1}E_{lk}^{2n}+\sum_{h=1}^{r-1}(-1)^{g^{i+h}}\delta
    _{ih}\delta _{hk}1+E_{i\left( r-1\right) }^{2n+1}E_{\left( r-1\right)
    k}^{2n}=-\sum_{l=1}^{r-1}E_{il}^{2n+1}E_{lk}^{2n}+\delta _{ik}1+E_{i\left(
    r-1\right) }^{2n+1}E_{\left( r-1\right) k}^{2n}
    \end{equation*}%
    still equals%
    \begin{equation*}
    -\sum_{h=1}^{r-2}E_{ih}^{2n+1}E_{hk}^{2n}+\delta _{ik}1.
    \end{equation*}%
    If \(j<r\) then we need to compare%
    \begin{equation*}
    \sum_{h=1}^{r-1}(-1)^{g^{i+h}}\delta _{ih}E_{hj}^{2n+1}E_{rk}^{2n}
    \end{equation*}%
    and%
    \begin{equation*}
    \sum_{h=1}^{r-1}(-1)^{g^{h+j}}\delta
    _{hj}E_{ih}^{2n+1}E_{rk}^{2n}-(-1)^{g^{i+r}}\delta
    _{ir}E_{rj}^{2n+1}E_{rk}^{2n}=E_{ij}^{2n+1}E_{rk}^{2n}-(-1)^{g^{i+r}}\delta
    _{ir}E_{rj}^{2n+1}E_{rk}^{2n}
    \end{equation*}%
    but if \(i=r\) then both of them are \(0\), while if \(i<r\) then both of them are 
    \(E_{ij}^{2n+1}E_{rk}^{2n}\). Thus, the ambiguity is resolvable.
    
    \item \ref{rel3}-\ref{rel5} \(E_{ir}^{2n+1}E_{rr}^{2n}E_{\left( r-1\right)
    j}^{2n+1}E_{k\left( r-1\right) }^{2n+2}\) for \(n\geq 1\)%
    \begin{eqnarray*}
    &&E_{ir}^{2n+1}E_{rr}^{2n}E_{\left( r-1\right) j}^{2n+1}E_{k\left(
    r-1\right) }^{2n+2}\overset{\ref{rel3}}{\rightarrow }%
    -\sum_{h=1}^{r-1}E_{ih}^{2n+1}E_{hr}^{2n}E_{\left( r-1\right)
    j}^{2n+1}E_{k\left( r-1\right) }^{2n+2}+\delta _{ir}E_{\left( r-1\right)
    j}^{2n+1}E_{k\left( r-1\right) }^{2n+2} \\
    &&\overset{\ref{rel5}}{\rightarrow }\sum_{h=1}^{r-1}%
    \sum_{l=1}^{r-2}E_{ih}^{2n+1}E_{hr}^{2n}E_{lj}^{2n+1}E_{kl}^{2n+2}-%
    \sum_{h=1}^{r-1}%
    \sum_{l=1}^{r-1}E_{ih}^{2n+1}E_{hl}^{2n}E_{lj}^{2n+1}E_{kr}^{2n+2}+%
    \sum_{h=1}^{r-1}\delta _{hj}E_{ih}^{2n+1}E_{kr}^{2n+2}-(-1)^{g^{k+j}}\delta
    _{kj}\left( \sum_{h=1}^{r-1}E_{ih}^{2n+1}E_{hr}^{2n}\right) +\delta
    _{ir}E_{\left( r-1\right) j}^{2n+1}E_{k\left( r-1\right) }^{2n+2} \\
    &=&\left( \sum_{h=1}^{r-1}E_{ih}^{2n+1}E_{hr}^{2n}\right) \left(
    \sum_{l=1}^{r-2}E_{lj}^{2n+1}E_{kl}^{2n+2}-(-1)^{g^{k+j}}\delta
    _{kj}1\right)
    -\sum_{h=1}^{r-1}%
    \sum_{l=1}^{r-1}E_{ih}^{2n+1}E_{hl}^{2n}E_{lj}^{2n+1}E_{kr}^{2n+2}+\left(
    1-\delta _{rj}\right) E_{ij}^{2n+1}E_{kr}^{2n+2}+\delta _{ir}E_{\left(
    r-1\right) j}^{2n+1}E_{k\left( r-1\right) }^{2n+2}
    \end{eqnarray*}%
    while%
    \begin{eqnarray*}
    &&E_{ir}^{2n+1}E_{rr}^{2n}E_{\left( r-1\right) j}^{2n+1}E_{k\left(
    r-1\right) }^{2n+2}\overset{\ref{rel5}}{\rightarrow }%
    -\sum_{l=1}^{r-2}E_{ir}^{2n+1}E_{rr}^{2n}E_{lj}^{2n+1}E_{kl}^{2n+2}+%
    \sum_{l=1}^{r-1}E_{ir}^{2n+1}E_{rl}^{2n}E_{lj}^{2n+1}E_{kr}^{2n+2}-\delta
    _{rj}E_{ir}^{2n+1}E_{kr}^{2n+2}+(-1)^{g^{k+j}}\delta
    _{kj}E_{ir}^{2n+1}E_{rr}^{2n} \\
    &=&-E_{ir}^{2n+1}E_{rr}^{2n}\left(
    \sum_{l=1}^{r-2}E_{lj}^{2n+1}E_{kl}^{2n+2}-(-1)^{g^{k+j}}\delta
    _{kj}1\right)
    +\sum_{l=1}^{r-1}E_{ir}^{2n+1}E_{rl}^{2n}E_{lj}^{2n+1}E_{kr}^{2n+2}-\delta
    _{rj}E_{ir}^{2n+1}E_{kr}^{2n+2} \\
    &&\overset{\ref{rel3}}{\rightarrow }\left(
    \sum_{h=1}^{r-1}E_{ih}^{2n+1}E_{hr}^{2n}-\delta _{ir}1\right) \left(
    \sum_{l=1}^{r-2}E_{lj}^{2n+1}E_{kl}^{2n+2}-(-1)^{g^{k+j}}\delta
    _{kj}1\right)
    -\sum_{l=1}^{r-1}%
    \sum_{h=1}^{r-1}E_{ih}^{2n+1}E_{hl}^{2n}E_{lj}^{2n+1}E_{kr}^{2n+2}+%
    \sum_{l=1}^{r-1}\delta _{il}E_{lj}^{2n+1}E_{kr}^{2n+2}-\delta
    _{rj}E_{ir}^{2n+1}E_{kr}^{2n+2} \\
    &=&\left( \sum_{h=1}^{r-1}E_{ih}^{2n+1}E_{hr}^{2n}\right) \left(
    \sum_{l=1}^{r-2}E_{lj}^{2n+1}E_{kl}^{2n+2}-(-1)^{g^{k+j}}\delta
    _{kj}1\right)
    -\sum_{l=1}^{r-1}%
    \sum_{h=1}^{r-1}E_{ih}^{2n+1}E_{hl}^{2n}E_{lj}^{2n+1}E_{kr}^{2n+2}-\delta
    _{ir}\left( \sum_{l=1}^{r-2}E_{lj}^{2n+1}E_{kl}^{2n+2}-(-1)^{g^{k+j}}\delta
    _{kj}1\right) +\left( 1-\delta _{ir}\right)
    E_{ij}^{2n+1}E_{kr}^{2n+2}-\delta _{rj}E_{ir}^{2n+1}E_{kr}^{2n+2}
    \end{eqnarray*}%
    and so we need to compare%
    \begin{equation*}
    \left( 1-\delta _{rj}\right) E_{ij}^{2n+1}E_{kr}^{2n+2}+\delta
    _{ir}E_{\left( r-1\right) j}^{2n+1}E_{k\left( r-1\right) }^{2n+2}\qquad \ 
    \mathrm{and}\qquad -\delta _{ir}\left(
    \sum_{l=1}^{r-2}E_{lj}^{2n+1}E_{kl}^{2n+2}-(-1)^{g^{k+j}}\delta
    _{kj}1\right) +\left( 1-\delta _{ir}\right)
    E_{ij}^{2n+1}E_{kr}^{2n+2}-\delta _{rj}E_{ir}^{2n+1}E_{kr}^{2n+2}.
    \end{equation*}%
    Let us proceed case by case. For \(i=r,j=r\) we need to compare \(E_{\left(
    r-1\right) r}^{2n+1}E_{k\left( r-1\right) }^{2n+2}\) and 
    \begin{equation*}
    -\sum_{l=1}^{r-2}E_{lr}^{2n+1}E_{kl}^{2n+2}+(-1)^{g^{k+r}}\delta
    _{kr}1-E_{rr}^{2n+1}E_{kr}^{2n+2}\overset{\ref{rel2}}{\rightarrow }%
    -\sum_{l=1}^{r-2}E_{lr}^{2n+1}E_{kl}^{2n+2}+(-1)^{g^{k+r}}\delta
    _{kr}1+\sum_{h=1}^{r-1}E_{hr}^{2n+1}E_{kh}^{2n+2}-(-1)^{g^{k+r}}\delta
    _{kr}1=E_{\left( r-1\right) r}^{2n+1}E_{k\left( r-1\right) }^{2n+2}.
    \end{equation*}%
    For \(i<r,j=r\) both are \(0\). For \(i=r,j<r\) we need to compare 
    \begin{equation*}
    E_{rj}^{2n+1}E_{kr}^{2n+2}+E_{\left( r-1\right) j}^{2n+1}E_{k\left(
    r-1\right) }^{2n+2}\overset{\ref{rel2}}{\rightarrow }%
    -\sum_{h=1}^{r-2}E_{hj}^{2n+1}E_{kh}^{2n+2}+(-1)^{g^{k+j}}\delta _{kj}1
    \end{equation*}%
    and%
    \begin{equation*}
    -\sum_{l=1}^{r-2}E_{lj}^{2n+1}E_{kl}^{2n+2}+(-1)^{g^{k+j}}\delta _{kj}1.
    \end{equation*}%
    Finally, for \(i<r,j<r\) they are both equal to \(E_{ij}^{2n+1}E_{kr}^{2n+2}\).
    Thus, the ambiguity is resolvable.
    
    \item \ref{rel3}-\ref{rel7} \(E_{ir}^{2n+3}E_{rj}^{2n+2}E_{l\left( r-1\right)
    }^{2n+1}E_{\left( r-1\right) k}^{2n}\) for \(n\geq 1\)%
    \begin{equation*}
    E_{ir}^{2n+3}E_{rj}^{2n+2}E_{l\left( r-1\right) }^{2n+1}E_{\left( r-1\right)
    k}^{2n}\overset{\ref{rel3}}{\rightarrow }\left(
    -\sum_{h=1}^{r-1}E_{ih}^{2n+3}E_{hj}^{2n+2}+\delta _{ij}1\right) E_{l\left(
    r-1\right) }^{2n+1}E_{\left( r-1\right) k}^{2n}
    \end{equation*}%
    while%
    \begin{eqnarray*}
    &&E_{ir}^{2n+3}E_{rj}^{2n+2}E_{l\left( r-1\right) }^{2n+1}E_{\left(
    r-1\right) k}^{2n}\overset{\ref{rel7}}{\rightarrow }%
    -\sum_{h=1}^{r-2}E_{ir}^{2n+3}E_{rj}^{2n+2}E_{lh}^{2n+1}E_{hk}^{2n}+%
    \sum_{h=1}^{r-1}E_{ir}^{2n+3}E_{hj}^{2n+2}E_{lh}^{2n+1}E_{rk}^{2n}+\delta
    _{lk}E_{ir}^{2n+3}E_{rj}^{2n+2}-(-1)^{g^{l+j}}\delta
    _{lj}E_{ir}^{2n+3}E_{rk}^{2n} \\
    &=&E_{ir}^{2n+3}E_{rj}^{2n+2}\left(
    -\sum_{h=1}^{r-2}E_{lh}^{2n+1}E_{hk}^{2n}+\delta _{lk}1\right)
    +E_{ir}^{2n+3}E_{\left( r-1\right) j}^{2n+2}E_{l\left( r-1\right)
    }^{2n+1}E_{rk}^{2n}+%
    \sum_{h=1}^{r-2}E_{ir}^{2n+3}E_{hj}^{2n+2}E_{lh}^{2n+1}E_{rk}^{2n}-(-1)^{g^{l+j}}\delta _{lj}E_{ir}^{2n+3}E_{rk}^{2n}
    \\
    &&\overset{\ref{rel8}}{\rightarrow }E_{ir}^{2n+3}E_{rj}^{2n+2}\left(
    -\sum_{h=1}^{r-2}E_{lh}^{2n+1}E_{hk}^{2n}+\delta _{lk}1\right)
    +\sum_{h=1}^{r-1}E_{ih}^{2n+3}E_{hj}^{2n+2}E_{lr}^{2n+1}E_{rk}^{2n}-\delta
    _{ij}E_{lr}^{2n+1}E_{rk}^{2n} \\
    &=&E_{ir}^{2n+3}E_{rj}^{2n+2}\left(
    -\sum_{h=1}^{r-2}E_{lh}^{2n+1}E_{hk}^{2n}+\delta _{lk}1\right) +\left(
    \sum_{h=1}^{r-1}E_{ih}^{2n+3}E_{hj}^{2n+2}-\delta _{ij}1\right)
    E_{lr}^{2n+1}E_{rk}^{2n} \\
    &&\overset{\ref{rel3}}{\rightarrow }\left(
    -\sum_{h=1}^{r-1}E_{ih}^{2n+3}E_{hj}^{2n+2}+\delta _{ij}1\right) \left(
    -\sum_{h=1}^{r-2}E_{lh}^{2n+1}E_{hk}^{2n}+\delta _{lk}1\right) -\left(
    -\sum_{h=1}^{r-1}E_{ih}^{2n+3}E_{hj}^{2n+2}+\delta _{ij}1\right) \left(
    -\sum_{h=1}^{r-1}E_{lh}^{2n+1}E_{hk}^{2n}+\delta _{lk}1\right)  \\
    &=&\left( -\sum_{h=1}^{r-1}E_{ih}^{2n+3}E_{hj}^{2n+2}+\delta _{ij}1\right)
    E_{l\left( r-1\right) }^{2n+1}E_{\left( r-1\right) k}^{2n}
    \end{eqnarray*}%
    and hence the ambiguity is resolvable.
    
    \item \ref{rel4}-\ref{rel6} \(E_{rj}^{2n+2}E_{rr}^{2n+1}E_{i\left( r-1\right)
    }^{2n+2}E_{\left( r-1\right) k}^{2n+3}\) for \(n\geq 0\)%
    \begin{eqnarray*}
    &&E_{rj}^{2n+2}E_{rr}^{2n+1}E_{i\left( r-1\right) }^{2n+2}E_{\left(
    r-1\right) k}^{2n+3}\overset{\ref{rel4}}{\rightarrow }%
    -\sum_{h=1}^{r-1}E_{hj}^{2n+2}E_{rh}^{2n+1}E_{i\left( r-1\right)
    }^{2n+2}E_{\left( r-1\right) k}^{2n+3}+(-1)^{g^{r+j}}\delta _{rj}E_{i\left(
    r-1\right) }^{2n+2}E_{\left( r-1\right) k}^{2n+3} \\
    &&\overset{\ref{rel6}}{\rightarrow }\sum_{h=1}^{r-1}%
    \sum_{l=1}^{r-2}E_{hj}^{2n+2}E_{rh}^{2n+1}E_{il}^{2n+2}E_{lk}^{2n+3}-%
    \sum_{h=1}^{r-1}%
    \sum_{l=1}^{r-1}E_{hj}^{2n+2}E_{lh}^{2n+1}E_{il}^{2n+2}E_{rk}^{2n+3}+%
    \sum_{h=1}^{r-1}(-1)^{g^{i+h}}\delta _{ih}E_{hj}^{2n+2}E_{rk}^{2n+3}-\delta
    _{ik}\left( \sum_{h=1}^{r-1}E_{hj}^{2n+2}E_{rh}^{2n+1}\right)
    +(-1)^{g^{r+j}}\delta _{rj}E_{i\left( r-1\right) }^{2n+2}E_{\left(
    r-1\right) k}^{2n+3} \\
    &=&\left( \sum_{h=1}^{r-1}E_{hj}^{2n+2}E_{rh}^{2n+1}\right) \left(
    \sum_{l=1}^{r-2}E_{il}^{2n+2}E_{lk}^{2n+3}-\delta _{ik}1\right)
    -\sum_{h=1}^{r-1}%
    \sum_{l=1}^{r-1}E_{hj}^{2n+2}E_{lh}^{2n+1}E_{il}^{2n+2}E_{rk}^{2n+3}+%
    \sum_{h=1}^{r-1}(-1)^{g^{i+h}}\delta
    _{ih}E_{hj}^{2n+2}E_{rk}^{2n+3}+(-1)^{g^{r+j}}\delta _{rj}E_{i\left(
    r-1\right) }^{2n+2}E_{\left( r-1\right) k}^{2n+3}
    \end{eqnarray*}%
    while%
    \begin{eqnarray*}
    &&E_{rj}^{2n+2}E_{rr}^{2n+1}E_{i\left( r-1\right) }^{2n+2}E_{\left(
    r-1\right) k}^{2n+3}\overset{\ref{rel6}}{\rightarrow }%
    -\sum_{l=1}^{r-2}E_{rj}^{2n+2}E_{rr}^{2n+1}E_{il}^{2n+2}E_{lk}^{2n+3}+%
    \sum_{l=1}^{r-1}E_{rj}^{2n+2}E_{lr}^{2n+1}E_{il}^{2n+2}E_{rk}^{2n+3}-(-1)^{g^{i+r}}\delta _{ir}E_{rj}^{2n+2}E_{rk}^{2n+3}+\delta _{ik}E_{rj}^{2n+2}E_{rr}^{2n+1}
    \\
    &=&E_{rj}^{2n+2}E_{rr}^{2n+1}\left(
    -\sum_{l=1}^{r-2}E_{il}^{2n+2}E_{lk}^{2n+3}+\delta _{ik}1\right)
    +\sum_{l=1}^{r-1}E_{rj}^{2n+2}E_{lr}^{2n+1}E_{il}^{2n+2}E_{rk}^{2n+3}-(-1)^{g^{i+r}}\delta _{ir}E_{rj}^{2n+2}E_{rk}^{2n+3}
    \\
    &&\overset{\ref{rel4}}{\rightarrow }\left(
    \sum_{h=1}^{r-1}E_{hj}^{2n+2}E_{rh}^{2n+1}-(-1)^{g^{r+j}}\delta
    _{rj}1\right) \left( \sum_{l=1}^{r-2}E_{il}^{2n+2}E_{lk}^{2n+3}-\delta
    _{ik}1\right)
    -\sum_{l=1}^{r-1}%
    \sum_{h=1}^{r-1}E_{hj}^{2n+2}E_{lh}^{2n+1}E_{il}^{2n+2}E_{rk}^{2n+3}+%
    \sum_{l=1}^{r-1}(-1)^{g^{l+j}}\delta
    _{lj}E_{il}^{2n+2}E_{rk}^{2n+3}-(-1)^{g^{i+r}}\delta
    _{ir}E_{rj}^{2n+2}E_{rk}^{2n+3} \\
    &=&\left( \sum_{h=1}^{r-1}E_{hj}^{2n+2}E_{rh}^{2n+1}\right) \left(
    \sum_{l=1}^{r-2}E_{il}^{2n+2}E_{lk}^{2n+3}-\delta _{ik}1\right)
    -\sum_{l=1}^{r-1}%
    \sum_{h=1}^{r-1}E_{hj}^{2n+2}E_{lh}^{2n+1}E_{il}^{2n+2}E_{rk}^{2n+3}-(-1)^{g^{r+j}}\delta _{rj}\left( \sum_{l=1}^{r-2}E_{il}^{2n+2}E_{lk}^{2n+3}-\delta _{ik}1\right) +\sum_{l=1}^{r-1}(-1)^{g^{l+j}}\delta _{lj}E_{il}^{2n+2}E_{rk}^{2n+3}-(-1)^{g^{i+r}}\delta _{ir}E_{rj}^{2n+2}E_{rk}^{2n+3}
    \end{eqnarray*}%
    so we need to compare%
    \begin{equation*}
    \sum_{h=1}^{r-1}(-1)^{g^{i+h}}\delta
    _{ih}E_{hj}^{2n+2}E_{rk}^{2n+3}+(-1)^{g^{r+j}}\delta _{rj}E_{i\left(
    r-1\right) }^{2n+2}E_{\left( r-1\right) k}^{2n+3}\qquad \mathrm{and}\qquad
    -(-1)^{g^{r+j}}\delta _{rj}\left(
    \sum_{l=1}^{r-2}E_{il}^{2n+2}E_{lk}^{2n+3}-\delta _{ik}1\right)
    +\sum_{l=1}^{r-1}(-1)^{g^{l+j}}\delta
    _{lj}E_{il}^{2n+2}E_{rk}^{2n+3}-(-1)^{g^{i+r}}\delta
    _{ir}E_{rj}^{2n+2}E_{rk}^{2n+3}.
    \end{equation*}%
    We proceed case by case. If \(i=r,j=r\) then both reduce to \(E_{r\left(
    r-1\right) }^{2n+2}E_{\left( r-1\right) k}^{2n+3}\) because%
    \begin{equation*}
    -\sum_{l=1}^{r-2}E_{rl}^{2n+2}E_{lk}^{2n+3}+\delta
    _{rk}1-E_{rr}^{2n+2}E_{rk}^{2n+3}\overset{\ref{rel1}}{\rightarrow }%
    -\sum_{l=1}^{r-2}E_{rl}^{2n+2}E_{lk}^{2n+3}+\delta
    _{rk}1+\sum_{h=1}^{r-1}E_{rh}^{2n}E_{hk}^{2n+1}-\delta _{rk}1=E_{r\left(
    r-1\right) }^{2n+2}E_{\left( r-1\right) k}^{2n+3}.
    \end{equation*}%
    If \(i<r,j=r\) then the first term reduces to%
    \begin{equation*}
    E_{ir}^{2n+2}E_{rk}^{2n+3}+E_{i\left( r-1\right) }^{2n+2}E_{\left(
    r-1\right) k}^{2n+3}\overset{\ref{rel1}}{\rightarrow }%
    -\sum_{h=1}^{r-1}E_{ih}^{2n}E_{hk}^{2n+1}+\delta _{ik}1+E_{i\left(
    r-1\right) }^{2n+2}E_{\left( r-1\right)
    k}^{2n+3}=-\sum_{l=1}^{r-2}E_{il}^{2n+2}E_{lk}^{2n+3}+\delta _{ik}1
    \end{equation*}%
    which coincides with the second. If \(i=r,j<r\) then both equal \(0\). Finally,
    if \(i<r,j<r\) then both reduce to \(E_{ij}^{2n+2}E_{rk}^{2n+3}\). Hence, the
    ambiguity is resolvable.
    
    \item \ref{rel4}-\ref{rel8} \(E_{rj}^{2n+4}E_{ir}^{2n+3}E_{\left( r-1\right)
    l}^{2n+2}E_{k\left( r-1\right) }^{2n+1}\) for \(n\geq 0\)%
    \begin{equation*}
    E_{rj}^{2n+4}E_{ir}^{2n+3}E_{\left( r-1\right) l}^{2n+2}E_{k\left(
    r-1\right) }^{2n+1}\overset{\ref{rel4}}{\rightarrow }\left(
    -\sum_{h=1}^{r-1}E_{hj}^{2n+4}E_{ih}^{2n+3}+(-1)^{g^{i+j}}\delta
    _{ij}\right) E_{\left( r-1\right) l}^{2n+2}E_{k\left( r-1\right) }^{2n+1}
    \end{equation*}%
    while%
    \begin{eqnarray*}
    &&E_{rj}^{2n+4}E_{ir}^{2n+3}E_{\left( r-1\right) l}^{2n+2}E_{k\left(
    r-1\right) }^{2n+1}\overset{\ref{rel8}}{\rightarrow }%
    -\sum_{h=1}^{r-2}E_{rj}^{2n+4}E_{ir}^{2n+3}E_{hl}^{2n+2}E_{kh}^{2n+1}+%
    \sum_{h=1}^{r-1}E_{rj}^{2n+4}E_{ih}^{2n+3}E_{hl}^{2n+2}E_{kr}^{2n+1}+(-1)^{g^{k+l}}\delta _{kl}E_{rj}^{2n+4}E_{ir}^{2n+3}-\delta _{il}E_{rj}^{2n+4}E_{kr}^{2n+1}
    \\
    &=&E_{rj}^{2n+4}E_{ir}^{2n+3}\left(
    -\sum_{h=1}^{r-2}E_{hl}^{2n+2}E_{kh}^{2n+1}+(-1)^{g^{k+l}}\delta
    _{kl}1\right) +E_{rj}^{2n+4}E_{i\left( r-1\right) }^{2n+3}E_{\left(
    r-1\right)
    l}^{2n+2}E_{kr}^{2n+1}+%
    \sum_{h=1}^{r-2}E_{rj}^{2n+4}E_{ih}^{2n+3}E_{hl}^{2n+2}E_{kr}^{2n+1}-\delta
    _{il}E_{rj}^{2n+4}E_{kr}^{2n+1} \\
    &&\overset{\ref{rel7}}{\rightarrow }E_{rj}^{2n+4}E_{ir}^{2n+3}\left(
    -\sum_{h=1}^{r-2}E_{hl}^{2n+2}E_{kh}^{2n+1}+(-1)^{g^{k+l}}\delta
    _{kl}1\right) +\left(
    \sum_{h=1}^{r-1}E_{hj}^{2n+4}E_{ih}^{2n+3}-(-1)^{g^{i+j}}\delta
    _{ij}1\right) E_{rl}^{2n+2}E_{kr}^{2n+1} \\
    &&\overset{\ref{rel4}}{\rightarrow }\left(
    -\sum_{h=1}^{r-1}E_{hj}^{2n+4}E_{ih}^{2n+3}+(-1)^{g^{i+j}}\delta
    _{ij}1\right) \left(
    -\sum_{h=1}^{r-2}E_{hl}^{2n+2}E_{kh}^{2n+1}+(-1)^{g^{k+l}}\delta
    _{kl}1\right) -\left(
    -\sum_{h=1}^{r-1}E_{hj}^{2n+4}E_{ih}^{2n+3}+(-1)^{g^{i+j}}\delta
    _{ij}1\right) \left(
    -\sum_{h=1}^{r-1}E_{hl}^{2n+2}E_{kh}^{2n+1}+(-1)^{g^{k+l}}\delta
    _{kl}1\right)  \\
    &=&\left( -\sum_{h=1}^{r-1}E_{hj}^{2n+4}E_{ih}^{2n+3}+(-1)^{g^{i+j}}\delta
    _{ij}\right) E_{\left( r-1\right) l}^{2n+2}E_{k\left( r-1\right) }^{2n+1},
    \end{eqnarray*}%
    so that the ambiguity is resolvable.
    
    \item \ref{rel5}-\ref{rel4} \(E_{ir}^{2n}E_{\left( r-1\right)
    j}^{2n+1}E_{r\left( r-1\right) }^{2n+2}E_{kr}^{2n+1}\) for \(n\geq 0\)%
    \begin{eqnarray*}
    &&E_{ir}^{2n}E_{\left( r-1\right) j}^{2n+1}E_{r\left( r-1\right)
    }^{2n+2}E_{kr}^{2n+1}\overset{\ref{rel5}}{\rightarrow }%
    -\sum_{h=1}^{r-2}E_{ir}^{2n}E_{hj}^{2n+1}E_{rh}^{2n+2}E_{kr}^{2n+1}+%
    \sum_{h=1}^{r-1}E_{ih}^{2n}E_{hj}^{2n+1}E_{rr}^{2n+2}E_{kr}^{2n+1}-\delta
    _{ij}E_{rr}^{2n+2}E_{kr}^{2n+1}+(-1)^{g^{r+j}}\delta
    _{rj}E_{ir}^{2n}E_{kr}^{2n+1} \\
    &&\overset{\ref{rel4}}{\rightarrow }\sum_{h=1}^{r-2}%
    \sum_{l=1}^{r-1}E_{ir}^{2n}E_{hj}^{2n+1}E_{lh}^{2n+2}E_{kl}^{2n+1}-%
    \sum_{h=1}^{r-2}(-1)^{g^{k+h}}\delta
    _{kh}E_{ir}^{2n}E_{hj}^{2n+1}-\sum_{h=1}^{r-1}%
    \sum_{l=1}^{r-1}E_{ih}^{2n}E_{hj}^{2n+1}E_{lr}^{2n+2}E_{kl}^{2n+1}+%
    \sum_{h=1}^{r-1}(-1)^{g^{k+r}}\delta
    _{kr}E_{ih}^{2n}E_{hj}^{2n+1}+\sum_{h=1}^{r-1}\delta
    _{ij}E_{hr}^{2n+2}E_{kh}^{2n+1}-(-1)^{g^{k+r}}\delta _{kr}\delta
    _{ij}1+(-1)^{g^{r+j}}\delta _{rj}E_{ir}^{2n}E_{kr}^{2n+1}
    \end{eqnarray*}%
    and%
    \begin{eqnarray*}
    &&E_{ir}^{2n}E_{\left( r-1\right) j}^{2n+1}E_{r\left( r-1\right)
    }^{2n+2}E_{kr}^{2n+1}\overset{\ref{rel4}}{\rightarrow }%
    -\sum_{h=1}^{r-1}E_{ir}^{2n}E_{\left( r-1\right) j}^{2n+1}E_{h\left(
    r-1\right) }^{2n+2}E_{kh}^{2n+1}+(-1)^{g^{k+\left( r-1\right) }}\delta
    _{k\left( r-1\right) }E_{ir}^{2n}E_{\left( r-1\right) j}^{2n+1} \\
    &&\overset{\ref{rel5}}{\rightarrow }\sum_{h=1}^{r-1}%
    \sum_{l=1}^{r-2}E_{ir}^{2n}E_{lj}^{2n+1}E_{hl}^{2n+2}E_{kh}^{2n+1}-%
    \sum_{h=1}^{r-1}%
    \sum_{l=1}^{r-1}E_{il}^{2n}E_{lj}^{2n+1}E_{hr}^{2n+2}E_{kh}^{2n+1}+%
    \sum_{h=1}^{r-1}\delta
    _{ij}E_{hr}^{2n+2}E_{kh}^{2n+1}-\sum_{h=1}^{r-1}(-1)^{g^{h+j}}\delta
    _{hj}E_{ir}^{2n}E_{kh}^{2n+1}+(-1)^{g^{k+\left( r-1\right) }}\delta
    _{k\left( r-1\right) }E_{ir}^{2n}E_{\left( r-1\right) j}^{2n+1}
    \end{eqnarray*}%
    so that we need to compare%
    \begin{equation*}
    -\sum_{h=1}^{r-2}(-1)^{g^{k+h}}\delta
    _{kh}E_{ir}^{2n}E_{hj}^{2n+1}+\sum_{h=1}^{r-1}(-1)^{g^{k+r}}\delta
    _{kr}E_{ih}^{2n}E_{hj}^{2n+1}-(-1)^{g^{k+r}}\delta _{kr}\delta
    _{ij}1+(-1)^{g^{r+j}}\delta _{rj}E_{ir}^{2n}E_{kr}^{2n+1}
    \end{equation*}%
    with%
    \begin{equation*}
    -\sum_{h=1}^{r-1}(-1)^{g^{h+j}}\delta
    _{hj}E_{ir}^{2n}E_{kh}^{2n+1}+(-1)^{g^{k+\left( r-1\right) }}\delta
    _{k\left( r-1\right) }E_{ir}^{2n}E_{\left( r-1\right) j}^{2n+1}.
    \end{equation*}%
    For \(k=r,j=r\)%
    \begin{equation*}
    \sum_{h=1}^{r-1}E_{ih}^{2n}E_{hr}^{2n+1}-\delta
    _{ir}1+E_{ir}^{2n}E_{rr}^{2n+1}\overset{\ref{rel1}}{\rightarrow }%
    \sum_{h=1}^{r-1}E_{ih}^{2n}E_{hr}^{2n+1}-\delta
    _{ir}1-\sum_{h=1}^{r-1}E_{ih}^{2n}E_{hr}^{2n+1}+\delta _{ir}1=0
    \end{equation*}%
    and the second term is identically \(0\), so that they are equal. For \(k=r,j<r\)%
    \begin{equation*}
    \sum_{h=1}^{r-1}E_{ih}^{2n}E_{hj}^{2n+1}-\delta _{ij}1
    \end{equation*}%
    and%
    \begin{equation*}
    -E_{ir}^{2n}E_{rj}^{2n+1}\overset{\ref{rel1}}{\rightarrow }%
    \sum_{h=1}^{r-1}E_{ih}^{2n}E_{hj}^{2n+1}-\delta _{ij}1
    \end{equation*}%
    and so they are equal. For \(k=r-1\)%
    \begin{equation*}
    (-1)^{g^{r+j}}\delta _{rj}E_{ir}^{2n}E_{\left( r-1\right) r}^{2n+1}
    \end{equation*}%
    and%
    \begin{equation*}
    -\sum_{h=1}^{r-1}(-1)^{g^{h+j}}\delta _{hj}E_{ir}^{2n}E_{\left( r-1\right)
    h}^{2n+1}+E_{ir}^{2n}E_{\left( r-1\right) j}^{2n+1}
    \end{equation*}%
    and they are always equal. For \(k\leq r-2\)%
    \begin{equation*}
    -E_{ir}^{2n}E_{kj}^{2n+1}+(-1)^{g^{r+j}}\delta _{rj}E_{ir}^{2n}E_{kr}^{2n+1}
    \end{equation*}%
    and%
    \begin{equation*}
    -\sum_{h=1}^{r-1}(-1)^{g^{h+j}}\delta _{hj}E_{ir}^{2n}E_{kh}^{2n+1}
    \end{equation*}%
    and they are always equal. Thus, the ambiguity is resolvable.
    
    \item \ref{rel5}-\ref{rel7} \(E_{ir}^{2n}E_{\left( r-1\right)
    j}^{2n+1}E_{r\left( r-1\right) }^{2n+2}E_{l\left( r-1\right)
    }^{2n+1}E_{\left( r-1\right) k}^{2n}\) for \(n\geq 1\)%
    \begin{eqnarray*}
    &&E_{ir}^{2n}E_{\left( r-1\right) j}^{2n+1}E_{r\left( r-1\right)
    }^{2n+2}E_{l\left( r-1\right) }^{2n+1}E_{\left( r-1\right) k}^{2n}\overset{%
    \ref{rel5}}{\rightarrow }%
    -\sum_{h=1}^{r-2}E_{ir}^{2n}E_{hj}^{2n+1}E_{rh}^{2n+2}E_{l\left( r-1\right)
    }^{2n+1}E_{\left( r-1\right)
    k}^{2n}+\sum_{h=1}^{r-1}E_{ih}^{2n}E_{hj}^{2n+1}E_{rr}^{2n+2}E_{l\left(
    r-1\right) }^{2n+1}E_{\left( r-1\right) k}^{2n}-\delta
    _{ij}E_{rr}^{2n+2}E_{l\left( r-1\right) }^{2n+1}E_{\left( r-1\right)
    k}^{2n}+(-1)^{g^{r+j}}\delta _{rj}E_{ir}^{2n}E_{l\left( r-1\right)
    }^{2n+1}E_{\left( r-1\right) k}^{2n} \\
    &=&-\sum_{h=1}^{r-2}E_{ir}^{2n}E_{hj}^{2n+1}E_{rh}^{2n+2}E_{l\left(
    r-1\right) }^{2n+1}E_{\left( r-1\right) k}^{2n}+\left(
    \sum_{h=1}^{r-1}E_{ih}^{2n}E_{hj}^{2n+1}-\delta _{ij}1\right)
    E_{rr}^{2n+2}E_{l\left( r-1\right) }^{2n+1}E_{\left( r-1\right)
    k}^{2n}+(-1)^{g^{r+j}}\delta _{rj}E_{ir}^{2n}E_{l\left( r-1\right)
    }^{2n+1}E_{\left( r-1\right) k}^{2n} \\
    &&\overset{\ref{rel7}}{\rightarrow }\sum_{h=1}^{r-2}%
    \sum_{m=1}^{r-2}E_{ir}^{2n}E_{hj}^{2n+1}E_{rh}^{2n+2}E_{lm}^{2n+1}E_{mk}^{2n}-\sum_{h=1}^{r-2}\sum_{m=1}^{r-1}E_{ir}^{2n}E_{hj}^{2n+1}E_{mh}^{2n+2}E_{lm}^{2n+1}E_{rk}^{2n}-\sum_{h=1}^{r-2}\delta _{lk}E_{ir}^{2n}E_{hj}^{2n+1}E_{rh}^{2n+2}+\sum_{h=1}^{r-2}(-1)^{g^{l+h}}\delta _{lh}E_{ir}^{2n}E_{hj}^{2n+1}E_{rk}^{2n}+
    \\
    &&\qquad \qquad +\left( \sum_{h=1}^{r-1}E_{ih}^{2n}E_{hj}^{2n+1}-\delta
    _{ij}1\right) \left(
    -\sum_{m=1}^{r-2}E_{rr}^{2n+2}E_{lm}^{2n+1}E_{mk}^{2n}+%
    \sum_{m=1}^{r-1}E_{mr}^{2n+2}E_{lm}^{2n+1}E_{rk}^{2n}+\delta
    _{lk}E_{rr}^{2n+2}-(-1)^{g^{l+r}}\delta _{lr}E_{rk}^{2n}\right)
    +(-1)^{g^{r+j}}\delta _{rj}E_{ir}^{2n}E_{l\left( r-1\right)
    }^{2n+1}E_{\left( r-1\right) k}^{2n} \\
    &=&\sum_{h=1}^{r-2}%
    \sum_{m=1}^{r-2}E_{ir}^{2n}E_{hj}^{2n+1}E_{rh}^{2n+2}E_{lm}^{2n+1}E_{mk}^{2n}-\sum_{h=1}^{r-2}\sum_{m=1}^{r-1}E_{ir}^{2n}E_{hj}^{2n+1}E_{mh}^{2n+2}E_{lm}^{2n+1}E_{rk}^{2n}-\sum_{h=1}^{r-2}\delta _{lk}E_{ir}^{2n}E_{hj}^{2n+1}E_{rh}^{2n+2}+\sum_{h=1}^{r-2}(-1)^{g^{l+h}}\delta _{lh}E_{ir}^{2n}E_{hj}^{2n+1}E_{rk}^{2n}+
    \\
    &&\qquad
    -\sum_{h=1}^{r-1}%
    \sum_{m=1}^{r-2}E_{ih}^{2n}E_{hj}^{2n+1}E_{rr}^{2n+2}E_{lm}^{2n+1}E_{mk}^{2n}+\sum_{h=1}^{r-1}\sum_{m=1}^{r-1}E_{ih}^{2n}E_{hj}^{2n+1}E_{mr}^{2n+2}E_{lm}^{2n+1}E_{rk}^{2n}+\sum_{h=1}^{r-1}\delta _{lk}E_{ih}^{2n}E_{hj}^{2n+1}E_{rr}^{2n+2}-\sum_{h=1}^{r-1}(-1)^{g^{l+r}}\delta _{lr}E_{ih}^{2n}E_{hj}^{2n+1}E_{rk}^{2n}+
    \\
    &&\qquad -\delta _{ij}1\left(
    -\sum_{m=1}^{r-2}E_{rr}^{2n+2}E_{lm}^{2n+1}E_{mk}^{2n}+%
    \sum_{m=1}^{r-1}E_{mr}^{2n+2}E_{lm}^{2n+1}E_{rk}^{2n}+\delta
    _{lk}E_{rr}^{2n+2}-(-1)^{g^{l+r}}\delta _{lr}E_{rk}^{2n}\right)
    +(-1)^{g^{r+j}}\delta _{rj}E_{ir}^{2n}E_{l\left( r-1\right)
    }^{2n+1}E_{\left( r-1\right) k}^{2n}
    \end{eqnarray*}%
    while%
    \begin{eqnarray*}
    &&E_{ir}^{2n}E_{\left( r-1\right) j}^{2n+1}E_{r\left( r-1\right)
    }^{2n+2}E_{l\left( r-1\right) }^{2n+1}E_{\left( r-1\right) k}^{2n}\overset{%
    \ref{rel7}}{\rightarrow }-\sum_{h=1}^{r-2}E_{ir}^{2n}E_{\left( r-1\right)
    j}^{2n+1}E_{r\left( r-1\right)
    }^{2n+2}E_{lh}^{2n+1}E_{hk}^{2n}+\sum_{h=1}^{r-1}E_{ir}^{2n}E_{\left(
    r-1\right) j}^{2n+1}E_{h\left( r-1\right)
    }^{2n+2}E_{lh}^{2n+1}E_{rk}^{2n}+\delta _{lk}E_{ir}^{2n}E_{\left( r-1\right)
    j}^{2n+1}E_{r\left( r-1\right) }^{2n+2}-(-1)^{g^{l+\left( r-1\right)
    }}\delta _{l\left( r-1\right) }E_{ir}^{2n}E_{\left( r-1\right)
    j}^{2n+1}E_{rk}^{2n} \\
    &=&E_{ir}^{2n}E_{\left( r-1\right) j}^{2n+1}E_{r\left( r-1\right)
    }^{2n+2}\left( -\sum_{h=1}^{r-2}E_{lh}^{2n+1}E_{hk}^{2n}+\delta
    _{lk}1\right) +\sum_{h=1}^{r-1}E_{ir}^{2n}E_{\left( r-1\right)
    j}^{2n+1}E_{h\left( r-1\right)
    }^{2n+2}E_{lh}^{2n+1}E_{rk}^{2n}-(-1)^{g^{l+\left( r-1\right) }}\delta
    _{l\left( r-1\right) }E_{ir}^{2n}E_{\left( r-1\right) j}^{2n+1}E_{rk}^{2n} \\
    &&\overset{\ref{rel5}}{\rightarrow }\left(
    -\sum_{m=1}^{r-2}E_{ir}^{2n}E_{mj}^{2n+1}E_{rm}^{2n+2}+%
    \sum_{m=1}^{r-1}E_{im}^{2n}E_{mj}^{2n+1}E_{rr}^{2n+2}-\delta
    _{ij}E_{rr}^{2n+2}+(-1)^{g^{r+j}}\delta _{rj}E_{ir}^{2n}\right) \left(
    -\sum_{h=1}^{r-2}E_{lh}^{2n+1}E_{hk}^{2n}+\delta _{lk}1\right) + \\
    &&\qquad \qquad
    -\sum_{h=1}^{r-1}%
    \sum_{m=1}^{r-2}E_{ir}^{2n}E_{mj}^{2n+1}E_{hm}^{2n+2}E_{lh}^{2n+1}E_{rk}^{2n}+\sum_{h=1}^{r-1}\sum_{m=1}^{r-1}E_{im}^{2n}E_{mj}^{2n+1}E_{hr}^{2n+2}E_{lh}^{2n+1}E_{rk}^{2n}-\sum_{h=1}^{r-1}\delta _{ij}E_{hr}^{2n+2}E_{lh}^{2n+1}E_{rk}^{2n}+\sum_{h=1}^{r-1}(-1)^{g^{h+j}}\delta _{hj}E_{ir}^{2n}E_{lh}^{2n+1}E_{rk}^{2n}-(-1)^{g^{l+\left( r-1\right) }}\delta _{l\left( r-1\right) }E_{ir}^{2n}E_{\left( r-1\right) j}^{2n+1}E_{rk}^{2n}
    \\
    &=&\sum_{h=1}^{r-2}%
    \sum_{m=1}^{r-2}E_{ir}^{2n}E_{mj}^{2n+1}E_{rm}^{2n+2}E_{lh}^{2n+1}E_{hk}^{2n}-\sum_{h=1}^{r-2}\sum_{m=1}^{r-1}E_{im}^{2n}E_{mj}^{2n+1}E_{rr}^{2n+2}E_{lh}^{2n+1}E_{hk}^{2n}+\sum_{h=1}^{r-2}\delta _{ij}E_{rr}^{2n+2}E_{lh}^{2n+1}E_{hk}^{2n}-\sum_{h=1}^{r-2}(-1)^{g^{r+j}}\delta _{rj}E_{ir}^{2n}E_{lh}^{2n+1}E_{hk}^{2n}+
    \\
    &&\qquad +\left(
    -\sum_{m=1}^{r-2}E_{ir}^{2n}E_{mj}^{2n+1}E_{rm}^{2n+2}+%
    \sum_{m=1}^{r-1}E_{im}^{2n}E_{mj}^{2n+1}E_{rr}^{2n+2}-\delta
    _{ij}E_{rr}^{2n+2}+(-1)^{g^{r+j}}\delta _{rj}E_{ir}^{2n}\right) \delta
    _{lk}1+ \\
    &&\qquad
    -\sum_{h=1}^{r-1}%
    \sum_{m=1}^{r-2}E_{ir}^{2n}E_{mj}^{2n+1}E_{hm}^{2n+2}E_{lh}^{2n+1}E_{rk}^{2n}+\sum_{h=1}^{r-1}\sum_{m=1}^{r-1}E_{im}^{2n}E_{mj}^{2n+1}E_{hr}^{2n+2}E_{lh}^{2n+1}E_{rk}^{2n}-\sum_{h=1}^{r-1}\delta _{ij}E_{hr}^{2n+2}E_{lh}^{2n+1}E_{rk}^{2n}+\sum_{h=1}^{r-1}(-1)^{g^{h+j}}\delta _{hj}E_{ir}^{2n}E_{lh}^{2n+1}E_{rk}^{2n}-(-1)^{g^{l+\left( r-1\right) }}\delta _{l\left( r-1\right) }E_{ir}^{2n}E_{\left( r-1\right) j}^{2n+1}E_{rk}^{2n}
    \end{eqnarray*}%
    so that we need to compare%
    \begin{equation*}
    \sum_{h=1}^{r-2}(-1)^{g^{l+h}}\delta
    _{lh}E_{ir}^{2n}E_{hj}^{2n+1}E_{rk}^{2n}-\sum_{h=1}^{r-1}(-1)^{g^{l+r}}%
    \delta _{lr}E_{ih}^{2n}E_{hj}^{2n+1}E_{rk}^{2n}+(-1)^{g^{l+r}}\delta
    _{ij}\delta _{lr}E_{rk}^{2n}+(-1)^{g^{r+j}}\delta _{rj}E_{ir}^{2n}E_{l\left(
    r-1\right) }^{2n+1}E_{\left( r-1\right) k}^{2n}
    \end{equation*}%
    with%
    \begin{equation*}
    -\sum_{h=1}^{r-2}(-1)^{g^{r+j}}\delta
    _{rj}E_{ir}^{2n}E_{lh}^{2n+1}E_{hk}^{2n}+(-1)^{g^{r+j}}\delta _{rj}\delta
    _{lk}E_{ir}^{2n}+\sum_{h=1}^{r-1}(-1)^{g^{h+j}}\delta
    _{hj}E_{ir}^{2n}E_{lh}^{2n+1}E_{rk}^{2n}-(-1)^{g^{l+\left( r-1\right)
    }}\delta _{l\left( r-1\right) }E_{ir}^{2n}E_{\left( r-1\right)
    j}^{2n+1}E_{rk}^{2n}.
    \end{equation*}%
    For \(j=r,l=r\)%
    \begin{eqnarray*}
    &&-\sum_{h=1}^{r-1}E_{ih}^{2n}E_{hr}^{2n+1}E_{rk}^{2n}+\delta
    _{ir}E_{rk}^{2n}+E_{ir}^{2n}E_{r\left( r-1\right) }^{2n+1}E_{\left(
    r-1\right) k}^{2n}\overset{\ref{rel3}}{\rightarrow }\sum_{h=1}^{r-1}%
    \sum_{m=1}^{r-1}E_{ih}^{2n}E_{hm}^{2n+1}E_{mk}^{2n}-\sum_{h=1}^{r-1}\delta
    _{hk}E_{ih}^{2n}+\delta _{ir}E_{rk}^{2n}+E_{ir}^{2n}E_{r\left( r-1\right)
    }^{2n+1}E_{\left( r-1\right) k}^{2n} \\
    &&\overset{\ref{rel1}}{\rightarrow }\sum_{h=1}^{r-1}%
    \sum_{m=1}^{r-1}E_{ih}^{2n}E_{hm}^{2n+1}E_{mk}^{2n}-\left( 1-\delta
    _{rk}\right) E_{ik}^{2n}+\delta
    _{ir}E_{rk}^{2n}-\sum_{h=1}^{r-1}E_{ih}^{2n}E_{h\left( r-1\right)
    }^{2n+1}E_{\left( r-1\right) k}^{2n}+\delta _{i\left( r-1\right) }E_{\left(
    r-1\right) k}^{2n} \\
    &=&\sum_{h=1}^{r-1}\sum_{m=1}^{r-2}E_{ih}^{2n}E_{hm}^{2n+1}E_{mk}^{2n}-%
    \left( 1-\delta _{ir}-\delta _{i\left( r-1\right) }\right)
    E_{ik}^{2n}+\delta _{rk}E_{ir}^{2n}
    \end{eqnarray*}%
    and%
    \begin{eqnarray*}
    &&-\sum_{h=1}^{r-2}E_{ir}^{2n}E_{rh}^{2n+1}E_{hk}^{2n}+\delta
    _{rk}E_{ir}^{2n}\overset{\ref{rel1}}{\rightarrow }\sum_{h=1}^{r-2}%
    \sum_{m=1}^{r-1}E_{im}^{2n}E_{mh}^{2n+1}E_{hk}^{2n}-\sum_{h=1}^{r-2}\delta
    _{ih}E_{hk}^{2n}+\delta _{rk}E_{ir}^{2n} \\
    &=&\sum_{h=1}^{r-1}\sum_{m=1}^{r-2}E_{ih}^{2n}E_{hm}^{2n+1}E_{mk}^{2n}-%
    \left( 1-\delta _{ir}-\delta _{i\left( r-1\right) }\right)
    E_{ik}^{2n}+\delta _{rk}E_{ir}^{2n},
    \end{eqnarray*}%
    so that they are equal. For \(j=r,l=r-1\)%
    \begin{equation*}
    E_{ir}^{2n}E_{\left( r-1\right) \left( r-1\right) }^{2n+1}E_{\left(
    r-1\right) k}^{2n}
    \end{equation*}%
    and%
    \begin{equation*}
    -\sum_{h=1}^{r-2}E_{ir}^{2n}E_{\left( r-1\right) h}^{2n+1}E_{hk}^{2n}+\delta
    _{\left( r-1\right) k}E_{ir}^{2n}-E_{ir}^{2n}E_{\left( r-1\right)
    r}^{2n+1}E_{rk}^{2n}\overset{\ref{rel3}}{\rightarrow }%
    -\sum_{h=1}^{r-2}E_{ir}^{2n}E_{\left( r-1\right) h}^{2n+1}E_{hk}^{2n}+\delta
    _{\left( r-1\right) k}E_{ir}^{2n}+\sum_{h=1}^{r-1}E_{ir}^{2n}E_{\left(
    r-1\right) h}^{2n+1}E_{hk}^{2n}-\delta _{\left( r-1\right)
    k}E_{ir}^{2n}=E_{ir}^{2n}E_{\left( r-1\right) \left( r-1\right)
    }^{2n+1}E_{\left( r-1\right) k}^{2n},
    \end{equation*}%
    and so they are equal. For \(j=r,l\leq r-2\)%
    \begin{equation*}
    E_{ir}^{2n}E_{lr}^{2n+1}E_{rk}^{2n}+E_{ir}^{2n}E_{l\left( r-1\right)
    }^{2n+1}E_{\left( r-1\right) k}^{2n}\overset{\ref{rel3}}{\rightarrow }%
    E_{ir}^{2n}\left( -\sum_{h=1}^{r-1}E_{lh}^{2n+1}E_{hk}^{2n}+\delta
    _{lk}1\right) +E_{ir}^{2n}E_{l\left( r-1\right) }^{2n+1}E_{\left( r-1\right)
    k}^{2n}=E_{ir}^{2n}\left( -\sum_{h=1}^{r-2}E_{lh}^{2n+1}E_{hk}^{2n}+\delta
    _{lk}1\right) 
    \end{equation*}%
    and%
    \begin{equation*}
    -\sum_{h=1}^{r-2}E_{ir}^{2n}E_{lh}^{2n+1}E_{hk}^{2n}+\delta
    _{lk}E_{ir}^{2n}=E_{ir}^{2n}\left(
    -\sum_{h=1}^{r-2}E_{lh}^{2n+1}E_{hk}^{2n}+\delta _{lk}1\right) ,
    \end{equation*}%
    so that they are equal. For \(j<r,l=r\)%
    \begin{equation*}
    -\sum_{h=1}^{r-1}E_{ih}^{2n}E_{hj}^{2n+1}E_{rk}^{2n}+\delta _{ij}E_{rk}^{2n}
    \end{equation*}%
    and%
    \begin{equation*}
    E_{ir}^{2n}E_{rj}^{2n+1}E_{rk}^{2n}\overset{\ref{rel1}}{\rightarrow }%
    -\sum_{h=1}^{r-1}E_{ih}^{2n}E_{hj}^{2n+1}E_{rk}^{2n}+\delta _{ij}E_{rk}^{2n}
    \end{equation*}%
    and so they are equal. For \(j<r,l=r-1\), the first term is identically \(0\) and%
    \begin{equation*}
    E_{ir}^{2n}E_{\left( r-1\right) j}^{2n+1}E_{rk}^{2n}-E_{ir}^{2n}E_{\left(
    r-1\right) j}^{2n+1}E_{rk}^{2n}=0
    \end{equation*}%
    so that they are equal. For \(j<r,l\leq r-2\,\), both terms are identically \(%
    E_{ir}^{2n}E_{lj}^{2n+1}E_{rk}^{2n}\) and so they are equal. Thus, the
    ambiguity is resolvable.
    
    \item \ref{rel6}-\ref{rel3} \(E_{rj}^{2n+1}E_{i\left( r-1\right)
    }^{2n+2}E_{\left( r-1\right) r}^{2n+3}E_{rk}^{2n+2}\) for \(n\geq 0\)%
    \begin{eqnarray*}
    &&E_{rj}^{2n+1}E_{i\left( r-1\right) }^{2n+2}E_{\left( r-1\right)
    r}^{2n+3}E_{rk}^{2n+2}\overset{\ref{rel6}}{\rightarrow }%
    -\sum_{h=1}^{r-2}E_{rj}^{2n+1}E_{ih}^{2n+2}E_{hr}^{2n+3}E_{rk}^{2n+2}+%
    \sum_{h=1}^{r-1}E_{hj}^{2n+1}E_{ih}^{2n+2}E_{rr}^{2n+3}E_{rk}^{2n+2}-(-1)^{g^{i+j}}\delta _{ij}E_{rr}^{2n+3}E_{rk}^{2n+2}+\delta _{ir}E_{rj}^{2n+1}E_{rk}^{2n+2}
    \\
    &=&-\sum_{h=1}^{r-2}E_{rj}^{2n+1}E_{ih}^{2n+2}E_{hr}^{2n+3}E_{rk}^{2n+2}+%
    \left( \sum_{h=1}^{r-1}E_{hj}^{2n+1}E_{ih}^{2n+2}-(-1)^{g^{i+j}}\delta
    _{ij}1\right) E_{rr}^{2n+3}E_{rk}^{2n+2}+\delta
    _{ir}E_{rj}^{2n+1}E_{rk}^{2n+2} \\
    &&\overset{\ref{rel3}}{\rightarrow }\sum_{h=1}^{r-2}%
    \sum_{l=1}^{r-1}E_{rj}^{2n+1}E_{ih}^{2n+2}E_{hl}^{2n+3}E_{lk}^{2n+2}-%
    \sum_{h=1}^{r-2}\delta _{hk}E_{rj}^{2n+1}E_{ih}^{2n+2}+\left(
    \sum_{h=1}^{r-1}E_{hj}^{2n+1}E_{ih}^{2n+2}-(-1)^{g^{i+j}}\delta
    _{ij}1\right) \left( -\sum_{h=1}^{r-1}E_{rh}^{2n+3}E_{hk}^{2n+2}+\delta
    _{rk}1\right) +\delta _{ir}E_{rj}^{2n+1}E_{rk}^{2n+2} \\
    &=&\sum_{h=1}^{r-2}%
    \sum_{l=1}^{r-1}E_{rj}^{2n+1}E_{ih}^{2n+2}E_{hl}^{2n+3}E_{lk}^{2n+2}-\left(
    1-\delta _{\left( r-1\right) k}-\delta _{rk}\right)
    E_{rj}^{2n+1}E_{ik}^{2n+2}+\left(
    \sum_{h=1}^{r-1}E_{hj}^{2n+1}E_{ih}^{2n+2}-(-1)^{g^{i+j}}\delta
    _{ij}1\right) \left( -\sum_{h=1}^{r-1}E_{rh}^{2n+3}E_{hk}^{2n+2}+\delta
    _{rk}1\right) +\delta _{ir}E_{rj}^{2n+1}E_{rk}^{2n+2} \\
    &=&\sum_{h=1}^{r-2}%
    \sum_{l=1}^{r-1}E_{rj}^{2n+1}E_{ih}^{2n+2}E_{hl}^{2n+3}E_{lk}^{2n+2}+\left(
    \sum_{h=1}^{r-1}E_{hj}^{2n+1}E_{ih}^{2n+2}-(-1)^{g^{i+j}}\delta
    _{ij}1\right) \left( -\sum_{h=1}^{r-1}E_{rh}^{2n+3}E_{hk}^{2n+2}\right)
    +\delta _{rk}\left(
    E_{rj}^{2n+1}E_{ir}^{2n+2}+%
    \sum_{h=1}^{r-1}E_{hj}^{2n+1}E_{ih}^{2n+2}-(-1)^{g^{i+j}}\delta
    _{ij}1\right) -\left( 1-\delta _{\left( r-1\right) k}-\delta _{ir}\right)
    E_{rj}^{2n+1}E_{ik}^{2n+2} \\
    &&\overset{\ref{rel2}}{\rightarrow }\sum_{h=1}^{r-2}%
    \sum_{l=1}^{r-1}E_{rj}^{2n+1}E_{ih}^{2n+2}E_{hl}^{2n+3}E_{lk}^{2n+2}+\left(
    \sum_{h=1}^{r-1}E_{hj}^{2n+1}E_{ih}^{2n+2}-(-1)^{g^{i+j}}\delta
    _{ij}1\right) \left( -\sum_{h=1}^{r-1}E_{rh}^{2n+3}E_{hk}^{2n+2}\right)
    -\left( 1-\delta _{\left( r-1\right) k}-\delta _{ir}\right)
    E_{rj}^{2n+1}E_{ik}^{2n+2}
    \end{eqnarray*}%
    and%
    \begin{eqnarray*}
    &&E_{rj}^{2n+1}E_{i\left( r-1\right) }^{2n+2}E_{\left( r-1\right)
    r}^{2n+3}E_{rk}^{2n+2}\overset{\ref{rel3}}{\rightarrow }%
    -\sum_{h=1}^{r-1}E_{rj}^{2n+1}E_{i\left( r-1\right) }^{2n+2}E_{\left(
    r-1\right) h}^{2n+3}E_{hk}^{2n+2}+\delta _{\left( r-1\right)
    k}E_{rj}^{2n+1}E_{i\left( r-1\right) }^{2n+2} \\
    &&\overset{\ref{rel6}}{\rightarrow }\sum_{h=1}^{r-1}%
    \sum_{l=1}^{r-2}E_{rj}^{2n+1}E_{il}^{2n+2}E_{lh}^{2n+3}E_{hk}^{2n+2}-%
    \sum_{h=1}^{r-1}%
    \sum_{l=1}^{r-1}E_{lj}^{2n+1}E_{il}^{2n+2}E_{rh}^{2n+3}E_{hk}^{2n+2}+%
    \sum_{h=1}^{r-1}(-1)^{g^{i+j}}\delta
    _{ij}E_{rh}^{2n+3}E_{hk}^{2n+2}-\sum_{h=1}^{r-1}\delta
    _{ih}E_{rj}^{2n+1}E_{hk}^{2n+2}+\delta _{\left( r-1\right)
    k}E_{rj}^{2n+1}E_{i\left( r-1\right) }^{2n+2} \\
    &=&\sum_{h=1}^{r-1}%
    \sum_{l=1}^{r-2}E_{rj}^{2n+1}E_{il}^{2n+2}E_{lh}^{2n+3}E_{hk}^{2n+2}+\left(
    \sum_{l=1}^{r-1}E_{lj}^{2n+1}E_{il}^{2n+2}-(-1)^{g^{i+j}}\delta
    _{ij}1\right) \left( -\sum_{h=1}^{r-1}E_{rh}^{2n+3}E_{hk}^{2n+2}\right)
    -\left( 1-\delta _{ir}-\delta _{\left( r-1\right) k}\right)
    E_{rj}^{2n+1}E_{ik}^{2n+2}
    \end{eqnarray*}%
    so that they are always equal and the ambiguity is resolvable.
    
    \item \ref{rel6}-\ref{rel8} \(E_{rj}^{2n+1}E_{i\left( r-1\right)
    }^{2n+2}E_{\left( r-1\right) r}^{2n+3}E_{\left( r-1\right)
    l}^{2n+2}E_{k\left( r-1\right) }^{2n+1}\) for \(n\geq 0\)%
    \begin{eqnarray*}
    &&E_{rj}^{2n+1}E_{i\left( r-1\right) }^{2n+2}E_{\left( r-1\right)
    r}^{2n+3}E_{\left( r-1\right) l}^{2n+2}E_{k\left( r-1\right) }^{2n+1}\overset%
    {\ref{rel6}}{\rightarrow }%
    -\sum_{h=1}^{r-2}E_{rj}^{2n+1}E_{ih}^{2n+2}E_{hr}^{2n+3}E_{\left( r-1\right)
    l}^{2n+2}E_{k\left( r-1\right)
    }^{2n+1}+\sum_{h=1}^{r-1}E_{hj}^{2n+1}E_{ih}^{2n+2}E_{rr}^{2n+3}E_{\left(
    r-1\right) l}^{2n+2}E_{k\left( r-1\right) }^{2n+1}-(-1)^{g^{i+j}}\delta
    _{ij}E_{rr}^{2n+3}E_{\left( r-1\right) l}^{2n+2}E_{k\left( r-1\right)
    }^{2n+1}+\delta _{ir}E_{rj}^{2n+1}E_{\left( r-1\right) l}^{2n+2}E_{k\left(
    r-1\right) }^{2n+1} \\
    &=&-\sum_{h=1}^{r-2}E_{rj}^{2n+1}E_{ih}^{2n+2}E_{hr}^{2n+3}E_{\left(
    r-1\right) l}^{2n+2}E_{k\left( r-1\right) }^{2n+1}+\left(
    \sum_{h=1}^{r-1}E_{hj}^{2n+1}E_{ih}^{2n+2}-(-1)^{g^{i+j}}\delta
    _{ij}1\right) E_{rr}^{2n+3}E_{\left( r-1\right) l}^{2n+2}E_{k\left(
    r-1\right) }^{2n+1}+\delta _{ir}E_{rj}^{2n+1}E_{\left( r-1\right)
    l}^{2n+2}E_{k\left( r-1\right) }^{2n+1} \\
    &&\overset{\ref{rel8}}{\rightarrow }\sum_{h=1}^{r-2}%
    \sum_{m=1}^{r-2}E_{rj}^{2n+1}E_{ih}^{2n+2}E_{hr}^{2n+3}E_{ml}^{2n+2}E_{km}^{2n+1}-\sum_{h=1}^{r-2}\sum_{m=1}^{r-1}E_{rj}^{2n+1}E_{ih}^{2n+2}E_{hm}^{2n+3}E_{ml}^{2n+2}E_{kr}^{2n+1}-\sum_{h=1}^{r-2}(-1)^{g^{k+l}}\delta _{kl}E_{rj}^{2n+1}E_{ih}^{2n+2}E_{hr}^{2n+3}+\sum_{h=1}^{r-2}\delta _{hl}E_{rj}^{2n+1}E_{ih}^{2n+2}E_{kr}^{2n+1}+
    \\
    &&\qquad \qquad +\left(
    \sum_{h=1}^{r-1}E_{hj}^{2n+1}E_{ih}^{2n+2}-(-1)^{g^{i+j}}\delta
    _{ij}1\right) \left(
    -\sum_{h=1}^{r-2}E_{rr}^{2n+3}E_{hl}^{2n+2}E_{kh}^{2n+1}+%
    \sum_{h=1}^{r-1}E_{rh}^{2n+3}E_{hl}^{2n+2}E_{kr}^{2n+1}+(-1)^{g^{k+l}}\delta
    _{kl}E_{rr}^{2n+3}-\delta _{rl}E_{kr}^{2n+1}\right) +\delta
    _{ir}E_{rj}^{2n+1}E_{\left( r-1\right) l}^{2n+2}E_{k\left( r-1\right)
    }^{2n+1} \\
    &=&E_{rj}^{2n+1}\left( \sum_{h=1}^{r-2}E_{ih}^{2n+2}E_{hr}^{2n+3}\right)
    \left( \sum_{m=1}^{r-2}E_{ml}^{2n+2}E_{km}^{2n+1}-(-1)^{g^{k+l}}\delta
    _{kl}1\right) -E_{rj}^{2n+1}\left(
    \sum_{h=1}^{r-2}\sum_{m=1}^{r-1}E_{ih}^{2n+2}E_{hm}^{2n+3}E_{ml}^{2n+2}%
    \right) E_{kr}^{2n+1}+\left( 1-\delta _{rl}-\delta _{\left( r-1\right)
    l}\right) E_{rj}^{2n+1}E_{il}^{2n+2}E_{kr}^{2n+1}+ \\
    &&\qquad \qquad +\left(
    \sum_{h=1}^{r-1}E_{hj}^{2n+1}E_{ih}^{2n+2}-(-1)^{g^{i+j}}\delta
    _{ij}1\right) E_{rr}^{2n+3}\left(
    -\sum_{h=1}^{r-2}E_{hl}^{2n+2}E_{kh}^{2n+1}+(-1)^{g^{k+l}}\delta
    _{kl}1\right) +\left(
    \sum_{h=1}^{r-1}E_{hj}^{2n+1}E_{ih}^{2n+2}-(-1)^{g^{i+j}}\delta
    _{ij}1\right) \left( \sum_{h=1}^{r-1}E_{rh}^{2n+3}E_{hl}^{2n+2}-\delta
    _{rl}1\right) E_{kr}^{2n+1}+\delta _{ir}E_{rj}^{2n+1}E_{\left( r-1\right)
    l}^{2n+2}E_{k\left( r-1\right) }^{2n+1} \\
    &=&E_{rj}^{2n+1}\left( \sum_{h=1}^{r-2}E_{ih}^{2n+2}E_{hr}^{2n+3}\right)
    \left( \sum_{m=1}^{r-2}E_{ml}^{2n+2}E_{km}^{2n+1}-(-1)^{g^{k+l}}\delta
    _{kl}1\right) -E_{rj}^{2n+1}\left(
    \sum_{h=1}^{r-2}\sum_{m=1}^{r-1}E_{ih}^{2n+2}E_{hm}^{2n+3}E_{ml}^{2n+2}%
    \right) E_{kr}^{2n+1}+\left( 1-\delta _{rl}-\delta _{\left( r-1\right)
    l}\right) E_{rj}^{2n+1}E_{il}^{2n+2}E_{kr}^{2n+1}+ \\
    &&\qquad +\left(
    \sum_{h=1}^{r-1}E_{hj}^{2n+1}E_{ih}^{2n+2}-(-1)^{g^{i+j}}\delta
    _{ij}1\right) E_{rr}^{2n+3}\left(
    -\sum_{h=1}^{r-2}E_{hl}^{2n+2}E_{kh}^{2n+1}+(-1)^{g^{k+l}}\delta
    _{kl}1\right) +\left( \sum_{h=1}^{r-1}E_{hj}^{2n+1}E_{ih}^{2n+2}\right)
    \left( \sum_{h=1}^{r-1}E_{rh}^{2n+3}E_{hl}^{2n+2}-\delta _{rl}1\right)
    E_{kr}^{2n+1}+ \\
    &&\qquad \qquad -(-1)^{g^{i+j}}\delta _{ij}\left(
    \sum_{h=1}^{r-1}E_{rh}^{2n+3}E_{hl}^{2n+2}-\delta _{rl}1\right)
    E_{kr}^{2n+1}+\delta _{ir}E_{rj}^{2n+1}E_{\left( r-1\right)
    l}^{2n+2}E_{k\left( r-1\right) }^{2n+1} \\
    &=&E_{rj}^{2n+1}\left( \sum_{h=1}^{r-2}E_{ih}^{2n+2}E_{hr}^{2n+3}\right)
    \left( \sum_{m=1}^{r-2}E_{ml}^{2n+2}E_{km}^{2n+1}-(-1)^{g^{k+l}}\delta
    _{kl}1\right) -E_{rj}^{2n+1}\left(
    \sum_{h=1}^{r-2}\sum_{m=1}^{r-1}E_{ih}^{2n+2}E_{hm}^{2n+3}E_{ml}^{2n+2}%
    \right) E_{kr}^{2n+1}+\left( 1-\delta _{rl}-\delta _{\left( r-1\right)
    l}\right) E_{rj}^{2n+1}E_{il}^{2n+2}E_{kr}^{2n+1}+ \\
    &&\qquad +\left(
    \sum_{h=1}^{r-1}E_{hj}^{2n+1}E_{ih}^{2n+2}-(-1)^{g^{i+j}}\delta
    _{ij}1\right) E_{rr}^{2n+3}\left(
    -\sum_{h=1}^{r-2}E_{hl}^{2n+2}E_{kh}^{2n+1}+(-1)^{g^{k+l}}\delta
    _{kl}1\right) +\left(
    \sum_{h=1}^{r-1}E_{hj}^{2n+1}E_{ih}^{2n+2}-(-1)^{g^{i+j}}\delta _{ij}\right)
    \left( \sum_{h=1}^{r-1}E_{rh}^{2n+3}E_{hl}^{2n+2}\right) E_{kr}^{2n+1}+ \\
    &&\qquad \qquad +\delta _{rl}\left(
    \sum_{h=1}^{r-1}E_{hj}^{2n+1}E_{ih}^{2n+2}-(-1)^{g^{i+j}}\delta _{ij}\right)
    E_{kr}^{2n+1}+\delta _{ir}E_{rj}^{2n+1}E_{\left( r-1\right)
    l}^{2n+2}E_{k\left( r-1\right) }^{2n+1} \\
    &=&E_{rj}^{2n+1}\left( \sum_{h=1}^{r-2}E_{ih}^{2n+2}E_{hr}^{2n+3}\right)
    \left( \sum_{m=1}^{r-2}E_{ml}^{2n+2}E_{km}^{2n+1}-(-1)^{g^{k+l}}\delta
    _{kl}1\right) -E_{rj}^{2n+1}\left(
    \sum_{h=1}^{r-2}\sum_{m=1}^{r-1}E_{ih}^{2n+2}E_{hm}^{2n+3}E_{ml}^{2n+2}%
    \right) E_{kr}^{2n+1}+\left( 1-\delta _{\left( r-1\right) l}\right)
    E_{rj}^{2n+1}E_{il}^{2n+2}E_{kr}^{2n+1}+ \\
    &&\qquad +\left(
    \sum_{h=1}^{r-1}E_{hj}^{2n+1}E_{ih}^{2n+2}-(-1)^{g^{i+j}}\delta
    _{ij}1\right) E_{rr}^{2n+3}\left(
    -\sum_{h=1}^{r-2}E_{hl}^{2n+2}E_{kh}^{2n+1}+(-1)^{g^{k+l}}\delta
    _{kl}1\right) +\left(
    \sum_{h=1}^{r-1}E_{hj}^{2n+1}E_{ih}^{2n+2}-(-1)^{g^{i+j}}\delta _{ij}\right)
    \left( \sum_{h=1}^{r-1}E_{rh}^{2n+3}E_{hl}^{2n+2}\right) E_{kr}^{2n+1}+ \\
    &&\qquad \qquad +\delta _{rl}\left(
    -E_{rj}^{2n+1}E_{ir}^{2n+2}+%
    \sum_{h=1}^{r-1}E_{hj}^{2n+1}E_{ih}^{2n+2}-(-1)^{g^{i+j}}\delta _{ij}\right)
    E_{kr}^{2n+1}+\delta _{ir}E_{rj}^{2n+1}E_{\left( r-1\right)
    l}^{2n+2}E_{k\left( r-1\right) }^{2n+1} \\
    &&\overset{\ref{rel2}}{\rightarrow }E_{rj}^{2n+1}\left(
    \sum_{h=1}^{r-2}E_{ih}^{2n+2}E_{hr}^{2n+3}\right) \left(
    \sum_{m=1}^{r-2}E_{ml}^{2n+2}E_{km}^{2n+1}-(-1)^{g^{k+l}}\delta
    _{kl}1\right) -E_{rj}^{2n+1}\left(
    \sum_{h=1}^{r-2}\sum_{m=1}^{r-1}E_{ih}^{2n+2}E_{hm}^{2n+3}E_{ml}^{2n+2}%
    \right) E_{kr}^{2n+1}+\left( 1-\delta _{\left( r-1\right) l}\right)
    E_{rj}^{2n+1}E_{il}^{2n+2}E_{kr}^{2n+1}+ \\
    &&\qquad \qquad +\left(
    \sum_{h=1}^{r-1}E_{hj}^{2n+1}E_{ih}^{2n+2}-(-1)^{g^{i+j}}\delta
    _{ij}1\right) E_{rr}^{2n+3}\left(
    -\sum_{h=1}^{r-2}E_{hl}^{2n+2}E_{kh}^{2n+1}+(-1)^{g^{k+l}}\delta
    _{kl}1\right) +\left(
    \sum_{h=1}^{r-1}E_{hj}^{2n+1}E_{ih}^{2n+2}-(-1)^{g^{i+j}}\delta _{ij}\right)
    \left( \sum_{h=1}^{r-1}E_{rh}^{2n+3}E_{hl}^{2n+2}\right)
    E_{kr}^{2n+1}+\delta _{ir}E_{rj}^{2n+1}E_{\left( r-1\right)
    l}^{2n+2}E_{k\left( r-1\right) }^{2n+1}
    \end{eqnarray*}%
    and%
    \begin{eqnarray*}
    &&E_{rj}^{2n+1}E_{i\left( r-1\right) }^{2n+2}E_{\left( r-1\right)
    r}^{2n+3}E_{\left( r-1\right) l}^{2n+2}E_{k\left( r-1\right) }^{2n+1}\overset%
    {\ref{rel8}}{\rightarrow }-\sum_{h=1}^{r-2}E_{rj}^{2n+1}E_{i\left(
    r-1\right) }^{2n+2}E_{\left( r-1\right)
    r}^{2n+3}E_{hl}^{2n+2}E_{kh}^{2n+1}+\sum_{h=1}^{r-1}E_{rj}^{2n+1}E_{i\left(
    r-1\right) }^{2n+2}E_{\left( r-1\right)
    h}^{2n+3}E_{hl}^{2n+2}E_{kr}^{2n+1}+(-1)^{g^{k+l}}\delta
    _{kl}E_{rj}^{2n+1}E_{i\left( r-1\right) }^{2n+2}E_{\left( r-1\right)
    r}^{2n+3}-\delta _{\left( r-1\right) l}E_{rj}^{2n+1}E_{i\left( r-1\right)
    }^{2n+2}E_{kr}^{2n+1} \\
    &=&E_{rj}^{2n+1}E_{i\left( r-1\right) }^{2n+2}E_{\left( r-1\right)
    r}^{2n+3}\left(
    -\sum_{h=1}^{r-2}E_{hl}^{2n+2}E_{kh}^{2n+1}+(-1)^{g^{k+l}}\delta
    _{kl}1\right) +\sum_{h=1}^{r-1}E_{rj}^{2n+1}E_{i\left( r-1\right)
    }^{2n+2}E_{\left( r-1\right) h}^{2n+3}E_{hl}^{2n+2}E_{kr}^{2n+1}-\delta
    _{\left( r-1\right) l}E_{rj}^{2n+1}E_{i\left( r-1\right)
    }^{2n+2}E_{kr}^{2n+1} \\
    &&\overset{\ref{rel6}}{\rightarrow }\left(
    -\sum_{h=1}^{r-2}E_{rj}^{2n+1}E_{ih}^{2n+2}E_{hr}^{2n+3}+%
    \sum_{h=1}^{r-1}E_{hj}^{2n+1}E_{ih}^{2n+2}E_{rr}^{2n+3}-(-1)^{g^{i+j}}\delta
    _{ij}E_{rr}^{2n+3}+\delta _{ir}E_{rj}^{2n+1}\right) \left(
    -\sum_{h=1}^{r-2}E_{hl}^{2n+2}E_{kh}^{2n+1}+(-1)^{g^{k+l}}\delta
    _{kl}1\right) + \\
    &&\qquad \qquad
    -\sum_{h=1}^{r-1}%
    \sum_{m=1}^{r-2}E_{rj}^{2n+1}E_{im}^{2n+2}E_{mh}^{2n+3}E_{hl}^{2n+2}E_{kr}^{2n+1}+\sum_{h=1}^{r-1}\sum_{m=1}^{r-1}E_{mj}^{2n+1}E_{im}^{2n+2}E_{rh}^{2n+3}E_{hl}^{2n+2}E_{kr}^{2n+1}-\sum_{h=1}^{r-1}(-1)^{g^{i+j}}\delta _{ij}E_{rh}^{2n+3}E_{hl}^{2n+2}E_{kr}^{2n+1}+\sum_{h=1}^{r-1}\delta _{ih}E_{rj}^{2n+1}E_{hl}^{2n+2}E_{kr}^{2n+1}-\delta _{\left( r-1\right) l}E_{rj}^{2n+1}E_{i\left( r-1\right) }^{2n+2}E_{kr}^{2n+1}
    \\
    &=&E_{rj}^{2n+1}\left( -\sum_{h=1}^{r-2}E_{ih}^{2n+2}E_{hr}^{2n+3}+\delta
    _{ir}1\right) \left(
    -\sum_{h=1}^{r-2}E_{hl}^{2n+2}E_{kh}^{2n+1}+(-1)^{g^{k+l}}\delta
    _{kl}1\right) +\left(
    \sum_{h=1}^{r-1}E_{hj}^{2n+1}E_{ih}^{2n+2}-(-1)^{g^{i+j}}\delta
    _{ij}1\right) E_{rr}^{2n+3}\left(
    -\sum_{h=1}^{r-2}E_{hl}^{2n+2}E_{kh}^{2n+1}+(-1)^{g^{k+l}}\delta
    _{kl}1\right) + \\
    &&-E_{rj}^{2n+1}\left(
    \sum_{h=1}^{r-1}\sum_{m=1}^{r-2}E_{im}^{2n+2}E_{mh}^{2n+3}E_{hl}^{2n+2}%
    \right) E_{kr}^{2n+1}+\left(
    \sum_{m=1}^{r-1}E_{mj}^{2n+1}E_{im}^{2n+2}-(-1)^{g^{i+j}}\delta _{ij}\right)
    \left( \sum_{h=1}^{r-1}E_{rh}^{2n+3}E_{hl}^{2n+2}\right)
    E_{kr}^{2n+1}+\left( 1-\delta _{ir}\right)
    E_{rj}^{2n+1}E_{il}^{2n+2}E_{kr}^{2n+1}-\delta _{\left( r-1\right)
    l}E_{rj}^{2n+1}E_{i\left( r-1\right) }^{2n+2}E_{kr}^{2n+1} \\
    &=&E_{rj}^{2n+1}\left( \sum_{h=1}^{r-2}E_{ih}^{2n+2}E_{hr}^{2n+3}\right)
    \left( \sum_{h=1}^{r-2}E_{hl}^{2n+2}E_{kh}^{2n+1}-(-1)^{g^{k+l}}\delta
    _{kl}1\right) -\delta _{ir}E_{rj}^{2n+1}\left(
    \sum_{h=1}^{r-2}E_{hl}^{2n+2}E_{kh}^{2n+1}-(-1)^{g^{k+l}}\delta
    _{kl}1\right) + \\
    &&\qquad +\left(
    \sum_{h=1}^{r-1}E_{hj}^{2n+1}E_{ih}^{2n+2}-(-1)^{g^{i+j}}\delta
    _{ij}1\right) E_{rr}^{2n+3}\left(
    -\sum_{h=1}^{r-2}E_{hl}^{2n+2}E_{kh}^{2n+1}+(-1)^{g^{k+l}}\delta
    _{kl}1\right) -E_{rj}^{2n+1}\left(
    \sum_{h=1}^{r-1}\sum_{m=1}^{r-2}E_{im}^{2n+2}E_{mh}^{2n+3}E_{hl}^{2n+2}%
    \right) E_{kr}^{2n+1}+ \\
    &&\qquad \qquad +\left(
    \sum_{m=1}^{r-1}E_{mj}^{2n+1}E_{im}^{2n+2}-(-1)^{g^{i+j}}\delta _{ij}\right)
    \left( \sum_{h=1}^{r-1}E_{rh}^{2n+3}E_{hl}^{2n+2}\right)
    E_{kr}^{2n+1}+\left( 1-\delta _{ir}\right)
    E_{rj}^{2n+1}E_{il}^{2n+2}E_{kr}^{2n+1}-\delta _{\left( r-1\right)
    l}E_{rj}^{2n+1}E_{i\left( r-1\right) }^{2n+2}E_{kr}^{2n+1} \\
    &=&E_{rj}^{2n+1}\left( \sum_{h=1}^{r-2}E_{ih}^{2n+2}E_{hr}^{2n+3}\right)
    \left( \sum_{h=1}^{r-2}E_{hl}^{2n+2}E_{kh}^{2n+1}-(-1)^{g^{k+l}}\delta
    _{kl}1\right) -\delta _{ir}E_{rj}^{2n+1}\left(
    E_{rl}^{2n+2}E_{kr}^{2n+1}+%
    \sum_{h=1}^{r-2}E_{hl}^{2n+2}E_{kh}^{2n+1}-(-1)^{g^{k+l}}\delta
    _{kl}1\right) + \\
    &&\qquad +\left(
    \sum_{h=1}^{r-1}E_{hj}^{2n+1}E_{ih}^{2n+2}-(-1)^{g^{i+j}}\delta
    _{ij}1\right) E_{rr}^{2n+3}\left(
    -\sum_{h=1}^{r-2}E_{hl}^{2n+2}E_{kh}^{2n+1}+(-1)^{g^{k+l}}\delta
    _{kl}1\right) -E_{rj}^{2n+1}\left(
    \sum_{h=1}^{r-1}\sum_{m=1}^{r-2}E_{im}^{2n+2}E_{mh}^{2n+3}E_{hl}^{2n+2}%
    \right) E_{kr}^{2n+1}+ \\
    &&\qquad \qquad +\left(
    \sum_{m=1}^{r-1}E_{mj}^{2n+1}E_{im}^{2n+2}-(-1)^{g^{i+j}}\delta _{ij}\right)
    \left( \sum_{h=1}^{r-1}E_{rh}^{2n+3}E_{hl}^{2n+2}\right)
    E_{kr}^{2n+1}+\left( 1-\delta _{\left( r-1\right) l}\right)
    E_{rj}^{2n+1}E_{il}^{2n+2}E_{kr}^{2n+1} \\
    &&\overset{\ref{rel4}}{\rightarrow }E_{rj}^{2n+1}\left(
    \sum_{h=1}^{r-2}E_{ih}^{2n+2}E_{hr}^{2n+3}\right) \left(
    \sum_{h=1}^{r-2}E_{hl}^{2n+2}E_{kh}^{2n+1}-(-1)^{g^{k+l}}\delta
    _{kl}1\right) +\delta _{ir}E_{rj}^{2n+1}E_{\left( r-1\right)
    l}^{2n+2}E_{k\left( r-1\right) }^{2n+1}+ \\
    &&\qquad +\left(
    \sum_{h=1}^{r-1}E_{hj}^{2n+1}E_{ih}^{2n+2}-(-1)^{g^{i+j}}\delta
    _{ij}1\right) E_{rr}^{2n+3}\left(
    -\sum_{h=1}^{r-2}E_{hl}^{2n+2}E_{kh}^{2n+1}+(-1)^{g^{k+l}}\delta
    _{kl}1\right) -E_{rj}^{2n+1}\left(
    \sum_{h=1}^{r-1}\sum_{m=1}^{r-2}E_{im}^{2n+2}E_{mh}^{2n+3}E_{hl}^{2n+2}%
    \right) E_{kr}^{2n+1}+ \\
    &&\qquad \qquad +\left(
    \sum_{m=1}^{r-1}E_{mj}^{2n+1}E_{im}^{2n+2}-(-1)^{g^{i+j}}\delta _{ij}\right)
    \left( \sum_{h=1}^{r-1}E_{rh}^{2n+3}E_{hl}^{2n+2}\right)
    E_{kr}^{2n+1}+\left( 1-\delta _{\left( r-1\right) l}\right)
    E_{rj}^{2n+1}E_{il}^{2n+2}E_{kr}^{2n+1}
    \end{eqnarray*}%
    and so they are always equal. Thus, the ambiguity is resolvable.
    
    \item \ref{rel7}-\ref{rel1} \(E_{rj}^{2n+2}E_{i\left( r-1\right)
    }^{2n+1}E_{\left( r-1\right) r}^{2n}E_{rk}^{2n+1}\) for \(n\geq 1\)%
    \begin{eqnarray*}
    &&E_{rj}^{2n+2}E_{i\left( r-1\right) }^{2n+1}E_{\left( r-1\right)
    r}^{2n}E_{rk}^{2n+1}\overset{\ref{rel7}}{\rightarrow }%
    -\sum_{h=1}^{r-2}E_{rj}^{2n+2}E_{ih}^{2n+1}E_{hr}^{2n}E_{rk}^{2n+1}+%
    \sum_{h=1}^{r-1}E_{hj}^{2n+2}E_{ih}^{2n+1}E_{rr}^{2n}E_{rk}^{2n+1}+\delta
    _{ir}E_{rj}^{2n+2}E_{rk}^{2n+1}-(-1)^{g^{i+j}}\delta
    _{ij}E_{rr}^{2n}E_{rk}^{2n+1} \\
    &=&-\sum_{h=1}^{r-2}E_{rj}^{2n+2}E_{ih}^{2n+1}E_{hr}^{2n}E_{rk}^{2n+1}+%
    \left( \sum_{h=1}^{r-1}E_{hj}^{2n+2}E_{ih}^{2n+1}-(-1)^{g^{i+j}}\delta
    _{ij}1\right) E_{rr}^{2n}E_{rk}^{2n+1}+\delta _{ir}E_{rj}^{2n+2}E_{rk}^{2n+1}
    \\
    &&\overset{\ref{rel1}}{\rightarrow }\sum_{h=1}^{r-2}%
    \sum_{l=1}^{r-1}E_{rj}^{2n+2}E_{ih}^{2n+1}E_{hl}^{2n}E_{lk}^{2n+1}-%
    \sum_{h=1}^{r-2}\delta _{hk}E_{rj}^{2n+2}E_{ih}^{2n+1}+\left(
    \sum_{h=1}^{r-1}E_{hj}^{2n+2}E_{ih}^{2n+1}-(-1)^{g^{i+j}}\delta
    _{ij}1\right) \left( -\sum_{h=1}^{r-1}E_{rh}^{2n}E_{hk}^{2n+1}+\delta
    _{rk}1\right) +\delta _{ir}E_{rj}^{2n+2}E_{rk}^{2n+1} \\
    &=&\sum_{h=1}^{r-2}%
    \sum_{l=1}^{r-1}E_{rj}^{2n+2}E_{ih}^{2n+1}E_{hl}^{2n}E_{lk}^{2n+1}+\left(
    \sum_{h=1}^{r-1}E_{hj}^{2n+2}E_{ih}^{2n+1}-(-1)^{g^{i+j}}\delta
    _{ij}1\right) \left( -\sum_{h=1}^{r-1}E_{rh}^{2n}E_{hk}^{2n+1}\right)
    +\delta _{rk}\left(
    \sum_{h=1}^{r-1}E_{hj}^{2n+2}E_{ih}^{2n+1}-(-1)^{g^{i+j}}\delta
    _{ij}1\right) -\left( 1-\delta _{rk}-\delta _{\left( r-1\right) k}\right)
    E_{rj}^{2n+2}E_{ik}^{2n+1}+\delta _{ir}E_{rj}^{2n+2}E_{rk}^{2n+1} \\
    &=&\sum_{h=1}^{r-2}%
    \sum_{l=1}^{r-1}E_{rj}^{2n+2}E_{ih}^{2n+1}E_{hl}^{2n}E_{lk}^{2n+1}+\left(
    \sum_{h=1}^{r-1}E_{hj}^{2n+2}E_{ih}^{2n+1}-(-1)^{g^{i+j}}\delta
    _{ij}1\right) \left( -\sum_{h=1}^{r-1}E_{rh}^{2n}E_{hk}^{2n+1}\right)
    +\delta _{rk}\left(
    E_{rj}^{2n+2}E_{ir}^{2n+1}+%
    \sum_{h=1}^{r-1}E_{hj}^{2n+2}E_{ih}^{2n+1}-(-1)^{g^{i+j}}\delta
    _{ij}1\right) -\left( 1-\delta _{\left( r-1\right) k}-\delta _{ir}\right)
    E_{rj}^{2n+2}E_{ik}^{2n+1} \\
    &&\overset{\ref{rel4}}{\rightarrow }\sum_{h=1}^{r-2}%
    \sum_{l=1}^{r-1}E_{rj}^{2n+2}E_{ih}^{2n+1}E_{hl}^{2n}E_{lk}^{2n+1}+\left(
    \sum_{h=1}^{r-1}E_{hj}^{2n+2}E_{ih}^{2n+1}-(-1)^{g^{i+j}}\delta
    _{ij}1\right) \left( -\sum_{h=1}^{r-1}E_{rh}^{2n}E_{hk}^{2n+1}\right)
    -\left( 1-\delta _{\left( r-1\right) k}-\delta _{ir}\right)
    E_{rj}^{2n+2}E_{ik}^{2n+1}
    \end{eqnarray*}%
    and%
    \begin{eqnarray*}
    &&E_{rj}^{2n+2}E_{i\left( r-1\right) }^{2n+1}E_{\left( r-1\right)
    r}^{2n}E_{rk}^{2n+1}\overset{\ref{rel1}}{\rightarrow }%
    -\sum_{h=1}^{r-1}E_{rj}^{2n+2}E_{i\left( r-1\right) }^{2n+1}E_{\left(
    r-1\right) h}^{2n}E_{hk}^{2n+1}+\delta _{\left( r-1\right)
    k}E_{rj}^{2n+2}E_{i\left( r-1\right) }^{2n+1} \\
    &&\overset{\ref{rel7}}{\rightarrow }\sum_{h=1}^{r-1}%
    \sum_{l=1}^{r-2}E_{rj}^{2n+2}E_{il}^{2n+1}E_{lh}^{2n}E_{hk}^{2n+1}+\left(
    \sum_{l=1}^{r-1}E_{lj}^{2n+2}E_{il}^{2n+1}-(-1)^{g^{i+j}}\delta _{ij}\right)
    \left( -\sum_{h=1}^{r-1}E_{rh}^{2n}E_{hk}^{2n+1}\right)
    -\sum_{h=1}^{r-1}\delta _{ih}E_{rj}^{2n+2}E_{hk}^{2n+1}+\delta _{\left(
    r-1\right) k}E_{rj}^{2n+2}E_{i\left( r-1\right) }^{2n+1} \\
    &=&\sum_{h=1}^{r-1}%
    \sum_{l=1}^{r-2}E_{rj}^{2n+2}E_{il}^{2n+1}E_{lh}^{2n}E_{hk}^{2n+1}+\left(
    \sum_{l=1}^{r-1}E_{lj}^{2n+2}E_{il}^{2n+1}-(-1)^{g^{i+j}}\delta _{ij}\right)
    \left( -\sum_{h=1}^{r-1}E_{rh}^{2n}E_{hk}^{2n+1}\right) -\left( 1-\delta
    _{ir}\right) E_{rj}^{2n+2}E_{ik}^{2n+1}+\delta _{\left( r-1\right)
    k}E_{rj}^{2n+2}E_{ik}^{2n+1} \\
    &=&\sum_{h=1}^{r-1}%
    \sum_{l=1}^{r-2}E_{rj}^{2n+2}E_{il}^{2n+1}E_{lh}^{2n}E_{hk}^{2n+1}+\left(
    \sum_{l=1}^{r-1}E_{lj}^{2n+2}E_{il}^{2n+1}-(-1)^{g^{i+j}}\delta _{ij}\right)
    \left( -\sum_{h=1}^{r-1}E_{rh}^{2n}E_{hk}^{2n+1}\right) -\left( 1-\delta
    _{ir}-\delta _{\left( r-1\right) k}\right) E_{rj}^{2n+2}E_{ik}^{2n+1}
    \end{eqnarray*}%
    so that they are always equal and the ambiguity is resolvable.
    
    \item \ref{rel7}-\ref{rel5} \(E_{rj}^{2n+2}E_{i\left( r-1\right)
    }^{2n+1}E_{\left( r-1\right) r}^{2n}E_{\left( r-1\right) l}^{2n+1}E_{k\left(
    r-1\right) }^{2n+2}\) for \(n\geq 1\)%
    \begin{eqnarray*}
    &&E_{rj}^{2n+2}E_{i\left( r-1\right) }^{2n+1}E_{\left( r-1\right)
    r}^{2n}E_{\left( r-1\right) l}^{2n+1}E_{k\left( r-1\right) }^{2n+2}\overset{%
    \ref{rel7}}{\rightarrow }%
    -\sum_{h=1}^{r-2}E_{rj}^{2n+2}E_{ih}^{2n+1}E_{hr}^{2n}E_{\left( r-1\right)
    l}^{2n+1}E_{k\left( r-1\right)
    }^{2n+2}+\sum_{h=1}^{r-1}E_{hj}^{2n+2}E_{ih}^{2n+1}E_{rr}^{2n}E_{\left(
    r-1\right) l}^{2n+1}E_{k\left( r-1\right) }^{2n+2}+\delta
    _{ir}E_{rj}^{2n+2}E_{\left( r-1\right) l}^{2n+1}E_{k\left( r-1\right)
    }^{2n+2}-(-1)^{g^{i+j}}\delta _{ij}E_{rr}^{2n}E_{\left( r-1\right)
    l}^{2n+1}E_{k\left( r-1\right) }^{2n+2} \\
    &=&-\sum_{h=1}^{r-2}E_{rj}^{2n+2}E_{ih}^{2n+1}E_{hr}^{2n}E_{\left(
    r-1\right) l}^{2n+1}E_{k\left( r-1\right) }^{2n+2}+\left(
    \sum_{h=1}^{r-1}E_{hj}^{2n+2}E_{ih}^{2n+1}-(-1)^{g^{i+j}}\delta
    _{ij}1\right) E_{rr}^{2n}E_{\left( r-1\right) l}^{2n+1}E_{k\left( r-1\right)
    }^{2n+2}+\delta _{ir}E_{rj}^{2n+2}E_{\left( r-1\right) l}^{2n+1}E_{k\left(
    r-1\right) }^{2n+2} \\
    &&\overset{\ref{rel5}}{\rightarrow }%
    -\sum_{h=1}^{r-2}E_{rj}^{2n+2}E_{ih}^{2n+1}\left(
    -\sum_{m=1}^{r-2}E_{hr}^{2n}E_{ml}^{2n+1}E_{km}^{2n+2}+%
    \sum_{m=1}^{r-1}E_{hm}^{2n}E_{ml}^{2n+1}E_{kr}^{2n+2}-\delta
    _{hl}E_{kr}^{2n+2}+(-1)^{g^{k+l}}\delta _{kl}E_{hr}^{2n}\right) + \\
    &&\qquad +\left(
    \sum_{h=1}^{r-1}E_{hj}^{2n+2}E_{ih}^{2n+1}-(-1)^{g^{i+j}}\delta
    _{ij}1\right) \left(
    -\sum_{h=1}^{r-2}E_{rr}^{2n}E_{hl}^{2n+1}E_{kh}^{2n+2}+%
    \sum_{h=1}^{r-1}E_{rh}^{2n}E_{hl}^{2n+1}E_{kr}^{2n+2}-\delta
    _{rl}E_{kr}^{2n+2}+(-1)^{g^{k+l}}\delta _{kl}E_{rr}^{2n}\right) +\delta
    _{ir}E_{rj}^{2n+2}E_{\left( r-1\right) l}^{2n+1}E_{k\left( r-1\right)
    }^{2n+2} \\
    &=&E_{rj}^{2n+2}\left( \sum_{h=1}^{r-2}E_{ih}^{2n+1}\left( E_{hr}^{2n}\left(
    \sum_{m=1}^{r-2}E_{ml}^{2n+1}E_{km}^{2n+2}-(-1)^{g^{k+l}}\delta _{kl}\right)
    -\left( \sum_{m=1}^{r-1}E_{hm}^{2n}E_{ml}^{2n+1}-\delta _{hl}1\right)
    E_{kr}^{2n+2}\right) \right) + \\
    &&\qquad +\left(
    \sum_{h=1}^{r-1}E_{hj}^{2n+2}E_{ih}^{2n+1}-(-1)^{g^{i+j}}\delta
    _{ij}1\right) \left( -E_{rr}^{2n}\left(
    \sum_{h=1}^{r-2}E_{hl}^{2n+1}E_{kh}^{2n+2}-(-1)^{g^{k+l}}\delta
    _{kl}1\right) +\left( \sum_{h=1}^{r-1}E_{rh}^{2n}E_{hl}^{2n+1}-\delta
    _{rl}1\right) E_{kr}^{2n+2}\right) +\delta _{ir}E_{rj}^{2n+2}E_{\left(
    r-1\right) l}^{2n+1}E_{k\left( r-1\right) }^{2n+2} \\
    &=&E_{rj}^{2n+2}\left( \sum_{h=1}^{r-2}E_{ih}^{2n+1}E_{hr}^{2n}\right)
    \left( \sum_{m=1}^{r-2}E_{ml}^{2n+1}E_{km}^{2n+2}-(-1)^{g^{k+l}}\delta
    _{kl}\right) -E_{rj}^{2n+2}\left( \sum_{h=1}^{r-2}E_{ih}^{2n+1}\left(
    \sum_{m=1}^{r-1}E_{hm}^{2n}E_{ml}^{2n+1}-\delta _{hl}1\right) \right)
    E_{kr}^{2n+2}+ \\
    &&\qquad -\left(
    \sum_{h=1}^{r-1}E_{hj}^{2n+2}E_{ih}^{2n+1}-(-1)^{g^{i+j}}\delta
    _{ij}1\right) E_{rr}^{2n}\left(
    \sum_{h=1}^{r-2}E_{hl}^{2n+1}E_{kh}^{2n+2}-(-1)^{g^{k+l}}\delta
    _{kl}1\right) +\left(
    \sum_{h=1}^{r-1}E_{hj}^{2n+2}E_{ih}^{2n+1}-(-1)^{g^{i+j}}\delta
    _{ij}1\right) \left( \sum_{h=1}^{r-1}E_{rh}^{2n}E_{hl}^{2n+1}-\delta
    _{rl}1\right) E_{kr}^{2n+2}+\delta _{ir}E_{rj}^{2n+2}E_{\left( r-1\right)
    l}^{2n+1}E_{k\left( r-1\right) }^{2n+2} \\
    &=&E_{rj}^{2n+2}\left( \sum_{h=1}^{r-2}E_{ih}^{2n+1}E_{hr}^{2n}\right)
    \left( \sum_{m=1}^{r-2}E_{ml}^{2n+1}E_{km}^{2n+2}-(-1)^{g^{k+l}}\delta
    _{kl}\right) -E_{rj}^{2n+2}\left(
    \sum_{h=1}^{r-2}\sum_{m=1}^{r-1}E_{ih}^{2n+1}E_{hm}^{2n}E_{ml}^{2n+1}\right)
    E_{kr}^{2n+2}+\left( 1-\delta _{rl}-\delta _{\left( r-1\right) l}\right)
    E_{rj}^{2n+2}E_{il}^{2n+1}E_{kr}^{2n+2}+ \\
    &&\qquad -\left(
    \sum_{h=1}^{r-1}E_{hj}^{2n+2}E_{ih}^{2n+1}-(-1)^{g^{i+j}}\delta
    _{ij}1\right) E_{rr}^{2n}\left(
    \sum_{h=1}^{r-2}E_{hl}^{2n+1}E_{kh}^{2n+2}-(-1)^{g^{k+l}}\delta
    _{kl}1\right) +\left(
    \sum_{h=1}^{r-1}E_{hj}^{2n+2}E_{ih}^{2n+1}-(-1)^{g^{i+j}}\delta
    _{ij}1\right) \left( \sum_{h=1}^{r-1}E_{rh}^{2n}E_{hl}^{2n+1}\right)
    E_{kr}^{2n+2}+ \\
    &&\qquad \qquad -\delta _{rl}\left(
    \sum_{h=1}^{r-1}E_{hj}^{2n+2}E_{ih}^{2n+1}-(-1)^{g^{i+j}}\delta
    _{ij}1\right) E_{kr}^{2n+2}+\delta _{ir}E_{rj}^{2n+2}E_{\left( r-1\right)
    l}^{2n+1}E_{k\left( r-1\right) }^{2n+2} \\
    &=&E_{rj}^{2n+2}\left( \sum_{h=1}^{r-2}E_{ih}^{2n+1}E_{hr}^{2n}\right)
    \left( \sum_{m=1}^{r-2}E_{ml}^{2n+1}E_{km}^{2n+2}-(-1)^{g^{k+l}}\delta
    _{kl}\right) -E_{rj}^{2n+2}\left(
    \sum_{h=1}^{r-2}\sum_{m=1}^{r-1}E_{ih}^{2n+1}E_{hm}^{2n}E_{ml}^{2n+1}\right)
    E_{kr}^{2n+2}+\left( 1-\delta _{\left( r-1\right) l}\right)
    E_{rj}^{2n+2}E_{il}^{2n+1}E_{kr}^{2n+2}+ \\
    &&\qquad -\left(
    \sum_{h=1}^{r-1}E_{hj}^{2n+2}E_{ih}^{2n+1}-(-1)^{g^{i+j}}\delta
    _{ij}1\right) E_{rr}^{2n}\left(
    \sum_{h=1}^{r-2}E_{hl}^{2n+1}E_{kh}^{2n+2}-(-1)^{g^{k+l}}\delta
    _{kl}1\right) +\left(
    \sum_{h=1}^{r-1}E_{hj}^{2n+2}E_{ih}^{2n+1}-(-1)^{g^{i+j}}\delta
    _{ij}1\right) \left( \sum_{h=1}^{r-1}E_{rh}^{2n}E_{hl}^{2n+1}\right)
    E_{kr}^{2n+2}+ \\
    &&\qquad \qquad -\delta _{rl}\left(
    E_{rj}^{2n+2}E_{ir}^{2n+1}+%
    \sum_{h=1}^{r-1}E_{hj}^{2n+2}E_{ih}^{2n+1}-(-1)^{g^{i+j}}\delta
    _{ij}1\right) E_{kr}^{2n+2}+\delta _{ir}E_{rj}^{2n+2}E_{\left( r-1\right)
    l}^{2n+1}E_{k\left( r-1\right) }^{2n+2} \\
    &&\overset{\ref{rel4}}{\rightarrow }E_{rj}^{2n+2}\left(
    \sum_{h=1}^{r-2}E_{ih}^{2n+1}E_{hr}^{2n}\right) \left(
    \sum_{m=1}^{r-2}E_{ml}^{2n+1}E_{km}^{2n+2}-(-1)^{g^{k+l}}\delta _{kl}\right)
    -E_{rj}^{2n+2}\left(
    \sum_{h=1}^{r-2}\sum_{m=1}^{r-1}E_{ih}^{2n+1}E_{hm}^{2n}E_{ml}^{2n+1}\right)
    E_{kr}^{2n+2}+\left( 1-\delta _{\left( r-1\right) l}\right)
    E_{rj}^{2n+2}E_{il}^{2n+1}E_{kr}^{2n+2}+ \\
    &&\qquad -\left(
    \sum_{h=1}^{r-1}E_{hj}^{2n+2}E_{ih}^{2n+1}-(-1)^{g^{i+j}}\delta
    _{ij}1\right) E_{rr}^{2n}\left(
    \sum_{h=1}^{r-2}E_{hl}^{2n+1}E_{kh}^{2n+2}-(-1)^{g^{k+l}}\delta
    _{kl}1\right) +\left(
    \sum_{h=1}^{r-1}E_{hj}^{2n+2}E_{ih}^{2n+1}-(-1)^{g^{i+j}}\delta
    _{ij}1\right) \left( \sum_{h=1}^{r-1}E_{rh}^{2n}E_{hl}^{2n+1}\right)
    E_{kr}^{2n+2}+\delta _{ir}E_{rj}^{2n+2}E_{\left( r-1\right)
    l}^{2n+1}E_{k\left( r-1\right) }^{2n+2}
    \end{eqnarray*}%
    and%
    \begin{eqnarray*}
    &&E_{rj}^{2n+2}E_{i\left( r-1\right) }^{2n+1}E_{\left( r-1\right)
    r}^{2n}E_{\left( r-1\right) l}^{2n+1}E_{k\left( r-1\right) }^{2n+2}\overset{%
    \ref{rel5}}{\rightarrow }-\sum_{m=1}^{r-2}E_{rj}^{2n+2}E_{i\left( r-1\right)
    }^{2n+1}E_{\left( r-1\right)
    r}^{2n}E_{ml}^{2n+1}E_{km}^{2n+2}+\sum_{m=1}^{r-1}E_{rj}^{2n+2}E_{i\left(
    r-1\right) }^{2n+1}E_{\left( r-1\right)
    m}^{2n}E_{ml}^{2n+1}E_{kr}^{2n+2}-\delta _{\left( r-1\right)
    l}E_{rj}^{2n+2}E_{i\left( r-1\right)
    }^{2n+1}E_{kr}^{2n+2}+(-1)^{g^{k+l}}\delta _{kl}E_{rj}^{2n+2}E_{i\left(
    r-1\right) }^{2n+1}E_{\left( r-1\right) r}^{2n} \\
    &=&E_{rj}^{2n+2}E_{i\left( r-1\right) }^{2n+1}E_{\left( r-1\right)
    r}^{2n}\left(
    -\sum_{m=1}^{r-2}E_{ml}^{2n+1}E_{km}^{2n+2}+(-1)^{g^{k+l}}\delta
    _{kl}1\right) +\sum_{m=1}^{r-1}E_{rj}^{2n+2}E_{i\left( r-1\right)
    }^{2n+1}E_{\left( r-1\right) m}^{2n}E_{ml}^{2n+1}E_{kr}^{2n+2}-\delta
    _{\left( r-1\right) l}E_{rj}^{2n+2}E_{i\left( r-1\right)
    }^{2n+1}E_{kr}^{2n+2} \\
    &&\overset{\ref{rel7}}{\rightarrow }\left(
    -\sum_{h=1}^{r-2}E_{rj}^{2n+2}E_{ih}^{2n+1}E_{hr}^{2n}+%
    \sum_{h=1}^{r-1}E_{hj}^{2n+2}E_{ih}^{2n+1}E_{rr}^{2n}+\delta
    _{ir}E_{rj}^{2n+2}-(-1)^{g^{i+j}}\delta _{ij}E_{rr}^{2n}\right) \left(
    -\sum_{m=1}^{r-2}E_{ml}^{2n+1}E_{km}^{2n+2}+(-1)^{g^{k+l}}\delta
    _{kl}1\right) + \\
    &&\qquad +\sum_{m=1}^{r-1}\left(
    -\sum_{h=1}^{r-2}E_{rj}^{2n+2}E_{ih}^{2n+1}E_{hm}^{2n}+%
    \sum_{h=1}^{r-1}E_{hj}^{2n+2}E_{ih}^{2n+1}E_{rm}^{2n}+\delta
    _{im}E_{rj}^{2n+2}-(-1)^{g^{i+j}}\delta _{ij}E_{rm}^{2n}\right)
    E_{ml}^{2n+1}E_{kr}^{2n+2}-\delta _{\left( r-1\right)
    l}E_{rj}^{2n+2}E_{i\left( r-1\right) }^{2n+1}E_{kr}^{2n+2} \\
    &=&\left( E_{rj}^{2n+2}\left(
    \sum_{h=1}^{r-2}E_{ih}^{2n+1}E_{hr}^{2n}-\delta _{ir}1\right) -\left(
    \sum_{h=1}^{r-1}E_{hj}^{2n+2}E_{ih}^{2n+1}-(-1)^{g^{i+j}}\delta
    _{ij}1\right) E_{rr}^{2n}\right) \left(
    \sum_{m=1}^{r-2}E_{ml}^{2n+1}E_{km}^{2n+2}-(-1)^{g^{k+l}}\delta
    _{kl}1\right) + \\
    &&\qquad +\left( \sum_{m=1}^{r-1}\left( -E_{rj}^{2n+2}\left(
    \sum_{h=1}^{r-2}E_{ih}^{2n+1}E_{hm}^{2n}-\delta _{im}1\right) +\left(
    \sum_{h=1}^{r-1}E_{hj}^{2n+2}E_{ih}^{2n+1}-(-1)^{g^{i+j}}\delta
    _{ij}1\right) E_{rm}^{2n}\right) E_{ml}^{2n+1}\right) E_{kr}^{2n+2}-\delta
    _{\left( r-1\right) l}E_{rj}^{2n+2}E_{i\left( r-1\right)
    }^{2n+1}E_{kr}^{2n+2} \\
    &=&E_{rj}^{2n+2}\left( \sum_{h=1}^{r-2}E_{ih}^{2n+1}E_{hr}^{2n}-\delta
    _{ir}1\right) \left(
    \sum_{m=1}^{r-2}E_{ml}^{2n+1}E_{km}^{2n+2}-(-1)^{g^{k+l}}\delta
    _{kl}1\right) -\left(
    \sum_{h=1}^{r-1}E_{hj}^{2n+2}E_{ih}^{2n+1}-(-1)^{g^{i+j}}\delta
    _{ij}1\right) E_{rr}^{2n}\left(
    \sum_{m=1}^{r-2}E_{ml}^{2n+1}E_{km}^{2n+2}-(-1)^{g^{k+l}}\delta
    _{kl}1\right) + \\
    &&\qquad -E_{rj}^{2n+2}\left( \sum_{m=1}^{r-1}\left(
    \sum_{h=1}^{r-2}E_{ih}^{2n+1}E_{hm}^{2n}-\delta _{im}1\right)
    E_{ml}^{2n+1}\right) E_{kr}^{2n+2}+\left(
    \sum_{h=1}^{r-1}E_{hj}^{2n+2}E_{ih}^{2n+1}-(-1)^{g^{i+j}}\delta
    _{ij}1\right) \left( \sum_{m=1}^{r-1}E_{rm}^{2n}E_{ml}^{2n+1}\right)
    E_{kr}^{2n+2}-\delta _{\left( r-1\right) l}E_{rj}^{2n+2}E_{i\left(
    r-1\right) }^{2n+1}E_{kr}^{2n+2} \\
    &=&E_{rj}^{2n+2}\left( \sum_{h=1}^{r-2}E_{ih}^{2n+1}E_{hr}^{2n}\right)
    \left( \sum_{m=1}^{r-2}E_{ml}^{2n+1}E_{km}^{2n+2}-(-1)^{g^{k+l}}\delta
    _{kl}1\right) -\delta _{ir}E_{rj}^{2n+2}\left(
    \sum_{m=1}^{r-2}E_{ml}^{2n+1}E_{km}^{2n+2}-(-1)^{g^{k+l}}\delta
    _{kl}1\right) -\left(
    \sum_{h=1}^{r-1}E_{hj}^{2n+2}E_{ih}^{2n+1}-(-1)^{g^{i+j}}\delta
    _{ij}1\right) E_{rr}^{2n}\left(
    \sum_{m=1}^{r-2}E_{ml}^{2n+1}E_{km}^{2n+2}-(-1)^{g^{k+l}}\delta
    _{kl}1\right) + \\
    &&\qquad -E_{rj}^{2n+2}\left( \sum_{m=1}^{r-1}\left(
    \sum_{h=1}^{r-2}E_{ih}^{2n+1}E_{hm}^{2n}-\delta _{im}1\right)
    E_{ml}^{2n+1}\right) E_{kr}^{2n+2}+\left(
    \sum_{h=1}^{r-1}E_{hj}^{2n+2}E_{ih}^{2n+1}-(-1)^{g^{i+j}}\delta
    _{ij}1\right) \left( \sum_{m=1}^{r-1}E_{rm}^{2n}E_{ml}^{2n+1}\right)
    E_{kr}^{2n+2}-\delta _{\left( r-1\right) l}E_{rj}^{2n+2}E_{i\left(
    r-1\right) }^{2n+1}E_{kr}^{2n+2} \\
    &=&E_{rj}^{2n+2}\left( \sum_{h=1}^{r-2}E_{ih}^{2n+1}E_{hr}^{2n}\right)
    \left( \sum_{m=1}^{r-2}E_{ml}^{2n+1}E_{km}^{2n+2}-(-1)^{g^{k+l}}\delta
    _{kl}1\right) -\delta _{ir}E_{rj}^{2n+2}\left(
    \sum_{m=1}^{r-2}E_{ml}^{2n+1}E_{km}^{2n+2}-(-1)^{g^{k+l}}\delta
    _{kl}1\right) -\left(
    \sum_{h=1}^{r-1}E_{hj}^{2n+2}E_{ih}^{2n+1}-(-1)^{g^{i+j}}\delta
    _{ij}1\right) E_{rr}^{2n}\left(
    \sum_{m=1}^{r-2}E_{ml}^{2n+1}E_{km}^{2n+2}-(-1)^{g^{k+l}}\delta
    _{kl}1\right) + \\
    &&\qquad -E_{rj}^{2n+2}\left(
    \sum_{m=1}^{r-1}\sum_{h=1}^{r-2}E_{ih}^{2n+1}E_{hm}^{2n}E_{ml}^{2n+1}\right)
    E_{kr}^{2n+2}+\left( 1-\delta _{ir}\right)
    E_{rj}^{2n+2}E_{il}^{2n+1}E_{kr}^{2n+2}+\left(
    \sum_{h=1}^{r-1}E_{hj}^{2n+2}E_{ih}^{2n+1}-(-1)^{g^{i+j}}\delta
    _{ij}1\right) \left( \sum_{m=1}^{r-1}E_{rm}^{2n}E_{ml}^{2n+1}\right)
    E_{kr}^{2n+2}-\delta _{\left( r-1\right) l}E_{rj}^{2n+2}E_{i\left(
    r-1\right) }^{2n+1}E_{kr}^{2n+2} \\
    &=&E_{rj}^{2n+2}\left( \sum_{h=1}^{r-2}E_{ih}^{2n+1}E_{hr}^{2n}\right)
    \left( \sum_{m=1}^{r-2}E_{ml}^{2n+1}E_{km}^{2n+2}-(-1)^{g^{k+l}}\delta
    _{kl}1\right) -\delta _{ir}E_{rj}^{2n+2}\left(
    E_{rl}^{2n+1}E_{kr}^{2n+2}+%
    \sum_{m=1}^{r-2}E_{ml}^{2n+1}E_{km}^{2n+2}-(-1)^{g^{k+l}}\delta
    _{kl}1\right) -\left(
    \sum_{h=1}^{r-1}E_{hj}^{2n+2}E_{ih}^{2n+1}-(-1)^{g^{i+j}}\delta
    _{ij}1\right) E_{rr}^{2n}\left(
    \sum_{m=1}^{r-2}E_{ml}^{2n+1}E_{km}^{2n+2}-(-1)^{g^{k+l}}\delta
    _{kl}1\right) + \\
    &&\qquad -E_{rj}^{2n+2}\left(
    \sum_{m=1}^{r-1}\sum_{h=1}^{r-2}E_{ih}^{2n+1}E_{hm}^{2n}E_{ml}^{2n+1}\right)
    E_{kr}^{2n+2}+\left(
    \sum_{h=1}^{r-1}E_{hj}^{2n+2}E_{ih}^{2n+1}-(-1)^{g^{i+j}}\delta
    _{ij}1\right) \left( \sum_{m=1}^{r-1}E_{rm}^{2n}E_{ml}^{2n+1}\right)
    E_{kr}^{2n+2}+\left( 1-\delta _{\left( r-1\right) l}\right)
    E_{rj}^{2n+2}E_{i\left( r-1\right) }^{2n+1}E_{kr}^{2n+2} \\
    &&\overset{\ref{rel2}}{\rightarrow }E_{rj}^{2n+2}\left(
    \sum_{h=1}^{r-2}E_{ih}^{2n+1}E_{hr}^{2n}\right) \left(
    \sum_{m=1}^{r-2}E_{ml}^{2n+1}E_{km}^{2n+2}-(-1)^{g^{k+l}}\delta
    _{kl}1\right) +\delta _{ir}E_{rj}^{2n+2}E_{\left( r-1\right)
    l}^{2n+1}E_{k\left( r-1\right) }^{2n+2}-\left(
    \sum_{h=1}^{r-1}E_{hj}^{2n+2}E_{ih}^{2n+1}-(-1)^{g^{i+j}}\delta
    _{ij}1\right) E_{rr}^{2n}\left(
    \sum_{m=1}^{r-2}E_{ml}^{2n+1}E_{km}^{2n+2}-(-1)^{g^{k+l}}\delta
    _{kl}1\right) + \\
    &&\qquad -E_{rj}^{2n+2}\left(
    \sum_{m=1}^{r-1}\sum_{h=1}^{r-2}E_{ih}^{2n+1}E_{hm}^{2n}E_{ml}^{2n+1}\right)
    E_{kr}^{2n+2}+\left(
    \sum_{h=1}^{r-1}E_{hj}^{2n+2}E_{ih}^{2n+1}-(-1)^{g^{i+j}}\delta
    _{ij}1\right) \left( \sum_{m=1}^{r-1}E_{rm}^{2n}E_{ml}^{2n+1}\right)
    E_{kr}^{2n+2}+\left( 1-\delta _{\left( r-1\right) l}\right)
    E_{rj}^{2n+2}E_{i\left( r-1\right) }^{2n+1}E_{kr}^{2n+2}
    \end{eqnarray*}%
    and so they are always equal. Thus, the ambiguity is resolvable.
    
    \item \ref{rel8}-\ref{rel2} \(E_{ir}^{2n+3}E_{\left( r-1\right)
    j}^{2n+2}E_{r\left( r-1\right) }^{2n+1}E_{kr}^{2n+2}\) for \(n\geq 0\)%
    \begin{eqnarray*}
    &&E_{ir}^{2n+3}E_{\left( r-1\right) j}^{2n+2}E_{r\left( r-1\right)
    }^{2n+1}E_{kr}^{2n+2}\overset{\ref{rel8}}{\rightarrow }%
    -\sum_{h=1}^{r-2}E_{ir}^{2n+3}E_{hj}^{2n+2}E_{rh}^{2n+1}E_{kr}^{2n+2}+%
    \sum_{h=1}^{r-1}E_{ih}^{2n+3}E_{hj}^{2n+2}E_{rr}^{2n+1}E_{kr}^{2n+2}+(-1)^{g^{r+j}}\delta _{rj}E_{ir}^{2n+3}E_{kr}^{2n+2}-\delta _{ij}E_{rr}^{2n+1}E_{kr}^{2n+2}
    \\
    &=&-\sum_{h=1}^{r-2}E_{ir}^{2n+3}E_{hj}^{2n+2}E_{rh}^{2n+1}E_{kr}^{2n+2}+%
    \left( \sum_{h=1}^{r-1}E_{ih}^{2n+3}E_{hj}^{2n+2}-\delta _{ij}1\right)
    E_{rr}^{2n+1}E_{kr}^{2n+2}+(-1)^{g^{r+j}}\delta
    _{rj}E_{ir}^{2n+3}E_{kr}^{2n+2} \\
    &&\overset{\ref{rel2}}{\rightarrow }\sum_{h=1}^{r-2}%
    \sum_{l=1}^{r-1}E_{ir}^{2n+3}E_{hj}^{2n+2}E_{lh}^{2n+1}E_{kl}^{2n+2}-%
    \sum_{h=1}^{r-2}(-1)^{g^{k+h}}\delta _{kh}E_{ir}^{2n+3}E_{hj}^{2n+2}+\left(
    \sum_{h=1}^{r-1}E_{ih}^{2n+3}E_{hj}^{2n+2}-\delta _{ij}1\right) \left(
    -\sum_{h=1}^{r-1}E_{hr}^{2n+1}E_{kh}^{2n+2}+(-1)^{g^{k+r}}\delta
    _{kr}1\right) +(-1)^{g^{r+j}}\delta _{rj}E_{ir}^{2n+3}E_{kr}^{2n+2} \\
    &=&\sum_{h=1}^{r-2}%
    \sum_{l=1}^{r-1}E_{ir}^{2n+3}E_{hj}^{2n+2}E_{lh}^{2n+1}E_{kl}^{2n+2}-\left(
    1-\delta _{kr}-\delta _{k\left( r-1\right) }\right)
    E_{ir}^{2n+3}E_{kj}^{2n+2}-\left(
    \sum_{h=1}^{r-1}E_{ih}^{2n+3}E_{hj}^{2n+2}-\delta _{ij}1\right) \left(
    \sum_{h=1}^{r-1}E_{hr}^{2n+1}E_{kh}^{2n+2}-(-1)^{g^{k+r}}\delta
    _{kr}1\right) +(-1)^{g^{k+r}}\delta _{kr}\left(
    \sum_{h=1}^{r-1}E_{ih}^{2n+3}E_{hj}^{2n+2}-\delta _{ij}1\right)
    +(-1)^{g^{r+j}}\delta _{rj}E_{ir}^{2n+3}E_{kr}^{2n+2} \\
    &=&\sum_{h=1}^{r-2}%
    \sum_{l=1}^{r-1}E_{ir}^{2n+3}E_{hj}^{2n+2}E_{lh}^{2n+1}E_{kl}^{2n+2}-\left(
    1-\delta _{k\left( r-1\right) }-\delta _{rj}\right)
    E_{ir}^{2n+3}E_{kj}^{2n+2}-\left(
    \sum_{h=1}^{r-1}E_{ih}^{2n+3}E_{hj}^{2n+2}-\delta _{ij}1\right) \left(
    \sum_{h=1}^{r-1}E_{hr}^{2n+1}E_{kh}^{2n+2}-(-1)^{g^{k+r}}\delta
    _{kr}1\right) +(-1)^{g^{k+r}}\delta _{kr}\left(
    E_{ir}^{2n+3}E_{rj}^{2n+2}+\sum_{h=1}^{r-1}E_{ih}^{2n+3}E_{hj}^{2n+2}-\delta
    _{ij}1\right)  \\
    &&\overset{\ref{rel3}}{\rightarrow }\sum_{h=1}^{r-2}%
    \sum_{l=1}^{r-1}E_{ir}^{2n+3}E_{hj}^{2n+2}E_{lh}^{2n+1}E_{kl}^{2n+2}-\left(
    1-\delta _{k\left( r-1\right) }-\delta _{rj}\right)
    E_{ir}^{2n+3}E_{kj}^{2n+2}-\left(
    \sum_{h=1}^{r-1}E_{ih}^{2n+3}E_{hj}^{2n+2}-\delta _{ij}1\right) \left(
    \sum_{h=1}^{r-1}E_{hr}^{2n+1}E_{kh}^{2n+2}-(-1)^{g^{k+r}}\delta
    _{kr}1\right) 
    \end{eqnarray*}%
    and%
    \begin{eqnarray*}
    &&E_{ir}^{2n+3}E_{\left( r-1\right) j}^{2n+2}E_{r\left( r-1\right)
    }^{2n+1}E_{kr}^{2n+2}\overset{\ref{rel2}}{\rightarrow }%
    -\sum_{l=1}^{r-1}E_{ir}^{2n+3}E_{\left( r-1\right) j}^{2n+2}E_{l\left(
    r-1\right) }^{2n+1}E_{kl}^{2n+2}+(-1)^{g^{k+\left( r-1\right) }}\delta
    _{k\left( r-1\right) }E_{ir}^{2n+3}E_{\left( r-1\right) j}^{2n+2} \\
    &&\overset{\ref{rel8}}{\rightarrow }-\sum_{l=1}^{r-1}\left(
    -\sum_{h=1}^{r-2}E_{ir}^{2n+3}E_{hj}^{2n+2}E_{lh}^{2n+1}+%
    \sum_{h=1}^{r-1}E_{ih}^{2n+3}E_{hj}^{2n+2}E_{lr}^{2n+1}+(-1)^{g^{l+j}}\delta
    _{lj}E_{ir}^{2n+3}-\delta _{ij}E_{lr}^{2n+1}\right)
    E_{kl}^{2n+2}+(-1)^{g^{k+\left( r-1\right) }}\delta _{k\left( r-1\right)
    }E_{ir}^{2n+3}E_{\left( r-1\right) j}^{2n+2} \\
    &=&\sum_{h=1}^{r-2}%
    \sum_{l=1}^{r-1}E_{ir}^{2n+3}E_{hj}^{2n+2}E_{lh}^{2n+1}E_{kl}^{2n+2}-\left(
    \sum_{h=1}^{r-1}E_{ih}^{2n+3}E_{hj}^{2n+2}-\delta _{ij}1\right) \left(
    \sum_{l=1}^{r-1}E_{lr}^{2n+1}E_{kl}^{2n+2}\right) -\left( 1-\delta
    _{rj}-\delta _{k\left( r-1\right) }\right) E_{ir}^{2n+3}E_{kj}^{2n+2}
    \end{eqnarray*}%
    and so they are always equal. Hence, the ambiguity is resolvable.
    
    \item \ref{rel8}-\ref{rel6} \(E_{ir}^{2n+3}E_{\left( r-1\right)
    j}^{2n+2}E_{r\left( r-1\right) }^{2n+1}E_{l\left( r-1\right)
    }^{2n+2}E_{\left( r-1\right) k}^{2n+3}\) for \(n\geq 0\)%
    \begin{eqnarray*}
    &&E_{ir}^{2n+3}E_{\left( r-1\right) j}^{2n+2}E_{r\left( r-1\right)
    }^{2n+1}E_{l\left( r-1\right) }^{2n+2}E_{\left( r-1\right) k}^{2n+3}\overset{%
    \ref{rel8}}{\rightarrow }%
    -\sum_{h=1}^{r-2}E_{ir}^{2n+3}E_{hj}^{2n+2}E_{rh}^{2n+1}E_{l\left(
    r-1\right) }^{2n+2}E_{\left( r-1\right)
    k}^{2n+3}+\sum_{h=1}^{r-1}E_{ih}^{2n+3}E_{hj}^{2n+2}E_{rr}^{2n+1}E_{l\left(
    r-1\right) }^{2n+2}E_{\left( r-1\right) k}^{2n+3}+(-1)^{g^{r+j}}\delta
    _{rj}E_{ir}^{2n+3}E_{l\left( r-1\right) }^{2n+2}E_{\left( r-1\right)
    k}^{2n+3}-\delta _{ij}E_{rr}^{2n+1}E_{l\left( r-1\right) }^{2n+2}E_{\left(
    r-1\right) k}^{2n+3} \\
    &=&-\sum_{h=1}^{r-2}E_{ir}^{2n+3}E_{hj}^{2n+2}\left( E_{rh}^{2n+1}E_{l\left(
    r-1\right) }^{2n+2}E_{\left( r-1\right) k}^{2n+3}\right) +\left(
    \sum_{h=1}^{r-1}E_{ih}^{2n+3}E_{hj}^{2n+2}-\delta _{ij}1\right)
    E_{rr}^{2n+1}E_{l\left( r-1\right) }^{2n+2}E_{\left( r-1\right)
    k}^{2n+3}+(-1)^{g^{r+j}}\delta _{rj}E_{ir}^{2n+3}E_{l\left( r-1\right)
    }^{2n+2}E_{\left( r-1\right) k}^{2n+3} \\
    &&\overset{\ref{rel6}}{\rightarrow }%
    -\sum_{h=1}^{r-2}E_{ir}^{2n+3}E_{hj}^{2n+2}\left(
    -\sum_{m=1}^{r-2}E_{rh}^{2n+1}E_{lm}^{2n+2}E_{mk}^{2n+3}+%
    \sum_{m=1}^{r-1}E_{mh}^{2n+1}E_{lm}^{2n+2}E_{rk}^{2n+3}-(-1)^{g^{l+h}}\delta
    _{lh}E_{rk}^{2n+3}+\delta _{lk}E_{rh}^{2n+1}\right) + \\
    &&\qquad +\left( \sum_{h=1}^{r-1}E_{ih}^{2n+3}E_{hj}^{2n+2}-\delta
    _{ij}1\right) \left(
    -\sum_{m=1}^{r-2}E_{rr}^{2n+1}E_{lm}^{2n+2}E_{mk}^{2n+3}+%
    \sum_{m=1}^{r-1}E_{mr}^{2n+1}E_{lm}^{2n+2}E_{rk}^{2n+3}-(-1)^{g^{l+r}}\delta
    _{lr}E_{rk}^{2n+3}+\delta _{lk}E_{rr}^{2n+1}\right) +(-1)^{g^{r+j}}\delta
    _{rj}E_{ir}^{2n+3}E_{l\left( r-1\right) }^{2n+2}E_{\left( r-1\right)
    k}^{2n+3} \\
    &=&E_{ir}^{2n+3}\left( \sum_{h=1}^{r-2}E_{hj}^{2n+2}E_{rh}^{2n+1}\right)
    \left( \sum_{m=1}^{r-2}E_{lm}^{2n+2}E_{mk}^{2n+3}-\delta _{lk}1\right)
    -E_{ir}^{2n+3}\left(
    \sum_{h=1}^{r-2}\sum_{m=1}^{r-1}E_{hj}^{2n+2}E_{mh}^{2n+1}E_{lm}^{2n+2}%
    \right) E_{rk}^{2n+3}+\sum_{h=1}^{r-2}(-1)^{g^{l+h}}\delta
    _{lh}E_{ir}^{2n+3}E_{hj}^{2n+2}E_{rk}^{2n+3}+ \\
    &&\qquad -\left( \sum_{h=1}^{r-1}E_{ih}^{2n+3}E_{hj}^{2n+2}-\delta
    _{ij}1\right) E_{rr}^{2n+1}\left(
    \sum_{m=1}^{r-2}E_{lm}^{2n+2}E_{mk}^{2n+3}-\delta _{lk}1\right) +\left(
    \sum_{h=1}^{r-1}E_{ih}^{2n+3}E_{hj}^{2n+2}-\delta _{ij}1\right) \left(
    \sum_{m=1}^{r-1}E_{mr}^{2n+1}E_{lm}^{2n+2}-(-1)^{g^{l+r}}\delta _{lr}\right)
    E_{rk}^{2n+3}+(-1)^{g^{r+j}}\delta _{rj}E_{ir}^{2n+3}E_{l\left( r-1\right)
    }^{2n+2}E_{\left( r-1\right) k}^{2n+3} \\
    &=&E_{ir}^{2n+3}\left( \sum_{h=1}^{r-2}E_{hj}^{2n+2}E_{rh}^{2n+1}\right)
    \left( \sum_{m=1}^{r-2}E_{lm}^{2n+2}E_{mk}^{2n+3}-\delta _{lk}1\right)
    -E_{ir}^{2n+3}\left(
    \sum_{h=1}^{r-2}\sum_{m=1}^{r-1}E_{hj}^{2n+2}E_{mh}^{2n+1}E_{lm}^{2n+2}%
    \right) E_{rk}^{2n+3}+\left( 1-\delta _{lr}-\delta _{l\left( r-1\right)
    }\right) E_{ir}^{2n+3}E_{lj}^{2n+2}E_{rk}^{2n+3}+ \\
    &&\qquad -\left( \sum_{h=1}^{r-1}E_{ih}^{2n+3}E_{hj}^{2n+2}-\delta
    _{ij}1\right) E_{rr}^{2n+1}\left(
    \sum_{m=1}^{r-2}E_{lm}^{2n+2}E_{mk}^{2n+3}-\delta _{lk}1\right) +\left(
    \sum_{h=1}^{r-1}E_{ih}^{2n+3}E_{hj}^{2n+2}-\delta _{ij}1\right) \left(
    \sum_{m=1}^{r-1}E_{mr}^{2n+1}E_{lm}^{2n+2}\right) E_{rk}^{2n+3}+ \\
    &&\qquad \qquad -(-1)^{g^{l+r}}\delta _{lr}\left(
    \sum_{h=1}^{r-1}E_{ih}^{2n+3}E_{hj}^{2n+2}-\delta _{ij}1\right)
    E_{rk}^{2n+3}+(-1)^{g^{r+j}}\delta _{rj}E_{ir}^{2n+3}E_{l\left( r-1\right)
    }^{2n+2}E_{\left( r-1\right) k}^{2n+3} \\
    &=&E_{ir}^{2n+3}\left( \sum_{h=1}^{r-2}E_{hj}^{2n+2}E_{rh}^{2n+1}\right)
    \left( \sum_{m=1}^{r-2}E_{lm}^{2n+2}E_{mk}^{2n+3}-\delta _{lk}1\right)
    -E_{ir}^{2n+3}\left(
    \sum_{h=1}^{r-2}\sum_{m=1}^{r-1}E_{hj}^{2n+2}E_{mh}^{2n+1}E_{lm}^{2n+2}%
    \right) E_{rk}^{2n+3}+\left( 1-\delta _{l\left( r-1\right) }\right)
    E_{ir}^{2n+3}E_{lj}^{2n+2}E_{rk}^{2n+3}+ \\
    &&\qquad -\left( \sum_{h=1}^{r-1}E_{ih}^{2n+3}E_{hj}^{2n+2}-\delta
    _{ij}1\right) E_{rr}^{2n+1}\left(
    \sum_{m=1}^{r-2}E_{lm}^{2n+2}E_{mk}^{2n+3}-\delta _{lk}1\right) +\left(
    \sum_{h=1}^{r-1}E_{ih}^{2n+3}E_{hj}^{2n+2}-\delta _{ij}1\right) \left(
    \sum_{m=1}^{r-1}E_{mr}^{2n+1}E_{lm}^{2n+2}\right) E_{rk}^{2n+3}+ \\
    &&\qquad \qquad -(-1)^{g^{l+r}}\delta _{lr}\left(
    E_{ir}^{2n+3}E_{rj}^{2n+2}+\sum_{h=1}^{r-1}E_{ih}^{2n+3}E_{hj}^{2n+2}-\delta
    _{ij}1\right) E_{rk}^{2n+3}+(-1)^{g^{r+j}}\delta
    _{rj}E_{ir}^{2n+3}E_{l\left( r-1\right) }^{2n+2}E_{\left( r-1\right)
    k}^{2n+3} \\
    &&\overset{\ref{rel3}}{\rightarrow }E_{ir}^{2n+3}\left(
    \sum_{h=1}^{r-2}E_{hj}^{2n+2}E_{rh}^{2n+1}\right) \left(
    \sum_{m=1}^{r-2}E_{lm}^{2n+2}E_{mk}^{2n+3}-\delta _{lk}1\right)
    -E_{ir}^{2n+3}\left(
    \sum_{h=1}^{r-2}\sum_{m=1}^{r-1}E_{hj}^{2n+2}E_{mh}^{2n+1}E_{lm}^{2n+2}%
    \right) E_{rk}^{2n+3}+\left( 1-\delta _{l\left( r-1\right) }\right)
    E_{ir}^{2n+3}E_{lj}^{2n+2}E_{rk}^{2n+3}+ \\
    &&\qquad -\left( \sum_{h=1}^{r-1}E_{ih}^{2n+3}E_{hj}^{2n+2}-\delta
    _{ij}1\right) E_{rr}^{2n+1}\left(
    \sum_{m=1}^{r-2}E_{lm}^{2n+2}E_{mk}^{2n+3}-\delta _{lk}1\right) +\left(
    \sum_{h=1}^{r-1}E_{ih}^{2n+3}E_{hj}^{2n+2}-\delta _{ij}1\right) \left(
    \sum_{m=1}^{r-1}E_{mr}^{2n+1}E_{lm}^{2n+2}\right)
    E_{rk}^{2n+3}+(-1)^{g^{r+j}}\delta _{rj}E_{ir}^{2n+3}E_{l\left( r-1\right)
    }^{2n+2}E_{\left( r-1\right) k}^{2n+3}
    \end{eqnarray*}%
    and%
    \begin{eqnarray*}
    &&E_{ir}^{2n+3}E_{\left( r-1\right) j}^{2n+2}E_{r\left( r-1\right)
    }^{2n+1}E_{l\left( r-1\right) }^{2n+2}E_{\left( r-1\right) k}^{2n+3}\overset{%
    \ref{rel6}}{\rightarrow }-\sum_{m=1}^{r-2}E_{ir}^{2n+3}E_{\left( r-1\right)
    j}^{2n+2}E_{r\left( r-1\right)
    }^{2n+1}E_{lm}^{2n+2}E_{mk}^{2n+3}+\sum_{m=1}^{r-1}E_{ir}^{2n+3}E_{\left(
    r-1\right) j}^{2n+2}E_{m\left( r-1\right)
    }^{2n+1}E_{lm}^{2n+2}E_{rk}^{2n+3}-(-1)^{g^{l+\left( r-1\right) }}\delta
    _{l\left( r-1\right) }E_{ir}^{2n+3}E_{\left( r-1\right)
    j}^{2n+2}E_{rk}^{2n+3}+\delta _{lk}E_{ir}^{2n+3}E_{\left( r-1\right)
    j}^{2n+2}E_{r\left( r-1\right) }^{2n+1} \\
    &=&E_{ir}^{2n+3}E_{\left( r-1\right) j}^{2n+2}E_{r\left( r-1\right)
    }^{2n+1}\left( -\sum_{m=1}^{r-2}E_{lm}^{2n+2}E_{mk}^{2n+3}+\delta
    _{lk}1\right) +\sum_{m=1}^{r-1}\left( E_{ir}^{2n+3}E_{\left( r-1\right)
    j}^{2n+2}E_{m\left( r-1\right) }^{2n+1}\right)
    E_{lm}^{2n+2}E_{rk}^{2n+3}-(-1)^{g^{l+\left( r-1\right) }}\delta _{l\left(
    r-1\right) }E_{ir}^{2n+3}E_{\left( r-1\right) j}^{2n+2}E_{rk}^{2n+3} \\
    &&\overset{\ref{rel8}}{\rightarrow }\left(
    -\sum_{h=1}^{r-2}E_{ir}^{2n+3}E_{hj}^{2n+2}E_{rh}^{2n+1}+%
    \sum_{h=1}^{r-1}E_{ih}^{2n+3}E_{hj}^{2n+2}E_{rr}^{2n+1}+(-1)^{g^{r+j}}\delta
    _{rj}E_{ir}^{2n+3}-\delta _{ij}E_{rr}^{2n+1}\right) \left(
    -\sum_{m=1}^{r-2}E_{lm}^{2n+2}E_{mk}^{2n+3}+\delta _{lk}1\right) + \\
    &&\qquad +\sum_{m=1}^{r-1}\left(
    -\sum_{h=1}^{r-2}E_{ir}^{2n+3}E_{hj}^{2n+2}E_{mh}^{2n+1}+%
    \sum_{h=1}^{r-1}E_{ih}^{2n+3}E_{hj}^{2n+2}E_{mr}^{2n+1}+(-1)^{g^{m+j}}\delta
    _{mj}E_{ir}^{2n+3}-\delta _{ij}E_{mr}^{2n+1}\right)
    E_{lm}^{2n+2}E_{rk}^{2n+3}-(-1)^{g^{l+\left( r-1\right) }}\delta _{l\left(
    r-1\right) }E_{ir}^{2n+3}E_{\left( r-1\right) j}^{2n+2}E_{rk}^{2n+3} \\
    &=&E_{ir}^{2n+3}\left(
    \sum_{h=1}^{r-2}E_{hj}^{2n+2}E_{rh}^{2n+1}-(-1)^{g^{r+j}}\delta
    _{rj}1\right) \left( \sum_{m=1}^{r-2}E_{lm}^{2n+2}E_{mk}^{2n+3}-\delta
    _{lk}1\right) -\left( \sum_{h=1}^{r-1}E_{ih}^{2n+3}E_{hj}^{2n+2}-\delta
    _{ij}1\right) E_{rr}^{2n+1}\left(
    \sum_{m=1}^{r-2}E_{lm}^{2n+2}E_{mk}^{2n+3}-\delta _{lk}1\right) + \\
    &&\qquad -E_{ir}^{2n+3}\left(
    \sum_{m=1}^{r-1}\sum_{h=1}^{r-2}E_{hj}^{2n+2}E_{mh}^{2n+1}E_{lm}^{2n+2}%
    \right) E_{rk}^{2n+3}+E_{ir}^{2n+3}\left(
    \sum_{m=1}^{r-1}(-1)^{g^{m+j}}\delta _{mj}E_{lm}^{2n+2}\right)
    E_{rk}^{2n+3}+\left( \sum_{h=1}^{r-1}E_{ih}^{2n+3}E_{hj}^{2n+2}-\delta
    _{ij}1\right) \left( \sum_{m=1}^{r-1}E_{mr}^{2n+1}E_{lm}^{2n+2}\right)
    E_{rk}^{2n+3}-(-1)^{g^{l+\left( r-1\right) }}\delta _{l\left( r-1\right)
    }E_{ir}^{2n+3}E_{\left( r-1\right) j}^{2n+2}E_{rk}^{2n+3} \\
    &=&E_{ir}^{2n+3}\left( \sum_{h=1}^{r-2}E_{hj}^{2n+2}E_{rh}^{2n+1}\right)
    \left( \sum_{m=1}^{r-2}E_{lm}^{2n+2}E_{mk}^{2n+3}-\delta _{lk}1\right)
    -(-1)^{g^{r+j}}\delta _{rj}E_{ir}^{2n+3}\left(
    \sum_{m=1}^{r-2}E_{lm}^{2n+2}E_{mk}^{2n+3}-\delta _{lk}1\right) -\left(
    \sum_{h=1}^{r-1}E_{ih}^{2n+3}E_{hj}^{2n+2}-\delta _{ij}1\right)
    E_{rr}^{2n+1}\left( \sum_{m=1}^{r-2}E_{lm}^{2n+2}E_{mk}^{2n+3}-\delta
    _{lk}1\right) + \\
    &&\qquad -E_{ir}^{2n+3}\left(
    \sum_{m=1}^{r-1}\sum_{h=1}^{r-2}E_{hj}^{2n+2}E_{mh}^{2n+1}E_{lm}^{2n+2}%
    \right) E_{rk}^{2n+3}+\left( 1-\delta _{rj}\right)
    E_{ir}^{2n+3}E_{lj}^{2n+2}E_{rk}^{2n+3}+\left(
    \sum_{h=1}^{r-1}E_{ih}^{2n+3}E_{hj}^{2n+2}-\delta _{ij}1\right) \left(
    \sum_{m=1}^{r-1}E_{mr}^{2n+1}E_{lm}^{2n+2}\right)
    E_{rk}^{2n+3}-(-1)^{g^{l+\left( r-1\right) }}\delta _{l\left( r-1\right)
    }E_{ir}^{2n+3}E_{\left( r-1\right) j}^{2n+2}E_{rk}^{2n+3} \\
    &=&E_{ir}^{2n+3}\left( \sum_{h=1}^{r-2}E_{hj}^{2n+2}E_{rh}^{2n+1}\right)
    \left( \sum_{m=1}^{r-2}E_{lm}^{2n+2}E_{mk}^{2n+3}-\delta _{lk}1\right)
    -(-1)^{g^{r+j}}\delta _{rj}E_{ir}^{2n+3}\left(
    E_{lr}^{2n+2}E_{rk}^{2n+3}+\sum_{m=1}^{r-2}E_{lm}^{2n+2}E_{mk}^{2n+3}-\delta
    _{lk}1\right) -\left( \sum_{h=1}^{r-1}E_{ih}^{2n+3}E_{hj}^{2n+2}-\delta
    _{ij}1\right) E_{rr}^{2n+1}\left(
    \sum_{m=1}^{r-2}E_{lm}^{2n+2}E_{mk}^{2n+3}-\delta _{lk}1\right) + \\
    &&\qquad -E_{ir}^{2n+3}\left(
    \sum_{m=1}^{r-1}\sum_{h=1}^{r-2}E_{hj}^{2n+2}E_{mh}^{2n+1}E_{lm}^{2n+2}%
    \right) E_{rk}^{2n+3}+\left(
    \sum_{h=1}^{r-1}E_{ih}^{2n+3}E_{hj}^{2n+2}-\delta _{ij}1\right) \left(
    \sum_{m=1}^{r-1}E_{mr}^{2n+1}E_{lm}^{2n+2}\right) E_{rk}^{2n+3}+\left(
    1-(-1)^{g^{l+\left( r-1\right) }}\delta _{l\left( r-1\right) }\right)
    E_{ir}^{2n+3}E_{lj}^{2n+2}E_{rk}^{2n+3} \\
    &&\overset{\ref{rel1}}{\rightarrow }E_{ir}^{2n+3}\left(
    \sum_{h=1}^{r-2}E_{hj}^{2n+2}E_{rh}^{2n+1}\right) \left(
    \sum_{m=1}^{r-2}E_{lm}^{2n+2}E_{mk}^{2n+3}-\delta _{lk}1\right)
    +(-1)^{g^{r+j}}\delta _{rj}E_{ir}^{2n+3}E_{l\left( r-1\right)
    }^{2n+2}E_{\left( r-1\right) k}^{2n+3}-\left(
    \sum_{h=1}^{r-1}E_{ih}^{2n+3}E_{hj}^{2n+2}-\delta _{ij}1\right)
    E_{rr}^{2n+1}\left( \sum_{m=1}^{r-2}E_{lm}^{2n+2}E_{mk}^{2n+3}-\delta
    _{lk}1\right) + \\
    &&\qquad -E_{ir}^{2n+3}\left(
    \sum_{m=1}^{r-1}\sum_{h=1}^{r-2}E_{hj}^{2n+2}E_{mh}^{2n+1}E_{lm}^{2n+2}%
    \right) E_{rk}^{2n+3}+\left(
    \sum_{h=1}^{r-1}E_{ih}^{2n+3}E_{hj}^{2n+2}-\delta _{ij}1\right) \left(
    \sum_{m=1}^{r-1}E_{mr}^{2n+1}E_{lm}^{2n+2}\right) E_{rk}^{2n+3}+\left(
    1-(-1)^{g^{l+\left( r-1\right) }}\delta _{l\left( r-1\right) }\right)
    E_{ir}^{2n+3}E_{lj}^{2n+2}E_{rk}^{2n+3}
    \end{eqnarray*}%
    and they are equal. Therefore, the ambiguity is resolvable.
    \end{itemize}%
\end{invisible}%

\section{One-sided Hopf monoids, Hopf modules and Frobenius functors}\label{sec:HopfFrobenius}

In this section we prove our main result which tightens the connection between Hopf and Frobenius properties. Namely, we characterise the existence of certain one-sided antipodes on a bimonoid \(B\) in terms of being Frobenius for the free right Hopf module functor \(- \ot B \colon \cM \to \cM^B_B\).

Throughout this section, let \((\cM,\ot,\I,\brd)\) be a braided monoidal category which admits the coequalisers of reflexive pairs (pairs of parallel arrows admitting a common section) and the equalisers of coreflexive pairs (pairs of parallel arrows admitting a common retraction).
Let \(\left( B,m,u,\Delta ,\varepsilon \right) \) be a bimonoid in \(\cM\). Many of the notions and constructions that we are going to use in this section already appeared in \cite{Lyubashenko}, even if over an abelian braided monoidal category.

\medskip

If \((M,\mu_M,\delta_M)\) is a right \(B\)-Hopf module, then we have a coreflexive pair
\[
\xymatrix @C=70pt {
M \ar@<+1.5ex>[r]^-{\delta_M} \ar@<-1.5ex>[r]_-{(M \ot u)\crho_M^{-1}} & M \ot B. \ar@{.>}[l]|-{\crho_M(M \ot \varepsilon)}
}
\]
This admits an equaliser in \(\cM\), that we denote by \(\coinv{M}{B}\) and we call the \emph{space of coinvariants}. That is, we have the following equaliser diagram in \(\cM\)
\begin{equation}\label{eq:coinv}
\xymatrix @C=30pt {
\coinv{M}{B} \ar[r]^-{e_M} & M \ar@<+0.7ex>[rr]^-{\delta_M} \ar@<-0.7ex>[rr]_-{(M \ot u)\crho_M^{-1}} & & M \ot B
}
\end{equation}
for every Hopf module \(M\). Similarly, we have the reflexive pair
\[
\xymatrix @C=70pt {
M \ot B \ar@<+1.5ex>[r]^-{\mu_M} \ar@<-1.5ex>[r]_-{\crho_M(M \ot \varepsilon)} & M. \ar@{.>}[l]|-{(M \ot u)\crho_M^{-1}}
}
\]
This admits a coequaliser in \(\cM\), that we denote by \(\inv{M}{B}\) and we call the \emph{space of invariants}. That is, we have the following coequaliser diagram in \(\cM\)
\begin{equation}\label{eq:coeq}
    \xymatrix @C=30pt {
        M \ot B \ar@<+0.7ex>[rr]^-{\mu_M} \ar@<-0.7ex>[rr]_-{\crho_M(M \ot \varepsilon)} & & M \ar[r]^-{q_M} & \inv{M}{B}
    }
\end{equation}
for every Hopf module \(M\).

\begin{remark}\label{rmk:inv-coinv} Taking the invariants, resp.\@ the coinvariants, of a Hopf module defines functors \(\cM^B_B\to \cM\). On morphisms, these functors behave as follows. 

For every morphism \(f\colon M\to N\) in \(\cM^B_B\), we have 
\[q_{N}  f  \mu_M=q_{N}  \mu_N  (f\ot B)=q_{N}  \crho_N (N\ot \varepsilon)  (f\ot B)=q_{N} \crho_N(f\ot \I )  (M\ot \varepsilon)= q_N  f \crho_M(M\ot \varepsilon).\] 
By the universal property of \((\inv{M}{B}, q_M)\), there exists a unique morphism \(\overline{f}^B\colon \inv{M}{B}\to \inv{N}{B}\) such that \(\overline{f}^B  q_M= q_N   f\). 

Similarly, for every morphism \(f\colon M\to N\) in \(\cM^B_B\), we have 
\[\delta_N  f  e_M=(f\ot B) \delta_M  e_M=(f\ot B)  (M\ot u)\crho^{-1}_M  e_M=(N\ot u)  (f\ot \I)\crho^{-1}_M  e_M=(N\ot u)\crho^{-1}_N  f  e_M.\] 
By the universal property of \((N^{\mathrm{co}B}, e_N)\), there exists a unique morphism \(f^{\mathrm{co}B}\colon M^{\mathrm{co}B}\to N^{\mathrm{co}B} \) such that \(f  e_M=e_N  f^{\mathrm{co}B}\). 
\end{remark}

The following result is a useful generalisation of the linear case.

\begin{proposition}\label{prop:ff}
    Let \(B\) be a bimonoid in \(\cM\). We have an adjoint triple
    \begin{equation}\label{eq:FundTriple}
    \begin{gathered}
    \xymatrix @R=30pt{
    \cM_{B}^{B} \ar@/_4ex/@<-0.3ex>[d]_-{\inv{(-)}{B}} \ar@/^4ex/@<+0.3ex>[d]^-{\coinv{(-)}{B}} \\ 
    \cM \ar[u]|-{-\ot B}
    }
    \end{gathered}
    \end{equation}
    where the functor \(-\ot B\colon \mathcal{M}\to \mathcal{M}^B_B\) is fully faithful.    
\end{proposition}

\begin{proof}
We first prove the right-hand side adjunction, by constructing its unit and counit. To begin with, suppose that \(X\) is an object in \(\cM\) and let us show that \((X,(X \ot u)   \crho_{X}^{-1})\) satisfies the conditions to be the equaliser \(\coinv{(X \ot B)}{B}\). 
Since 
\begin{align*}
    (X\ot\Delta)(X\ot u)\crho^{-1}_X & =(X\ot  u\ot u)(X\ot \crho_{\I}^{-1})\crho_X^{-1} =(X\ot B\ot u)(X\ot u\ot \I)\crho_{X\ot \I}^{-1}\crho^{-1}_X \\
    & =(X\ot B\ot u)\crho^{-1}_{X\ot B}(X\ot u)\crho^{-1}_X,
\end{align*} 
we have that \((X \ot u)   \crho_{X}^{-1}\) equalizes \(X \ot \Delta\) and \((X \ot B \ot u)   \crho_{X \ot B}^{-1}\). For any morphism \(f \colon Y \to X \ot B\) in \(\cM\) such that 
\[(X \ot \Delta)   f = (X \ot B \ot u)   \crho_{X \ot B}^{-1}   f,\] 
we have that
\begin{align*}
f & = (\crho_X \ot B) \, (X \ot \varepsilon \ot B) \, (X \ot \Delta) \, f = (\crho_X \ot B) \, (X \ot \varepsilon \ot B) \, (X \ot B \ot u) \, \crho_{X \ot B}^{-1} \, f \\
& = (\crho_X \ot B) \, (X \ot \I \ot u) \, (X \ot \varepsilon \ot \I) \, \crho_{X \ot B}^{-1} \, f = (\crho_X \ot B) \, (X \ot \I \ot u) \, \crho_{X \ot \I}^{-1} \, (X \ot \varepsilon) \, f \\
& = (X \ot u) \, \crho_{X}^{-1} \, \crho_X \, (X \ot \varepsilon) \, f
\end{align*}
i.e.,
\[
\xymatrix @C=30pt {
X \ar[rr]^{(X \ot u)\crho_{X}^{-1}} & & X \ot B \ar@<+0.7ex>[rr]^-{X \ot \Delta} \ar@<-0.7ex>[rr]_-{(X \ot B \ot u)\crho_{X \ot B}^{-1}} & & X \ot B \ot B \\
& Y \ar[ur]_-{f} \ar@{.>}[ul]^-{\crho_X (X \ot \varepsilon) f} & & &
}
\]
which means that \((X,(X \ot u)   \crho_{X}^{-1})\) is, in fact, the equaliser of \(X \ot \Delta\) and \((X \ot B \ot u)   \crho_{X \ot B}^{-1}\). 

As a consequence, we have that there exists an isomorphism \(\eta_X \colon X \to \coinv{(X \ot B)}{B}\), natural in \(X\), which realises \(X\) as the space of coinvariants of \(X \ot B\), and which satisfies
\begin{equation}\label{eq:eta}
    e_{X \ot B}   \eta_X = (X \ot u)   \crho_{X}^{-1}
\end{equation}
and whose inverse is
\begin{equation}\label{eq:etainv}
    \eta_X^{-1} = \crho_X   (X \ot \varepsilon)   e_{X \ot B}.
\end{equation}
As for the counit of the right-hand side adjunction, the morphism
\begin{equation}\label{eq:epsilon}
    \epsilon_M \colon \left(\coinv{M}{B} \ot B \xrightarrow{e_M \ot B} M \ot B \xrightarrow{\mu_M} M\right)
\end{equation}
for \(M\) varying in \(\cM^B_B\), defines a natural transformation such that
\[
    \begin{gathered}
        \xymatrix @!0 @C=100pt @R=50pt {
            X \ot B \ar[d]_-{\crho_X^{-1} \ot B} \ar[rr]^-{\eta_{X} \ot B} & & \coinv{(X\ot B)}{B} \ot B \ar[dl]|-{e_{X \ot B} \ot B} \ar@/^3ex/[ddl]^{\epsilon_{X \ot B}} \ar@{}[ddl]|-{\color{gray}\eqref{eq:epsilon}} \\
            X \ot \I \ot B \ar@{}[urr]^(0.4){\color{gray}\eqref{eq:eta}} \ar[r]^-{X \ot u \ot B} \ar[dr]_-{X \ot \clambda_B} & X \ot B \ot B \ar@{}[dl]|(0.3){\color{gray}\text{unitality}} \ar[d]|-{X \ot m} \\
            & X \ot B
        }
    \end{gathered}
\]
commutes and
\[
\begin{gathered}
\xymatrix @!0 @C=120pt @R=70pt {
\coinv{M}{B} \ar@{}[dr]|-{\color{gray}\text{naturality}} \ar[d]_-{e_M} \ar[r]^-{\eta_{\coinv{M}{B}}} & \coinv{(\coinv{M}{B} \ot B)}{B} \ar[d]|-{\coinv{(e_{M} \ot B)}{B}} \ar[dr]^-{\coinv{\epsilon_{M}}{B}} & \\
M \ar@{}[dr]|-{\color{gray}\eqref{eq:eta}} \ar[r]^-{\eta_M} \ar[d]_-{\crho_M^{-1}} & \coinv{(M \ot B)}{B} \ar@{}[ur]|(0.3){\color{gray}\eqref{eq:epsilon}} \ar@{}[dr]|-{\color{gray}\text{naturality}} \ar[r]^-{\coinv{\mu_M}{B}} \ar[d]^-{e_{M \ot B}} & \coinv{M}{B} \ar[d]^-{e_M} \\
M \ot \I \ar[r]_-{M \ot u} & M \ot B \ar[r]_-{\mu_M} & M
}
\end{gathered}
\]
commutes, i.e., \(e_M   \coinv{\epsilon_M}{B}   \eta_{\coinv{M}{B}} = e_M\). Since the equaliser morphism \(e_M\) is a monomorphism, we get the first of the two needed identities: \(\coinv{\epsilon_M}{B}   \eta_{\coinv{M}{B}}=\id_{\coinv{M}{B}}\). 

We now discuss the left-hand side adjunction. As before, let us show that \((X,\crho_X   (X \ot \varepsilon))\) satisfies the conditions to the be the coequaliser \(\inv{(X \ot B)}{B}\). For every object \(X\) in \(\cM\), \(\crho_X   (X \ot \varepsilon)\) coequalises \(X \ot m\) and \(\crho_{X \ot B}   (X \ot B \ot \varepsilon)\) and for any \(f \colon X \ot B \to Y\) such that
\[f   (X \ot m) = f   \crho_{X \ot B}   (X \ot B \ot \varepsilon),\]
we have that
\begin{align*}
f & = f \, (X \ot m) \, (X \ot u \ot B) \, (\crho_X^{-1} \ot B) = f \, \crho_{X \ot B} \, (X \ot B \ot \varepsilon)  \, (X \ot u \ot B) \, (\crho_X^{-1} \ot B) \\
& = f \, (X \ot u) \, \crho_X^{-1} \, \crho_{X} \, (X \ot \varepsilon)
\end{align*}
i.e.,
\[
\xymatrix @C=30pt {
X \ot B \ot B \ar@<+0.7ex>[rr]^-{X \ot m} \ar@<-0.7ex>[rr]_-{\crho_{X \ot B}(X \ot B \ot \varepsilon)} & & X \ot B \ar[rr]^-{\crho_X(X \ot \varepsilon)} \ar[dr]_-{f} & & X \ar@{.>}[dl]^-{ f (X \ot u) \crho_X^{-1}} \\
 & & & Y &
}
\]
which means that, in fact, \((X,\crho_X   (X \ot \varepsilon))\) is the coequaliser of \(X \ot m\) and \(\crho_{X \ot B}   (X \ot B \ot \varepsilon)\). As a consequence, we have that there exists an isomorphism \(\theta_{X} \colon \inv{(X \ot B)}{B} \to X\), natural in \(X\), which realises \(X\) as the space of invariants of \(X \ot B\) and which satisfies 
\begin{equation}\label{eq:theta}
    \theta_X   q_{X \ot B} = \crho_X   (X \ot \varepsilon). 
\end{equation}
It follows that \(\theta^{-1}_X   \crho_X   (X\ot\varepsilon)=q_{X\ot B}\). This \(\theta\) is the counit of the left-hand side adjunction.

The unit of the adjunction is given by
\begin{equation}\label{eq:gamma}
    \gamma_M \colon \left( M \xrightarrow{\delta_M} M \ot B \xrightarrow{q_M \ot B} \inv{M}{B} \ot B \right)
\end{equation}
which, for \(M\) varying in \(\cM^B_B\), defines a natural transformation such that
\vspace{5pt}

\[
    \xymatrix @=50pt {
        X \ot B \ar@{}[ddr]|(0.625){\color{gray}\text{counitality}\quad } \ar[rr]^{\gamma_{X \ot B}} \ar[dr]|-{X \ot \Delta} \ar@/_5ex/[ddr]_-{X \ot \clambda_B^{-1}} & & \inv{X \ot B}{B} \ot B \ar[d]^-{\theta_X \ot B} \\
        & X \ot B \ot B \ar@{}[r]|-{\color{gray}\eqref{eq:theta}} \ar@{}[u]|-{\color{gray}\eqref{eq:gamma}} \ar[ur]|-{q_{X \ot B} \ot B} \ar[d]|-{X \ot \varepsilon \ot B} & X \ot B \\
         & X \ot \I \ot B \ar[ur]_-{\crho_X \ot B} & 
    }
\]
\vspace{5pt}

commutes, i.e., \(\left(\theta_X \ot B\right)   \gamma_{X \ot B} = \id_{X \ot B}\). Since

\vspace{5pt}
\[
    \xymatrix @=50pt {
        M \ar[r]^-{q_M} \ar[dr]_-{\delta_M} & \inv{M}{B} \ar@{}[d]|-{\color{gray}\text{naturality}} \ar[rr]^-{\inv{\gamma_M}{B}} \ar[dr]^-{\inv{\delta_M}{B}} & & \inv{\inv{M}{B} \ot B}{B} \ar[d]^-{\theta_{\inv{M}{B}}} \\
         & M \ot B \ar@{}[dr]|-{\color{gray}\eqref{eq:theta}} \ar[r]_-{q_{M \ot B}} \ar[d]_-{M \ot \varepsilon} & \inv{M \ot B}{B} \ar@{}[r]|-{\color{gray}\text{naturality}} \ar@{}[u]|-{\color{gray}\eqref{eq:gamma}} \ar[ur]|-{\inv{q_M \ot B}{B}} \ar[d]_-{\theta_M} & \inv{M}{B} \\
          & M \ot \I \ar[r]_-{\crho_M} & M \ar[ur]_-{q_M} & 
    }
\]

\vspace{5pt}

commutes as well, i.e., \(\theta_{\inv{M}{B}}   \inv{\gamma_M}{B}   q_M = q_M\), and since \(q_M\) is an epimorphism, we obtain the second needed identity: \(\theta_{\inv{M}{B}}   \inv{\gamma_M}{B} = \id_{\inv{M}{B}}\). \end{proof}

In this setting, the natural transformation \(\sigma\colon (-)^{\mathrm{co}B}\to \inv{(-)}{B}\) (see \eqref{def:sigma}) is given by
\begin{equation}\label{eq:sigma}
    \sigma _{M}\colon M^{\mathrm{co}B}\xrightarrow{\theta_{\coinv{M}{B}}^{-1}} \inv{\coinv{M}{B} \ot B}{B} \xrightarrow{\inv{\epsilon_M}{B}} \inv{M}{B}.
\end{equation}
We want to investigate what can be said if this is a natural isomorphism, that is to say, we are interested in characterizing when the functor \(-\ot B\) is a Frobenius functor, extending \cite[Theorem 3.7]{Saracco-Frobenius}. 
Recall also the following result.

\vspace{8pt}

\begin{remark}  \label{rmk:sigma-splitmono-iso} 
Let \(\mathcal{L}\dashv \mathcal{F}\dashv \mathcal{R}\) be an adjoint triple with \(\mathcal{F}\) fully faithful, such as \eqref{eq:FundTriple} from \zcref{prop:ff}. By \cite[Proposition 2.26 and Theorem 2.24]{AB22} it follows that \(\mathcal{R}\) is Frobenius if and only if \(\sigma\) as in \eqref{def:sigma} is a split-mono, if and only if \(\sigma\) is a natural isomorphism.
\end{remark}

\vspace{8pt}

\begin{theorem}
\label{prop:Frobenius-OneSidedHopf}
   Let \((\cM,\ot,\I,\brd)\) be a braided monoidal category which admits the coequalisers of reflexive pairs and the equalisers of coreflexive pairs. Let \((B,m, u, \Delta,\varepsilon)\) be a bimonoid in \(\cM\) for which the endofunctor \(B \ot -\) preserves reflexive coequalisers. 
   Then, the following assertions are equivalent:
   \begin{enumerate}[label=\((\roman*)\)]
       \item\label{item:Frobenius-OneSidedHopf1} \(B\) is a right Hopf monoid in \(\cM\) whose right antipode \(S\) is anti-multiplicative and anti-comultiplicative; 
       \item\label{item:Frobenius-OneSidedHopf2} \(\sigma\) is a split-mono;
       \item\label{item:Frobenius-OneSidedHopf3} \(\sigma\) is a natural isomorphism;
       \item\label{item:Frobenius-OneSidedHopf3bis} the free Hopf module functor \(- \ot B \colon \cM \to \cM^B_B\) from \eqref{eq:FundTriple} is Frobenius;
       \item\label{item:Frobenius-OneSidedHopf4} the component \(\sigma_{\rmod{B} \ot \hopfmod{B}}\) is an isomorphism;
       \item\label{item:Frobenius-OneSidedHopf5} the following canonical morphism $i_B$ is an isomorphism:
       \[i_B \coloneqq  \left(B \xrightarrow{\eta_B} \coinv{(\rmod{B} \ot \hopfmod{B})}{B} \xrightarrow{\sigma_{\rmod{B} \ot \hopfmod{B}}} \inv{\rmod{B} \ot \hopfmod{B}}{B}\right).\]
   \end{enumerate}
\end{theorem}
In order to prove \zcref{prop:Frobenius-OneSidedHopf}, we need a preliminary result.

\begin{lemma}\label{lem:easysigma}
    Let \((\cM,\ot,\I,\brd)\) be a braided monoidal category which admits the coequalisers of reflexive pairs and the equalisers of coreflexive pairs. Let \((B,m, u, \Delta,\varepsilon)\) be a bimonoid in \(\cM\) and let \((M,\mu_M,\delta_M)\) be a right Hopf module over \(B\). Then 
\begin{equation}\label{eq:easysigma}    \sigma_M = q_M   e_M.
    \end{equation}\end{lemma}

\begin{proof}
    Let us show that \(q_M   e_M   \theta_{\coinv{M}{B}} = \inv{\epsilon_M}{B}\): in view of the definition \eqref{eq:sigma} of \(\sigma\), this implies \eqref{eq:easysigma}. To this aim, consider the following computation
    \begin{align*}
        q_M  \, e_M  \, \theta_{\coinv{M}{B}}  \, q_{\coinv{M}{B} \ot B}  \overset{\eqref{eq:theta}}&{=} q_M  \, e_M  \, \crho_{\coinv{M}{B}}  \, (\coinv{M}{B} \ot \varepsilon) = q_M  \, \crho_{M}  \, (e_M \ot \I)  \, (\coinv{M}{B} \ot \varepsilon) \\
        & = q_M  \, \crho_{M}  \, (M \ot \varepsilon)  \, (e_M \ot B) \overset{\eqref{eq:coeq}}{=} q_M  \, \mu_M  \, (e_M \ot B) \\
        \overset{\eqref{eq:epsilon}}&{=} q_M  \, \epsilon_M = \inv{\epsilon_M}{B}  \, q_{\coinv{M}{B} \ot B}
    \end{align*}
    where in the last equality we used the fact that \(\epsilon_M \colon \coinv{M}{B} \ot \hopfmod{B} \to \hopfmod{M}\) is a morphism of Hopf modules, in view of \zcref{rmk:inv-coinv}. Since \(q_{\coinv{M}{B} \ot B}\) is an epimorphism in \(\cM\), \(q_M   e_M   \theta_{\coinv{M}{B}} = \inv{\epsilon_M}{B}\) holds and the proof is over.
\end{proof}

\begin{proof}[Proof of \zcref{prop:Frobenius-OneSidedHopf}.]

\ref{item:Frobenius-OneSidedHopf1} \(\Rightarrow\) \ref{item:Frobenius-OneSidedHopf2}. Suppose that \(B\) admits a right antipode \(S\) which is a bimonoid anti-homomorphism. We show that \(\sigma\) is a split-mono. 

First of all, for every right \(B\)-Hopf module \((M, \mu_M, \delta_M)\), we show that the projection \(\mu_M(M\ot S)\delta_M\) factors through \(\coinv{M}{B}\). The commutativity of
\[
    \begin{tikzcd}[column sep=1.25cm,row sep=1cm]
    \arrow["{\begin{array}{c} \color{gray}\mathsf{anti}\mbox{-} \\ \color{gray}\mathsf{comultiplicativity} \end{array}}"{description,xshift=5pt, yshift=25pt}, draw=none, from=1-3, to=3-4]
    	M && {M \ot B} & {M\ot B} & M \\
    	{M \ot B} & {M \ot B \ot B} & {M\ot B\ot B\ot B} \\
    	&& {M\ot B\ot B\ot B} & {M\ot B\ot B\ot B} \\
    	{M\ot B} & {M \ot \I \ot B} & {M\ot B\ot B} & {M\ot B\ot B\ot B} \\
    	& {M \ot B \ot \I} && {M\ot B\ot B} \\
    	M & {M\ot \I} &&& {M\ot B}
    	\arrow["{\delta_M}", from=1-1, to=1-3]
    	\arrow["{\delta_M}"{description}, from=1-1, to=2-1]
    	\arrow["{M\ot S}", from=1-3, to=1-4]
    	\arrow["{\delta_M\ot\Delta}"{description}, from=1-3, to=2-3]
    	\arrow["{\mu_M}", from=1-4, to=1-5]
    	\arrow["{\delta_M\ot\Delta}"{description}, from=1-4, to=3-4]
    	\arrow["{\delta_M}"{description}, from=1-5, to=6-5]
    	\arrow["{M \ot \Delta}", from=2-1, to=2-2]
    	\arrow["{M \ot S}"{description}, from=2-1, to=4-1]
    	\arrow["{M \ot \clambda_B^{-1}}"{description}, from=2-1, to=4-2]
    	\arrow["{M \ot \Delta \ot B}", from=2-2, to=2-3]
    	\arrow["{M \ot \varepsilon \ot B}"{description}, from=2-2, to=4-2]
    	\arrow["{M\ot B\ot S \ot S}"{description}, from=2-3, to=3-3]
    	\arrow["{M \ot B \ot\brd_{B,B}}"{yshift=3pt}, from=3-3, to=3-4]
    	\arrow["{M \ot m \ot B}"{description}, from=3-3, to=4-3]
    	\arrow["{M \ot\brd_{B,B} \ot B}"{description}, from=3-4, to=4-4]
    	\arrow["{\mu_M}"{description}, from=4-1, to=6-1]
    	\arrow["{M \ot u \ot  S}", from=4-2, to=4-3]
    	\arrow["{M \ot \brd_{\I,B}}"{description}, from=4-2, to=5-2]
    	\arrow["{M \ot \crho_B^{-1}}"{description}, curve={height=13pt}, from=2-1, to=5-2]
    	\arrow["{M \ot S \ot u}"', from=5-2, to=5-4]
    	\arrow["{M\ot\brd_{B,B}}"{description}, from=4-3, to=5-4]
    	\arrow["{M\ot B \ot m}"{description}, from=4-4, to=5-4]
    	\arrow["{\mu_M\ot B}"{description}, from=5-4, to=6-5]
    	\arrow["{\crho_M^{-1}}"', from=6-1, to=6-2]
    	\arrow["{M \ot u}"', from=6-2, to=6-5]
        \arrow["{\color{gray}\mathsf{coassociativity}}"{description}, draw=none, from=1-1, to=2-3]
        \arrow["{\begin{array}{c} \color{gray}\mathsf{Hopf} \\ \color{gray}\mathsf{module} \end{array}}"{description,xshift=5pt}, draw=none, from=1-4, to=6-5]
        \arrow["{\color{gray}\mathsf{antipode}}"{description,xshift=-5pt}, draw=none, from=2-2, to=4-3]
        \arrow["{\color{gray}\brd \, \mathsf{natural}}"{description, pos=0.4,xshift=10pt}, draw=none, from=3-3, to=5-4]
        \arrow["{\color{gray}\brd \, \mathsf{natural}}"{yshift=5pt}, draw=none, from=4-3, to=5-2]
        \arrow["{\color{gray}\mathsf{counitality}}"{xshift=-10pt,yshift=30pt}, draw=none, from=2-2, to=4-1]
        \arrow["{\color{gray}\mathsf{naturality \, + \, functoriality}}"{yshift=-5pt}, draw=none, from=6-1, to=5-4]
        \arrow["{\color{gray}\eqref{eq:braidI}}"{description,yshift=6pt}, draw=none, from=2-1, to=5-2]
    \end{tikzcd}
\]
shows that \(\delta_M\mu_M(M\ot S)\delta_M = (M\ot u)\crho^{-1}_M\mu_M(M\ot S)\delta_M\) and
so, by the universal property of the equaliser \((M^{\mathrm{co}B}, e_M)\), there exists a unique morphism \(\zeta_M\colon M\to M^{\mathrm{co}B}\) such that 
\begin{equation}\label{eq:defZeta}
    e_M\zeta_M=\mu_M (M\ot S)\delta_M. 
\end{equation}

We now show that \(\zeta_M\) factors through \(\inv{M}{B}\), i.e.\@ that \(\zeta_M\mu_M=\zeta_M \crho_M(M\ot \varepsilon)\). We have
\begin{align*}
&e_M\zeta_M\mu_M \overset{\eqref{eq:defZeta}}{=} \mu_M(M\ot S)\delta_M\mu_M \\
\overset{\eqref{eq:HopfMod}}&{=} \mu_M(M\ot S)(\mu_M\ot B)(M\ot B\ot m)(M\ot \brd_{B,B}\ot B)(\delta_M\ot\Delta) \\
& = \mu_M(\mu_M\ot B)(M\ot B\ot Sm)(M\ot \brd_{B,B}\ot B)(\delta_M\ot\Delta) \\
\overset{\eqref{eq:antibimon}}&{=} \mu_M(\mu_M\ot B)(M\ot B\ot m)(M\ot B\ot S\ot S)(M\ot B\ot \brd_{B,B})(M\ot \brd_{B,B}\ot B)(\delta_M\ot\Delta) \\
\overset{\eqref{eq:hexagon}}&{=} \mu_M(M\ot m)(M\ot B\ot m)(M\ot B\ot S\ot S)(M\ot \brd_{B,B\ot B})(M\ot B\ot \Delta)(\delta_M\ot B)\\
& = \mu_M(M\ot m)(M\ot m\ot B)(M\ot B\ot S\ot S)(M\ot  \Delta\ot B)(M\ot \brd_{B,B})(\delta_M\ot B)\\
\overset{\eqref{eq:right-antipode}}&{=} \mu_M(M\ot m)(M\ot u\ot B)(M\ot \varepsilon\ot B)(M\ot  B\ot S)(M\ot \brd_{B,B})(\delta_M\ot B)\\
& = \mu_M(M\ot \clambda_B)(M\ot \I\ot S)(M\ot \varepsilon\ot B)(M\ot \brd_{B,B})(\delta_M\ot B)\\
& = \mu_M(M\ot \clambda_B)(M\ot \I\ot S)(M\ot \brd_{\I,B})(M\ot B\ot \varepsilon)(\delta_M\ot B)\\
\overset{\eqref{eq:braidI}}&{=}\mu_M(M\ot \clambda_B)(M\ot \I\ot S)(M\ot \clambda^{-1}_{B})(M\ot \crho_{B})(\delta_M\ot \I)(M\ot \varepsilon)\\
& = \mu_M(M\ot \clambda_B)(M\ot \clambda^{-1}_B)(M\ot S)(M\ot \crho_{B})(\delta_M\ot \I)(M\ot \varepsilon)\\
& = \mu_M(M\ot S)\delta_M\crho_M(M\ot \varepsilon) \overset{\eqref{eq:defZeta}}{=}e_M\zeta_M \crho_M(M\ot \varepsilon).
\end{align*}
Since \(e_M\) is a monomorphism, we get
\begin{equation}\label{eq:zeta-mi}
    \zeta_M\mu_M=\zeta_M \crho_M(M\ot\varepsilon).
\end{equation}
Thus, by the universal property of the coequaliser \((\inv{M}{B}, q_M)\), there exists a unique morphism \(\rho_M\colon \inv{M}{B}\to M^{\mathrm{co}B}\) such that
\begin{equation}\label{eq:def-gamma}
    \rho_M q_M=\zeta_M.
\end{equation}
It remains to check that \(\rho\coloneqq  (\rho_M)_{M\in \cM^B_B }\) is a left inverse of \(\sigma\). Note that from \(m(\Id\ot S)\Delta=u\varepsilon\) it follows that \(Su=u\). In fact, by composing both sides on the right by \(u\), we get \(m(\Id\ot S)\Delta u=u\varepsilon u=u\), and the left hand side is \(m(\Id\ot S)\Delta u=m(\Id\ot S)(u\ot u)\clambda_\I^{-1}= m(u\ot B)(\I\ot S)(\I\ot u)\clambda^{-1}_{\I}=\clambda_B(\I\ot S)(\I\ot u)\clambda^{-1}_{\I}=S\clambda_B(\I\ot u)\clambda^{-1}_{\I}=S u\clambda_\I\clambda^{-1}_{\I}=Su\), so 
\begin{equation}\label{eq:Sunital}
Su=u.
\end{equation}
Then, for every \(M\in \cM^B_B\), we have 
\begin{align*}
    & e_M \rho_M \sigma_M \crho_{M^{\mathrm{co}B}}(M^{\mathrm{co}B}\ot \varepsilon)\overset{\eqref{eq:sigma}}{=} e_M \rho_M  \overline{\epsilon_{M}}^B\theta^{-1}_{M^{\mathrm{co}B}} \crho_{M^{\mathrm{co}B}}(M^{\mathrm{co}B}\ot \varepsilon) \overset{\eqref{eq:theta}}{=} e_M\rho_M\overline{\epsilon_{M}}^B q_{M^{\mathrm{co}B}\ot B} \\
    & = e_M\rho_M q_M \epsilon_M \overset{\eqref{eq:def-gamma}}{=} e_M \zeta_M \epsilon_M \overset{\eqref{eq:epsilon}}{=} e_M \zeta_M \mu_M (e_M\ot B) \overset{\eqref{eq:zeta-mi}}{=}e_M \zeta_M \crho_M (M\ot \varepsilon)(e_M\ot B) \\
    &= e_M \zeta_M e_M \crho_{M^{\mathrm{co}B}} (M^{\mathrm{co}B}\ot \varepsilon) \overset{\eqref{eq:defZeta}}{=}\mu_M(M\ot S)\delta_M e_M \crho_{M^{\mathrm{co}B}} (M^{\mathrm{co}B}\ot \varepsilon) \\
    \overset{\eqref{eq:coinv}}&{=} \mu_M (M\ot S)(M\ot u) \crho_M^{-1}e_M \crho_{M^{\mathrm{co}B}} (M^{\mathrm{co}B}\ot \varepsilon) \overset{\eqref{eq:Sunital}}{=} \mu_M (M\ot u) \crho_M^{-1}e_M \crho_{M^{\mathrm{co}B}} (M^{\mathrm{co}B}\ot \varepsilon)\\
    & = e_M \crho_{M^{\mathrm{co}B}} (M^{\mathrm{co}B}\ot \varepsilon),
\end{align*}
from where the equality \(\rho_M \sigma_M = \Id_{M^{\mathrm{co}B}}\) follows by cancelling the monomorphism \(e_M\) on the left and the epimorphism \(\crho_{M^{\mathrm{co}B}} (M^{\mathrm{co}B}\ot \varepsilon)\) on the right. Thus, \(\sigma_M\) is a split-mono.

\ref{item:Frobenius-OneSidedHopf2} \(\Leftrightarrow\) \ref{item:Frobenius-OneSidedHopf3}. It follows from \zcref{rmk:sigma-splitmono-iso}.

\ref{item:Frobenius-OneSidedHopf3} \(\Leftrightarrow\) \ref{item:Frobenius-OneSidedHopf3bis}. This is \zcref{lemma:frobenius}.

\ref{item:Frobenius-OneSidedHopf3} \(\Rightarrow\) \ref{item:Frobenius-OneSidedHopf4}. Evident.

\ref{item:Frobenius-OneSidedHopf4} \(\Leftrightarrow\) \ref{item:Frobenius-OneSidedHopf5}. It is obvious since \(\eta\) is a natural isomorphism.

\ref{item:Frobenius-OneSidedHopf4} \(\Rightarrow\) \ref{item:Frobenius-OneSidedHopf1}. 
Consider the Hopf module \(\hat{B} \coloneqq  \rmod{B} \ot \hopfmod{B}\) and 
\[\sigma_{\hat{B}}^{-1} \colon \inv{(\rmod{B} \ot \hopfmod{B})}{B} \to \coinv{(\rmod{B} \ot \hopfmod{B})}{B}.\]
Define
\begin{equation}\label{eq:defS}
    S \coloneqq  \left(B \xrightarrow{\clambda^{-1}_B} \I \ot B \xrightarrow{u \ot B} B \ot B \xrightarrow{q_{\hat B}} \inv{(B \ot B)}{B} \xrightarrow{\sigma_{\hat{B}}^{-1}} \coinv{(B \ot B)}{B} \xrightarrow{\eta_B^{-1}} B\right).
\end{equation}
We claim that this \(S\) is a bimonoid anti-homomorphism and a right antipode for \(B\).

First of all, we show that \(S   u = u\), as we need it to prove that \(m   (B \ot S)   \Delta = u   \varepsilon\). To this aim, observe that \((u \ot B)   \clambda_B^{-1}   u \colon \I \to B \ot B\)
equalises \(B \ot \Delta\) and \((B \ot B \ot u)   \crho_{B \ot B}^{-1}\). Hence there exists a unique \(v \colon \I \to \coinv{(B \ot B)}{B}\) such that
\begin{equation}\label{eq:def-v}
    e_{\hat B}   v = (u \ot B)   \clambda_B^{-1}   u.
\end{equation}
Therefore,
\begin{align*}
    S \, u \overset{\eqref{eq:defS}}&{=} \eta_B^{-1} \, \sigma_{\hat B}^{-1} \, q_{\hat B} \, (u \ot B) \, \clambda_{B}^{-1} \, u 
    \overset{\eqref{eq:def-v}}{=} \eta_B^{-1} \, \sigma_{\hat B}^{-1} \, q_{\hat B} \, e_{\hat B} \, v  \overset{\eqref{eq:easysigma}}{=} \eta_B^{-1} \, v 
    \overset{\eqref{eq:etainv}}{=} \crho_B \, (B \ot \varepsilon) \, e_{\hat B} \, v \\
    \overset{\eqref{eq:def-v}}&{=} \crho_B \, (B \ot \varepsilon) \, (u \ot B) \, \clambda_B^{-1} \, u = u \, \crho_\I \, \clambda_{\I}^{-1} \, \varepsilon \, u = u,
\end{align*}
i.e.
\begin{equation}\label{eq:Sunit}
    S   u = u.
\end{equation}

Secondly, we will need the following  observation. Since
\begin{align*}
    & \mu_{\hat B} (u \ot B \ot B)   \left(\clambda_B^{-1} \ot B\right)   (u \ot B)   \clambda_B^{-1} \\
     \overset{\eqref{eq:diagonal}}&{=} (m \ot m)   (B \ot \brd_{B,B} \ot B)   (B \ot B \ot \Delta)   (u \ot B \ot B)   \left(\clambda_B^{-1} \ot B\right)   (u \ot B)   \clambda_B^{-1} \\
    & = (m \ot m)   (B \ot \brd_{B,B} \ot B)   (u \ot B \ot B \ot B)   \left(\clambda_B^{-1} \ot B \ot B\right)   (u \ot B \ot B)   \left(\clambda_B^{-1} \ot B\right)   \Delta \\
    & = (m \ot m)   (u \ot B \ot B \ot B)   \left(\clambda_B^{-1} \ot B \ot B\right)  (B \ot u \ot B)   (\brd_{\I,B} \ot B)    \left(\clambda_B^{-1} \ot B\right)   \Delta \\
     \overset{\eqref{eq:braidI}}&{=} (B \ot m)   (B \ot u \ot B)   \left(\crho_B^{-1} \ot B\right)   \Delta = \Delta,
\end{align*}%
we see that
\begin{equation}\label{eq:Deltamu}
    \Delta = \mu_{\hat B}   (u \ot B \ot B)   \left(\clambda_B^{-1} \ot B\right)   (u \ot B)   \clambda_B^{-1}.
\end{equation}

Thirdly, we remark that \(\coinv{(\rmod{B} \ot \hopfmod{B})}{B}\) is a left \(B\)-module with respect to the left \(B\)-action induced by
\begin{equation}\label{eq:regularleftBaction}
    B \ot (B \ot B) = (B \ot B) \ot B \xrightarrow{m \ot B} B \ot B.
\end{equation}
This follows from the fact that the forgetful functor \(\Cc^{\mathbb{T}}\to \Cc\) from the category of monoids \(\Cc^{\mathbb{T}}\) over a monad \(\mathbb{T}\) to the underlying category \(\Cc\) induces, on \(\Cc^{\mathbb{T}}\), all the limits that exist in \(\Cc\) (see \cite[Proposition 4.3.1]{Borceux2}). However, it can also be checked directly via the following diagram:
\[
\xymatrix @C=35pt{
    B \ot \coinv{(B \ot B)}{B} \ar[r]^-{B \ot e_{\hat B}} \ar@{.>}[d]_-{\mu_{\coinv{\hat{B}}{B}}} & B \ot B \ot B \ar[d]_-{m \ot B} \ar@<+0.5ex>[rrrr]^-{B \ot B \ot \Delta} \ar@<-0.5ex>[rrrr]_-{(B \ot B \ot B \ot u)\left(B \ot B \ot \crho_B^{-1}\right)} & & & & B \ot B \ot B \ot B \ar[d]^-{m \ot B \ot B}  \\
    \coinv{(B \ot B)}{B} \ar[r]_-{e_{\hat B}} & B \ot B \ar@<+0.5ex>[rrrr]^-{B \ot \Delta} \ar@<-0.5ex>[rrrr]_-{(B \ot B \ot u)\left(B \ot \crho_B^{-1}\right)} & & & & B \ot B \ot B,
}
\]
whose horizontal rows are equalisers. The two right-hand side squares (with the top arrows, resp.\@ the bottom arrows from each row) both commute. From this, it is immediate that \((m\ot B)(B\ot e_{\hat B})\) equalizes \(B \ot \Delta\) and \((B\ot B\ot u)(B\ot \crho^{-1}_{B})\), whence the definition of \(\mu_{\coinv{\hat{B}}{B}}\). The commutativity of the left-hand side square is the left \(B\)-linearity of \(e_{\hat B}\), namely:
\begin{equation}\label{eq:mucoinv}
    e_{\hat B}   \mu_{\coinv{\hat B}{B}} = (m \ot B)   (B \ot e_{\hat B}).
\end{equation}
In addition, it turns out that \(\eta_B \colon B \to \coinv{(\rmod{B} \ot \hopfmod{B})}{B}\) is left \(B\)-linear with respect to this action. In fact, since the right \(B\)-action on \(\rmod{B} \ot \hopfmod{B}\) is not involved in the construction of \(\coinv{(\rmod{B} \ot \hopfmod{B})}{B}\), we can verify directly that
\begin{align*}
    e_{\hat B}  \, \eta_B  \, m \overset{\eqref{eq:eta}}&{=} (B \ot u)  \, \crho_B^{-1}  \, m = (m \ot B)   (B \ot B \ot u)   \crho_{B \ot B}^{-1}  = (m \ot B)  (B \ot B \ot u)  \left( B \ot \crho_{B}^{-1}\right) \\
     \overset{\eqref{eq:eta}}&{=} (m \ot B)   \left(B \ot e_{\hat B}\right)  \left(B \ot \eta_B\right) \overset{\eqref{eq:mucoinv}}{=} e_{\hat B}  \, \mu_{\coinv{\hat B}{B}}   \left(B \ot \eta_B\right),
\end{align*}
from which we deduce that
\begin{equation}\label{eq:etalin}
    \eta_B \, m = \mu_{\coinv{\hat B}{B}} \, \left(B \ot \eta_B\right),
\end{equation}
because \(e_{\hat B}\) is a monomorphism. In particular, \(\eta_B^{-1}\) is left \(B\)-linear as well. 

Under the hypothesis that \(B \ot -\) preserves reflexive coequalisers, \(\inv{(\rmod{B} \ot \hopfmod{B})}{B}\) is a left \(B\)-module with respect to the left \(B\)-action induced by \eqref{eq:regularleftBaction}, too. Again, this can be checked directly via the following commutative diagram with coequaliser rows:
\[
\xymatrix @C=35pt{
     B \ot B \ot B \ot B \ar[d]_-{m \ot B \ot B} \ar@<+0.5ex>[rrrr]^-{B \ot \mu_{B \ot B}} \ar@<-0.5ex>[rrrr]_-{\left(B \ot B \ot \crho_B\right)(B \ot B \ot B \ot \varepsilon)} & & & &  B \ot B \ot B \ar[d]^-{m \ot B} \ar[r]^-{B \ot q_{\hat B}} & B \ot \inv{(B \ot B)}{B} \ar@{.>}[d]^-{\mu_{\inv{\hat{B}}{B}}} \\
    B \ot B \ot B \ar@<+0.5ex>[rrrr]^-{\mu_{B \ot B}} \ar@<-0.5ex>[rrrr]_-{\left(B \ot \crho_B\right)(B \ot B \ot \varepsilon)} & & & &  B \ot B \ar[r]_-{q_{\hat B}} & \inv{(B \ot B)}{B},
}
\]
from which we also see that \(q_{\hat B}\) is left \(B\)-linear with respect to this action, i.e.,
\begin{equation}\label{eq:muinv}
    q_{\hat B}   (m \ot B) = \mu_{\inv{\hat B}{B}}   (B \ot q_{\hat B}).
\end{equation}
In addition, it turns out that \(\sigma_{\hat B}\) is left \(B\)-linear with respect to the action \(\mu_{\coinv{\hat{B}}{B}}\) on the domain and \(\mu_{\inv{\hat{B}}{B}}\) on the codomain, because \(\sigma_{\hat B} = q_{\hat B}   e_{\hat B}\) by \zcref{lem:easysigma}, and \(q_{\hat B}\) and we have seen that \(e_{\hat B}\) are both left \(B\)-linear. Therefore, so is \(\sigma_{\hat B}^{-1}\), i.e.,
\begin{equation}\label{eq:sigmalin}
    \sigma_{\hat B}^{-1}   \mu_{\inv{\hat B}{B}} = \mu_{\coinv{\hat B}{B}}   \left(B \ot \sigma_{\hat B}^{-1}\right).
\end{equation}
We can also make use of these actions to express \(\sigma_{\hat B}^{-1}\) in terms of \(S\):
\begin{align*}
    \sigma_{\hat B}^{-1}    q_{\hat B} & = \sigma_{\hat B}^{-1}    q_{\hat B}    (m \ot B)    (B \ot u \ot B)    (\crho_B^{-1} \ot B) \overset{\eqref{eq:muinv}}{=} \sigma_{\hat B}^{-1}    \mu_{\inv{\hat B}{B}}    \left(B \ot q_{\hat B}\right)    (B \ot u \ot B)    (B \ot \clambda_B^{-1}) \\
    \overset{\eqref{eq:sigmalin}}&{=} \mu_{\coinv{\hat B}{B}}    \left(B \ot \sigma_{\hat B}^{-1}\right)    \left(B \ot q_{\hat B}\right)    (B \ot u \ot B)    (B \ot \clambda_B^{-1}) \overset{\eqref{eq:defS}}{=} \mu_{\coinv{\hat B}{B}}    \left(B \ot \eta_B\right)    \left(B \ot S\right) \\ \overset{\eqref{eq:etalin}}&{=} \eta_B \,   m    \left(B \ot S\right),
\end{align*}
that is to say,
\begin{equation}\label{eq:sigmainvS}
    \sigma_{\hat B}^{-1}   q_{\hat B} = \eta_B   m   \left(B \ot S\right).
\end{equation}
Finally, this last relation allows us to see that
\begin{align*}
    \eta_B  \,   m   \,  (B \ot S)     \mu_{\hat B}  \overset{\eqref{eq:sigmainvS}}&{=} \sigma_{\hat B}^{-1}    \, q_{\hat B}  \,   \mu_{\hat B} \overset{\eqref{eq:coeq}}{=} \sigma_{\hat B}^{-1}\,     q_{\hat B}  \,   (B \ot \crho_B)     (B \ot B \ot \varepsilon) \\
    \overset{\eqref{eq:sigmainvS}}&{=} \eta_B \,    m  \,   (B \ot S)     (B \ot \crho_B)     (B \ot B \ot \varepsilon),
\end{align*}
that is to say,
\begin{equation}\label{eq:rightantipode}
    m   (B \ot S)   \mu_{\hat B} = m   (B \ot S)   (B \ot \crho_B)   (B \ot B \ot \varepsilon)
\end{equation}
because \(\eta_B\) is an isomorphism.

Now, we can prove that \(S\) is a right antipode for \(B\) by the following direct computation:
\begin{align*}
     m  \left(B \ot S\right)  \Delta \overset{\eqref{eq:Deltamu}}&{=} m  \left(B \ot S\right)  \mu_{\hat B}  (u \ot B \ot B)  (\clambda_B^{-1} \ot B)  (u \ot B)  \,\clambda_B^{-1} \\
    \overset{\eqref{eq:rightantipode}}&{=} m  (B \ot S)  (B \ot \crho_B)  (B \ot B \ot \varepsilon)  (u \ot B \ot B)  (\clambda_B^{-1} \ot B)  (u \ot B) \, \clambda_B^{-1} \\
    & = m \, (u \ot B) \, \clambda_B^{-1} \, S \, \crho_B \, (B \ot \varepsilon) \, (u \ot B) \, \clambda_B^{-1} = S \, \crho_B \, (B \ot \varepsilon) \, (u \ot B) \, \clambda_B^{-1} \\
    & = S \, \crho_B \, (u \ot \I) \, \clambda_\I^{-1} \, \varepsilon = S \, u \, \varepsilon  \overset{\eqref{eq:Sunit}}{=} u \, \varepsilon,
\end{align*}
i.e., \(m   \left(B \ot S\right)   \Delta = u   \varepsilon\) as expected.

Thus, we are left with proving that \(S\) is counital, anti-comultiplicative and anti-mul\-ti\-pli\-ca\-ti\-ve. However, the fact that \(S\) is a right antipode already guarantees that it is counital:
\begin{align*}
    \varepsilon & = \varepsilon \, u \, \varepsilon = \varepsilon \, m \, (B \ot S) \, \Delta = \clambda_\I \, (\varepsilon \ot \varepsilon) \, (B \ot S) \, \Delta = \clambda_\I \, \big(\I \ot \left(\varepsilon \, S\right)\big) \, (\varepsilon \ot B) \, \Delta = \varepsilon \, S.
\end{align*}

In order to show that \(S\) is anti-multiplicative, we need a few additional intermediate formulas. First of all, \(B \ot B\) is a left \(B\)-module also with respect to the action
\begin{equation}\label{eq:twistaction}
    B \ot B \ot B \xrightarrow{\brd_{B,B} \ot B} B \ot B \ot B \xrightarrow{B \ot m} B \ot B.
\end{equation}

Under the hypothesis that \(B \ot -\) preserves reflexive coequalisers, \(\inv{(\rmod{B} \ot \hopfmod{B})}{B}\) is a left \(B\)-module also with respect to the left \(B\)-action induced by \eqref{eq:twistaction}. Again, this can be checked directly via the following diagram
\[
\xymatrix @C=35pt @R=35pt{
     B \ot B \ot B \ot B \ar[d]|-{(B \ot m \ot B)(\brd_{B,B}\ot B \ot B)} \ar@<+0.5ex>[rrrr]^-{B \ot \mu_{B \ot B}} \ar@<-0.5ex>[rrrr]_-{\left(B \ot B \ot \crho_B\right)(B \ot B \ot B \ot \varepsilon)} & & & &  B \ot B \ot B \ar[d]|-{(B \ot m)(\brd_{B,B}\ot B)} \ar[r]^-{B \ot q_{\hat B}} & B \ot \inv{(B \ot B)}{B} \ar@{.>}[d]^-{\nu_{\inv{\hat{B}}{B}}} \\
    B \ot B \ot B \ar@<+0.5ex>[rrrr]^-{\mu_{B \ot B}} \ar@<-0.5ex>[rrrr]_-{\left(B \ot \crho_B\right)(B \ot B \ot \varepsilon)} & & & &  B \ot B \ar[r]_-{q_{\hat B}} & \inv{(B \ot B)}{B},
}
\]
which commutes sequentially since
\begin{align*}
    & (B \ot m)  (\brd_{B,B} \ot B)  (B \ot \mu_{\hat B}) \\ 
     \overset{\eqref{eq:diagonal}}&{=} (B \ot m)  (\brd_{B,B} \ot B)  (B \ot m \ot m)  (B \ot B \ot \brd_{B,B} \ot B) (B \ot B \ot B \ot \Delta) \\
    & = (B \ot m)  (B \ot B \ot m)   (\brd_{B,B} \ot B \ot B)   (B \ot m \ot B \ot B)   (B \ot B \ot \brd_{B,B} \ot B)  (B \ot B \ot B \ot \Delta) \\
    & = (B \ot m)   (B \ot m \ot B)   (\brd_{B,B} \ot B \ot B)   (B \ot m \ot B \ot B)   (B \ot B \ot \brd_{B,B} \ot B)  (B \ot B \ot B \ot \Delta) \\ 
    & = (B \ot m)   (m \ot m \ot B)   (\brd_{B,B \ot B} \ot B \ot B)   (B \ot B \ot \brd_{B,B} \ot B)   (B \ot B \ot B \ot \Delta) \\
     \overset{\eqref{eq:hexagon}}&{=} (B \ot m)   (m \ot m \ot B)   (B \ot \brd_{B,B} \ot B \ot B)   (\brd_{B,B}\ot \brd_{B,B} \ot B)    (B \ot B \ot B \ot \Delta) 
    \\ 
    \overset{\eqref{eq:hexagon}}&{=}  (B \ot m)   (m \ot B \ot B)   (B \ot B \ot m \ot B)  (B \ot \brd_{B \ot B,B} \ot B) (\brd_{B,B} \ot B \ot \Delta) \\
    & = (B \ot m)   (m \ot B \ot B)   (B \ot \brd_{B,B} \ot B)   (B \ot m \ot B \ot B)   (\brd_{B,B} \ot B \ot \Delta) \\
    \overset{\eqref{eq:diagonal}}&{=} \mu_{\hat B}   (B \ot m \ot B)   (\brd_{B,B} \ot B \ot B),
\end{align*}
and the commutativity of the other square is simply the functoriality of \(\ot\). From the above, we see that \(q_{\hat B}\) is left \(B\)-linear with respect to the action \( (B \ot m)   (\brd_{B,B} \ot B)\) on the domain, and this new action \(\nu_{\inv{\hat B}{B}}\) on the codomain, i.e.,
\begin{equation}\label{eq:nuinv}
    q_{\hat B}   (B \ot m)   (\brd_{B,B} \ot B) = \nu_{\inv{\hat B}{B}}   (B \ot q_{\hat B}).
\end{equation}
In addition, \(q_{\hat B}\) and \(\nu_{\inv{\hat B}{B}}\) satisfy
\begin{align*}
    q_{\hat B} & = q_{\hat B} \, (B \ot m) \, (B \ot B \ot u) \, (B \ot \crho_B^{-1}) \, \brd_{B,B} \, \brd_{B,B}^{-1} \\
    & = q_{\hat B} \, (B \ot m) \, (B \ot B \ot u) \, \crho_{B \ot B}^{-1} \, \brd_{B,B} \, \brd_{B,B}^{-1} \\
    & = q_{\hat B} \, (B \ot m) \, (B \ot B \ot u)  \, \left(\brd_{B,B} \ot \I\right) \, \crho_{B \ot B}^{-1} \, \brd_{B,B}^{-1} \\
    & = q_{\hat B} \, (B \ot m) \, \left(\brd_{B,B} \ot B\right) \, (B \ot B \ot u) \, \left(B \ot \crho_{B}^{-1}\right) \, \brd_{B,B}^{-1} \\
   \overset{\eqref{eq:nuinv}}&{=} \nu_{\inv{\hat B}{B}} \, (B \ot q_{\hat B}) \, (B \ot B \ot u) \, \left(B \ot \crho_{B}^{-1}\right) \, \brd_{B,B}^{-1}.
\end{align*}
That is,
\begin{equation}\label{eq:qnew}
    q_{\hat B} = \nu_{\inv{\hat B}{B}}   (B \ot q_{\hat B})   (B \ot B \ot u)   \left(B \ot \crho_{B}^{-1}\right)   \brd_{B,B}^{-1}.
\end{equation}
Now, we can finally show that \(S\) is anti-multiplicative as in \eqref{eq:antibimon}. 
We compute
\begin{align*}
    \eta_B  \,  m  \,  (S \ot S)    \brd_{B,B} \overset{\eqref{eq:sigmainvS}}&{=} \sigma_{\hat B}^{-1}    q_{\hat B}    (S \ot B)    \brd_{B,B} \\
    \overset{\eqref{eq:qnew}}&{=} \sigma_{\hat B}^{-1}   \nu_{\inv{\hat B}{B}}    (B \ot q_{\hat B})    (B \ot B \ot u)    \left(B \ot \crho_{B}^{-1}\right)    \brd^{-1}_{B,B}    (S \ot B)   \brd_{B,B} \\
    \overset{\eqref{eq:eta}}&{=} \sigma_{\hat B}^{-1}   \nu_{\inv{\hat B}{B}}    (B \ot q_{\hat B})    (B \ot e_{\hat B})    \left(B \ot \eta_{B}\right)    (B \ot S) \\
    \overset{\eqref{eq:easysigma}}&{=} \sigma_{\hat B}^{-1}   \nu_{\inv{\hat B}{B}}    (B \ot \sigma_{\hat B})    \left(B \ot \eta_{B}\right)    (B \ot S) \\
    \overset{\eqref{eq:defS}}&{=} \sigma_{\hat B}^{-1}   \nu_{\inv{\hat B}{B}}    (B \ot q_{\hat B})    \left(B \ot u \ot B\right)    (B \ot \clambda_B^{-1}) \\
    \overset{\eqref{eq:nuinv}}&{=} \sigma_{\hat B}^{-1}    q_{\hat B}    (B \ot m)    (\brd_{B,B} \ot B)    \left(B \ot u \ot B\right)    (\crho_B^{-1} \ot B) \\
    \overset{\eqref{eq:braidI}}&{=} \sigma_{\hat B}^{-1}    q_{\hat B}    (B \ot m)    \left(u \ot B \ot B\right)    (\clambda_B^{-1} \ot B) \\
    & = \eta_B    \eta_B^{-1}    \sigma_{\hat B}^{-1}    q_{\hat B}    \left(u \ot B\right)    \clambda_B^{-1}    m \stackrel{\eqref{eq:defS}}{=} \eta_B  S  m,
\end{align*}
and hence we deduce that \(m   (S \ot S)   \brd_{B,B} = S   m\), because \(\eta_B\) is an isomorphism.

It remains to verify the anti-comultiplicativity. To this aim, consider the composition
\begin{equation}\label{eq:deflambda}\lambda \coloneqq  (m \ot m)   (B \ot S \ot B \ot S)   (B \ot \brd_{B,B} \ot B)   (\Delta \ot \Delta^{\mathrm{cop}})\end{equation}
from \(B \ot B\) to \(B \ot B\), where \(\Delta^{\mathrm{cop}} = \brd_{B,B}\Delta\).
We claim that it factors through \(\inv{(\rmod{B} \ot \hopfmod{B})}{B}\). To verify this, we need to prove that \(\lambda\) equalises \(\mu_{\hat B}\) and \((B \ot B \ot \crho_B)   (B \ot B \ot \varepsilon)\). 

One has
{\small
\begin{align*}
    \lambda \mu_{\hat B}
    \overset{\eqref{eq:deflambda}}&{=} (m\ot m)(B\ot S\ot B\ot S) (B\ot \brd_{B\ot B, B})(\Delta \ot \Delta)\mu_{\hat B}\\
    \overset{(\dagger)}&{=} (m\ot B)(B\ot S\ot B) (B\ot \brd_{B, B})(B\ot m\ot B)(B\ot B\ot S\ot B)(\Delta \ot \Delta)\mu_{\hat B}\\
    \overset{\eqref{eq:diagonal},\eqref{eq:bialgebra-compatibility}}&{=} (m\ot B)(B\ot S\ot B) (B\ot \brd_{B, B})(B\ot m\ot B)(B\ot B\ot S\ot B) (m\ot m\ot m\ot m)\\
        & \hspace{2em}(B\ot \brd_{B,B}\ot B\ot B\ot \brd_{B,B}\ot B) (\Delta \ot \Delta \ot \Delta \ot \Delta) (B\ot \brd_{B,B}\ot B)(B\ot B\ot \Delta) \\
    \overset{(\dagger),\eqref{eq:hexagon}}&{=} (m\ot B)(B\ot S\ot B) (B\ot \brd_{B, B})(B\ot m\ot B)(B\ot B\ot S\ot B) (m\ot m\ot m\ot m)\\
        & \hspace{2em} (B\ot \brd_{B,B}\ot \brd_{B,B}\ot \brd_{B,B}\ot B) (B \ot B \ot \brd_{B,B}\ot \brd_{B,B}\ot B\ot B) (B^{\ot 3}\ot \brd_{B,B}\ot B^{\ot 3}) \\
        & \hspace{2em} (B^{\ot 4} \ot \Delta \ot \Delta) (\Delta\ot \Delta\ot \Delta) \\
    &=(m\ot B)(B\ot S\ot B) (B\ot \brd_{B, B})(B\ot m\ot B)(B\ot B\ot S\ot B) (m\ot m\ot m\ot m) \\
        & \hspace{2em}(B\ot \brd_{B,B}\ot \brd_{B,B}\ot \brd_{B,B}\ot B) (B \ot B \ot \brd_{B,B}\ot \brd_{B,B}\ot B\ot B) (B^{\ot 3}\ot \brd_{B,B}\ot B^{\ot 3}) \\
        & \hspace{2em}(B^{\ot 5} \ot \Delta \ot B)(B^{\ot 5}\ot \Delta) (\Delta \ot \Delta \ot \Delta) \\
    \overset{\eqref{eq:hexagon}}&{=} (m\ot B)(B\ot S\ot B) (B\ot \brd_{B, B})(B\ot m\ot B)(B\ot B\ot S\ot B) (m\ot m\ot m\ot m) \\
        & \hspace{2em}(B\ot \brd_{B,B}\ot \brd_{B,B}\ot B^{\ot 3}) (B^{\ot 4} \ot \brd_{B, B\ot B}\ot B)(B^{\ot 5} \ot \Delta \ot B) \\
        & \hspace{2em} (B^{\ot 2}\ot \brd_{B \ot B,B}\ot B^{\ot 2})(B^{\ot 5}\ot \Delta) (\Delta \ot \Delta \ot \Delta) \\
    \overset{(\dagger)}&{=} (m\ot B)(B\ot S\ot B) (B\ot \brd_{B, B})(B\ot m\ot B)(B\ot B\ot S\ot B) (m\ot m\ot m\ot m) \\
        & \hspace{2em}(B^{\ot 3} \ot \brd_{B,B}\ot B^{\ot 3})(B^{\ot 4} \ot \Delta \ot B^{\ot 2}) (B\ot  \brd_{B,B}\ot B^{\ot 5})(B^{\ot 4} \ot \brd_{B, B}\ot B) \\
        & \hspace{2em} (B^{\ot 2}\ot \brd_{B \ot B,B}\ot B^{\ot 2})(B^{\ot 5}\ot \Delta) (\Delta \ot \Delta \ot \Delta) \\
    &= (m\ot B)(B\ot S\ot B) (B\ot \brd_{B, B}) \Bigg(  B\!\ot \Big(\! m(B\ot S) (m\ot m)(B\ot \brd_{B,B}\ot B)(B^{\ot 2}\ot \Delta)\!\Big)\!\ot\! B \Bigg)\\
    & \hspace{2em} (m\ot B^{\ot 3}\ot m)(B\ot \brd_{B,B}\ot B\ot \brd_{B,B}\ot B) (B^{\ot 2}\ot \brd_{B \ot B,B}\ot B^{\ot 2})(B^{\ot 5}\ot \Delta) (\Delta \ot \Delta \ot \Delta) 
    \\
\overset{\eqref{eq:rightantipode}}&{=} (m\ot B)(B\ot S\ot B) (B\ot \brd_{B, B}) \Bigg(  B\ot \Big(\! m(B\ot S)(B\ot \crho_B)(B\ot B\ot \varepsilon)\!\Big) \ot  B \Bigg)\\
    & \hspace{2em} (m\ot B^{\ot 3}\ot m)(B\ot \brd_{B,B}\ot B\ot \brd_{B,B}\ot B) (B^{\ot 2}\ot \brd_{B \ot B,B}\ot B^{\ot 2})(B^{\ot 5}\ot \Delta) (\Delta \ot \Delta \ot \Delta)
    \\
&= (m\ot B)(B\ot S\ot B) (B\ot \brd_{B, B})(B \ot m \ot B)(B \ot B \ot S \ot B)(B \ot B \ot \crho_B \ot B) (B^{\ot 3} \ot \varepsilon \ot B)\\
    & \hspace{2em} (m\ot B^{\ot 3}\ot m)(B\ot \brd_{B,B}\ot B\ot \brd_{B,B}\ot B) (B^{\ot 2}\ot \brd_{B \ot B,B}\ot B^{\ot 2})(B^{\ot 5}\ot \Delta) (\Delta \ot \Delta \ot \Delta)\\
&= (m\ot B)(B\ot S\ot B) (B\ot \brd_{B, B})(B \ot m \ot B)(B \ot B \ot S \ot B)(B \ot B \ot \crho_B \ot B)\\
    & \hspace{2em} (m\ot B^{\ot\! 2} \ot \I\ot m)(B\ot \brd_{\!B,B}\ot B\ot \brd_{\!B,\I}\ot B) (B^{\ot\! 2}\ot \brd_{\!B\! \ot\! B,B} \ot \I\ot B)(B^{\ot\! 5}\ot \clambda_B^{-1}) (\Delta \ot \Delta \ot \Delta)\\
\overset{\eqref{eq:braidI}}&{=} (m\ot B)(B\ot S\ot B) (B\ot \brd_{B, B})(B \ot m \ot B)(B \ot B \ot S \ot B)\\
    & \hspace{2em} (m\ot B^{\ot 2} \ot m)(B\ot \brd_{B,B}\ot B^{\ot 3}) (B^{\ot 2}\ot \brd_{B \ot B,B}\ot B) (\Delta \ot \Delta \ot \Delta)\\
\overset{\eqref{eq:hexagon}}&{=} (m\ot B) (B\ot S\ot B) (B\ot \brd_{B,B}) (m\ot B\ot B)(B^{\ot 2} \ot m \ot B) (B^{\ot 2} \ot B\ot S\ot B)  \\ 
    & \hspace{2em} (B\ot \brd_{B \ot B,B}\ot B\ot B)(B^{\ot 4}\ot m)(B\ot B\ot B\ot \brd_{B,B}\ot B)(\Delta \ot \Delta \ot \Delta)\\
\overset{(\dagger)}&{=} (m\ot B) (B\ot S\ot B)(m\ot \brd_{B,B}) (B\ot \brd_{B,B}\ot B) (B\ot m\ot B\ot m) \\ 
    & \hspace{2em}(B\ot B\ot S \ot \brd_{B,B}\ot B)(\Delta \ot \Delta \ot \Delta) \\
\overset{(\dagger)}&{=} (m\ot B) (B\ot S\ot B) (m\ot m\ot B)(B \ot B \ot \brd_{B,B \ot B})(B\ot \brd_{B,B}\ot B\ot B) \\ 
    &\hspace{2em} (B\ot m\ot B\ot B \ot B)(B\ot B\ot S \ot \brd_{B,B}\ot B)(\Delta \ot \Delta \ot \Delta)\\
\overset{\eqref{eq:hexagon}}&{=} (m\ot B) (B\ot S\ot B) (m\ot m\ot B)(B \ot B \ot B \ot \brd_{B,B})\\
    &\hspace{2em} (B \ot B \ot \brd_{B,B} \ot B)(B\ot \brd_{B,B}\ot B\ot B)(B\ot  B \ot \brd_{B,B}\ot B) \\ 
    &\hspace{2em} (B\ot m\ot B\ot B \ot B)(B\ot B\ot S \ot B \ot B\ot B)(\Delta \ot \Delta \ot \Delta)\\
\overset{(\star)}&{=} (m\ot B) (B\ot S\ot B) (m\ot m\ot B)(B \ot B \ot B \ot \brd_{B,B})\\
    &\hspace{2em} (B\ot \brd_{B,B}\ot B\ot B)(B \ot B \ot \brd_{B,B} \ot B)(B\ot \brd_{B,B}\ot B\ot B) \\ 
    &\hspace{2em} (B\ot m\ot B\ot B \ot B)(B\ot B\ot S \ot B \ot B\ot B)(\Delta \ot \Delta \ot \Delta)\\
\overset{\eqref{eq:hexagon}}&{=} (m\ot B) (B\ot S\ot B) (m\ot m\ot B)(B\ot \brd_{B,B}\ot B\ot B)(B \ot B \ot \brd_{B,B \ot B})\\
    &\hspace{2em} (B\ot \brd_{B,B}\ot B\ot B)(B\ot m\ot B\ot B \ot B)(B\ot B\ot S \ot B \ot B\ot B)(\Delta \ot \Delta \ot \Delta)\\
\overset{(\dagger)}&{=} (m\ot B) (B\ot S\ot B) (m\ot m\ot B)(B\ot \brd_{B,B}\ot B\ot B)(B \ot B \ot \Delta \ot B)(B \ot B \ot \brd_{B,B})\\
    &\hspace{2em} (B\ot \brd_{B,B}\ot B)(B\ot m\ot B\ot B)(B\ot B\ot S \ot B \ot B)(\Delta \ot \Delta \ot B)\\
\overset{\eqref{eq:hexagon},\eqref{eq:rightantipode}}&{=} \Big(\big(  m(B\ot S)(B\ot \crho_B)(B\ot B\ot \varepsilon)\big)\ot B\Big)(B\ot \brd_{B, B\ot B}) (B\ot m\ot B\ot B) \\ 
    &\hspace{2em}(\Delta \ot S\ot B\ot B) (B\ot \Delta \ot B) \\
\overset{(\dagger)}&{=} (m\ot m)(B\ot S \ot B\ot S)(B\ot \brd_{B,B}\ot B)(B\ot B\ot \brd_{B,B})(\Delta \ot B\ot \crho_B)(B\ot \Delta \ot \varepsilon)\\
\overset{\eqref{eq:deflambda}}&{=} \lambda (B\ot B\ot \crho_B)(B\ot B\ot \varepsilon),
\end{align*}
}%
where in \((\dagger)\) we used naturality of \(\brd\), and in \((\star)\) the third Reidemeister's move (a.k.a, the braid relation).
Thus \(\lambda\) equalises \(\mu_{\hat B}\) and \((B\ot B\ot \crho_B)(B\ot B\ot \varepsilon)\) as desired and hence there exists a unique \(\bar\lambda \colon \inv{\rmod{B} \ot \hopfmod{B}}{B} \to B \ot B\) such that $\bar\lambda q_{\hat B} = \lambda$.

Now, remark that
\begin{align*}
    \lambda (B\ot u)\crho_B^{-1}
    \overset{\eqref{eq:deflambda}}&{=} (m\ot m) (B\ot S\ot B\ot S) (B\ot \brd_{B,B}\ot B)(\Delta \ot \Delta^{\mathrm{cop}})(B\ot u)\crho_B^{-1}\\
    &=(m\ot m) (B\ot S\ot B\ot S) (B\ot \brd_{B,B}\ot B)(\Delta \ot u\ot u) (\crho_B^{-1}\ot\I)\crho_{B}^{-1}\\
    &= (m\ot m) (B\ot S\ot B\ot S) (B\ot u\ot B\ot u)(\crho_B^{-1}\ot \crho_B^{-1})\Delta\\
    \overset{\eqref{eq:Sunit}}&{=} (m\ot m)(B\ot u\ot B\ot u)(\crho_B^{-1}\ot \crho_B^{-1})\Delta = \Delta,
\end{align*}
that is 
\begin{equation}
    \label{eq:lambdaunitDelta} \lambda (B\ot u)\crho_B^{-1} =\Delta,
\end{equation}
whence
\[\Delta   S   \clambda_B \!\overset{\eqref{eq:lambdaunitDelta}}{=}\! \lambda   (B \ot u)   \crho_B^{-1}   S   \clambda_ B \!\overset{\eqref{eq:eta}}{=}\! \lambda   e_{\hat B}   \eta_B   S   \clambda_B \!\overset{\eqref{eq:defS}}{=} \!\bar\lambda   q_{\hat B}   e_{\hat B}   \sigma_{\hat B}^{-1}   q_{\hat B}   (u \ot B) \!\overset{\eqref{eq:easysigma}}{=} \!\bar\lambda   q_{\hat B}   (u \ot B) = \lambda   (u \ot B)\]
and, by definition of \(\lambda\), 
\begin{align*}
    \lambda \, (u \ot B) \overset{\eqref{eq:deflambda}}&{=}(m \ot m) \, (B \ot S \ot B \ot S) \, (B \ot \brd_{B,B} \ot B) \, (\Delta \ot \Delta^{\mathrm{cop}}) \, (u \ot B) \\
    & = (m \ot m) \, (B \ot S \ot B \ot S) \, (B \ot \brd_{B,B} \ot B) \, (u \ot B \ot B \ot B) \, (\clambda_B^{-1} \ot B \ot B) \, (u \ot \Delta^{\mathrm{cop}}) \\
    & = (m \ot m) \, (u \ot B \ot B \ot B) \, (\clambda_B^{-1} \ot B \ot B) \, (S \ot B \ot S) \, (\brd_{B,B} \ot B) \, (u \ot \Delta^{\mathrm{cop}}) \\
    & = (B \ot m) \, (S \ot B \ot S) \, (\brd_{B,B} \ot B) \, (u \ot \Delta^{\mathrm{cop}}) \\
    & = (B \ot m) \, (\brd_{B,B} \ot B) \, (B \ot S \ot S) \, (u \ot \Delta^{\mathrm{cop}}) \\
    & = (B \ot m) \, (\brd_{B,B} \ot B) \, (B \ot S \ot S) \, (u \ot B \ot B) \, (\clambda_B^{-1}\ot B) \, \Delta^{\mathrm{cop}} \, \clambda_B \\
    & = (B \ot m) \, (B \ot u \ot B) \, (\crho_B^{-1}\ot B) \, (S \ot S) \, \Delta^{\mathrm{cop}} \, \clambda_B = (S \ot S) \, \Delta^{\mathrm{cop}} \, \clambda_B,
\end{align*}%
that is, \(\Delta   S = (S \ot S)   \Delta^{\mathrm{cop}}\) as in \eqref{eq:antibimon} and the proof is complete.
\end{proof}

In \zcref{ex:trivial}, we have seen that every right Hopf algebra \(H\) in the sense of \cite{GreenNicholsTaft} whose antipode is a bialgebra anti-homomorphism gives rise to right Hopf monoids in categories of modules \({}_B\mathfrak{M}\) over cocommutative bialgebras \(B\). In particular, \({}_B\mathfrak{M}\) is complete and cocomplete and both \(H \ot -\) and \(- \ot H\) preserve all colimits. The braided structure on \({}_B\mathfrak{M}\) is simply the flip. Therefore, the hypotheses of \zcref{prop:Frobenius-OneSidedHopf} are satisfied and so
\[
    \begin{gathered}
    \xymatrix {
    \left({}_B\mathfrak{M}\right)_{H}^{H} \ar@/_4ex/@<-0.3ex>[d]_-{\inv{(-)}{H}} \ar@/^4ex/@<+0.3ex>[d]^-{\coinv{(-)}{H}} \\ 
    {{}_B\mathfrak{M}}^{\phantom{H}}_{\phantom{H}} \ar[u]|-{-\ot H}
    }
    \end{gathered}
\]
is an ambidextrous adjunction and \(- \ot H\) is Frobenius. In particular, \(\coinv{M}{H} \cong M/MH^+\) for every object \(M\) in \(\left({}_B\mathfrak{M}\right)_{H}^{H}\), where \(H^+ = \ker(\varepsilon)\). 

Similarly, the right Hopf algebra \(H \coloneqq \rH(M_r(\Bbbk))\) in \({}^{\Z_2}\mathfrak{M}\) from \zcref{ex:MatcomodZ2} satisfies the hypotheses of \zcref{prop:Frobenius-OneSidedHopf} and so
\[
    \begin{gathered}
    \xymatrix {
    \left({}^{\Z_2}\mathfrak{M}\right)_{H}^{H} \ar@/_4ex/@<-0.3ex>[d]_-{\inv{(-)}{H}} \ar@/^4ex/@<+0.3ex>[d]^-{\coinv{(-)}{H}} \\ 
    {{}^{\Z_2}\mathfrak{M}}^{\phantom{H}}_{\phantom{H}} \ar[u]|-{-\ot H}
    }
    \end{gathered}
\]
is an ambidextrous adjunction and \(- \ot H\) is Frobenius.

\begin{invisible}
    \begin{example}
        Let \(C = \C c \oplus \C s\), cocommutative coalgebra with respect to 
        \begin{align*}
            \Delta(c) & = c \ot c - s \ot s, \qquad \varepsilon(c) = 1, \\
            \Delta(s) & = s \ot c + c \ot s, \qquad \varepsilon(s) = 0.
        \end{align*}
        Then \(C^{\mathrm{cop}} = C\) and the coalgebra direct sum \emph{\`a la} Takeuchi \(\bigoplus_n C_n\), which is the vector space direct sum with the induced comultiplication and counit, is
        \[\mathsf{span}_\C \left\{c_n,s_n \mid n \in \N\right\}.\]
        Then 
        \[T\left(\bigoplus_n C_n\right) \cong \C\langle c_n,s_n \mid n \in \N\rangle\]
        and the ideal \(I + TS(I)T\) of Green-Nichols-Taft has generators
        \begin{align*}
            & c_nc_{n+1} - s_ns_{n+1} = 1 & & \forall\, n \geq 0, \\
            & c_ns_{n+1} + s_nc_{n+1} = 0 & & \forall\, n \geq 0, \\
            & c_{n+1}c_n - s_{n+1}s_n = 1 & & \forall\, n \geq 1, \\
            & s_{n+1}c_n + c_{n+1}s_n = 0 & & \forall\, n \geq 1.
        \end{align*}
        {\paolo[To Be Continued \ldots]}
    \end{example}
\end{invisible}

\begin{invisible}
    \begin{example}\label{ex:MatinZ2}
        Let us keep the setting from \zcref{ex:Z2inZ2}, but let us denote by \(u\) and \(g\) the elements of \(\Z_2\) and let us now consider the matrix coalgebra
        \[C \coloneqq  M_2(\K)\]
        with
        \[\Delta(E_{ij}) = E_{i1} \ot E_{1j} + E_{i2} \ot E_{2j} \qquad \text{and} \qquad \varepsilon(E_{ij}) = \delta_{ij}\]
        for all \(i,j \in \{1,2\}\). It admits an action of \(\K \Z_2\) uniquely determined by
        \[g \cdot \begin{pmatrix}
            a & b \\ c & d
        \end{pmatrix} = \begin{pmatrix}
            0 & 1 \\ 1 & 0
        \end{pmatrix}\begin{pmatrix}
            a & b \\ c & d
        \end{pmatrix}\begin{pmatrix}
            0 & 1 \\ 1 & 0
        \end{pmatrix} = \begin{pmatrix}
            d & c \\ b & a
        \end{pmatrix}.\]
        Namely, \(g\) acts on the element of the canonical basis via
        \[E_{11} \mapsto E_{22}, \quad E_{12} \mapsto E_{21}, \quad E_{21} \mapsto E_{12}, \quad E_{22} \mapsto E_{11}.\]
        With this action, \(M_2(\K)\) becomes a \(\K \Z_2\)-module coalgebra, because
        \begin{align*}
            g \cdot E_{11} \ot g \cdot E_{11} + g \cdot E_{12} \ot g \cdot E_{21} & = E_{22} \ot E_{22} + E_{21} \ot E_{12} = \Delta(E_{22}) = \Delta(g \cdot E_{11}), \\
            g \cdot E_{11} \ot g \cdot E_{12} + g \cdot E_{12} \ot g \cdot E_{22} & = E_{22} \ot E_{21} + E_{21} \ot E_{11} = \Delta(E_{21}) = \Delta(g \cdot E_{12}), \\
            g \cdot E_{21} \ot g \cdot E_{11} + g \cdot E_{22} \ot g \cdot E_{21} & = E_{12} \ot E_{22} + E_{11} \ot E_{12} = \Delta(E_{12}) = \Delta(g \cdot E_{21}), \\
            g \cdot E_{21} \ot g \cdot E_{12} + g \cdot E_{22} \ot g \cdot E_{22} & = E_{12} \ot E_{21} + E_{11} \ot E_{11} = \Delta(E_{11}) = \Delta(g \cdot E_{22}),
        \end{align*}
        and clearly \(\varepsilon\) is left \(\K\Z_2\)-linear, too. In this case, \(C^\mathrm{cop}\) has comultiplication
        \begin{align*}
            \Delta^{\mathrm{cop}}(E_{11}) & = \frac{1}{2}\Big(1 \ot 1 + 1 \ot g + g \ot 1 - g \ot g \Big) \cdot \Big(E_{11} \ot E_{11} + E_{21} \ot E_{12}\Big) \\
            & = \frac{1}{2}\Big(E_{11} \ot E_{11} + E_{11} \ot g \cdot E_{11} + g \cdot E_{11} \ot E_{11} - g \cdot E_{11} \ot g \cdot E_{11} + \\
            & \hspace{1cm} + E_{21} \ot E_{12} + E_{21} \ot g \cdot E_{12} + g \cdot E_{21} \ot E_{12} - g \cdot E_{21} \ot g \cdot E_{12}\Big) \\
            & = \frac{1}{2}\Big(E_{11} \ot E_{11} + E_{11} \ot E_{22} + E_{22} \ot E_{11} - E_{22} \ot E_{22} + \\
            & \hspace{1cm} + E_{21} \ot E_{12} + E_{21} \ot E_{21} + E_{12} \ot E_{12} - E_{12} \ot E_{21}\Big), \\
            \Delta^{\mathrm{cop}}(E_{12}) & = \frac{1}{2}\Big(1 \ot 1 + 1 \ot g + g \ot 1 - g \ot g \Big) \cdot \Big(E_{12} \ot E_{11} + E_{22} \ot E_{12}\Big) \\
            & = \frac{1}{2}\Big(E_{12} \ot E_{11} + E_{12} \ot g \cdot E_{11} + g \cdot E_{12} \ot E_{11} - g \cdot E_{12} \ot g \cdot E_{11} + \\
            & \hspace{1cm} + E_{22} \ot E_{12} + E_{22} \ot g \cdot E_{12} + g \cdot E_{22} \ot E_{12} - g \cdot E_{22} \ot g \cdot E_{12}\Big) \\
            & = \frac{1}{2}\Big(E_{12} \ot E_{11} + E_{12} \ot E_{22} + E_{21} \ot E_{11} - E_{21} \ot E_{22} + \\
            & \hspace{1cm} + E_{22} \ot E_{12} + E_{22} \ot E_{21} + E_{11} \ot E_{12} - E_{11} \ot E_{21}\Big), \\
            \Delta^{\mathrm{cop}}(E_{21}) & = \frac{1}{2}\Big(1 \ot 1 + 1 \ot g + g \ot 1 - g \ot g \Big) \cdot \Big(E_{11} \ot E_{21} + E_{21} \ot E_{22}\Big) \\
            & = \frac{1}{2}\Big(E_{11} \ot E_{21} + E_{11} \ot g \cdot E_{21} + g \cdot E_{11} \ot E_{21} - g \cdot E_{11} \ot g \cdot E_{21} + \\
            & \hspace{1cm} + E_{21} \ot E_{22} + E_{21} \ot g \cdot E_{22} + g \cdot E_{21} \ot E_{22} - g \cdot E_{21} \ot g \cdot E_{22}\Big) \\
            & = \frac{1}{2}\Big(E_{11} \ot E_{21} + E_{11} \ot E_{12} + E_{22} \ot E_{21} - E_{22} \ot E_{12} + \\
            & \hspace{1cm} + E_{21} \ot E_{22} + E_{21} \ot E_{12} + E_{121} \ot E_{22} - E_{12} \ot E_{11}\Big),\\
            \Delta^{\mathrm{cop}}(E_{22}) & = \frac{1}{2}\Big(1 \ot 1 + 1 \ot g + g \ot 1 - g \ot g \Big) \cdot \Big(E_{12} \ot E_{21} + E_{22} \ot E_{22}\Big) \\
            & = \frac{1}{2}\Big(E_{12} \ot E_{21} + E_{12} \ot g \cdot E_{21} + g \cdot E_{12} \ot E_{21} - g \cdot E_{12} \ot g \cdot E_{21} + \\
            & \hspace{1cm} + E_{22} \ot E_{22} + E_{22} \ot g \cdot E_{22} + g \cdot E_{22} \ot E_{22} - g \cdot E_{22} \ot g \cdot E_{22}\Big) \\
            & = \frac{1}{2}\Big(E_{12} \ot E_{21} + E_{12} \ot E_{12} + E_{21} \ot E_{21} - E_{21} \ot E_{12} + \\
            & \hspace{1cm} + E_{22} \ot E_{22} + E_{22} \ot E_{11} + E_{11} \ot E_{22} - E_{11} \ot E_{11}\Big).
        \end{align*}
    
        \paolo{Perform NGT on this example and see what's the outcome \ldots}
    \end{example}

    \begin{example}
        Let us keep the same notation as in \zcref{ex:MatinZ2}. The left \(\K\Z_2\)-module \(M_2(\K)\) with the same action above can be made, in fact, a \(\Z_2\)-graded \(\Z_2\)-module by setting
        \[\delta(E_{ij}) = g^{i+j} \ot E_{ij}\]
        for all \(i,j \in \{1,2\}\). As a consequence, \(M_2(\K)\) with these structures is a left-left Yetter-Drinfeld module over \(\K\Z_2\). We already know that \((\Delta,\varepsilon)\) are \(\K\Z_2\)-linear with respect to the diagonal action on the tensor product. Let us verify that they are also colinear with respect to the diagonal coaction:
        \begin{align*}
            \delta\Delta(E_{11}) & = \delta\Big(E_{11} \ot E_{11} + E_{12} \ot E_{21}\Big) = u \ot E_{11} \ot E_{11} + u \ot E_{12} \ot E_{21} \\
            & = (\K\Z_2 \ot \Delta)\delta(E_{11}), \\
            \delta\Delta(E_{12}) & = \delta\Big(E_{11} \ot E_{12} + E_{12} \ot E_{22}\Big) = g \ot E_{11} \ot E_{12} + g \ot E_{12} \ot E_{22} \\
            & = (\K\Z_2 \ot \Delta)\delta(E_{12}), \\
            \delta\Delta(E_{21}) & = \delta\Big(E_{21} \ot E_{11} + E_{22} \ot E_{21}\Big) = g \ot E_{21} \ot E_{11} + g \ot E_{22} \ot E_{21} \\
            & = (\K\Z_2 \ot \Delta)\delta(E_{21}), \\
            \delta\Delta(E_{22}) & = \delta\Big(E_{21} \ot E_{12} + E_{22} \ot E_{22}\Big) = u \ot E_{21} \ot E_{12} + u \ot E_{22} \ot E_{22} \\
            & = (\K\Z_2 \ot \Delta)\delta(E_{22}),
        \end{align*}
        and
        \[(\K\Z_2 \ot \varepsilon)\delta(E_{11}) = u, \quad (\K\Z_2 \ot \varepsilon)\delta(E_{12}) = 0, \quad (\K\Z_2 \ot \varepsilon)\delta(E_{21}) = 0, \quad (\K\Z_2 \ot \varepsilon)\delta(E_{22}) = u.\]
        Therefore, \(M_2(\K)\) is a coalgebra in the braided monoidal category \(\prescript{\Z_2}{\Z_2}{\mathcal{YD}}\), with braiding
        \[\mf{c}(x \ot y) = x_{[-1]} \cdot y \ot x_{[0]}.\]
         In this case, \(C^\mathrm{cop}\) has comultiplication
        \begin{align*}
            \Delta^{\mathrm{cop}}(E_{11}) & = \mf{c}\Big(E_{11} \ot E_{11} + E_{12} \ot E_{21}\Big) = E_{11} \ot E_{11} + g \cdot E_{21} \ot E_{12} \\
            & = E_{11} \ot E_{11} + E_{12} \ot E_{12}, \\
            \Delta^{\mathrm{cop}}(E_{12}) & = \mf{c}\Big(E_{11} \ot E_{12} + E_{12} \ot E_{22}\Big) = E_{12} \ot E_{11} + g \cdot E_{22} \ot E_{12} \\
            & = E_{12} \ot E_{11} + E_{11} \ot E_{12}, \\
            \Delta^{\mathrm{cop}}(E_{21}) & = \mf{c}\Big(E_{21} \ot E_{11} + E_{22} \ot E_{21}\Big) = g \cdot E_{11} \ot E_{21} + E_{21} \ot E_{22} \\
            & = E_{22} \ot E_{21} + E_{21} \ot E_{22}, \\
            \Delta^{\mathrm{cop}}(E_{22}) & = \mf{c}\Big(E_{21} \ot E_{12} + E_{22} \ot E_{22}\Big) = g \cdot E_{12} \ot E_{21} + E_{22} \ot E_{22} \\
            & = E_{21} \ot E_{21} + E_{22} \ot E_{22}. \\
        \end{align*}
    
        \paolo{Perform NGT on this example and see what's the outcome \ldots}
    \end{example}
\end{invisible}

\begin{invisible}

\section{Free one-sided Hopf algebra generated by a bialgebra (work in progress)}

\marginpar{\paolo{\tiny We may even assume \(\cM\) cocomplete: a category which admits coproducts and reflecive coequalisers is automatically cocomplete}}
Let \((\cM,\ot,\I,\brd)\) be a braided monoidal category admitting coequalisers of reflexive pairs, equalisers of coreflexive pairs, and denumerable coproducts. Assume also that the endofunctors \(X \ot -\) and \(- \ot X\) preserve these coproducts and the reflexive coequalisers. Then the forgetful functor \(\Mon(\cM) \to \cM\) admits a left adjoint (see, e.g., \cite[Theorem VII.3.2]{Maclane}) and we can apply \cite[Theorem 3.5]{Porst} to conclude that \(\Mon(\cM)\) admits denumerable coproducts. In this setting, also the category \(\mathrm{Bimon}(\cM)\) admits denumerable coproducts as follows (the proof is a straightforward adaptation of \cite[Corollary 2.6.2]{Pareigis}): suppose that \(\{B_n \mid n \in \N\}\) is a family of bialgebras in \(\cM\). Consider their coproduct \(\cB \coloneqq  \bigsqcup_n B_n\) as algebras. Since the tensor product of algebras in a braided monoidal category is an algebra itself with multiplication \((m \ot m)(\id \ot \brd \ot \id)\), there exist unique algebra maps \(\Delta \colon \cB \to \cB \ot \cB\) and \(\varepsilon \colon \cB \to \I\) for which the following diagrams commute
\[
\xymatrix @C=70pt {
B_m \ar[r]^{j_m} \ar[d]_-{\Delta_m} & \bigsqcup_n B_n \ar@{.>}[d]^-{\Delta} \\
B_m \ot B_m \ar[r]_-{j_m \ot j_m} & \left(\bigsqcup_n B_n\right) \ot \left(\bigsqcup_n B_n\right)
}
\qquad 
\xymatrix{
B_m \ar[rr]^-{j_m} \ar[dr]_-{\varepsilon_m} & & \bigsqcup_n B_n \ar@{.>}[dl]^-{\varepsilon} \\
& \I & 
}
\]
In view of the uniqueness part of the universal property of the coproduct, one can verify that \((\Delta,\varepsilon)\) is a coalgebra structure on \(\cB\) and that \(\cB\) satisfies the universal property of the coproduct in the category \(\mathrm{Bimon}(\cM)\).

\medskip

Now, let \((B,m,u,\Delta,\varepsilon)\) be a bialgebra in \(\cM\) and write \(B^{\mathrm{op},\mathrm{cop}}\) for the bialgebra whose multiplication is \(m   \brd^{-1}\) and whose comultiplication is \(\brd   \Delta\). It is still a bialgebra in \(\cM\). 

\begin{remark}
    If \(H\) is a Hopf algebra in \(\cM\), then one can verify that 
    \[S   m^{\mathrm{op}} = m   (S \ot S) \qquad \text{and} \qquad \Delta   S = (S \ot S)   \Delta^{\mathrm{cop}}\]
    where \(m^{\mathrm{op}} = m   \brd^{-1}\) and \(\Delta^{\mathrm{cop}} = \brd   \Delta\), by, for instance, verifying that \(S   m^{\mathrm{op}}\) is left convolution inverse of \(m^{\mathrm{op}}\) in \(\Hom(H^{\mathrm{op},\mathrm{cop}} \ot H^{\mathrm{op},\mathrm{cop}},H^{\mathrm{op},\mathrm{cop}})\), while \(m   (S \ot S)\) is right convolution inverse of it therein.
\end{remark}

Define a family of bialgebras \(\{B_n \mid n \in \N\}\) as follows:
\[B_{2n} \coloneqq  B \qquad \text{and} \qquad B_{2n+1} \coloneqq  B^{\mathrm{op},\mathrm{cop}},\]
and consider their bialgebra coproduct \(\cB \coloneqq  \bigsqcup_n B_n\) in \(\cM\), with canonical morphisms \(j_n \colon B_n \to \cB\). By the universal property of the coproduct, there is a unique bialgebra morphism \(S \colon \cB^{\mathrm{op},\mathrm{cop}} \to \cB\) such that the following diagram commutes
\[
\xymatrix{
B_n^{\mathrm{op},\mathrm{cop}} \ar[r]^-{j_n} \ar[d]_-{=} & \cB \ar@{.>}[d]^-{S} \\
B_{n+1} \ar[r]_-{j_{n+1}} & \cB
}
\]
Now, define a morphism \((S*\id)_0 \colon \cB \to \cB\) as the one arising from \(B_0 \xrightarrow{\varepsilon} \I \xrightarrow{u_\cB} \cB\) and \(B_n \xrightarrow{j_n} \cB \xrightarrow{S * \id} \cB\) for all \(n \geq 1\) {[\paolo CAVEAT: are these morphisms of bialgebras in \(\cM\)? If this is not the case, we cannot invoke the universal property]} via the universal property \ldots

\zorro

{\paolo [In case we need it, we shall find a way of completing this proof]}

\begin{example}
    Let \(B \coloneqq  \C\langle g,x \mid g^2=g,x^2=0,gx+xg=0\rangle\) with
    \[\Delta(g) = g \ot g, \qquad \Delta(x) = x \ot 1 + g \ot x,\]
    in \(\cM = \Vect_\C\) with the flip \(\tau\). In this case, \(B^{\mathrm{op,cop}} = \C\langle h,y \mid h^2 = h, y^2 = 0, hy+yh = 0\rangle\) with
    \[\Delta(h) = h \ot h, \qquad \Delta(y) = y \ot h + 1 \ot y.\]
    In this case, the coproduct \(\bigsqcup_n B_n\) is isomorphic to
    \[\C\langle g_n,h_n,x_n,y_n \mid g_n^2 = g_n, h_n^2 = h_n, x_n^2 = 0 = y_n^2, x_ng_n + g_nx_n = 0 = y_nh_n + h_ny_n\rangle,\]
    with
    \[\Delta(g_n) = g_n \ot g_n, \ \Delta(x_n) = x_n \ot 1 + g_n \ot x_n, \ \Delta(h_n) = h_n \ot h_n, \ \Delta(y_n) = y_n \ot h_n + 1 \ot y_n.\]
    Suppose that we quotient it by the ideal 
    \[I\coloneqq  \langle b_1S(b_2) - \varepsilon(b)1 \mid b \in B_n, n\geq 0\rangle.\]
    Then, in the quotient,
    \[1 = g_nS(g_n) = g_nh_n \qquad \text{and} \qquad 1 = h_nS(h_n) = h_ng_{n+1}\]
    from which it follows that \(g_n = g_{n+1}\) and \(h_n = h_{n+1}\) for all \(n \in \N\). Moreover,
    \[g_n = g_n^2h_n = g_nh_n = 1 = h_ng_{n+1} = h_n^2g_{n+1} = h_n\]
    for all \(n \in \N\). Concerning the skew-primitive generators, instead,
    \[0 = x_nS(1) + g_nS(x_n) = x_n + y_n \qquad \text{and} \qquad 0 = y_nS(h_n) + 1S(y_n) = y_n + x_{n+1},\]
    for all \(n \in \N\), but, in any case, \(x_ng_n + g_nx_n = 0\) already entails that \(x_n = 0\) and similarly for \(y_n = 0\). From these observations, we conclude that
    \[\bigsqcup_{n \in \N}B_n/I \cong \C\]
    and hence that
    \[\bigsqcup_{n \in \N}B_n\big/\big(I + \langle S(I)\rangle\big) \cong \C.\]
\end{example}

\begin{example}
    Let \(B \coloneqq  \C\langle g,x \mid g^3=g,x^2=0,gx+xg=0\rangle\) with
    \[\Delta(g) = g \ot g, \qquad \Delta(x) = x \ot 1 + g \ot x,\]
    in \(\cM = \Vect_\C\) with the flip \(\tau\). In this case, the coproduct \(\bigsqcup_n B_n\) is isomorphic to
    \[\C\langle g_n,x_n \mid g_n^3 = g_n, x_n^2 = 0, x_ng_n + g_nx_n = 0\rangle,\]
    with
    \begin{gather*}
        \Delta(g_n) = g_n \ot g_n, \qquad \Delta(x_{2n}) = x_{2n} \ot 1 + g_{2n} \ot x_{2n}, \\ 
        \Delta(x_{2n+1}) = x_{2n+1} \ot g_{2n+1} + 1 \ot x_{2n+1}.
    \end{gather*}
    Suppose that we quotient it by the ideal 
    \[I\coloneqq  \langle b_1S(b_2) - \varepsilon(b)1 \mid b \in B_n, n\geq 0\rangle.\]
    Then, in the quotient,
    \[1 = g_nS(g_n) = g_ng_{n+1}\]
    from which it follows that \(g_n = g_{n+2}\) and \(g_{n+1} = g_n^{-1}\) for all \(n \in \N\). Moreover,
    \[1 = g_ng_{n+1} = g_n^3g_{n+1} = g_n^2\]
    for all \(n \in \N\) entails that, in fact, \(g_n = g_{n+1}\) for all \(n \in \N\). Concerning the skew-primitive generators, instead,
    \[0 = x_{2n}S(1) + gS(x_{2n}) = x_{2n} + gx_{2n+1} \qquad \text{and} \qquad 0 = x_{2n+1}g + x_{2n+2},\]
    for all \(n \in \N\), imply that
    \[-gx_{2n+3} = x_{2n+2} = gx_{2n+1} = -x_{2n},\]
    while from
    \[0 = gx_{2n}g + S(x_{2n}) = -x_{2n} + x_{2n+1} \qquad \text{and} \qquad 0 = gx_{2n+1} + gx_{2n+2}g = gx_{2n+1} - x_{2n+2}\]
    we deduce that \(x_{2n} = x_{2n+1}\) for all \(n \in \N\).
    Summing up,
    \[\bigsqcup_{n \in \N}B_n/I \cong H_4 = \C\langle g,x \mid g^2=1,x^2 = 0, gx+xg = 0\rangle\]
    and hence
    \[\bigsqcup_{n \in \N}B_n\big/\big(I + \langle S(I)\rangle\big) \cong H_4.\]
\end{example}

\appendix

\section*{Appendix}

{\paolo
\textbf{Questions/ideas:} 
\begin{enumerate}[leftmargin=*,itemsep=10pt,topsep=10pt]
\item What are the Hopf modules over the Green-Nichols-Taft one-sided Hopf algebra? Can we describe them in some "explicit" way that may make them attractive for the community?

\item Are there nice examples of objects in \( \left({}_B\mathfrak{M}\right)_{H}^{H}\) that we may use to "sell" our result?

\item Up to some verifications, the Frobenius property for \(- \ot H\) shall tell us that any Hopf module decomposes "canonically" as \(M \cong \coinv{M}{H} \oplus MH^+\) (for \(H\) itself, it is always true that \(H \cong \Bbbk 1_H \oplus H^+ = \coinv{H}{H} \oplus H^+\)). Does this carry some interesting information?

\item \(E_\infty = \K\langle g,x_n \mid g^2 = 1, gx_i + x_{i+1} = 0, x_i^2=0\rangle\) with \(S(x_i) = x_{i+1}\), \(S(g) = g\),
\[x_1 \mapsto x_1 \ot 1 + g \ot x_1 \mapsto S(x_1) + gx_1 = x_2 - x_2 = 0\]
\[x_1 \mapsto x_1 \ot 1 + g \ot x_1 \mapsto x_1 + gS(x_1) = x_1 + gx_2 = 0\]









\end{enumerate}
}

\end{invisible}

{\small
\subsection*{Acknowledgements} This paper was written while the authors were members of the ``National Group for Algebraic and Geometric Structures and their Applications'' (GNSAGA-INdAM). This work was partially supported by the project funded by the European Union - NextGenerationEU under NRRP, Mission 4 Component 2 CUP D53D23005960006 - Call PRIN 2022 No.\@ 104 of February 2, 2022 of Italian Ministry of University and Research; Project 2022S97PMY \textit{Structures for Quivers, Algebras and Representations (SQUARE).} 
L.B.\@ was partially supported by a postdoctoral fellowship at the Department of Mathematics of the Universit\'e Libre de Bruxelles, within the framework of the ARC project ``From algebra to combinatorics, and back''.
D.F.\@  was funded by the University of Torino though a PNRR DM 118 scholarship; by the Vrije Universiteit Brussel through the bench fee OZR3762; and by Leandro Vendramin through his FWO Senior Research Project G004124N.  
P.S.\@ was supported by the project ``HADRA - Hopf Algebroids: Duality, Representations and Applications'' funded by the Italian Ministero dell'Universit\`a e della Ricerca, grant number PGR21S8QB2, and was partially supported by the project PID2024-157173NB-I00 funded by MCIN/AEI/10.13039/501100011033 and by FEDER, UE. 
Esta publicaci\'on ha sido parcialmente financiada con fondos propios de la Junta de Andaluc\'ia, en el marco de la ayuda DGP\_EMEC\_2023\_00216.
}

\bibliography{references}
\bibliographystyle{acm}

\vspace{8pt}

\end{document}